\documentclass[10pt, a4paper]{amsart}

\usepackage{amssymb,amsfonts}
\usepackage[all,arc]{xy}
\usepackage{enumerate}
\usepackage{mathrsfs}
\usepackage{tikz-cd}
\usepackage{enumitem}
\usepackage{empheq}
\usepackage{kotex}
\usepackage{ifpdf}
\usepackage{mathtools}
\usepackage{pst-node, pst-plot, pstricks}
\usepackage{mathrsfs}
\usepackage{amsmath}
\usepackage{epsfig}
\usepackage{amscd}
\usepackage{graphicx, color}
\usepackage{soul}
\graphicspath{ {./images/} }

\def\E{\ifmmode{\mathbb E}\else{$\mathbb E$}\fi} 
\def\N{\ifmmode{\mathbb N}\else{$\mathbb N$}\fi} 
\def\R{\ifmmode{\mathbb R}\else{$\mathbb R$}\fi} 
\def\Q{\ifmmode{\mathbb Q}\else{$\mathbb Q$}\fi} 
\def\C{\ifmmode{\mathbb C}\else{$\mathbb C$}\fi} 
\def\H{\ifmmode{\mathbb H}\else{$\mathbb H$}\fi} 
\def\Z{\ifmmode{\mathbb Z}\else{$\mathbb Z$}\fi} 
\def\P{\ifmmode{\mathbb P}\else{$\mathbb P$}\fi} 
\def\T{\ifmmode{\mathbb T}\else{$\mathbb T$}\fi} 
\def\SS{\ifmmode{\mathbb S}\else{$\mathbb S$}\fi} 
\def\DD{\ifmmode{\mathbb D}\else{$\mathbb D$}\fi} 

\DeclareSymbolFont{yhlargesymbols}{OMX}{yhex}{m}{n}

\DeclareMathAccent{\widetriangle}{\mathord}{yhlargesymbols}{"E6}

\def\E{\ifmmode{\mathbb E}\else{$\mathbb E$}\fi} 
\def\N{\ifmmode{\mathbb N}\else{$\mathbb N$}\fi} 
\def\R{\ifmmode{\mathbb R}\else{$\mathbb R$}\fi} 
\def\Q{\ifmmode{\mathbb Q}\else{$\mathbb Q$}\fi} 
\def\C{\ifmmode{\mathbb C}\else{$\mathbb C$}\fi} 
\def\H{\ifmmode{\mathbb H}\else{$\mathbb H$}\fi} 
\def\Z{\ifmmode{\mathbb Z}\else{$\mathbb Z$}\fi} 
\def\P{\ifmmode{\mathbb P}\else{$\mathbb P$}\fi} 
\def\T{\ifmmode{\mathbb T}\else{$\mathbb T$}\fi} 
\def\SS{\ifmmode{\mathbb S}\else{$\mathbb S$}\fi} 
\def\DD{\ifmmode{\mathbb D}\else{$\mathbb D$}\fi} 

\newcommand{\ben}{\begin{enumerate}}
\newcommand{\een}{\end{enumerate}}
\newcommand{\be}{\begin{equation}}
\newcommand{\ee}{\end{equation}}
\newcommand{\bea}{\begin{eqnarray}}
\newcommand{\eea}{\end{eqnarray}}
\newcommand{\beastar}{\begin{eqnarray*}}
\newcommand{\eeastar}{\end{eqnarray*}}
\newcommand{\bc}{\begin{center}}
\newcommand{\ec}{\end{center}}

\theoremstyle{theorem}
\newtheorem{thm}{Theorem}[section]
\newtheorem{cor}[thm]{Corollary}
\newtheorem{lem}[thm]{Lemma}
\newtheorem{prop}[thm]{Proposition}
\newtheorem{'thm'}[thm]{'Theorem'}

\theoremstyle{definition}
\newtheorem{defn}[thm]{Definition}
\newtheorem{rem}[thm]{Remark}

\newtheorem{exam}[thm]{Example}

\newtheorem{proof-sketch}[thm]{Proof-Sketch}

\newtheorem{lem-defn}[thm]{Lemma-Definition}
\newtheorem{prop-defn}[thm]{Proposition-Definition}
\newtheorem{thm-defn}[thm]{Theorem-Definition}
\newtheorem{exam-defn}[thm]{Example-Definition}

\newtheorem{notation}[thm]{\rm\bfseries{Notation}}

\newtheorem*{thm*}{Theorem}

\numberwithin{equation}{section}

\def\R{{\mathbb R}}

\def\E{{\mathbb E}}
\def\Z{{\mathbb Z}}
\def\C{{\mathbb C}}
\def\R{{\mathbb R}}
\def\P{{\mathbb P}}

\def\N{{\mathbb N}}

\def\11{{\mathbb I}}

\def\sgn{{\text{\rm sgn}}}

\def\C{\mathbb{C}}
\def\Z{\mathbb{Z}}

\def\T{\mathbb{T}}

\def\Q{\mathbb{Q}}

\def\E{\ifmmode{\mathbb E}\else{$\mathbb E$}\fi} 
\def\N{\ifmmode{\mathbb N}\else{$\mathbb N$}\fi} 
\def\R{\ifmmode{\mathbb R}\else{$\mathbb R$}\fi} 
\def\Q{\ifmmode{\mathbb Q}\else{$\mathbb Q$}\fi} 
\def\C{\ifmmode{\mathbb C}\else{$\mathbb C$}\fi} 
\def\H{\ifmmode{\mathbb H}\else{$\mathbb H$}\fi} 
\def\Z{\ifmmode{\mathbb Z}\else{$\mathbb Z$}\fi} 
\def\P{\ifmmode{\mathbb P}\else{$\mathbb P$}\fi} 
\def\SS{\ifmmode{\mathbb S}\else{$\mathbb S$}\fi} 
\def\DD{\ifmmode{\mathbb D}\else{$\mathbb D$}\fi} 

\def\R{{\mathbb R}}

\def\E{{\mathbb E}}
\def\Z{{\mathbb Z}}
\def\C{{\mathbb C}}
\def\R{{\mathbb R}}

\def\N{{\mathbb N}}

\def\darr#1{\raise1.5ex\hbox{$\leftrightarrow$}
\mkern-16.5mu #1}

\def\roughly#1{\raise.3ex\hbox{$#1$\kern-.75em
\lower1ex\hbox{$\sim$}}}

\def\opname#1{\mathop{\kern0pt{\rm #1}}\nolimits}

\def\dim{\opname{dim}}

\def\Incl{\operatorname{Incl}}
\def\Eval{\operatorname{Eval}}

\begin{document}

\quad \vskip1.375truein

\bibliographystyle{plain}

\title[Kuranishi spaces in a category of fibrant objects]
{Kuranishi spaces in a category of fibrant objects and their homotopy fiber products}

\author{Taesu Kim}
\address{Department of Mathematics Education, Silla University, 140 Baegyang-daero, Sasang-Gu, Busan, Korea 46958}
\email{kimtaesu@silla.ac.kr}



\maketitle

\begin{abstract}
We study several fundamental properties of the category $\mathbf{Kur}$ introduced in \cite{Kim1}. By defining fibrations and weak equivalences, we show that morphisms with manifold targets admit homotopy fiber products; this is then applied to obtain a notion of generalized intersections of two submanifolds. A  definition of the Weinstein category in the context of Kuranishi spaces is also presented. Finally, we establish that $\mathbf{Kur}$ forms a category of fibrant objects in the sense of Brown \cite{Brown}, and discuss its homotopy category $\operatorname{Ho}(\mathbf{Kur})$.
\end{abstract}

\tableofcontents

\section{Introduction}

The study of Kuranishi spaces plays a central role in modern symplectic geometry, providing a framework for constructing virtual fundamental classes when transversality fails. However, their lack of canonical and categorical structures has long been pointed out. To address this issue, \cite{Kim1} proposed a categorical approach by introducing the category $\mathbf{Kur}$ of a modified version of Kuranishi spaces\footnote{Throughout this paper, for simplicity of presentation, we use the term \textit{Kuranishi space} to mean the $L_{\infty}$-\textit{Kuranishi space} defined in \cite{Kim1}.} equipped with $L_{\infty}[1]$-algebraic structures. It is also shown in \cite{Kim1} that smooth manifolds embed into $\mathbf{Kur}$ as Kuranishi spaces with zero obstruction bundles:
\begin{equation}\label{bhbnn}
\mathscr{F} : \mathbf{Man} \hookrightarrow \mathbf{Kur}.
\end{equation}

In this paper, we investigate the homotopy-theoretic aspects of $\mathbf{Kur}$ and establish two foundational results: first, it admits homotopy fiber products for appropriately defined fibrations; second, it carries the structure of a \textit{category of fibrant objects} in the sense of Brown \cite{Brown}. Specifically, we formulate the notions of fibrations and weak equivalences by building upon the technical results of \cite{Kim1}. As a consequence, we can explicitly describe the homotopy category $\mathrm{Ho}(\mathbf{Kur})$.

Recall that morphisms in $\mathbf{Kur}$ are defined by taking the quotient of \textit{pre-morphisms} by an equivalence relation. A technical advantage of this quotient construction is that the surjectivity and submersivity of the base maps can always be achieved by replacing a representative with one of sufficiently high dimension and extending it suitably within its equivalence class. This observation leads to several key applications, including proper definitions of fibrations and weak equivalences that satisfy the axioms for a category of fibrant objects.

We then illustrate a typical application of this homotopical language: distinct presentations of the same geometric object can sometimes be identified via weak equivalences, thereby defining isomorphic objects in the homotopy category. To see this, let $\widehat{\mathcal{U}}^{\mathrm{man}}$ denote the atlas consisting of the Kuranishi charts from the image of $\mathscr{F}$ in (\ref{bhbnn}). At the same time, any smooth manifold $M$ can be presented as the underlying space of an atlas, say $\widehat{\mathcal{U}}^{\mathrm{tr}},$ with transversal Kuranishi sections. We then obtain a canonical weak equivalence
\begin{equation}\label{gshg}
\left(M, \left[\widehat{\mathcal{U}}^{\mathrm{man}}\right]\right) \xrightarrow{\exists \mathrm{w. e.}} \left(M, \left[\widehat{\mathcal{U}}^{\mathrm{tr}}\right]\right),
\end{equation}
where $[\,\cdots\,]$ denotes the equivalence class of an atlas (cf. Definition \ref{eqats}):

\subsection*{Fibrations and weak equivalences}

\ \

Before proceeding to the statements of the main results, we first recall some preliminary notions, which are essentially reformulations of results from \cite{Kim1}.

Let $E \to U$ and $E' \to U'$ be vector bundles with sections $s$ and $s'$, respectively. Let $\widetilde{\phi} \colon E \to E'$ be a bundle embedding covering a smooth map $\phi \colon U \to U'$ and satisfying
\[
   \pi_{\mathrm{Im} \widetilde{\phi}}(s' \circ \phi) = \widetilde{\phi} \circ s,
\]
where $\pi_{\operatorname{Im} \widetilde{\phi}}(s' \circ \phi)$ denotes the projection of $s' \circ \phi$ onto the component $\operatorname{Im} \widetilde{\phi}$ in $\phi^* E'$.

In \cite{Kim1}, we assigned an $L_{\infty}[1]$-algebra and an open neighborhood $W_x \subset U$ to each zero $x$ of a chart $(U,E,s,\Gamma,\psi)$. As a component, this algebra contains $\bigwedge\nolimits^{-\bullet}\Gamma(E^*\vert{}_{W_x})$, the Koszul complex whose differential is given by the signed contraction $\iota_{s\vert{}_{W_x}}$ with trivial higher operations.

Then we have:
\begin{thm}\label{bwggg}
\begin{enumerate}[label = (\roman*)]
\item The bundle embedding $\widetilde{\phi}$ induces a chain map, and hence an $L_\infty[1]$-morphism
\[
\widehat{\eta}(\widetilde{\phi})_x : \bigwedge\nolimits^{-\bullet}\Gamma\left(E^*|_{W_x}\right) \rightarrow \bigwedge\nolimits^{-\bullet}\Gamma\left(E^{'*}|_{W'_{\phi(x)}}\right).
\]
\item Suppose, in addition, that the data in (i) satisfies the compatibility condition for the sections
\[
    s' \circ \phi = \widetilde{\phi} \circ s,
\]
and the tangent bundle condition; that is $ds'$ induces an isomorphism
\begin{equation}\label{tancond}
\left[ds'_{\phi(x)}\right] : \frac{T_{\phi(x)}U'}{\phi_*(T_xU)} \xrightarrow{\simeq} \frac{E'_{\phi(x)}}{\widetilde{\phi}(E_{x})},
\end{equation}
at each $x \in s^{-1}(0).$ Then $\widehat{\eta}(\widetilde{\phi})_x$ is a quasi-isomorphism.
\end{enumerate}
\end{thm}

Building on the preceding theorem, we introduce the notions of \textit{fibrations} and \textit{weak equivalences} in $\operatorname{Mor}(\mathbf{Kur})$ (see Definition \ref{adfaf} for their precise definitions).
\begin{defn}

Let $F \colon \mathfrak{X} \to \mathfrak{X}'$ be a morphism of Kuranishi spaces represented by a pre-morphism
\[
\overline{F} = \left(\widehat{\mathcal{U}}, \widehat{\mathcal{U}}', f, \left\{f_{p}\right\}, \left\{\widehat{f}_{p,x}\right\}\right).
\] We say:

\begin{enumerate}[label = (\alph*)]
\item $F$ is a \textit{fibration} if there exists a representative pre-morphism $\overline{F}$ of $F$ satisfying the following properties:

\begin{enumerate}[label = (\roman*)]

    \item $\widehat{\mathcal{U}}$ consists of bases of minimal dimension among all representing atlases for the equivalence class $[\widehat{\mathcal{U}}]$.
    \item For each $p$, the base map $f_p$ is surjective and submersive.
    \item The Koszul $L_\infty$-component map is given by the composition
\[
\widehat{\eta}(\widetilde{\phi}_p^{(\mathrm{i})}) \circ \widehat{\eta}(\widetilde{\phi}_p^{(\mathrm{ii})})^{-1},
\] 
where $\widehat{\eta}(\widetilde{\phi}_p^{(\mathrm{i})})$ and $\widehat{\eta}(\widetilde{\phi}_p^{(\mathrm{ii})})$ are the $L_\infty[1]$-morphisms arising from cases (i) and (ii) of Theorem \ref{bwggg}, respectively, and $\widehat{\eta}(\widetilde{\phi}_p^{(\mathrm{ii})})^{-1}$ denotes a homotopy inverse of the quasi-isomorphism $\widehat{\eta}(\widetilde{\phi}_p^{(\mathrm{ii})})$.

    \item Consider a decomposition of $E_p$ into subbundles
\begin{equation}\label{jbwwsxxx}
E_p = E_p^{\mathrm{i}} \oplus {E}^{\mathrm{c}}_p := \pi_p^{*}{\mathrm{Im}}_p\left(\widetilde{\phi}_p^{(\mathrm{i})}\right) \oplus {E}^{\mathrm{c}}_p
\end{equation}
for each $p$. Here, $\pi_p \colon U_p \twoheadrightarrow \operatorname{Im}_p(\phi_p^{(\mathrm{i})})$ is a chosen projection between contractible open sets. We then require that the Koszul $L_{\infty}$-coordinate change for $\widehat{\mathcal{U}}'$ decompose into those on the Koszul complexes of $E_p^{\mathrm{i}}$ and $E^{\mathrm{c}}_p$.

     \item The zero loci satisfy
\[
[s_p^{-1}(0)] \cap f_{p}^{-1}((s_{f(p)}^{-1}(0)) = s_p^{-1}(0)
\]
for each $p$, where $[(\cdots)]$ denotes the image of $(\cdots)$ at $p \in X$ in the quotient bundle $E'_{f(p)} / \pi_p^{*}\!\left(\operatorname{Im}(\widetilde{\phi}_p^{(\mathrm{i})})\right)$.
\end{enumerate}

     \item $F$ is a \textit{weak equivalence} if there exists a representative pre-morphism $\overline{F}$ of $F$ satisfying the following properties:

\begin{enumerate}[label = (\roman*)]
\item The induced map $f$ is a homeomorphism.
\item $\widehat{f}_{p,x}$ is a quasi-isomorphism for all $p \in X$ and $x \in s_p^{-1}(0)$.
\end{enumerate}
\end{enumerate}

We denote by $\mathrm{Fib}$ and $\mathrm{W}$ the classes of fibrations and weak equivalences, respectively.

\end{defn}

\subsection*{Main results}

\ \

The first main theorem of this paper is the existence of the weak equivalence (\ref{gshg}):

\begin{thm}\label{gsjjagmhmh}
There exists a weak equivalence from any smooth manifold to its presentation as a
transversal zero locus in the Kuranishi setting.
\end{thm}

The second main result concerns the existence of fiber products in $\mathbf{Kur}$. In Section~4, we construct an explicit model for the fiber product of two morphisms in $\mathbf{Kur}$, which yields:

\begin{thm}
The fiber product of a pair of morphisms exists whenever at least one of them is a fibration.
\end{thm}

Even without the fibration assumption, fiber products may still exist. In fact, any morphism whose target is a smooth manifold can be made to satisfy the axioms of a fibration upon choosing an atlas with bases of sufficiently large dimension. In other words, we have the following result:

\begin{thm}\label{twomdoku}
Any two morphisms in $\mathbf{Kur}$ with a common manifold target admit a homotopy fiber product.
\end{thm}

This description of fiber products (and of direct products as a special case) is expected to contribute to a better understanding of moduli spaces in symplectic geometry, whose boundaries frequently arise as naive fiber products of lower-dimensional moduli. While various approaches exist in the literature, none appears to provide an entirely satisfying or well-behaved machinery for this purpose. In this regard, we remark that the moduli space of Morse trajectories can be studied within our proposed setting, producing algebraic structures such as Morse chain complexes in a canonical manner \cite{Kim2}.

On the other hand, the term \textit{homotopy} fiber product in the preceding paragraph emphasizes that its construction inevitably involves \textit{noncanonical} choices, such as representatives of morphisms; nevertheless, the resulting structure is well-defined \textit{up to weak equivalence}. Consequently, these constructions determine the ordinary fiber product in the homotopy category obtained by formally inverting weak equivalences.

With $\operatorname{Fib}$ and $\mathrm{W}$ as defined above, we establish the following:

\begin{thm}\label{hhba}
The triple $(\mathbf{Kur}, \operatorname{Fib}, \mathrm{W})$ forms a category of fibrant objects.
\end{thm}
\noindent (See Definition~\ref{njjs} for the axioms satisfied by $(\mathbf{Kur}, \operatorname{Fib}, \mathrm{W})$.)

By virtue of Theorem~\ref{hhba}, we obtain an explicit construction of the homotopy category $\operatorname{Ho}(\mathbf{Kur})$, in which weak equivalences (such as the one in Theorem~\ref{gshg}) identify any smooth manifold with its transversal Kuranishi presentations up to isomorphism.

An immediate corollary of Theorem~\ref{twomdoku} is then the following:

\begin{cor}
Any two morphisms in $\operatorname{Ho}(\mathbf{Kur})$ with a common smooth manifold target admit a fiber product in the ordinary sense.
\end{cor}

\subsection*{Geometric applications}

\ \

The aforementioned results lead to interesting geometric applications as well. In particular, the constructions in this paper prove useful in at least three contexts: (i) generalized intersections of smooth submanifolds, (ii) boundary descriptions of the moduli space of pseudoholomorphic discs, and (iii) the formulation of the Weinstein category.

\ \

\noindent (i) (Generalized intersections) 
Given two submanifolds $\iota_i \colon M_i \hookrightarrow M$ ($i = 1, 2$) of a compact smooth manifold $M$, we define their \textit{generalized intersection} by
\[
(M_1, \iota_1) \cap (M_2, \iota_2) := M_1 \,{}_{\iota_1}\!\times_{\iota_2} M_2
\]
as an object of $\mathbf{Kur}$, which extends the classical intersection of submanifolds to a general setting.

\ \

\noindent (ii) (The moduli space of pseudoholomorphic disks)
Consider the Kuranishi space structure on the moduli space $\mathcal{M}_k(L; \beta)$ (and the family in $k$) discussed in \cite{Kim3} and Example~\ref{moduli} of this paper, equipped with the evaluation morphisms $\mathrm{Ev}_i \colon \mathcal{M}_k(L; \beta) \to L$. Since their common target $L$ is a smooth (Lagrangian) manifold, the \textit{homotopy} fiber product
\[
\partial \mathcal{M}_k(L; \beta) = \bigcup\limits_{\substack{k = k_1 + k_2,\\ \beta = \beta_1 + \beta_2}} \mathcal{M}_{k_1}(L; \beta_1) {}_{\mathrm{Ev}_1} \!\times_{\mathrm{Ev}_2} \mathcal{M}_{k_2}(L; \beta_2),
\]
is well-defined as an object in $\mathbf{Kur}$.

\ \

\noindent (iii) (Weinsten category)
The existence of fiber products in $\operatorname{Ho}(\mathbf{Kur})$ allows us to define the \textit{Weinstein category}, denoted by $\mathbf{Wein}$, without any transversality assumptions: its objects are symplectic manifolds naturally regarded as Kuranishi spaces, and the morphisms from $M$ to $M'$ are \textit{Lagrangian} Kuranishi correspondences.
\[
\mathrm{Mor}_{\mathbf{Wein}}\left(\mathfrak{X}, \mathfrak{X}'\right) := \bigl\{ \mathfrak{L} \overset{\iota}{\hookrightarrow} \mathfrak{X} \times \mathfrak{X}', \text{ Lagrangian} \bigr\},
\]
with composition given by fiber products and identities by diagonal Lagrangians. (See Example-Definition~\ref{bjsjbhb} for a more precise statement.)

\subsection*{Outline of the paper} 

\ \

The paper is organized as follows. In Section~2, we recall the necessary background on Kuranishi spaces and fix our conventions. Section~3 introduces two distinguished classes of morphisms in the category $\mathbf{Kur}$, namely fibrations and weak equivalences. Section~4 provides an explicit model of homotopy fiber products for fibrations, alongside several geometric applications. In Section~5, we develop the homotopy theory of $\mathbf{Kur}$ by establishing the factorization axiom for a category of fibrant objects. Finally, Section~6 concludes with the construction of the homotopy category of $\mathbf{Kur}$. Appendices~A through~C collect related supplementary materials from earlier papers by the author.

\section{Kuranishi spaces and their category}

In this section, we recall the definition of Kuranishi spaces in \cite{Kim1} and show that their collection form a category where the category of smooth manifolds embeds.

\subsection{Kuranishi charts and atlases}

We define Kuranishi charts by associating $L_{\infty}[1]$-algebras to presymplectic neighborhoods of the zero locus of the Kuranishi section.

\begin{defn}[$L_\infty$-Kuranishi charts]\label{kurdef}
Let $X$ be a compact metrizable space. An $L_\infty$-\textit{Kuranishi chart} of $X$ is given by a tuple
\begin{equation}\nonumber
\mathcal{U} = (U, E, s, \Gamma, \psi),
\end{equation}
where
\begin{enumerate}
\item[--] $U = (U, \beta)$ is a pair of a smooth manifold with a closed two-form $\beta \in Z^2(U).$
\item[--] $\pi : E \rightarrow U$ is a (finite rank) vector bundle.
\item[--] $s : U \rightarrow E$ is a smooth section.
\item[--] $\Gamma$ is a finite group acting on $U$ that restricts to the zero set of $s$, that is, $\Gamma \cdot s^{-1}(0) \subset s^{-1}(0).$
\item[--] $\psi : {s^{-1}(0)}/{\Gamma} \overset{\simeq}{\hookrightarrow} X$ is a homeomorphism onto the image.
\end{enumerate}
The \textit{dimension} of $\mathcal{U}$ is defined by $\dim \mathcal{U} := \dim U - \text{rk}E.$

We require that the chart  $\mathcal{U}$ be endowed with the following structures:
\begin{itemize}
\item[--] $U$ has a decomposition 
\begin{equation}\label{wstri}
U = \bigcup\limits_i \mathcal{S}_i,
\end{equation}
into (possibly nonconnected) submanifolds,
\[
\mathcal{S}_i := \{x \in U \mid \text{rk} (\ker\beta_x) =i \}, \ 0 \leq i \leq \dim U,
\]
together with their tubular neighborhoods for each $i$.

\item[--] To each zero point $x \in s^{-1}(0),$ we assign:
\begin{enumerate}[label = (\roman*)]
\item A presymplectic open neighborhood $W_x$ of $x$ in $U$ with $W_x \simeq B^n,$
\item A local $ L_{\infty}[1]$-algebra $\mathcal{C}_x$
\[
\mathcal{C}_x := \overbrace{\bigwedge\nolimits^{-\bullet}\Gamma(E^*|_{W_x})}^{\text{Koszul}} \oplus \overbrace{\Omega^{\bullet + 1}_{\mathrm{aug}}(\mathcal{F}_x)}^{\text{de Rham}},
\]
\end{enumerate}
whose detailed descriptions are provided in Appendix B.
\end{itemize}
\end{defn}

We  now define morphisms of Kuranishi charts, with embeddings as a special case, which yields the construction of a global object called a Kuranishi atlas.

Let $f : X \rightarrow X'$ be a continuous map between compact topological spaces. 

\begin{defn}[Chart morphisms]
A \textit{chart morphism} between Kuranishi charts ${\mathcal{U}} = (U,E,\Gamma,s,\psi)$ and ${\mathcal{U}}' = (U',E',\Gamma',s',\psi')$ of $X$ and $X',$ respectively, is defined by a pair
\[
\Phi = \left(\phi, \widehat{\phi}\right) : {\mathcal{U}} \rightarrow {\mathcal{U}}' ,
\]
given by:
\begin{enumerate}
\item[--] $\phi : U \rightarrow U', \text{ a } (\Gamma, \Gamma')\text{-equivariant map of manifolds,}$
\item[--] $\widehat{\phi} = \left\{\widehat{\phi}_x :  \mathcal{C}'_{\phi(x)} \rightarrow \mathcal{C}_x \right\}_{x \in s^{-1}(0)}$ is a family of $L_{\infty}[1]$-morphisms,
\end{enumerate}
satisfying 
\begin{enumerate}[label = (\roman*)]
\item $\psi' \circ \phi = f \circ \psi$ on $s^{-1}(0),$
\item $\phi(W_x) \subset W'_{\phi(W_x)},$
\item $\widehat{\phi}_x$ factors through $\mathcal{C}'_{\phi(x), \phi},$ that is, we have $\widehat{\phi}_x = \widehat{\phi}_x^{\mathrm{c}} \circ \widehat{\varepsilon}_{\phi(x), \phi}$ for some $L_{\infty}[1]$-morphism, $\widehat{\phi}^{\mathrm{c}}_x :  \mathcal{C}'_{\phi(x), \phi} \rightarrow \mathcal{C}_x.$
\end{enumerate}
Here, $\mathcal{C}_{\phi(x),\phi}'$ stands for \textit{the completion} of $\mathcal{C}'_{\phi(x)}$ \textit{at the image of} $\phi$, induced from the \textit{completed} V-algebra at the image (see Definition\ref{ladef} for details). We remark that for surjective $\phi$, we have an isomorphism $\mathcal{C}'_{\phi(x), \phi} \simeq \mathcal{C}'_{\phi(x)}$ (cf.\ Lemma \ref{surjcp}).

On the other hand, $\widehat{\varepsilon}_{\phi(x), \phi} : \mathcal{C}'_{\phi(x)} \rightarrow \mathcal{C}'_{\phi(x), \phi}$ is the $L_{\infty}[1]$-morphism given by considering elements in $\mathcal{C}'_{\phi(x)}$ naturally as those in $\mathcal{C}'_{\phi(x), \phi}$, which can be easily shown to define an $L_{\infty}[1]$-morphism. (See Definition \ref{varep} and Lemma \ref{vecpf}.)
\end{defn}

\begin{defn}[Embedding of charts]\label{ourcemb}
A chart morphism $\left(\phi, \widehat{\phi}\right)$ is called an \textit{embedding} if 
\begin{enumerate}
\item[--]$\phi$ is an (equivariant) embedding of manifolds, 
\item[--] $\widehat{\phi}^{\mathrm{c}}_{x} : \mathcal{C}'_{\phi(x), \phi} \rightarrow \mathcal{C}_x$ is a quasi-isomorphism of $L_{\infty}[1]$-algebras for each $x.$
\end{enumerate}
\end{defn}

A \textit{coordinate change} of Kuranishi charts provides a main example of embedding:

\begin{defn}[Coordinate changes]\label{kstr}
For two points $p, q \in X$ with $\mathrm{Im}\psi_p \cap \mathrm{Im}\psi_q \neq \emptyset,$ we define a \textit{coordinate change} $\Phi_{pq} : \mathcal{U}_p \rightarrow \mathcal{U}_q$ by a tuple
\[
\Phi_{pq} := \left(U_{pq}, \phi_{pq}, \left\{\widehat{\phi}_{pq,x}\right\}\right),
\]
where $U_{pq} \subset U_p$ is an open submanifold, and
\[
\left(\phi_{pq}, \widehat{\phi}_{pq}\right) : \mathcal{U}_p|_{U_{pq}} \rightarrow \mathcal{U}_q
\]
is an embedding of Kuranishi charts from $\mathcal{U}_p|_{U_{pq}},$ that is, the chart restricted to $U_{pq}.$ They are required to satisfy:

\begin{enumerate}[label = (\roman*)]
\item $\Phi_{pp} = \mathrm{id}_{\mathcal{U}_p},$
\item $\psi_q \circ \phi_{pq} = \psi_p$ on $s_{p}^{-1}(0) \cap U_{pq},$
\item $\phi_{qr} \circ \phi_{pq} = \phi_{pr}$ on $s_p^{-1}(0) \cap \phi_{qr}^{-1}(U_{pq}) \cap U_{pr},$
\item $\psi_p\big(s_p^{-1}(0) \cap U_{pq}\big) =\mathrm{Im} \psi_p \cap\mathrm{Im} \psi_q,$
\end{enumerate}
\end{defn}

\begin{rem}
Here the cocycle condition is imposed {on the base maps alone,} and not on the $L_{\infty}[1]$-component. The reason for this is that $L_{\infty}$-compatibilities always hold. We refer the reader to \cite{Kim3} for more detailed discussion on this issue.
\end{rem}

\begin{defn}
A pair of the compact topological space $X$ and a collection of Kuranishi charts with coordinate changes
\[
\left(X, \widehat{\mathcal{U}}\right),
\]
where $ \widehat{\mathcal{U}}=\left(\left\{\widehat{\mathcal{U}}_p\right\}, \left\{\Phi_{pq}\right\}\right),$ is called a \textit{Kuranishi atlas}. For technical reasons, we assume that $\max\limits_{p \in X} \dim U_p < \infty$ with the compactness of $X$.
\end{defn}

\subsection{Kuranishi spaces and their categorical structures}

By considering equivalence relations on atlases, we define Kuranishi spaces. We also define morphisms between them, obtaining the category of Kuranishi spaces.

Two atlases are said to be equivalent, denoted $\left(X, \widehat{\mathcal{U}}\right) \sim \left(X, \widehat{\mathcal{U}}'\right)$, or simply $\widehat{\mathcal{U}} \sim \widehat{\mathcal{U}}'$, if we have an \textit{equality between the expanded atlases}
\[
\widehat{\mathcal{U}}^0_1 \times \mathbb{R}^{m_1} = \widehat{\mathcal{U}}^0_2 \times \mathbb{R}^{m_2},
\]
for some $m_1, m_2 \geq 0$, where $\widehat{\mathcal{U}}^0_i$ is an open subcharts of $\widehat{\mathcal{U}}_i, \ i = 0,1.$ Since we are mostly interested in the abstract categorical structure, we omit the full details regarding them. Their precise definitions can be found in \cite[Definition 3.7]{Kim1}. By \cite[Lemma 3.8 (iii)]{Kim1}, this defines an equivalence relation.

\begin{defn}[Kuranishi spaces]
We define a \textit{Kuranishi space} to be an equivalence class with respect to the relation $\sim$
\[
\mathfrak{X} := \left(X, \left[\widehat{\mathcal{U}}\right]\right).
\]
\end{defn}

Let $\mathfrak{X} = \left(X, \left[\widehat{\underline{\mathcal{U}}}\right]\right)$ and $\mathfrak{Y} = \left(Y,\left[\underline{\widehat{\mathcal{U}}}'\right]\right)$ be Kuranishi spaces.
We consider two atlases $\left(X, \widehat{\mathcal{U}}\right)$ and $\left(X', \widehat{\mathcal{U}}'\right)$ such that $\widehat{\mathcal{U}} \sim \underline{\widehat{\mathcal{U}}}$ and $\widehat{\mathcal{U}}' \sim \underline{\widehat{\mathcal{U}}}'$ with $\widehat{\mathcal{U}} = \left(\{\widehat{\mathcal{U}}_p\}, \left\{\Phi_{pq}\right\}\right) $ and $\widehat{\mathcal{U}}' = \left(\{\widehat{\mathcal{U}}'_{p'}\}, \left\{\Phi'_{p'q'}\right\}\right).$

\begin{defn}[Pre-morphisms]\label{morphismkur}
A \textit{pre-morphism} is defined by the following tuple
\[
\overline{F} := \left(\widehat{\mathcal{U}}, \widehat{\mathcal{U}}', f, \{f_p\}, \left\{\widehat{f}_{p,x}\right\}\right).
\]
$f : X \rightarrow Y$ is a continuous map between the zero loci, while $\left\{\left(f_p, \left\{\widehat{f}_{p,x}\right\}\right)\right\}$ is a collection of chart morphisms. Then $\overline{F}$ is required to satisfy the following compatibilities with respect to the coordinate change $\Phi_{pq}=\left(\phi_{pq}, \left\{\widehat{\phi}_{pq,x}\right\}\right)$:
For $p,q \in X$ with Im$\psi_p \cap\mathrm{Im} \psi_q \neq \emptyset,$ we require
\begin{enumerate}[label = (\roman*)]
\item $\phi'_{f(p)f(q)} \circ f_p = f_q \circ \phi_{pq}$ on $s_{p}^{-1}(0) \cap U_{pq},$
\item $\widehat{\phi}_{pq,x} \circ \widehat{f}_{q, \phi_{pq}(x)} = \widehat{f}_{p,x} \circ \widehat{\phi}'_{f(p)f(q),f_p(x)}$ for each $x \in s_p^{-1}(0) \cap U_{pq}$ up to $L_{\infty}[1]$-homotopy.
\end{enumerate}
\end{defn}

We can define an equivalence relation on the set of pre-morphisms as in \cite[Definition 3.12]{Kim1} and \cite[Lemma 3.13]{Kim1}. 

\begin{defn}[Morphism of Kuranishi spaces]
We define a \textit{morphism} from $\mathfrak{X} = \left(X, \left[\widehat{\mathcal{U}}\right]\right)$ to $\mathfrak{X}' = \left(X', \left[\widehat{\mathcal{U}}'\right]\right)$
by an equivalence class of a pre-morphism $\overline{F}$ from $\mathfrak{X}$ to $\mathfrak{X}':$
\[F := \left[\overline{F}\right] : \mathfrak{X} \rightarrow \mathfrak{X}'.\]
\end{defn}

\begin{defn}[Composition of morphisms]
Let $\mathfrak{X} = \left(X, \left[\widehat{\mathcal{U}}\right]\right),$ $\mathfrak{X}' = \left(X', \left[\widehat{\mathcal{U}'}\right]\right),$ and $\mathfrak{X}'' = \left(X'', \left[\widehat{\mathcal{U}''}\right]\right)$ be Kuranishi spaces. Let $F : \mathfrak{X} \rightarrow \mathfrak{X}'$ and $G : \mathfrak{X}' \rightarrow \mathfrak{X}''$ be morphisms between them represented by 
\begin{equation}\nonumber
\begin{cases}
\overline{F} = \left(\widehat{\mathcal{U}}, \widehat{\mathcal{U}'}, f, \left\{f_{p}\right\}, \left\{\widehat{f}_{p,x}\right\}\right),\\
\overline{G} = \left(\underline{\widehat{\mathcal{U}'}}, \widehat{\mathcal{U}''}, g, \left\{g_{f(p)}\right\}, \left\{\widehat{g}_{f(p),y}\right\}\right),
\end{cases}
\end{equation}
respectively with $\left[ \widehat{\mathcal{U}'}\right] = \left[\underline{\widehat{\mathcal{U}'}}\right]$.

In fact, we may assume that ${\widehat{\mathcal{U}'}} = \underline{\widehat{\mathcal{U}'}}$; if not, we take the \textit{extension} of the pre-morphisms having equivalent charts as components. (See \cite[(3.6) \& Definition 3.15]{Kim1} for its precise definition). We then define the \textit{composition} $G \circ F$ to be the following equivalence class:
\begin{equation}\label{morcom}
G \circ F := \left[\left(\widehat{\mathcal{U}}, \widehat{\mathcal{U}}'', g\circ f, \left\{{g}_{f(p)} \circ {f}_{p}\right\}, \left\{{\widehat{f}}_{p,x} \circ {\widehat{g}}_{f(p),f_p(x)}\right\} \right)\right].
\end{equation}
\end{defn}

\begin{prop}\cite[Proposition 3.16]{Kim1}\label{ppidx}
The composition is well-defined and associative with the identity given by
\begin{equation}\label{idxm}
\mathrm{id}_{\mathfrak{X}}:= \left[\left(\widehat{\mathcal{U}}, \widehat{\mathcal{U}}, \mathrm{id}_X, \left\{\mathrm{id}_p\right\}, \left\{\widehat{\mathrm{id}}_{p,x}\right\}\right)\right]
\end{equation}
of each $\mathfrak{X} = \left(X, \left[\widehat{\mathcal{U}}\right]\right).$
\end{prop}

The above data give rise to a category denoted by $\mathbf{Kur}$ that consists of:
\begin{equation}\nonumber
\begin{cases}
\text{Ob}(\textbf{Kur}) =\{L_{\infty}\text{-Kuranishi spaces}\}\\
\text{Mor}(\textbf{Kur}) = \{\text{Equivalence classes of pre-morphisms} \}.
\end{cases}
\end{equation}

Indeed, $\mathbf{Kur}$ contains the category of smooth manifolds as a subcategory, allowing us to treat Kuranishi spaces and smooth manifolds on equal footing.
\begin{thm}\cite[Propositions 3.16 \& 3.18]{Kim1}\label{popl}
Kuranishi spaces form a category that admits an embedding from the category of smooth manifolds.
\end{thm}

\subsection{Examples of Kuranishi space}

{Smooth manifold as a Kuranishi space}\label{masks}

\begin{exam}[{Smooth manifolds}]\label{man}
Smooth manifolds are Kuranishi spaces endowed with a Kuranishi atlas 
\[
\widehat{\mathcal{U}}^{\mathrm{man}} = \left(\left\{\mathcal{U}^{\mathrm{man}}_p\right\}, \left\{\Phi_{pq}\right\}\right) = \left(\left\{\left(U_p, E_p, s_p, \Gamma_p, \psi_p\right)\right\}, \left\{\left(U_{pq}, \phi_{pq}, \left\{\widehat{\phi}_{pq,x}\right\}\right)\right\}\right)
\] 
of the following restrictive type:
\begin{itemize}
\item[--] $U_p = (U_p, \alpha)$ is the pair of a Euclidean space $\mathbb{R}^n$ of fixed dimension $n$ for all $p$ and the zero form $\alpha \equiv 0.$
\item[--] The decompositions (\ref{wstri}) exist for trivial reasons.
\item[--] $E_p = U_p \times \{0\} \simeq U_p$ is the zero-rank vector bundle.
\item[--] $s_p : U_p \xrightarrow{\simeq} E_p$ is the zero section.
\item[--] $\Gamma_p$ is the trivial group.
\item[--] $\psi_p : s^{-1}_p(0) \simeq U_p \hookrightarrow \mathbb{R}^n$ is the manifold coordinate chart.
\item[--] $x \in W_x \subset U_p$ is an open ball $\simeq B^n.$
\item[--] $T\mathcal{F}_x = TU_p|_{W_x}$ is the total tangent bundle,
\item[--] $\mathcal{C}_{p,x} := \overset{\text{Koszul}}{C^{\infty}(W_x)} \oplus \overset{\text{de Rham}}{\Omega_{\mathrm{aug}}^{\bullet +1}(W_x)}$ is given by the direct sum of the one component trivial chain complex $C^{\infty}(W_x)$ and the augmented de Rham chain complex $\Omega_{\mathrm{aug}}^{\bullet +1}(W_x)$ with the $L_{\infty}[1]$-algebra $\{l^{\mathrm{man}}_k\}_{k \geq 1}$ with $l^{\mathrm{man}}_{k \geq 2} = 0$. In other words, $\mathcal{C}_{p,x}$ is only a chain complex whose cohomolgy is isomorphic to $C^{\infty}(W_x).$

Let $\mathcal{U}_p$ and $\mathcal{U}_q$ be Kuranishi charts at $p$ and $q,$ respectively. The coordinate change $\Phi_{pq} := \left(U_{pq}, \phi_{pq}, \left\{\widehat{\phi}_{pq,x}\right\}\right) : \mathcal{U}_p \rightarrow \mathcal{U}_q$ is given by:
\begin{enumerate}
\item[--] $U_{pq} := \psi_p^{-1}\left(\mathrm{Im} \psi_p \cap \mathrm{Im} \psi_q\right).$
\item[--] $\phi_{pq} : U_{pq} \rightarrow U_q$ is the (usual) coordinate change for manifolds
\begin{equation}\nonumber
\phi_{pq} := \psi_q^{-1} \circ \psi_p\big|_{U_{pq}},
\end{equation}
which is an open embedding.
\item[--] $\widehat{\phi}_{pq, x}: \mathcal{C}'_{\phi_{pq}(x)} \rightarrow \mathcal{C}_x$ at each $x \in s^{-1}_p(0) \cap U_{pq} = U_{pq}$ is an isomorphism constructed as in \cite[Example 3.3]{Kim1}.
\end{enumerate}
\end{itemize}
\end{exam}

Example \ref{man} illustrates that a smooth manifold can be regarded as a special type of Kuranishi space in the sense that it naturally determines a Kuranishi atlas, and therefore gives rise to a Kuranishi space. In fact, we further obtain the following result.

\begin{prop}\label{pman}\cite[Proposition 6.20]{Kim1}
The category of smooth manifolds $\mathbf{Man}$ is a subcategory of $\mathbf{Kur}$. In other words, there exists an embedding of categories
\begin{equation}\nonumber
\mathscr{F} : \mathbf{Man} \hookrightarrow \mathbf{Kur}
\end{equation}
\end{prop}

\begin{exam}[Moduli space of pseudoholomorphic disks]
Under the assumption on the stratifications on the bases, \label{moduli}{$\mathcal{M}_{k+1}(L, \beta)$ as a Kuranishi space}. (See \cite{Kim4} for more details on the assumption.)

\begin{thm-defn}\label{krnsthm}\cite[Theorem-Definition 4.13]{Kim3}
$$\mathcal{M}_{k+1}(L, \beta) := \left( \mathcal{M}_{k+1}(L, \beta), \left[\{\mathcal{U}_{\mathbf{p}}\}, \{ \Phi_{\mathbf{p} \mathbf{q}} \}\right] \right)$$ is a Kuranishi space, which we call \textit{the moduli space of pseudoholomorphic curves with Lagrangian boundary conditions.}
\end{thm-defn}

\begin{thm-defn}[Evaluation morphisms]\cite[Theorem-Definiton 5.2]{Kim3}
The equivalence class 
\[\mathrm{Ev}_i := \left[\left(\widehat{\mathcal{U}}_{\mathbf{p}}, \widehat{\mathcal{U}}^{\mathrm{man}}_{\mathrm{ev}_i(\mathbf{p})}, \mathrm{ev}_i, \{\mathrm{ev}_{i,\mathbf{p}}\}, \{\widehat{\mathrm{ev}}_{i,\mathbf{p},\mathbf{x}}\} \right)\right]
\] 
defines a morphism of Kuranishi spaces
\[ 
\mathrm{Ev}_i : \mathcal{M}_{k+1}(L, \beta) \rightarrow L
\]
for each $i$, which we call the \textit{evaluation morphism} of the moduli space $\mathcal{M}_{k+1}(L, \beta).$
\end{thm-defn}

\begin{thm-defn}[Forgetful morphisms]\label{thqi}\cite[Theorem-Definition 5.6]{Kim3}
The equivalence class
\[
\mathrm{Ft}_i:=\left[\left(\mathcal{U}_{\mathbf{p}^+}, \mathcal{U}_{\mathbf{p}}, {\mathrm{ft}_i}, \left\{{\mathrm{ft}_i}_{\mathbf{p}^+}\right\}, \left\{\widehat{\mathrm{ft}}_{i, \mathbf{p}^+, \mathbf{x}^+}\right\}\right)\right]
\]
defines a morphism of Kuranishi spaces
\begin{equation}\nonumber
{\mathrm{Ft}_i} : \left( \mathcal{M}_{k+1}(L, \beta), \left[\widehat{\mathcal{U}}^+\right] \right) \rightarrow \left( \mathcal{M}_{k}(L, \beta), \left[\widehat{\mathcal{U}}\right] \right)
\end{equation}
for each $i,$ which we call the $i$-th \textit{forgetful morphism} of the moduli space $\mathcal{M}_{k+1}(L, \beta).$
\end{thm-defn}

Consider the forgetful morphism $\mathrm{Ft}_{i}$ that forgets the $i$-th marked point of each element in $\mathcal{M}_{k+1}(L, \beta)$ for $k \geq 1.$ For the 0-th evaluation morphism $\mathrm{Ev}_0^{(k+1)}$ and  $\mathrm{Ev}_0^{(k)}$ from  $\mathcal{M}_{k+1}(L, \beta)$ and  $\mathcal{M}_{k}(L, \beta)$ to $L,$ respectively, we have the following diagram.
\[
\begin{tikzcd}
\mathcal{M}_{k+1}(L, \beta) \arrow{rd}[swap]{\mathrm{Ev}_0^{(k+1)}} \arrow{rr}{\mathrm{Ft}_{i}} & {} & \mathcal{M}_{k}(L, \beta) \arrow{ld}{\mathrm{Ev}_0^{(k)}}\\
{} & L. & {}
\end{tikzcd}
\]

\begin{cor}\label{efe}\cite[Corollary 5.7]{Kim3}
As morphisms of Kuranishi spaces, the equality
\begin{equation}\nonumber
\mathrm{Ev}^{(k)}_0 \circ \mathrm{Ft}_i = \mathrm{Ev}^{(k+1)}_0
\end{equation}
holds for each $1 \leq i \leq k.$
\end{cor}
\end{exam}

\section{Special classes of morphisms}

In this section, we define several distinguished classes of morphisms in the category $\mathbf{Kur}$ in preparation for the construction of the homotopy category $\mathrm{Ho}(\mathbf{Kur})$ in Section~6.

\subsection{Induced $L_{\infty}[1]$-morphisms from bundle embeddings}\label{geeeee}

We briefly review the relevant results from \cite{Kim1}, in particular, the theorem stating that a compatible bundle embedding induces an $L_{\infty}[1]$-morphism. We will use this result repeatedly throughout the paper.

Let $E$ and $E'$ be vector bundles over $U$ and $U'$, respectively, with sections
$s$ and $s'$. Let
\[
\widetilde{\phi}:E\to E'
\]
be a bundle embedding covering a smooth map
\[
\phi:U\to U',
\]
and suppose that
\[
\pi_{\operatorname{Im}\widetilde{\phi}}(s'\circ\phi)
=
\widetilde{\phi}\circ s,
\]
where
$\pi_{\operatorname{Im}\widetilde{\phi}}(s'\circ\phi)$
denotes the component of $s'\circ\phi$ in
$\operatorname{Im}\widetilde{\phi}\subset\phi^*E'$.
Then, for each $x\in s^{-1}(0)$, the bundle embedding induces a map
\[
\widehat{\eta}(\widetilde{\phi})_x:
\Gamma\left(
\bigwedge\nolimits^{-\bullet}E^*|_{W_x}
\right)
\longrightarrow
\Gamma\left(
\bigwedge\nolimits^{-\bullet}E^{\prime *}|_{W'_{\phi(x)}}
\right).
\]

\subsubsection*{(Preparatory constructions)}

To define $\Xi(\widetilde{\phi})_x$, we first take some preliminary steps.
We begin by choosing a projection
\[
\pi:U'\twoheadrightarrow\phi(U)
\]
that restricts to the identity map
\[
\pi|_{\phi(U)}=\mathrm{id}_{\phi(U)}
\]
on $\phi(U)\subset U'$. Such a projection exists because $\phi$ is an embedding.

The bundle embedding
\[
\widetilde{\phi}:E\hookrightarrow E'
\]
naturally induces a bundle embedding
\[
\widetilde{\phi}:
\bigwedge\nolimits^{-\bullet}E^*
\hookrightarrow
\bigwedge\nolimits^{-\bullet}E^{\prime *}.
\]
By abuse of notation, we denote this induced embedding by the same symbol
$\widetilde{\phi}$.

\subsubsection*{(Definition of $\widehat{\eta}(\widetilde{\phi})_x$)}

We now define
\begin{equation}\label{haahb}
\widehat{\eta}_{x,k}
:=
\widehat{\eta}(\widetilde{\phi})_{x,k}
:
\Gamma\left(\bigwedge\nolimits^{-\bullet}E^*|_{W_x}\right)^{\otimes k}
\rightarrow
\Gamma\left(\bigwedge\nolimits^{-\bullet}E^{\prime *}|_{W'_{\phi(x)}}\right)
\end{equation}
by
\begin{equation}\label{eta1k}
\widehat{\eta}_{x,k}(a_1,\ldots,a_k)
:=
\begin{cases}
\overline{a_1}, & \text{if } k=1,\\
0, & \text{if } k\geq 2,
\end{cases}
\end{equation}
for
\[
a_i\in
\Gamma\left(\bigwedge\nolimits^{-\bullet}E^*|_{W_x}\right),
\qquad i=1,\ldots,k,
\]
where
\begin{equation}
\overline{a}
:=
\pi^*\circ\widetilde{\phi}(a)
\in
\Gamma\left(\bigwedge\nolimits^{-\bullet}E^{\prime *}|_{W'_{\phi(x)}}\right).
\end{equation}

Here, the pullback by $\pi$,
\[
\pi^*:
\Gamma\left(\bigwedge\nolimits^{-\bullet}E^{\prime *}|_{\operatorname{Im}\phi}\right)
\rightarrow
\Gamma\left(\bigwedge\nolimits^{-\bullet}E^{\prime *}|_{W'_{\phi(x)}}\right),
\]
is defined by
\[
\pi^*(b)(u)
:=
\pi^*\left(b\left(u|_{\operatorname{Im}\phi}\right)\right)
\]
for
\[
u\in
\Gamma\left(\bigwedge\nolimits^{-\bullet}E'|_{W'_{\phi(x)}}\right),
\]
inductively on the degree of
\[
b\in
\Gamma\left(\bigwedge\nolimits^{-\bullet}E^{\prime *}|_{\operatorname{Im}\phi}\right).
\]

\begin{lem}\label{pokhgjb}
$\widehat{\eta}_{x}$ is an $L_{\infty}[1]$-morphism.
\end{lem}

\begin{proof}
It suffices to show that
\begin{equation}\nonumber
l'_k\big(
\widehat{\eta}_{x,1}(a_1),\ldots,
\widehat{\eta}_{x,1}(a_k)
\big)
=
\widehat{\eta}_{x,1}
\big(
l_k(a_1,\ldots,a_k)
\big).
\end{equation}

If $k=1$, then
\begin{equation}\nonumber
\begin{split}
l'_1\big(\widehat{\eta}_{x,1}(a)\big)
&=
l^{'\mathrm{K}}_1(\overline{a})
=
\iota_{s'}\overline{a}\\
&\overset{*}{=}
\overline{\iota_s a}
=
\overline{l^K_1(a)}
=
\widehat{\eta}_{x,1}\big(l^K_1(a)\big)
=
\widehat{\eta}_{x,1}\big(l_1(a)\big).
\end{split}
\end{equation}

Assume that $a$ is homogeneous and write
\[
a=a_1\wedge\cdots\wedge a_l
\]
as a product of degree-$1$ elements. The equality $*$ follows from the fact that the operations $\iota_s$ and $\iota_{s'}$ are compatible with the restriction maps. Indeed,
\begin{equation}\label{isp1}
\begin{split}
\iota_{s'}\overline{a}
&=
\iota_{s'}\left(\pi^*\big(\widetilde{\phi}(a)\big)\right)\\
&=
\pi^*\big(\widetilde{\phi}(a)\big)(s')\\
&=
\pi^*\left(
\widetilde{\phi}(a)(s'|_{\operatorname{Im}\phi})
\right)\\
&=
\pi^*\left(
\widetilde{\phi}(a_1)\wedge\cdots\wedge
\widetilde{\phi}(a_l)
(s'|_{\operatorname{Im}\phi})
\right)\\
&=
\pi^*\left(
\sum_i(-1)^{i+1}
\widetilde{\phi}(a_i)(s'|_{\operatorname{Im}\phi})
\cdot
\widetilde{\phi}(a_1)\wedge\cdots\wedge
\widehat{\widetilde{\phi}(a_i)}
\wedge\cdots\wedge
\widetilde{\phi}(a_l)
\right),
\end{split}
\end{equation}
where
\begin{equation}\nonumber
\begin{split}
\widetilde{\phi}(a_i)
\left(s'|_{\operatorname{Im}\phi}\right)
&=
\widetilde{\phi}(a_i)
\left(
\pi_{\operatorname{Im}\widetilde{\phi}}
\left(s'|_{\operatorname{Im}\phi}\right)
\right)\\
&=
a_i\left(
\left(\widetilde{\phi}^{-1}
\circ
\pi_{\operatorname{Im}\widetilde{\phi}}\right)
\left(s'|_{\operatorname{Im}\phi}\right)
\right)\\
&=
a_i\left(
\widetilde{\phi}^{-1}
\left(s'|_{\operatorname{Im}\phi}\right)
\right)
=
a_i(s),
\end{split}
\end{equation}
by the bundle map property of $\widetilde{\phi}$. Thus, \eqref{isp1} becomes
\begin{equation}\nonumber
\begin{split}
&\pi^*\left(
\sum_i(-1)^{i+1}
a_i(s)\cdot
\widetilde{\phi}(a_1)\wedge\cdots\wedge
\widehat{\widetilde{\phi}(a_i)}
\wedge\cdots\wedge
\widetilde{\phi}(a_l)
\right)\\
&=
\pi^*\circ\widetilde{\phi}
\left(
\sum_i(-1)^{i+1}
a_i(s)\cdot
a_1\wedge\widehat{a_i}\wedge\cdots\wedge a_l
\right)\\
&=
\pi^*\circ\widetilde{\phi}(\iota_s a)
=
\overline{\iota_s a}.
\end{split}
\end{equation}

If $k\geq 2$, then for
\[
a_i\in
\Gamma\left(\bigwedge\nolimits^{-\bullet}E^*|_{W_x}\right),
\qquad 1\leq i\leq k,
\]
we have
\begin{equation}\nonumber
\begin{split}
l'_k\big(
\widehat{\eta}_{x,1}(a_1),\ldots,
\widehat{\eta}_{x,1}(a_k)
\big)
&=
l'_k(\overline{a_1},\ldots,\overline{a_k})\\
&=
l^{'\mathrm{K}}_k(\overline{a_1},\ldots,\overline{a_k})\\
&=
0
=
\widehat{\eta}_{x,1}
\big(l^K_k(a_1,\ldots,a_k)\big).
\end{split}
\end{equation}
\end{proof}

We recall the following definitions and results.

\begin{defn}[FOOO Kuranishi charts]\label{deflkur}
Let $X$ be a compact metrizable space. We call a tuple
$\mathscr{U} := (U,E,s,\Gamma,\psi)$ an \textit{FOOO Kuranishi chart}
of $X$ if the following conditions are satisfied:
\begin{enumerate}[label=(\roman*)]
\item[--] $U$ is a simply connected orbifold.
\item[--] $E$ is a finite-rank vector bundle over $U$.
\item[--] $s:U\rightarrow E$ is a smooth section.
\item[--] $\Gamma$ is a finite group acting on $U$ and preserving $s^{-1}(0)$.
\item[--] $\psi:s^{-1}(0)/\Gamma\overset{\simeq}{\hookrightarrow}X$ is a homeomorphism onto its image.
\end{enumerate}
\end{defn}

\begin{defn}\label{foemb}
Given two FOOO Kuranishi charts $\mathcal{U}$ and $\mathcal{U}'$ of $X$,
an \textit{FOOO embedding}
$\Phi=(\phi,\widetilde{\phi}):\mathcal{U}\hookrightarrow\mathcal{U}'$
consists of:
\begin{enumerate}
\item[--] an orbifold embedding $\phi:U\hookrightarrow U'$,
\item[--] a vector bundle embedding $\widetilde{\phi}:E\hookrightarrow E'$,
\end{enumerate}
satisfying the following conditions:
\begin{enumerate}
\item[(i)] $\widetilde{\phi}\circ s=s'\circ\phi$,
\item[(ii)] $\psi'\circ\phi=\psi$ on $s^{-1}(0)$,
\item[(iii)] (\textit{Tangent bundle condition}) $ds'$ induces an isomorphism
\begin{equation}\label{tancond}
\left[ds'_{\phi(x)}\right]:
\frac{T_{\phi(x)}U'}{\phi_*(T_xU)}
\xrightarrow{\simeq}
\frac{E'_{\phi(x)}}{\widetilde{\phi}(E_x)},
\end{equation}
at each $x\in s^{-1}(0)$.
\end{enumerate}
\end{defn}

\begin{prop}\cite[Proposition 2.30, Theorem 2.36]{Kim1}\label{afec}
In the setting of Proposition \ref{afec}, an FOOO embedding induces an embedding of Kuranishi charts in the sense of Definition \ref{ourcemb}. In particular, the tangent bundle condition yields the quasi-isomorphism condition for the associated $L_{\infty}[1]$-algebras.
\end{prop}

\begin{lem}\cite[Lemma 2.37]{Kim1}\label{cochqim}
Suppose that the bases $U$ and $U'$ are contractible. Let
$\Phi:\mathcal{U}\rightarrow\mathcal{U}'$ be an embedding induced by an
$\mathrm{FOOO}$ embedding (cf. Proposition \ref{afec}). Then, for each
$x\in s^{-1}(0)$, the $L_{\infty}[1]$-morphism
\[
\widehat{\varepsilon}_{\phi(x),\phi}:
\mathcal{C}'_{\phi(x)}
\rightarrow
\mathcal{C}'_{\phi(x),\phi}
\]
is a quasi-isomorphism.
\end{lem}

\begin{cor}\cite[Corollary 2.38]{Kim1}\label{crtbc}
For an embedding of Kuranishi charts satisfying the tangent bundle condition,
the $L_{\infty}[1]$-component
\[
\widehat{\phi}_x
=
\widehat{\phi}^{\mathrm{c}}_x
\circ
\widehat{\varepsilon}_{\phi(x),\phi}
\]
is a quasi-isomorphism for each $x\in s^{-1}(0)$.
\end{cor}

\begin{cor}\label{ygsdnn}
In the situation of Lemma \ref{pokhgjb}, suppose that the tangent bundle
condition holds and that the sections satisfy
\[
s'\circ\phi=\widetilde{\phi}\circ s.
\]
Then $\widehat{\eta}_x$ is a quasi-isomorphism.
\end{cor}

\begin{proof}[Proof sketch]
According to the definition of the map $\widehat{\phi}_x^\mathrm{c}$
in the proof of \cite[Proposition 2.30]{Kim1}, we have
\[
\widehat{\phi}_x^\mathrm{c}
=
\widehat{\varepsilon}_{\phi(x),\phi}
\circ
\widehat{\eta}_x.
\]
By \cite[Proposition 2.30]{Kim1} and Lemma \ref{cochqim}, both
$\widehat{\phi}_x^\mathrm{c}$ and
$\widehat{\varepsilon}_{\phi(x),\phi}$ are quasi-isomorphisms.
Hence $\widehat{\eta}_x$ is a quasi-isomorphism.
\end{proof}

\subsection{Embeddings, fibrations, weak equivalences, and isomorphisms}

In this subsection, we define fibrations and weak equivalences for later use in the construction a category of fibrant objects. We also consider embeddings as natural generalizations of the corresponding notions in the topological and smooth categories.

Let $F:\mathfrak{X}\to\mathfrak{X}'$ be a morphism of Kuranishi spaces represented by a pre-morphism
\[
\overline{F}
=
\left(
\widehat{\mathcal{U}},
\widehat{\mathcal{U}}',
f,
\left\{f_p\right\},
\left\{\widehat{f}_{p,x}\right\}
\right).
\]
We say:

\begin{defn}
\begin{enumerate}[label=(\alph*)]
\item
$F$ is an \textit{embedding} if there exists $\overline{F}$ such that each chart morphism
$\left(f_p,\left\{\widehat{f}_{p,x}\right\}\right)$ is an embedding for each
$p\in X$ and $x\in s_p^{-1}(0)$.

\item
$F$ is a \textit{base embedding} if there exists $\overline{F}$ such that each base component
$f_p$ is an embedding for each $p\in X$.
\end{enumerate}
\end{defn}

Denote by $\mathscr{A}(\mathfrak{X})$ the set of all atlases of $\mathfrak{X}$.
Consider the subset $\mathscr{A}_0(\mathfrak{X})\subset\mathscr{A}(\mathfrak{X})$
consisting of atlases with the smallest base dimension, namely those satisfying
$\dim U_p\leq\dim U'_p$ for each $p\in X$.

More precisely,
\begin{equation}\label{cevb}
\begin{split}
\mathscr{A}_0(\mathfrak{X})
:=
\Bigl\{\,
\widehat{\mathcal{U}}\in\mathscr{A}(\mathfrak{X})
\;\Big|\;
&\dim U_p\leq\dim U'_p \text{ for each }p\in X,\\
&\text{for some }\widehat{\mathcal{U}}'\in\mathscr{A}(\mathfrak{X})
\text{ such that }
\widehat{\mathcal{U}}\sim\widehat{\mathcal{U}}'
\,\Bigr\}.
\end{split}
\end{equation}

This is well-defined by the definition of the equivalence relation for
$L_\infty$-Kuranishi atlases (cf. Definition \ref{eqats}).

\begin{defn}\label{adfaf}
Let $F:\mathfrak{X}\to\mathfrak{X}'$ be a morphism of Kuranishi spaces represented by a pre-morphism
\[
\overline{F}
=
\left(
\widehat{\mathcal{U}},
\widehat{\mathcal{U}}',
f,
\left\{f_p\right\},
\left\{\widehat{f}_{p,x}\right\}
\right).
\]
We say:

\begin{enumerate}[label=(\alph*)]

\item
$F$ is a \textit{fibration} if there exists $\overline{F}$ satisfying the following properties:

\begin{enumerate}[label=(\roman*)]

\item
$\widehat{\mathcal{U}}\in\mathscr{A}_0(\mathfrak{X})$.

\item
$f_p$ is surjective and a submersion for each $p\in X$.

\item
The Koszul part morphism $\widehat{f}^{\mathrm{K}}_{p,x}$ is homotopic to the degree-wise injective $L_{\infty}[1]$-morphism
\[
\widehat{\eta}_{p,x}:
\bigwedge\nolimits^{-\bullet}
\Gamma\left(E^{\prime *}_{f(p)}|_{W'_{f(p)}}\right)
\rightarrow
\bigwedge\nolimits^{-\bullet}
\Gamma\left(E^*_p|_{W_x}\right)
\]
defined by
\begin{equation}\label{bhuaiuhnn}
\widehat{\eta}_{p,x}
:=
\widehat{\eta}^{\vee}_{p,x}
\circ
\underline{\widehat{\eta}}_{p,x},
\end{equation}
where $\widehat{\eta}^{\vee}_{p,x}$ and
$\underline{\widehat{\eta}}_{p,x}$ are defined as follows.

Let $\underline{E}^{\prime}_{f(p)}$ be a vector bundle over
$\underline{U}'_{f(p)}$.

There exists a pair of bundle embeddings
\[
\begin{cases}
\mathrm{Em}_p := (\overline{\mathrm{em}}_p, \mathrm{em}_p),\\
\mathrm{Em}'_p := (\overline{\mathrm{em}}'_p, \mathrm{em}'_p),
\end{cases}
\]

\[
\begin{tikzcd}
E'_{f(p)} \arrow{d} &  \underline{E}^{\prime}_{f(p)} \arrow{d} \arrow[hook]{l}[swap]{\overline{\mathrm{em}}'_p} \arrow[hook]{r}{\overline{\mathrm{em}}_p}  & E_p  \arrow{d} \\
U'_{f(p)} & \underline{U}'_{f(p)} \arrow[hook]{l}[swap]{{\mathrm{em}}'_p} \arrow[hook]{r}{\mathrm{em}_p} & U_p \\
\end{tikzcd}
\]

for each $p,$ satisfying the compatibilities
\begin{equation}\label{gadgashhhh2}
\begin{split}
\left(s_p \circ \mathrm{em}_p\right)|_{\mathrm{Im}\left(\overline{\mathrm{em}}_p\right)} = \overline{\mathrm{em}}_p \circ \underline{s}^{\prime \mathrm{c}}_{f(p)},\\
\left(s'_{f(p)} \circ \mathrm{em}'_p\right) = \overline{\mathrm{em}}'_p \circ \underline{s}^{\prime \mathrm{c}}_{f(p)}
\end{split} 
\end{equation}
together with the tangent bundle condition for $\mathrm{Em}'_p.$

By Lemma \ref{pokhgjb}, we obtain $L_{\infty}[1]$-morphisms
\[
\widehat{\eta}^{\vee}_{p,x}:
\bigwedge\nolimits^{-\bullet}
\Gamma\left(
\underline{E}^{\prime *}_{f(p)}
\big|_{\underline{W}'_{f(p)}}
\right)
\rightarrow
\bigwedge\nolimits^{-\bullet}
\Gamma\left(
E^*_p|_{W_p}
\right)
\]
and
\[
\underline{\widehat{\eta}}_{p,x}^{-1}:
\bigwedge\nolimits^{-\bullet}
\Gamma\left(
E^{\prime *}_{f(p)}|_{W'_{f(p)}}
\right)
\xrightarrow{\simeq}
\bigwedge\nolimits^{-\bullet}
\Gamma\left(
\underline{E}^{\prime *}_{f(p)}
\big|_{\underline{W}'_{f(p)}}
\right).
\]
In particular, $\underline{\widehat{\eta}}_{p,x}^{-1}$ is a quasi-isomorphism by Lemma \ref{ygsdnn}, since it has homotopy inverse $\underline{\widehat{\eta}}_{p,x}$.

\item
Consider a decomposition of the vector bundle
\begin{equation}\label{jbwwsxxx}
E_p
=
E_p^{\mathrm{i}}\oplus E_p^{\mathrm{c}}
:=
\pi_p^*\overline{\mathrm{em}}_p\left(E'_{f(p)}\right)
\oplus E_p^{\mathrm{c}}
\end{equation}
for some subbundle $E_p^{\mathrm{c}}\subset E_p$. Here,
\[
\pi_p:
U_p\twoheadrightarrow
\mathrm{em}_p(\underline{U}'_{f(p)})
\]
is a choice of projection between the contractible open sets.

We require that the $L_{\infty}$-coordinate change
$\widehat{\phi}_{pq,x}$ respect the decomposition \eqref{jbwwsxxx} in the following sense.
There exist well-defined $L_{\infty}[1]$-morphisms
\begin{equation}\label{hqrwbb}
\begin{split}
\widehat{\phi}^{\mathrm{i}}_{pq,x}
&:=
\widehat{\phi}_{pq,x}
\big|_{\bigwedge\nolimits^{-\bullet}
\left(E_{q}^{\mathrm{i}*}|_{W_{\phi_{pq}(x)}}\right)},\\
\widehat{\phi}^{\mathrm{c}}_{pq,x}
&:=
\widehat{\phi}_{pq,x}
\big|_{\bigwedge\nolimits^{-\bullet}
\left(E_{q}^{\mathrm{c}*}|_{W_{\phi_{pq}(x)}}\right)}
\end{split}
\end{equation}
for each $x\in s_p^{-1}(0)$ and $p,q\in X$ with
$\operatorname{Im}\psi_p\cap\operatorname{Im}\psi_q\neq\emptyset$.

We denote by
\[
[E_p]
:=
E_p\big/
\pi_p^*\overline{\mathrm{em}}_p\left(E'_{f(p)}\right)
\simeq
E_p^{\mathrm{c}}
\]
the quotient bundle, and by $[s_p]$ the corresponding section induced by $s_p$.

\item
\[
[s_{p'}']^{-1}(0)
\cap
f_p^{-1}\bigl(s_p^{-1}(0)\bigr)
=
(s_{p'}')^{-1}(0)
\]
for each $p\in X$.

\end{enumerate}

\item
$F$ is a \textit{weak equivalence} if there exists $\overline{F}$ satisfying the following properties:

\begin{enumerate}[label=(\roman*)]

\item
$f$ is a homeomorphism.

\item
$\widehat{f}_{p,x}$ is a quasi-isomorphism for each
$x\in s_p^{-1}(0)$ and $p\in X$.

\end{enumerate}

\item
$F$ is a \textit{trivial fibration} if it is both a fibration and a weak equivalence.

\end{enumerate}
\end{defn}

\begin{lem}\label{afdhjdasf}
\begin{enumerate}[label=(\roman*)]
\item An isomorphism in $\mathbf{Kur}$ is represented by a pre-morphism
\[
\overline{F}
=
\left(
\widehat{\mathcal{U}},
\widehat{\mathcal{U}}',
f,
\left\{f_p\right\},
\left\{\widehat{f}_{p,x}\right\}
\right)
\]
with the following properties: for each $p\in X$ and
$x\in s_p^{-1}(0)$,
\begin{enumerate}[label=(\roman*)]
\item[--] $f:X\rightarrow X'$ is a homeomorphism.
\item[--] $f_p:U_p\rightarrow U'_{f(p)}$ is a diffeomorphism of smooth manifolds.
\item[--] $\widehat{f}_{p,x}:{\mathcal{C}}'_{f(p),f_p(x)}
\rightarrow {\mathcal{C}}_{p,x}$ is a quasi-isomorphism.
\end{enumerate}

\item Any isomorphism in $\mathbf{Kur}$ is a fibration.
\end{enumerate}
\end{lem}

\begin{proof}
\begin{enumerate}[label=(\roman*)]
\item
By \eqref{morcom}, the maps on the underlying topological spaces and
the base components are given by isomorphisms in the corresponding
categories. By condition (iii) in Definition \ref{defmoreq},
the $L_{\infty}$-component is homotopic to a quasi-isomorphism and is
therefore itself a quasi-isomorphism.

\item
For condition $(\mathrm{iii})$ of (a) in Definition \ref{adfaf}, we recall that
\[
[s_p]^{-1}(0)\simeq U_p
\quad\text{and}\quad
(\mathrm{id}_p)^{-1}\bigl(s_p^{-1}(0)\bigr)
\simeq s_p^{-1}(0).
\]
Hence their intersection coincides with $s_p^{-1}(0)$.

All other conditions hold trivially. In particular, the bundle embedding
in $(\mathrm{ii})$ can be chosen to be the identity.

Any quasi-isomorphism is homotopic to an isomorphism induced by the
identity embedding.
\end{enumerate}
\end{proof}

We consider the special case of fibrations with manifold targets for later use.

\begin{lem}\label{genbvf}
Any morphism with a manifold target admits a representing pre-morphism
that satisfies all the axioms in $\mathrm{(a)}$ for a fibration except $\mathrm{(i)}$.
\end{lem}

\begin{proof}
Let
\[
\overline{F}
=
\left(
\widehat{\mathcal{U}},
\widehat{\mathcal{U}}^{\mathrm{man}},
f,
\left\{f_p\right\},
\left\{\widehat{f}_{p,x}\right\}
\right)
\]
be a choice of pre-morphism.
\begin{enumerate}[label=(\roman*), start=2]
\item
By replacing it with another representing pre-morphism with
sufficiently large contractible domain bases $U_p$, we may assume that
$f_p$ is surjective.

\item
$\widehat{f}_{p,x}$ is homotopic to the zero morphism, as given by the
composition of the $L_{\infty}[1]$-morphisms in (iv) and (v).

\item
Recall that
${E}^{\mathrm{man}}_{f(p)}\equiv {U}^{\mathrm{man}}_{f(p)}$
is the rank-zero vector bundle. We may take
\[
\left(
\underline{E}^{\mathrm{man}}_{f(p)},
\underline{U}^{\mathrm{man}}_{f(p)}
\right)
=
\left(
{E}^{\mathrm{man}}_{f(p)},
{U}^{\mathrm{man}}_{f(p)}
\right)
\]
with the identity bundle map between them. In this case,
$\underline{\widehat{\eta}}_{p,x}$ is the identity
$L_{\infty}[1]$-morphism.

\item
Choose any embedding
\[
\underline{U}^{\mathrm{man}}_{f(p)}\hookrightarrow U_p.
\]
We take the bundle embedding $\overline{\mathrm{em}}_p$ to be the
trivial embedding into the zero section, so that \eqref{gadgashhhh2}
holds trivially. Then $\left[E_p\right]$ is given by the whole
$E_p$.

Since $[s_p]\equiv s_p$ and
$\left(s_p^{\mathrm{man}}\right)^{-1}(0)=U_p^{\mathrm{man}}$, it
follows that
\[
\left[s_p'\right]^{-1}(0)
\cap
f_p^{-1}\left(s_p^{-1}(0)\right)
=
\left[s_p'\right]^{-1}(0)
\cap
f_p^{-1}\left(U_p^{\mathrm{man}}\right)
=
\left[s_p'\right]^{-1}(0).
\]
\end{enumerate}
\end{proof}

Finally, we give a nontrivial example of a weak equivalence. Let $M$ be a smooth manifold, and let $E\rightarrow U$ be a vector bundle over a smooth manifold with a section $s$ that is transverse to the zero section at every zero point, satisfying $s^{-1}(0)\simeq M$. Then we have:

\begin{prop}
There exists a weak equivalence from a smooth manifold to its
presentation as a transverse zero set.
\end{prop}

\begin{proof}
We interpret the data $(U,E,s,M)$ as a Kuranishi space, where the isotropy group is taken to be trivial and is therefore omitted from the
notation. Namely, for each $p\in M$, take an open neighborhood $U_p$ of $p$ in $U$, where $p\in s^{-1}(0)\simeq M$. Consider the restricted bundle
\[
E\big|_{U_p}\rightarrow U_p
\]
with section $s_p:=s|_{U_p}$. Denote
\[
V_p:=s_p^{-1}(0),
\]
which is an open neighborhood of $p$ in $M$ and is a smooth manifold of dimension
\[
\dim U_p-\mathrm{rk}E|_{U_p}
=
\dim U-\mathrm{rk}E.
\]
Take the atlas $\widehat{\mathcal{U}}^{\mathrm{man}}_M$ for $M$ given by $\{V_p\}$, together with the data in Example \ref{man}.

For the collection of charts
\[
\bigl\{
(U_p,E|_{U_p},s_p,\psi_p)
\bigr\},
\]
consider, for each pair $p,q\in M$, the tuple
\[
\Phi_{pq}
=
\left(
U_{pq},\phi_{pq},\left\{\widehat{\phi}_{pq,x}\right\}
\right)_{x\in s_p^{-1}(0)},
\]
where the components are defined as follows. Set
\[
U_{pq}:=U_p\cap U_q.
\]
The map $\phi_{pq}:U_{pq}\rightarrow U_q$ is any smooth map covering
the base map
\[
\phi^{\mathrm{man}}_{pq}:V_{pq}\rightarrow V_q,
\]
defined by
\[
\phi^{\mathrm{man}}_{pq}
=
\psi_q^{-1}\circ\psi_p,
\]
where
\[
V_{pq}:=V_p\cap V_q.
\]

Thus, we have the commutative diagram
\[
\begin{tikzcd}
U_{pq} \arrow[r, "\phi_{pq}"] & U_q \\
V_{pq} \arrow[u, hook] \arrow[r, "\phi^{\mathrm{man}}_{pq}
=\psi_q^{-1}\circ\psi_p"'] &
V_q. \arrow[u, hook]
\end{tikzcd}
\]

The map
\[
\widehat{\phi}_{pq,x}:
\mathcal{C}_{q,\phi_{pq}(x)}
\rightarrow
\mathcal{C}_{p,x}
\]
is the $L_{\infty}[1]$-morphism induced by the bundle embedding
\[
\begin{tikzcd}
E|_{U_{pq}} \arrow[r, hook] \arrow[d] &
E|_{U_q} \arrow[d] \\
U_{pq} \arrow[u, "s_q|_{U_{pq}}", bend left=20] \arrow[r, hook] &
U_q \arrow[u, swap, "s_q", bend right=20].
\end{tikzcd}
\]
as provided by Lemma \ref{pokhgjb}. Moreover, it is a
quasi-isomorphism by Corollary \ref{crtbc}. It is straightforward to
verify that the data
\[
\left(\left\{\mathcal{U}_p\right\},
\left\{\Phi_{pq}\right\}\right)
\]
define an atlas, and hence we obtain a Kuranishi space
\[
\mathcal{M}
:=
\left(
M,\,
\left[
\left(
\left\{\mathcal{U}_p\right\},
\left\{\Phi_{pq}\right\}
\right)
\right]
\right).
\]

We then consider the pre-morphism
\[
I_M
:=
\Bigl(
\widehat{\mathcal{U}}^{\mathrm{man}}_M,\,
\widehat{\mathcal{U}},\,
i_M,\,
\{i_{M,p}\},\,
\{\widehat{i}_{M,p,x}\}
\Bigr).
\]
Here, the atlases
$\widehat{\mathcal{U}}^{\mathrm{man}}_M$ and
$\widehat{\mathcal{U}}$ are chosen as above, while
$i_M:M\rightarrow M$ is the identity map.

The base map
\[
i_{M,p}:V_p\rightarrow U_p
\]
is the natural inclusion, viewed as a homeomorphism from $V_p$ onto
the subset
\[
V_p\simeq s_p^{-1}(0)\subset U_p,
\]
for each $p\in M$.

The $L_{\infty}[1]$-map
\[
\widehat{i}_{M,p,x}:
\bigwedge\nolimits^{-\bullet}
\Gamma\left(E_p^*\big|_{W_{M,x}}\right)
\oplus
\Omega^{\bullet+1}_{\mathrm{aug}}(W_{M,x})
\rightarrow
C^\infty(W_{M,x})
\oplus
\Omega^{\bullet+1}_{\mathrm{aug}}(W_x)
\]
is the $L_{\infty}[1]$-morphism induced by Lemma \ref{pokhgjb} and the
bundle embedding
\[
\begin{tikzcd}
V_p \times \{*\} \arrow[r, hook, swap, "\widetilde{i}_p"'] \arrow[d] &
E|_{U_p} \arrow[d] \\
V_p \arrow[u, "s_p^0", bend left=20] \arrow[r, swap, hook, "i_p"'] &
U_p \arrow[u, swap, "s_p", bend right=20],
\end{tikzcd}
\]
where $V_p \times \{*\}$ is the total space of the zero-rank vector
bundle over $V_p$, and $s_p^0$ is its trivial zero section. Observe
that
\[
s_p \circ i_p
=
\widetilde{i}_p \circ s_p^0
=
0.
\]

By the transversality assumption on $s$, namely, that $ds_p$ has full
rank at each zero point, and by the fact that $V_p$ is a manifold of
dimension $\dim U_p - \mathrm{rk}E|_{U_p}$, the above data satisfy the
tangent bundle condition. Indeed, at each $y \in V_p$,
\[
\frac{T_{i_p(y)}U_p}{i_{p*}(T_yV_p)}
\xrightarrow{[d_{i_p(y)}s_p], \simeq}
\frac{(E|_{U_p})_{i_p(y)}}
{\widetilde{i}_p(T_yV_p \times \{*\})}.
\]
Thus, this is an FOOO embedding, and by Proposition \ref{afec},
$\widehat{i}_{M,p,x}$ is a quasi-isomorphism. Therefore, the morphism
represented by $I_M$ is a weak equivalence.

\end{proof}

\section{A model of homotopy fiber products}

In this section, we construct an explicit model of fiber products for Kuranishi spaces. Specifically, when one of two given morphisms is a fibration, we show that there exists an object completing the pullback diagram and satisfying the corresponding universal property.

\subsection{Chart-level construction}

In this subsection, we construct an atlas for the fiber product by working at the level of Kuranishi charts, under the assumption that one of the two morphisms is a fibration.

Let $\mathfrak{X}$, $\mathfrak{X}'$, and $\mathfrak{X}''$ be Kuranishi spaces, and let $F \colon \mathfrak{X}' \to \mathfrak{X}$ and $G \colon \mathfrak{X}'' \to \mathfrak{X}$ be morphisms represented by pre-morphisms
\[
\overline{F} := \left(\widehat{\mathcal{U}}', \widehat{\mathcal{U}}, f, \left\{f'_{p'}\right\}, \left\{\widehat{f}'_{p',x'}\right\}\right)
\quad\text{and}\quad
\overline{G} := \left(\widehat{\mathcal{U}}'', \widehat{\mathcal{U}}, g, \left\{g''_{p''}\right\}, \left\{\widehat{g}''_{p'',x''}\right\}\right),
\]
respectively. We assume that $G$ satisfies all the axioms of a fibration in (a) with the exception of (i).

The primary objective of this subsection is to construct a (homotopy) fiber product Kuranishi space
\[
\mathfrak{X}' {}_F \times_G \mathfrak{X}''
\]
associated with the diagram
\[
\begin{tikzcd}
    & \mathfrak{X}'' \arrow[d, "G"] \\
    \mathfrak{X}' \arrow[r, "F"'] & \mathfrak{X}.
\end{tikzcd}
\]

We begin by observing the relation
\[
f_{p'} \circ \pi' = g_{p''} \circ \pi'',
\]
where
\[
\pi' \colon U'_{p'} {}_{f_{p'}} \!\times_{g_{p''}} U''_{p''} \to U'_{p'},
\qquad
\pi'' \colon U'_{p'} {}_{f_{p'}} \!\times_{g_{p''}} U''_{p''} \to U''_{p''}
\]
denote the canonical projections from the fiber product. In other words, the following diagram commutes:
\[
\begin{tikzcd}
U'_{p'} {}_{f_{p'}} \times_{g_{p''}} U''_{p''} \arrow[r, "\pi''"] \arrow[d, "\pi'"'] & U''_{p''} \arrow[d, "g_{p''}"] \\
U'_{p'} \arrow[r, "f_{p'}"] & U_{f(p') = g(p'')}.
\end{tikzcd}
\]

We set
\[
\bigl[E''_{p''}\bigr] := E''_{p''} / \operatorname{Im}\left(\overline{\operatorname{em}}_{f(p')}\right),
\]
and write $\bigl[\tilde{s}''_{p''}\bigr]$ for the induced section.

To each pair $(p',p'') \in X' {}_f \times_g X''$, we assign a Kuranishi chart
\[
    \mathcal{U}^\mathrm{f}_{(p',p'')}
    := \Bigl( U^\mathrm{f}_{(p',p'')},\, E^\mathrm{f}_{(p',p'')},\, s^\mathrm{f}_{(p',p'')},\, \psi^\mathrm{f}_{(p',p'')} \Bigr)
\]
consisting of the following data:

\begin{itemize}
    \item[--] $U^\mathrm{f}_{(p',p'')} := U'_{p'} {}_{f_{p'}} \times_{g_{p''}} U''_{p''}$, which is a smooth manifold since $g_{p''}$ is a submersion.

    \item[--] $E^\mathrm{f}_{(p',p'')} := \pi^{\prime*} E'_{p'} \oplus \pi^{\prime\prime*}\bigl[E''_{p''}\bigr]
    = \pi^{\prime*} E'_{p'} \oplus \pi^{\prime\prime*}\!\left(E''_{p''} / \operatorname{Im}\left(\overline{\operatorname{em}}_{f(p')}\right)\right)$.

    \item[--] The section $s^\mathrm{f}_{(p',p'')} := \pi^{\prime*} s'_{p'} \oplus \pi^{\prime\prime*}\bigl[\tilde{s}''_{p''}\bigr] \colon U_{(p',p'')}^\mathrm{f} \to E_{(p',p'')}^\mathrm{f}$, given explicitly by
    \[
    (y', y'') \mapsto \pi^{\prime*} s'_{p'}(y') \oplus \pi^{\prime\prime*} \bigl[\tilde{s}''_{p''}\bigr](y'').
    \]

    \item[--] The parametrization $\psi^\mathrm{f}_{(p',p'')}$, defined by the composition
    \begin{equation}\nonumber
    \begin{split}
    \psi^\mathrm{f}_{(p',p'')} \colon \left(s^\mathrm{f}_{(p',p'')}\right)^{-1}(0) &= \Bigl(\pi^{\prime*} s'_{p'} \oplus \pi^{\prime\prime*}\bigl[\tilde{s}''_{p''}\bigr]\Bigr)^{-1}(0) \\
        &\overset{\text{Lemma~\ref{agagsdgaa}}}{\simeq}
        (s'_{p'})^{-1}(0) {}_{f_{p'}} \times_{g_{p''}} (s''_{p''})^{-1}(0) \\
        &\longrightarrow X' {}_f \times_g X'', \\
    (y',y'') &\longmapsto \left(\psi'_{p'}(y'), \psi''_{p''}(y'')\right).
    \end{split}
    \end{equation}
\end{itemize}

\begin{lem}\label{agagsdgaa}
For each pair $(p', p'')$, there is a natural identification
\[
\left(s^\mathrm{f}_{(p',p'')}\right)^{-1}(0) \;\simeq\;
(s'_{p'})^{-1}(0) {}_{f_{p'}}\!\times_{g_{p''}} (s''_{p''})^{-1}(0).
\]
\end{lem}

\begin{proof}
This follows immediately from Definition~\ref{adfaf}~(a)~(iv).
\end{proof}

For $(x',x'') \in \left(s^\mathrm{f}_{(p',p'')}\right)^{-1}(0)$, the presymplectic local neighborhood is given by an open subset
\[
W^\mathrm{f}_{(x',x'')}
\subset
W'_{x'} {}_{f_{p'}}\!\times_{g_{p''}} W''_{x''}
\]
equipped with the closed two-form
\[
\alpha^\mathrm{f}_{W^{\mathrm{f}}_{(x',x'')}}
:=
\left(
\pi^{\prime*} \alpha'_{W'_{x'}}
+
\pi^{\prime\prime*} \alpha''_{W''_{x''}}
\right)\!\Bigg|_{W^{\mathrm{f}}_{(x',x'')}}.
\]
We then set
\[
T\mathcal{F}^\mathrm{f}_{(x',x'')}
:=
\ker \alpha^\mathrm{f}_{W^{\mathrm{f}}_{(x',x'')}}.
\]

The local $L_{\infty}[1]$-algebra is defined by
\begin{equation}\nonumber
\begin{split}
\mathcal{C}^\mathrm{f}_{(p',p''),(x',x'')}
&:=
\bigwedge\nolimits^{-\bullet}
\Gamma\left(
\left(E^\mathrm{f}_{(p',p'')}\right)^*
\Bigg|_{W^\mathrm{f}_{(x',x'')}}
\right)
\;\oplus\;
\Omega^{\bullet+1}_{\mathrm{aug}}
\left(
\mathcal{F}^\mathrm{f}_{(x',x'')}
\right) \\
&= \left( \bigoplus_{i \geq 0} \bigwedge\nolimits^{-\bullet + i} \Gamma\left(\pi^{\prime*} (E'_{p'})^* \Big|_{W^{\mathrm{f}}_{(x',x'')}}\right) \otimes \bigwedge\nolimits^{-i} \Gamma\left(\pi^{\prime\prime*} \left[E''_{p''}\right]^* \Big|_{W^{\mathrm{f}}_{(x',x'')}}\right) \right) \\
&\quad \oplus \Omega_{\mathrm{aug}}^{\bullet+1}\left(\mathcal{F}^\mathrm{f}_{(x',x'')}\right),
\end{split}
\end{equation}
where the $L_{\infty}$-operations
\[
l^{\mathrm{f}}_{(p',p''),(x',x''),k} \colon
\left(\mathcal{C}^\mathrm{f}_{(p',p''),(x',x'')}\right)^{\otimes k}
\to
\mathcal{C}^\mathrm{f}_{(p',p''),(x',x'')}
\qquad (k \geq 1)
\]
are given by
\[
l^{\mathrm{f}}_{(p',p''),(x',x''),k} := l^{\mathrm{f}, \mathrm{K}}_{(p',p''),(x',x''),k} \oplus l^{\mathrm{f}, \mathrm{dR}}_{(p',p''),(x',x''),k}.
\]
Here, $l^{\mathrm{f}, \mathrm{dR}}_{(p',p''),(x',x''),k}$ denotes the $L_{\infty}[1]$-algebra structure given in Proposition~\ref{augomega}.

For each homogeneous element
\[
a = a' \otimes a'' \in \bigwedge\nolimits^{-i}\Gamma\left(\pi^{\prime*} \left(E'_{p'}\right)^* \Big|_{W^{\mathrm{f}}_{(x',x'')}}\right) \otimes \bigwedge\nolimits^{-j}\Gamma\left(\pi^{\prime\prime*} \left[E''_{p''}\right]^* \Big|_{W^{\mathrm{f}}_{(x',x'')}}\right),
\]
the Koszul operation $l^{\mathrm{f}, \mathrm{K}}_{(p', p''), (x', x''), m}$ is defined by
\[
l^{\mathrm{f}, \mathrm{K}}_{(p', p''), (x', x''), m}(a) := 
\begin{cases} 
\pi^{\prime*} \left(l^{\mathrm{f}, \mathrm{K}}_{p',x', m}(a')\right) & \text{if } j=0 \text{ and } i \geq 1, \\
\pi^{\prime\prime*} \left(\left[l^{\mathrm{f}, \mathrm{K}}_{p'', x'', m}(a'')\right]\right) & \text{if } i=0 \text{ and } j \geq 1, \\
0 & \text{otherwise}.
\end{cases}
\]

For points $p',q' \in X'$, $p'',q'' \in X''$, and $(x',x'') \in \left(s^{\mathrm{f}}_{(p', p'')}\right)^{-1}(0) \cap \left(s^{\mathrm{f}}_{(q', q'')}\right)^{-1}(0)$, we consider the tuple
\[
\Phi^\mathrm{f}_{p'q',\,p''q''}
:=
\left(
U^\mathrm{f}_{p'q',\,p''q''},
\;
\phi^\mathrm{f}_{p'q',\,p''q''},
\;
\left\{\widehat{\phi}^\mathrm{f}_{p'q',\,p''q'',(x',x'')}\right\}
\right),
\]
consisting of:
\begin{itemize}
\item[--]
$U^\mathrm{f}_{p'q',\,p''q''} := U'_{p'q'} {}_{f_{p'}}\!\times_{g_{p''}} U''_{p''q''}$, which is a smooth manifold since $g_{p''}$ is a submersion.

\item[--] The coordinate change map
\[
\phi^\mathrm{f}_{p'q',\,p''q''}
\colon
U^\mathrm{f}_{p'q',\,p''q''}
\to
U'_{q'} {}_{f_{q'}}\!\times_{g_{q''}} U''_{q''},
\]
defined by
\[
\phi^\mathrm{f}_{p'q',\,p''q''}(y',y'')
:=
\left(
\phi'_{p'q'}(y'),
\;
\phi''_{p''q''}(y'')
\right),
\]
where $\phi'_{p'q'} \colon U'_{p'q'} \to U'_{q'}$ and $\phi''_{p''q''} \colon U''_{p''q''} \to U''_{q''}$ are the coordinate changes of $\mathfrak{X}'$ and $\mathfrak{X}''$, respectively.

\item[--] The cocycle condition for the base coordinate change $\phi^\mathrm{f}_{(\cdots)}$ follows immediately from the fact that it holds for each component, $\phi'_{p'q'}$ and $\phi''_{p''q''}$.

\item[--] The $L_\infty[1]$-morphism
\[
\widehat{\phi}^\mathrm{f}_{p'q',\,p''q'',(x',x'')}
\colon
\mathcal{C}^\mathrm{f}_{(q',q''),\phi^\mathrm{f}_{p'q',\,p''q''}(x',x'')}
\to
\mathcal{C}^\mathrm{f}_{(p',p''),(x',x'')},
\]
defined by
\begin{equation}\label{tvev}
\widehat{\phi}^\mathrm{f}_{p'q',\,p''q'',(x',x'')} := \widehat{\phi}^\mathrm{f,c}_{p'q',\,p''q'',(x',x'')} \circ \widehat{\varepsilon}_{(q',q''),(x',x''), \phi^{\mathrm{f}}_{p'q',\,p''q''}}.
\end{equation}
\end{itemize}

Here, we decompose
\[
\widehat{\phi}^{\mathrm{f,c}}_{p'q',\,p''q'',\,(x',x'')} := \widehat{\phi}^{\mathrm{f},\mathrm{K,c}}_{p'q',\,p''q'',\,(x',x'')} \oplus \widehat{\phi}_{(p'q',\,p''q''),(x',x'')}^{\mathrm{f,dR,c}}.
\]

\noindent\textit{(Koszul part)}
The Koszul part
\[
\begin{split}
\widehat{\phi}_{(p'q',\,p''q''),\,(x',x'')}^{\mathrm{f,K,c}} &\colon
\bigwedge\nolimits^{-\bullet}\Gamma\left(\left(E_{q',q''}^\mathrm{f}\right)^* \Big|_{W^\mathrm{f}_{\phi^\mathrm{f}_{p'q',p''q''}(x',x'')}}\right)_{\phi^\mathrm{f}_{p'q',p''q''}}\\
& \quad\quad\quad\quad\quad\quad\quad \longrightarrow
\bigwedge\nolimits^{-\bullet}\Gamma\left(\left(E_{p',p''}^\mathrm{f}\right)^* \Big|_{W^\mathrm{f}_{(x',x'')}}\right)
\end{split}
\]
is given by the restriction to the open subset
\[
W^\mathrm{f}_{\phi^\mathrm{f}_{p'q',p''q''}(x',x'')}
\subset
W'_{\phi'_{p'q'}(x')} {}_{f_{q'}}\!\times_{g_{q''}} W''_{\phi''_{p''q''}(x'')}
\]
of the $L_{\infty}[1]$-morphism
\begin{equation}\nonumber
\begin{split}
\bigwedge\nolimits^{-\bullet}\Gamma\left(
\pi^{\prime*}(E'_{q'})^*
\oplus
\pi^{\prime\prime*}\left[E''_{q''}\right]^*
\Bigg|_{W'_{\phi'_{p'q'}(x')} \times W''_{\phi''_{p''q''}(x'')}}
\right)_{\phi'_{p'q'} \times \phi''_{p''q''}} \\
\longrightarrow \bigwedge\nolimits^{-\bullet}\Gamma\left(
\pi^{\prime*}(E'_{p'})^* \oplus \pi^{\prime\prime*}\left[E''_{p''}\right]^*
\Bigg|_{W'_{x'} \times W''_{x''}}
\right).
\end{split}
\end{equation}

\begin{lem}\label{wnqb}
The following statements hold:
\begin{itemize}
    \item[(i)] There exists an $\mathbb{R}$-module isomorphism 
    \[
    \left(\pi^{\prime*} C^\infty\left(W'_{x'}\right) + \pi^{\prime\prime*} C^\infty\left(W''_{x''}\right)\right) / \mathbb{R} \;\simeq\; \pi^{\prime*} C^\infty\left(W'_{x'}\right) \oplus \pi^{\prime\prime*} C^\infty\left(W''_{x''}\right).
    \]
    \item[(ii)] There exists an $\mathbb{R}$-module projection
    \begin{equation}\label{kkoslg}
    \mathrm{p}_0 \colon C^\infty\left(W^\mathrm{f}_{(x', x'')}\right) \to \pi^{\prime*} C^\infty\left(W'_{x'}\right) \oplus \pi^{\prime\prime*} C^\infty\left(W''_{x''}\right),
    \end{equation}
    \[
    g \mapsto \left(\pi^{\prime*} \mathrm{p}_0'(g),\, \pi^{\prime\prime*} \mathrm{p}_0''(g)\right).
    \]
\end{itemize}
\end{lem}

\begin{proof}
\begin{itemize}
    \item[(i)] Suppose that $\pi^{\prime*} g_1 + \pi^{\prime\prime*} g_2 = \pi^{\prime*} h_1 + \pi^{\prime\prime*} h_2$ represents two expressions for the sum of two functions. Then
    \[
    \pi^{\prime*}(g_1 - h_1) = \pi^{\prime\prime*}(h_2 - g_2) \in \pi^{\prime*} C^\infty\left(W'_{x'}\right) \cap \pi^{\prime\prime*} C^\infty\left(W''_{x''}\right).
    \]
    Since the intersection $\pi^{\prime*} C^\infty\left(W'_{x'}\right) \cap \pi^{\prime\prime*} C^\infty\left(W''_{x''}\right)$ consists solely of constant functions, the assignment
    \[
    \pi^{\prime*} g_1 + \pi^{\prime\prime*} g_2 \mapsto \left( \pi^{\prime*} g_1 - \pi^{\prime*} g_1(0,0),\, \pi^{\prime\prime*} g_2 - \pi^{\prime\prime*} g_2(0,0) \right)
    \]
    defines a canonical linear map
    \[
    \begin{split}
    \left(\pi^{\prime*} C^\infty\left(W'_{x'}\right) + \pi^{\prime\prime*} C^\infty\left(W''_{x''}\right)\right) / \mathbb{R} &\to \pi^{\prime*} C^\infty\left(W'_{x'}\right) \oplus \pi^{\prime\prime*} C^\infty\left(W''_{x''}\right), \\
    \left[\pi^{\prime*} g_1 + \pi^{\prime\prime*} g_2\right] &\mapsto \left( \pi^{\prime*} g_1 - \pi^{\prime*} g_1(0,0),\, \pi^{\prime\prime*} g_2 - \pi^{\prime\prime*} g_2(0,0) \right),
    \end{split}
    \]
    which is readily seen to be bijective, completing the proof of (i).
    
    \item[(ii)] Consider a direct sum decomposition
    \[
    C^\infty\left(W^\mathrm{f}_{(x', x'')}\right) = \left( \pi^{\prime*} C^\infty\left(W'_{x'}\right) + \pi^{\prime\prime*} C^\infty\left(W''_{x''}\right) \right) \oplus V
    \]
    and the corresponding projection onto the first summand $\pi^{\prime*} C^\infty\left(W'_{x'}\right) + \pi^{\prime\prime*} C^\infty\left(W''_{x''}\right)$. Composing this projection with the map obtained in (i) yields the desired morphism.
\end{itemize}
\end{proof}

For each homogeneous element
\[
\begin{split}
a &= a' \otimes a'' \\
  &\in \bigwedge\nolimits^{-\bullet + i}\Gamma\left(\pi^{\prime*} (E'_{q'})^* \Big|_{W'_{\phi'_{p'q'}(x')} \times W''_{\phi''_{p''q''}(x'')}}\right)_{\phi'_{p'q'} \times \phi''_{p''q''}} \\
  &\quad\quad\quad \otimes \bigwedge\nolimits^{-i}\Gamma\left(\pi^{\prime\prime*} \left[E''_{q''}\right]^* \Big|_{W'_{\phi'_{p'q'}(x')} \times W''_{\phi''_{p''q''}(x'')}}\right)_{\phi'_{p'q'} \times \phi''_{p''q''}} \\
  &= \bigwedge\nolimits^{-\bullet}\Gamma\left(
  \pi^{\prime*}(E'_{q'})^*
  \oplus
  \pi^{\prime\prime*}\left[E''_{q''}\right]^*
  \Bigg|_{W'_{\phi'_{p'q'}(x')} \times W''_{\phi''_{p''q''}(x'')}}
  \right)_{\phi'_{p'q'} \times \phi''_{p''q''}},
\end{split}
\]
we define the Koszul component
\begin{equation}\label{bbexx}
\widehat{\phi}^{\mathrm{f, K, c}}_{(p'q',\, p''q''),\, (x', x''),\, k}(a) := 
\begin{cases} 
\pi^{\prime*} \left( \widehat{\phi}^{\prime, \mathrm{K, c}}_{p'q',\, x',\, k}(a') \right) & \text{if } j = 0 \text{ and } i \geq 1, \\
\pi^{\prime\prime*} \left( \left[ \widehat{\phi}^{\prime\prime, \mathrm{K, c}}_{p''q'',\, x'',\, k}(a'') \right] \right) & \text{if } i = 0 \text{ and } j \geq 1, \\
\pi^{\prime*} \left( \widehat{\phi}^{\prime, \mathrm{K, c}}_{p'q',\, x',\, k}(\operatorname{pr}'(a)) \right) \\
\quad +\, \pi^{\prime\prime*} \left( \left[ \widehat{\phi}^{\prime\prime, \mathrm{K, c}}_{p''q'',\, x'',\, k}(\operatorname{pr}''(a)) \right] \right) & \text{if } i = 0 \text{ and } j = 0, \\
0 & \text{otherwise}.
\end{cases}
\end{equation}

\noindent\textit{(de Rham part)}
For the de Rham component
\[
\widehat{\phi}_{p'q',\, p''q'',\, (x',x'')}^{\mathrm{f}, \mathrm{dR,c}} \colon \Omega_{\mathrm{aug}}^{\bullet+1}\left(\mathcal{F}_{\left(\phi'_{p'q'}(x'),\, \phi''_{p''q''}(x'')\right)}^\mathrm{f}\right)_{\phi^\mathrm{f}_{p'q',\,p''q''}} \longrightarrow \Omega_{\mathrm{aug}}^{\bullet+1}\left(\mathcal{F}_{(x', x'')}^\mathrm{f}\right),
\]
we consider the following sequence of morphisms:
\begin{equation}\label{gageeeeee}
\begin{split}
&\Omega^{\bullet+1}\left(\mathcal{F}_{\left(\phi'_{p'q'}(x'),\, \phi''_{p''q''}(x'')\right)}^\mathrm{f}\right)_{\phi^{\mathrm{f}}_{p'q',\,p''q''}} = \Gamma\left(\bigwedge\nolimits^{\bullet+1} T^*\mathcal{F}^\mathrm{f}_{\left(\phi'_{p'q'}(x'),\, \phi''_{p''q''}(x'')\right)}\right)_{\phi^{\mathrm{f}}_{p'q',\,p''q''}} \\
&\xrightarrow{\ (1)\ } \Gamma\left(\bigwedge\nolimits^{\bullet+1} T^*\widetilde{\mathcal{F}}_{\left(\phi'_{p'q'}(x'),\, \phi''_{p''q''}(x'')\right)}\right)_{\phi'_{p'q'} \times \phi''_{p''q''}} \\
&\xleftarrow{\ (2)\ } \pi^{\prime*}\Gamma\left(\bigwedge\nolimits^{\bullet+1} T^*\mathcal{F}_{\phi'_{p'q'}(x')}'\right)_{\phi'_{p'q'} \times \phi''_{p''q''}} \oplus \pi^{\prime\prime*}\Gamma\left(\bigwedge\nolimits^{\bullet+1} T^*\mathcal{F}_{\phi''_{p''q''}(x'')}^{\prime\prime}\right)_{\phi'_{p'q'} \times \phi''_{p''q''}} \\
&\xrightarrow{\widehat{\phi}_{p'q',x'}^{\prime\mathrm{dR,c}} \oplus \widehat{\phi}_{p''q'', x''}^{\prime\prime\mathrm{dR,c}}} \pi^{\prime*}\Gamma \left(\bigwedge\nolimits^{\bullet+1} T^*\mathcal{F}_{x'}'\right) \oplus \pi^{\prime\prime*}\Gamma\left(\bigwedge\nolimits^{\bullet+1} T^*\mathcal{F}_{x''}^{\prime\prime}\right) \\
&\xrightarrow{\ (3)\ } \Gamma\left(\bigwedge\nolimits^{\bullet+1} T^*\widetilde{\mathcal{F}}_{(x', x'')}\right) \xleftarrow{\ (4)\ } \Gamma\left(\bigwedge\nolimits^{\bullet+1} T^*\mathcal{F}^\mathrm{f}_{(x', x'')}\right) = \Omega^{\bullet+1}\left(\mathcal{F}_{(x', x'')}^\mathrm{f}\right),
\end{split}
\end{equation}
where
\[
T\widetilde{\mathcal{F}} := \ker \widetilde{\alpha}
\]
for the closed two-form $\widetilde{\alpha} := \pi^{\prime*} \alpha'_{W'_{x'}} - \pi^{\prime\prime*} \alpha''_{W''_{x''}} \in \Omega^2\left(W'_{\phi'_{p'q'}(x')} \times W''_{\phi''_{p''q''}(x'')}\right)$.

The maps (1) through (4) are defined as follows:

\noindent (1) Let $\pi \colon W^\times_{(x', x'')} \to W^\mathrm{f}_{(x', x'')}$ be a choice of retraction between the contractible domains. It satisfies
\[
d\pi\left(\ker \alpha^\times\right) \subset \ker \alpha^\mathrm{f},
\]
where
\[
\alpha^\times := \pi^{\prime*} \alpha'_{W'_{x'}} - \pi^{\prime\prime*} \alpha''_{W''_{x''}} \in \Omega^2\left(W^\times_{(x',x'')}\right)
\quad\text{and}\quad
\alpha^\mathrm{f} := \alpha^\times \Big|_{W^\mathrm{f}_{(x', x'')}}.
\]
Consequently, we obtain pullback morphisms between the foliation de Rham complexes:
\[
\pi^* \colon \Omega^{\bullet+1}\left(\mathcal{F}^\mathrm{f}_{\left(\phi'_{p'q'}(x'),\, \phi''_{p''q''}(x'')\right)}\right) \longrightarrow \Omega^{\bullet+1}\left(\mathcal{F}^\times_{\left(\phi'_{p'q'}(x'),\, \phi''_{p''q''}(x'')\right)}\right),
\]
and similarly on the localized sections,
\[
\pi^* \colon \Omega^{\bullet+1}\left(\mathcal{F}^\mathrm{f}_{\left(\phi'_{p'q'}(x'),\, \phi''_{p''q''}(x'')\right)}\right)_{\phi^\mathrm{f}_{p'q',\,p''q''}} \longrightarrow \Omega^{\bullet+1}\left(\mathcal{F}^\times_{\left(\phi'_{p'q'}(x'),\, \phi''_{p''q''}(x'')\right)}\right)_{\phi^\times_{p'q',\,p''q''}}.
\]
Passing to augmentations yields the induced $L_\infty[1]$-morphism:
\[
\Omega^{\bullet+1}_{\mathrm{aug}}\left(\mathcal{F}^\mathrm{f}_{\left(\phi'_{p'q'}(x'),\, \phi''_{p''q''}(x'')\right)}\right)_{\phi^\mathrm{f}_{p'q',\,p''q''}} \longrightarrow \Omega^{\bullet+1}_{\mathrm{aug}}\left(\mathcal{F}^\times_{\left(\phi'_{p'q'}(x'),\, \phi''_{p''q''}(x'')\right)}\right)_{\phi^\times_{p'q',\,p''q''}}.
\]

\noindent (2) In an analogous manner, the canonical projections
\[
\mathrm{pr}' \colon W^\times_{(x', x'')} \to W'_{x'}
\quad\text{and}\quad
\mathrm{pr}'' \colon W^\times_{(x', x'')} \to W''_{x''}
\]
satisfy
\[
d\mathrm{pr}'\left(\ker \alpha^\times_{W^\times_{(x', x'')}}\right) \subset \ker \alpha'_{W'_{x'}},
\qquad
d\mathrm{pr}''\left(\ker \alpha^\times_{W^\times_{(x', x'')}}\right) \subset \ker \alpha''_{W''_{x''}}.
\]
We therefore obtain the pullback maps between the respective foliation de Rham complexes:
\[
\Omega^{\bullet+1}\left(\mathcal{F}'_{\phi'_{p'q'}(x')}\right)_{\phi'_{p'q'} \times \phi''_{p''q''}} \longrightarrow \Omega^{\bullet+1}\left(\mathcal{F}^\times_{\left(\phi'_{p'q'}(x'),\, \phi''_{p''q''}(x'')\right)}\right)_{\phi^\times_{p'q',\,p''q''}},
\]
\[
\Omega^{\bullet+1}\left(\mathcal{F}''_{\phi''_{p''q''}(x'')}\right)_{\phi'_{p'q'} \times \phi''_{p''q''}} \longrightarrow \Omega^{\bullet+1}\left(\mathcal{F}^\times_{\left(\phi'_{p'q'}(x'),\, \phi''_{p''q''}(x'')\right)}\right)_{\phi^\times_{p'q',\,p''q''}}.
\]
Taking their direct sum and passing to augmentations provides the morphism
\[
\begin{split}
\Omega^{\bullet+1}_{\mathrm{aug}}\left(\mathcal{F}'_{\phi'_{p'q'}(x')}\right)_{\phi'_{p'q'} \times \phi''_{p''q''}} \oplus \Omega^{\bullet+1}_{\mathrm{aug}}\left(\mathcal{F}''_{\phi''_{p''q''}(x'')}\right)_{\phi'_{p'q'} \times \phi''_{p''q''}} \\
\longrightarrow \Omega^{\bullet+1}_{\mathrm{aug}}\left(\mathcal{F}^\times_{\left(\phi'_{p'q'}(x'),\, \phi''_{p''q''}(x'')\right)}\right)_{\phi^\times_{p'q',\,p''q''}}.
\end{split}
\]

\noindent (3) and (4) are constructed in the same manner as (2) and (1), respectively.

The morphisms pointing in the opposite direction, namely (2) and (4), become homotopy invertible once extended to the augmentations. Inverting these arrows up to homotopy in the sequence \eqref{gageeeeee} yields the required de Rham $L_\infty[1]$-morphism
\[
\widehat{\phi}_{p'q',\, p''q'',\, (x',x'')}^{\mathrm{f}, \mathrm{dR,c}} \colon \Omega_{\mathrm{aug}}^{\bullet+1}\left(\mathcal{F}_{\left(\phi'_{p'q'}(x'),\, \phi''_{p''q''}(x'')\right)}^\mathrm{f}\right)_{\phi^\mathrm{f}_{p'q',\,p''q''}} \longrightarrow \Omega_{\mathrm{aug}}^{\bullet+1}\left(\mathcal{F}_{(x', x'')}^\mathrm{f}\right).
\]

Before proceeding to the next lemma concerning the quasi-isomorphism property of $\widehat{\phi}^{\mathrm{f,c}}_{p'q',p''q'', (x', x'')}$, we introduce the following notation for simplicity of presentation:
\[
\begin{split}
\overline{\mathrm{A}}_q &:= \bigwedge\nolimits^{-\bullet}\Gamma\left(\pi^{\prime*} E'_{q'} \oplus \pi^{\prime\prime*} E''_{q''}\big|_{W^\mathrm{f}_{\left(\phi'_{p'q'}(x'),\phi''_{p''q''}(x'')\right)}}\right)_{\phi'_{p'q'}\times\phi''_{p''q''}}, \\
\mathrm{A}_q &:= \bigwedge\nolimits^{-\bullet}\Gamma\left(\pi^{\prime*} E'_{q'} \oplus \pi^{\prime\prime*} E''_{q''}\big|_{W^\mathrm{f}_{\left(\phi'_{p'q'}(x'),\phi''_{p''q''}(x'')\right)}}\right), \\
\mathrm{A}_p &:= \bigwedge\nolimits^{-\bullet}\Gamma\left(\pi^{\prime*} E'_{p'} \oplus \pi^{\prime\prime*} E''_{p''}\big|_{W^\mathrm{f}_{(x', x'')}}\right);
\end{split}
\]

\[
\begin{split}
\overline{\mathrm{B}}_q &:= \bigwedge\nolimits^{-\bullet}\Gamma\left(\pi^{\prime\prime*}E_{q}\big|_{W^\mathrm{f}_{\left(\phi'_{p'q'}(x'),\phi''_{p''q''}(x'')\right)}}\right)_{\phi'_{p'q'}\times \phi''_{p''q''}}, \\
\mathrm{B}_q &:= \bigwedge\nolimits^{-\bullet}\Gamma\left(\pi^{\prime\prime*}E_{q}\big|_{W^\mathrm{f}_{\left(\phi'_{p'q'}(x'),\phi''_{p''q''}(x'')\right)}}\right), \\
\mathrm{B}_p &:= \bigwedge\nolimits^{-\bullet}\Gamma\!\left(\pi^{\prime*} E'_{p'} \oplus \pi^{\prime\prime*}E''_{p''}\big|_{W'_{x'} \times W''_{x''}}\right);
\end{split}
\]

\[
\begin{split}
\overline{\mathrm{C}}_q &:= \bigwedge\nolimits^{-\bullet}\Gamma\!\left(\pi^{\prime*} E'_{q'} \oplus \pi^{\prime\prime*} E''_{q''}\big|_{W'_{\phi'_{p'}(x')}\times W''_{\phi''_{p''}(x'')}}\right)_{\phi'_{p'q'}\times\phi''_{p''q''}}, \\
\mathrm{C}_q &:= \bigwedge\nolimits^{-\bullet}\Gamma\!\left(\pi^{\prime*} E'_{q'} \oplus \pi^{\prime\prime*} E''_{q''}\big|_{W'_{\phi'_{p'}(x')}\times W''_{\phi''_{p''}(x'')}}\right), \\
\mathrm{C}_p &:= \bigwedge\nolimits^{-\bullet}\Gamma\!\left(\pi^{\prime*} E'_{p'} \oplus \pi^{\prime\prime*}E''_{p''}\big|_{W'_{x'} \times W''_{x''}}\right);
\end{split}
\]

\[
\begin{split}
\overline{\mathrm{D}}_q &:= \bigwedge\nolimits^{-\bullet}\Gamma\!\left(\pi^{\prime*} E'_{q'} \oplus \pi^{\prime\prime*} \left[E''_{q''}\right]\big|_{W^\mathrm{f}_{\left(\phi'_{p'q'}(x'),\phi''_{p''q''}(x'')\right)}}\right)_{\phi'_{p'q'}\times\phi''_{p''q''}}, \\
\mathrm{D}_q &:= \bigwedge\nolimits^{-\bullet}\Gamma\!\left(\pi^{\prime*} E'_{q'} \oplus \pi^{\prime\prime*} \left[E''_{q''}\right]\big|_{W^\mathrm{f}_{\left(\phi'_{p'q'}(x'),\phi''_{p''q''}(x'')\right)}}\right), \\
\mathrm{D}_p &:= \bigwedge\nolimits^{-\bullet}\Gamma\!\left(\pi^{\prime*} E'_{p'} \oplus \pi^{\prime\prime*}\left[E''_{p''}\right]\big|_{W^\mathrm{f}_{(x', x'')}}\right).
\end{split}
\]

\begin{lem}
$\widehat{\phi}^{\mathrm{f,c}}_{p'q',p''q'', (x', x'')}$ is a quasi-isomorphic $L_\infty[1]$-morphisms.
\end{lem}

\begin{proof}
\noindent(\textit{Koszul part})

In the following proof, we use the abbreviations $\mathrm{A}, \mathrm{B}, \mathrm{B'}, \mathrm{C},$ and $\mathrm{D}$ for simplicity in each respective case.

\medskip
\noindent(1) (\textit{$\widehat{\phi}^{\mathrm{c}}_\mathrm{A}$ is a quasi-isomorphism.})

We define the map $\widehat{\phi}^{\mathrm{c}}_\mathrm{A} \colon \overline{\mathrm{A}}_q \rightarrow \mathrm{A}_p$ given by
\[
\begin{split}
&\bigwedge\nolimits^{-\bullet}\Gamma\left(\pi^{\prime*} E'_{q'} \oplus \pi^{\prime\prime*} E''_{q''}\big|_{W^\mathrm{f}_{\left(\phi'_{p'q'}(x'),\phi''_{p''q''}(x'')\right)}}\right)_{\phi'_{p'q'}\times\phi''_{p''q''}} \\
&\qquad\qquad\qquad\xrightarrow{\widehat{\phi}^{\mathrm{K,c}}_{p'q', p''q'', (x',x'')}} \bigwedge\nolimits^{-\bullet}\Gamma\left(\pi^{\prime*} E'_{p'} \oplus \pi^{\prime\prime*} E''_{p''}\big|_{W^\mathrm{f}_{(x', x'')}}\right)
\end{split}
\]
as the restriction to $W^\mathrm{f}_{\left(\phi'_{p'q'}(x'),\phi''_{p''q''}(x'')\right)} \subset W_{\left(\phi'_{p'q'}(x'),\phi''_{p''q''}(x'')\right)}$ of the morphism
\[
\begin{split}
&\bigwedge\nolimits^{-\bullet}\Gamma\!\left(\pi^{\prime*} E'_{q'} \oplus \pi^{\prime\prime*} E''_{q''}\big|_{W'_{\phi'_{p'}(x')}\times W''_{\phi''_{p''}(x'')}}\right)_{\phi'_{p'q'}\times\phi''_{p''q''}} \\
&\qquad\qquad\qquad\xrightarrow{\widehat{\phi}^{\mathrm{K,c}}_{p'q', p''q'', (x',x'')}} \bigwedge\nolimits^{-\bullet}\Gamma\!\left(\pi^{\prime*} E'_{p'} \oplus \pi^{\prime\prime*}E''_{p''}\big|_{W'_{x'} \times W''_{x''}}\right).
\end{split}
\]
Since the Koszul differentials on both section spaces commute with the restriction maps, the quasi-isomorphism property of $\widehat{\phi}^{\mathrm{K,c}}_{p'q',p''q'',(x',x'')}$ implies that of $\widehat{\phi}^{\mathrm{c}}_\mathrm{A}$.

\noindent(2) (\textit{The quasi-isomorphism $\widehat{\phi}^{\mathrm{c}}_\mathrm{B}$ is induced by $\widehat{\phi}^{\mathrm{c}}_\mathrm{C}$.})

The morphism $\widehat{\phi}^{\mathrm{c}}_\mathrm{B} \colon \overline{\mathrm{B}} \rightarrow \mathrm{B}$ given by
\[
\bigwedge\nolimits^{-\bullet}\Gamma\left(\pi^{\prime\prime*}E_{q}\big|_{W^\mathrm{f}_{\left(\phi'_{p'q'}(x'),\phi''_{p''q''}(x'')\right)}}\right)_{\phi'_{p'q'}\times \phi''_{p''q''}}
\xrightarrow{\widehat{\phi}^{\mathrm{K,c}}_{pq,x}}
\bigwedge\nolimits^{-\bullet}\Gamma\left(\pi^{\prime\prime*}E_p\big|_{W^\mathrm{f}_{(x', x'')}}\right)
\]
is induced by $\widehat{\phi}^{\mathrm{c}}_\mathrm{C} \colon \overline{\mathrm{C}} \rightarrow \mathrm{C}$ given by
\[
\bigwedge\nolimits^{-\bullet}\Gamma\left(E_{q}\big|_{W_{\phi_{pq}(x)}}\right)_{\phi_{pq}}
\xrightarrow{\widehat{\phi}^{\mathrm{K,c}}_{pq,x}}
\bigwedge\nolimits^{-\bullet}\Gamma\left(E_p\big|_{W^\mathrm{f}_{x}}\right)
\]
via pullback along the projection $\pi^{\prime\prime}$.

\noindent(3) (\textit{The quasi-isomorphism $\widehat{\phi}^{\mathrm{c}}_{\mathrm{B}'}$ is induced by $\widehat{\phi}^{\mathrm{c}}_\mathrm{B}$.})

We make use of the fibration property of $\overline{G}$. Recall that, in Definition~\ref{adfaf}, $[E''_{p''}]$ is the quotient with respect to the bundle embedding $\overline{\mathrm{em}}_p$.

It is straightforward to see that the Koszul $L_{\infty}[1]$-component for $\overline{G}$ yields an $L_{\infty}[1]$-morphism
\[
\widehat{\gamma}_q \colon \mathrm{B}_q \longrightarrow \mathrm{A}_q
\]
given by
\[
\bigwedge\nolimits^{-\bullet}\Gamma\left(\pi^{\prime\prime*}E_q\big|_{W^\mathrm{f}_{(x', x'')}}\right) \longrightarrow \bigwedge\nolimits^{-\bullet}\Gamma\left(\pi^{\prime*} E'_{q'} \oplus \pi^{\prime\prime*} E''_{q''}\big|_{W^\mathrm{f}_{(x', x'')}}\right)
\]
between the section spaces of the pullback bundles. Then the fibration property of $G$ implies that there exists a decomposition
\begin{equation}\label{qvvagag}
\widehat{\gamma}_q \colon \mathrm{B}_q \xrightarrow{\ (\widehat{\eta}^{\vee}_q)^{-1},\ \simeq\ } \mathrm{B}'_q \xrightarrow{\ \widehat{\underline{\eta}}_q\ } \mathrm{A}_q,
\end{equation}
where $\mathrm{B}'_q$ is defined as in \eqref{bhuaiuhnn}. Here $\mathrm{B}'_q$, $\widehat{\underline{\eta}}_q$, and $\widehat{\eta}^{\vee}_q$ arise from the bundle embedding pair $(\overline{\mathrm{em}}_q, \overline{\mathrm{em}}_q')$ in Definition~\ref{adfaf}\,(iii).

Similarly, we obtain
\[
\widehat{\gamma}_p \colon \mathrm{B}_p \longrightarrow \mathrm{B}'_p \longrightarrow \mathrm{A}_p
\]
with the decomposition
\[
\widehat{\gamma}_p = \widehat{\underline{\eta}}_p \circ (\widehat{\eta}^{\vee}_p)^{-1}
\]
for some $\mathrm{B}'_p$.

In fact, \eqref{qvvagag} composed with the quotient map $\mathrm{A}_q \rightarrow \mathrm{D}_q$ naturally induces maps on the completions:
\[
\overline{\mathrm{B}}_q \xleftarrow{\ \widehat{\overline{\eta}}_q^{\vee},\ \simeq\ } \overline{\mathrm{B}}'_q \xrightarrow{\ \widehat{\underline{\overline{\eta}}}_q\ } \overline{\mathrm{A}}_q \longrightarrow \overline{\mathrm{D}}_q,
\]
where $\overline{\mathrm{B}}_q \xleftarrow{\ \widehat{\overline{\eta}}_q^{\vee},\ \simeq\ } \overline{\mathrm{B}}'_q$ is a quasi-isomorphism, following the construction of quasi-isomorphisms between completed section spaces from compatible bundle embeddings in Subsection~5.2.

Note that we have a short exact sequence
\[
0 \longrightarrow \overline{\mathrm{B}}'_q \xrightarrow{\ \widehat{\underline{\overline{\eta}}}_q\ } \overline{\mathrm{A}}_q \longrightarrow \overline{\mathrm{D}}_q \longrightarrow 0
\]
by Definition~\ref{adfaf}\,(iii) and the definition of the quotient bundle $[E''_{q''}]$. In the same way, for $p$ (without completion), we also have another short exact sequence:
\[
0 \longrightarrow {\mathrm{B}}'_p \xrightarrow{\ \widehat{\underline{\eta}}_p\ } {\mathrm{A}}_p \longrightarrow {\mathrm{D}}_p \longrightarrow 0.
\]

Setting
\[
\widehat{\phi}^{\mathrm{c}}_{\mathrm{B}'} := (\widehat{\eta}_p^\vee)^{-1} \circ \widehat{\phi}^{\mathrm{c}}_\mathrm{B} \circ \widehat{\overline{\eta}}_q^\vee,
\]
we obtain the following \textit{homotopy commutative} diagram:
\begin{equation}\label{abbgh}
\begin{tikzcd}[column sep=large, row sep=large]
0 \ar[r] & \overline{\mathrm{B}}'_q \ar[r, "\widehat{\underline{\overline{\eta}}}_q"] \ar[d, swap, "\widehat{\overline{\eta}}^{\vee}_q"'] \ar[ddd, bend right=20, swap, "\widehat{\phi}^{\mathrm{c}}_{\mathrm{B}'}"] & \overline{\mathrm{A}}_q \ar[r, "\widehat{\overline{\kappa}}_q"] \ar[ddd, "\widehat{\phi}^{\mathrm{c}}_\mathrm{A}"] & \overline{\mathrm{D}}_q \ar[r] \ar[ddd, "\widehat{\phi}^{\mathrm{c}}_\mathrm{D}"] & 0 \\
{} & \overline{\mathrm{B}}_q \ar[d, "\widehat{\phi}^{\mathrm{c}}_{\mathrm{B}}"] & {} & {} & {} \\
{} & \mathrm{B}_p \ar[d, swap, "(\widehat{\eta}^{\vee}_p)^{-1}"'] & {} & {} & {} \\
0 \ar[r] & \mathrm{B}'_p \ar[r, "\widehat{\underline{\eta}}_p"] & \mathrm{A}_p \ar[r, "\widehat{\kappa}_p"] & \mathrm{D}_p \ar[r] & 0.
\end{tikzcd}
\end{equation}
The homotopy commutativity of \eqref{abbgh} follows from the short exactness of the rows and the decomposition axiom of the Koszul $L_{\infty}$-coordinate changes for fibrations in Definition~\ref{adfaf}\,(iv), together with the quasi-isomorphisms $\widehat{\overline{\eta}}^{\vee}_q$ and $\widehat{\eta}^{\vee}_p$.

\noindent(4) (\textit{$\widehat{\phi}^{\mathrm{c}}_\mathrm{D}$ is a quasi-isomorphism.})

The morphism $\widehat{\phi}^{\mathrm{c}}_\mathrm{D} \colon \overline{\mathrm{D}}_q \rightarrow \mathrm{D}_p$, given explicitly by
\[
\begin{split}
&\bigwedge\nolimits^{-\bullet}\Gamma\!\left(\pi^{\prime*} E'_{q'} \oplus \pi^{\prime\prime*} \left[E''_{q''}\right]\big|_{W^\mathrm{f}_{\left(\phi'_{p'q'}(x'),\phi''_{p''q''}(x'')\right)}}\right)_{\phi'_{p'q'}\times\phi''_{p''q''}}\\ 
&\quad\quad\quad\quad\quad\quad\quad\quad\quad\quad\quad   \xrightarrow{\widehat{\phi}^{\mathrm{K,c}}_{p'q', p''q'', (x',x'')}} \bigwedge\nolimits^{-\bullet}\Gamma\!\left(\pi^{\prime*} E'_{p'} \oplus \pi^{\prime\prime*}\left[E''_{p''}\right]\big|_{W^\mathrm{f}_{(x', x'')}}\right),
\end{split}
\]
is the Koszul coordinate change $\widehat{\phi}^{\mathrm{f,c}}_{p'q',p''q'', (x', x'')}$ for the fiber product, and we show that it is a quasi-isomorphism.

By taking cohomology in \eqref{abbgh} with the quasi-isomorphisms $\widehat{\phi}^{\mathrm{c}}_\mathrm{A}$, $\widehat{\phi}^{\mathrm{c}}_\mathrm{B}$, and $\widehat{\eta}^{\vee}$, we obtain the following \textit{commutative} diagram with exact rows:
\begin{equation}\label{bkllass}
\begin{tikzcd}[column sep=large, row sep=large]
\cdots \ar[r] & H^*(\overline{\mathrm{B}}'_q) \ar[r, "\overline{\gamma}"] \ar[d, "\simeq"'] & H^*(\overline{\mathrm{A}}_q) \ar[r, "\overline{\kappa}"] \ar[d, "\simeq"'] & H^*(\overline{\mathrm{D}}_q) \ar[r, "\overline{\delta}"] \ar[d, "\widehat{\phi}^{\mathrm{c}}_{\mathrm{D}*}"] & H^*(\overline{\mathrm{B}'}_q) \ar[r, "\overline{\gamma}"] \ar[d, "\simeq"'] & H^*(\overline{\mathrm{A}}_q) \ar[r] \ar[d, "\simeq"'] & \cdots \\
\cdots \ar[r] & H^*(\mathrm{B'}_p) \ar[r, "\gamma"] & H^*(\mathrm{A}_p) \ar[r, "\kappa"] & H^*(\mathrm{D}_p) \ar[r, "\delta"] & H^*(\mathrm{B}'_p) \ar[r, "\gamma"] & H^*(\mathrm{A}_p) \ar[r] & \cdots.
\end{tikzcd}
\end{equation}
Splitting these long exact sequences into short exact sequences yields:
\[
\begin{tikzcd}[column sep=large, row sep=large]
0 \ar[r] \ar[d, "\simeq"'] & \ker \overline{\delta} \ar[r, hook] \ar[d, "(\mathrm{i})"] & H^*(\overline{\mathrm{D}}_q) \ar[r] \ar[d, "\widehat{\phi}^{\mathrm{c}}_{\mathrm{D}*}"] & \operatorname{Im} \overline{\delta} \ar[r] \ar[d, "(\mathrm{ii})"] & 0 \ar[d, "\simeq"'] \\
0 \ar[r] & \ker \delta \ar[r, hook] & H^*(\mathrm{D}_p) \ar[r] & \operatorname{Im} \delta \ar[r] & 0.
\end{tikzcd}
\]
We then apply the five lemma to verify that $\widehat{\phi}^{\mathrm{c}}_{\mathrm{D}*}$ is an isomorphism. Therefore, it suffices to show that $(\mathrm{i})$ and $(\mathrm{ii})$ are isomorphisms.

\begin{enumerate}[label = (\roman*)]
\item The exactness of the rows in \eqref{bkllass} yields the isomorphisms:
\[
\ker \overline{\delta} \simeq \operatorname{Im} \overline{\kappa} \quad \text{and} \quad \ker \delta \simeq \operatorname{Im} \kappa.
\]
Applying the four lemma to the commutative diagram
\[
\begin{tikzcd}[column sep=large, row sep=large]
H^*(\overline{\mathrm{B}}'_q) \ar[r, "\overline{\gamma}"] \ar[d, "\simeq"'] & H^*(\overline{\mathrm{A}}_q) \ar[d, "\simeq"'] \ar[r, "\overline{\kappa}"] & \operatorname{Im} \overline{\kappa} \ar[d] \ar[r] & 0 \ar[d, "\simeq"'] \\
H^*(\mathrm{B}'_p) \ar[r, "\gamma"] & H^*(\mathrm{A}_p) \ar[r, "\kappa"] & \operatorname{Im} \kappa \ar[r] & 0
\end{tikzcd}
\]
with exact rows, we obtain the isomorphism $\operatorname{Im} \overline{\kappa} \simeq \operatorname{Im} \kappa$.

\item The exactness of the rows in \eqref{bkllass} yields the isomorphisms:
\[
\operatorname{Im} \overline{\delta} \simeq \ker \overline{\gamma} \quad \text{and} \quad \operatorname{Im} \delta \simeq \ker \gamma.
\]
From the commutative diagram
\[
\begin{tikzcd}[column sep=large, row sep=large]
H^*(\overline{\mathrm{B}}'_q) \ar[r, "\overline{\gamma}"] \ar[d, "\simeq"'] & H^*(\overline{\mathrm{A}}_q) \ar[d, "\simeq"'] \\
H^*(\mathrm{B}'_p) \ar[r, "\gamma"] & H^*(\mathrm{A}_p)
\end{tikzcd}
\]
with isomorphic vertical maps, we obtain the isomorphism $\ker \overline{\gamma} \simeq \ker \gamma$ induced by the isomorphism of the left column.
\end{enumerate}

\noindent\textit{(de Rham part)}

The morphism $\widehat{\phi}^{\mathrm{f}, \mathrm{dR,c}}_{p'q',p''q'',(x',x'')} \colon \Omega_{\mathrm{aug}}^{\bullet+1}\left(\mathcal{F}^{\mathrm{f}}_{(\phi'_{p'q'}(x'), \phi''_{p''q''}(x''))}\right)_{\phi^\mathrm{f}_{p'q'}} \longrightarrow \Omega_{\mathrm{aug}}^{\bullet+1}\left(\mathcal{F}^{\mathrm{f}}_{(x', x'')}\right)$ is a quasi-isomorphic $L_\infty[1]$-morphism, as it is a map between acyclic $L_{\infty}[1]$-algebras.
\end{proof}

\subsection{Fiber product Kuranishi spaces}

In this subsection, we show that the data from the previous subsection
provide a fiber product as a Kuranishi space, that is, as an object in
the category $\mathbf{Kur}$. We also interpret the canonical projection
maps from this object to each factor as morphisms in $\mathbf{Kur}$.
This discussion will be used to show that morphisms with manifold
targets always give rise to \textit{homotopy} fiber products. Roughly
speaking, this means that the construction involves some noncanonical
choices, while the resulting object is well defined up to weak
equivalence.

We consider the Kuranishi atlas
$\widehat{\mathcal{U}}
:=
\left(
\left\{\widehat{\mathcal{U}}_p\right\},
\left\{\widehat{\Phi}_{pq}\right\}
\right)$.
Let
$F: \mathfrak{X}' \rightarrow \mathfrak{X}$
and
$G: \mathfrak{X}'' \rightarrow \mathfrak{X}$
be morphisms of Kuranishi spaces, and assume that $G$ is a fibration.
In this situation, we define their \textit{fiber product} to be the
Kuranishi space
\[
\mathfrak{X}' {}_F \!\times_G \mathfrak{X}''
:=
\left(
X' {}_f\!\times_g X'',
\left[\widehat{\mathcal{U}}^\mathrm{f}\right]
\right),
\]
where $\widehat{\mathcal{U}}^\mathrm{f}$ is the atlas constructed in
the previous subsection.

We further define a morphism
\[
\mathfrak{X}'{}_F \!\times_G \mathfrak{X}''
\xrightarrow{\mathrm{Pr}'}
\mathfrak{X}'
\]
determined by the following pre-morphism:
\[
\overline{\mathrm{Pr}}'
:=
\left(
\widehat{\mathcal{U}}^\mathrm{f},
\widehat{\mathcal{U}}',
\mathrm{pr}',
\mathrm{pr}'_{(p',p'')},
\widehat{\mathrm{pr}}'_{(p',p''),(x',x'')}
\right).
\]
The underlying maps are the projections to the first factors:
\[
\begin{split}
&\mathrm{pr}':
X' {}_f\!\times_g X''
\twoheadrightarrow X',
\\
&\mathrm{pr}_{(p',p'')}' :
U'_{p'} {}_{f_{p'}}\!\times_{g_{p''}} U''_{p''}
\twoheadrightarrow U_{p'}'.
\end{split}
\]
Since $g_{p''}$ is assumed to be a surjective submersion, the base
projection $\mathrm{pr}_{(p',p'')}'$ is surjective. Hence,
\[
\mathcal{C}_{p',x',\mathrm{pr}_{(p',p'')}'}'
=
\mathcal{C}_{p',x'}'.
\]

For each $(p',p'')\in X'{}_f\!\times_g X''$ and
$(x',x'')\in(s_{p',p''}^{\mathrm{f}})^{-1}(0)$, we have
\[
(x',x'')
\in
(s'_{p'})^{-1}(0)
{}_{f_{p'}}\!\times_{g_{p''}}
(s''_{p''})^{-1}(0).
\]
The corresponding $L_{\infty}[1]$-algebra is
\[
\mathcal{C}_{(p',p''),(x',x'')}^\mathrm{f}
:=
\bigwedge\nolimits^{-\bullet}
\Gamma\left(
(E_{p',p''}^{\mathrm{f}})^*
\big|_{{W}_{(x',x'')}^\mathrm{f}}
\right)
\oplus
\Omega_{\mathrm{aug}}^{\bullet+1}
(\mathcal{F}_{(x',x'')}^\mathrm{f}).
\]

The obstruction bundle is
\[
{E}^\mathrm{f}_{p',p''}
:=
\pi^{\prime*}{E}'_{p'}
{}_{f_{p'}}\!\times_{g_{p''}}
\pi^{\prime\prime*}{E}''_{p''}.
\]

Moreover,
\[
\begin{split}
&\Gamma\Bigl(
E^{\mathrm{f}*}_{p',p''}
\big|_{{W}_{(x',x'')}^\mathrm{f}}
\Bigr)
\\
&\qquad:=
\Gamma\Bigl(
(\pi'{}^{*}E'_{p'}
{}_{f_{p'}}\!\times_{g_{p''}}
\pi''{}^{*}E''_{p''})^{*}
\big|_{{W}_{(x',x'')}^\mathrm{f}}
\Bigr)
\\
&\qquad\simeq
\Gamma\Bigl(
(\pi'{}^{*}E'_{p'}
\oplus
\pi''{}^{*}E''_{p''})^{*}
\big|_{W'_{x'}\times W''_{x''}}
\Bigr)
\\
&\qquad\simeq
\Gamma\Bigl(
\pi'{}^{*}(E'_{p'})^{*}
\oplus
\pi''{}^{*}(E''_{p''})^{*}
\big|_{W'_{x'}\times W''_{x''}}
\Bigr).
\end{split}
\]

The target $L_{\infty}[1]$-algebra for the base projection is
\[
\mathcal{C}_{p',x',\mathrm{pr}_{(p',p'')}'}'
:=
\bigwedge\nolimits^{-\bullet}
\Gamma\left(
E_{p'}^{\prime*}
\big|_{{W}_{x'}}
\right)_{\mathrm{pr}_{(p',p'')}'}
\oplus
\Omega_{\mathrm{aug}}^{\bullet+1}
(\mathcal{F}'_{x'})_{\mathrm{pr}_{(p',p'')}'}.
\]

Its $L_{\infty}[1]$-component map
\[
\widehat{\mathrm{pr}}^{\prime\mathrm{c}}_{(p',p''),(x',x'')}:
\mathcal{C}_{p',x',\mathrm{pr}'_{(p',p'')}}'
\rightarrow
\mathcal{C}_{(p',p''),(x',x'')}^\mathrm{f}
\]
is defined as follows.

For the Koszul part, set
\begin{equation}\label{bzqqqqgbbbq}
\begin{split}
\widehat{\mathrm{pr}}^{\prime\mathrm{K}}_{(p',p''),(x',x''),1} :
\bigwedge\nolimits^{-\bullet}
\Gamma\left(
(E_{p'}')^{*}
\big|_{{W}_{x'}}
\right)
&\rightarrow
\bigwedge\nolimits^{-\bullet}
\Gamma\left(
(E_{p',p''}^{\mathrm{f}})^*
\big|_{{W}_{(x',x'')}^\mathrm{f}}
\right)
\\
&\simeq
\bigwedge\nolimits^{-\bullet}
\Gamma\Bigl(
\pi'{}^{*}(E'_{p'})^{*}
\oplus
\pi''{}^{*}(E''_{p''})^{*}
\big|_{W'_{x'}\times W''_{x''}}
\Bigr),
\\
\widehat{\mathrm{pr}}^{\prime\mathrm{K}}_{(p',p''),(x',x''),1}(a')
&:=
\left(\pi^{\prime*}a',0\right).
\end{split}
\end{equation}
We set
\[
\widehat{\mathrm{pr}}^{\prime\mathrm{K}}_{(p',p''),(x',x''),m}
:=
\begin{cases}
\widehat{\mathrm{pr}}^{\prime\mathrm{K}}_{(p',p''),(x',x''),1}
& \text{if }m=1,\\
0
& \text{if }m\geq2.
\end{cases}
\]
It is straightforward to verify that
$\widehat{\mathrm{pr}}^{\prime\mathrm{K}}_{(p',p''),(x',x'')}$
with vanishing higher-order components is an
$L_{\infty}[1]$-morphism. In other words, it is a chain map. We omit
the verification.

For the de Rham part, observe that for
$X\in\ker\alpha^{\mathrm{f}}$, we have
\begin{equation}\label{ahbbbsxxxx}
\begin{split}
(\mathrm{pr}'^*_{p'}\alpha'
-
\mathrm{pr}''^*_{p''}\alpha'')(X,\cdot)
&=
\alpha'(d\mathrm{pr}'_{p'}(X,\cdot))
-
\alpha''(d\mathrm{pr}''_{p''}(X,\cdot))
\\
&=
\alpha'(X_1,d\mathrm{pr}'_{p'}(\cdot))
-
\alpha''(X_2,d\mathrm{pr}''_{p''}(\cdot))
=
0,
\end{split}
\end{equation}
where we use the decomposition
$X=(X_1,X_2)$ induced by the projections, with
\[
X_1:=d\mathrm{pr}'_{p'}(X),
\qquad
X_2:=d\mathrm{pr}''_{p''}(X).
\]

Since $X$ is arbitrary, \eqref{ahbbbsxxxx} implies that
$X_1\in\ker\alpha'$ and $X_2\in\ker\alpha''$. Consequently,
\[
\mathrm{pr}'_{p'}:
U^\mathrm{f}_{p',p''}\rightarrow U'_{p'},
\qquad
\mathrm{pr}''_{p''}:
U^\mathrm{f}_{p',p''}\rightarrow U''_{p''}
\]
satisfy
\[
d\mathrm{pr}'_{p'}
(\ker\alpha^\mathrm{f}_{x',x''})
\subset
\ker\alpha'_{x'},
\qquad
d\mathrm{pr}''_{p''}
(\ker\alpha^\mathrm{f}_{x',x''})
\subset
\ker\alpha''_{x''}.
\]
Hence, they induce pullback maps between the corresponding foliation
de Rham complexes, which we denote by
$\widehat{\mathrm{pr}}'^{\mathrm{dR}}_{p',x'}$ and
$\widehat{\mathrm{pr}}''^{\mathrm{dR}}_{p'',x''}$, respectively:
\begin{equation}\label{bwbwbwb}
\begin{split}
\widehat{\mathrm{pr}}'^{\mathrm{dR}}_{p',x'} :
\Omega^{\bullet+1}_{\mathrm{aug}}(\mathcal{F}'_{x'})
&\rightarrow
\Omega^{\bullet+1}_{\mathrm{aug}}
(\mathcal{F}^\mathrm{f}_{(x',x'')}),\\
\widehat{\mathrm{pr}}''^{\mathrm{dR}}_{p'',x''} :
\Omega^{\bullet+1}_{\mathrm{aug}}(\mathcal{F}''_{x''})
&\rightarrow
\Omega^{\bullet+1}_{\mathrm{aug}}
(\mathcal{F}^\mathrm{f}_{(x',x'')}).
\end{split}
\end{equation}

Similarly,
\[
\mathfrak{X}'{}_F\!\times_G\mathfrak{X}''
\xrightarrow{\mathrm{Pr}''}
\mathfrak{X}''
\]
is defined by
\[
\overline{\mathrm{Pr}}''
:=
\left(
\widehat{\mathcal{U}}^\mathrm{f},
\widehat{\mathcal{U}}'',
\mathrm{pr}'',
\mathrm{pr}''_{(p',p'')},
\widehat{\mathrm{pr}}''_{(p',p''),(x',x'')}
\right),
\]
where $\mathrm{pr}''$, $\mathrm{pr}''_{(p',p'')}$, and
$\widehat{\mathrm{pr}}''_{(p',p''),(x',x'')}$ are given as follows.

The first two maps
\[
\begin{split}
&\mathrm{pr}'':
X'{}_f\!\times_g X''
\twoheadrightarrow X'',
\\
&\mathrm{pr}''_{(p',p'')}:
U'_{p'}{}_{f_{p'}}\!\times_{g_{p''}}U''_{p''}
\twoheadrightarrow U''_{p''}
\end{split}
\]
are the projections to the second factors, as before.

The $L_{\infty}[1]$-component
\[
\widehat{\mathrm{pr}}''_{p'',x''}:
\mathcal{C}''_{p'',x''}
\rightarrow
\mathcal{C}^{\mathrm{f}}_{(p',p''),(x',x'')}
\]
is defined differently in this case. Its de Rham part is given in
\eqref{bwbwbwb}, while the Koszul part
$\widehat{\mathrm{pr}}''^{\mathrm{K}}_{p'',x''}$ is defined as follows.

\[
\widehat{\mathrm{pr}}''^{\mathrm{K}}_{p'',x''}:
\bigwedge\nolimits^{-\bullet}
\Gamma(E''_{p''}|_{W''_{x''}})
\rightarrow
\bigwedge\nolimits^{-\bullet}
\Gamma(E^\mathrm{f}_{(p',p'')}|_{W^\mathrm{f}_{(x',x'')}}).
\]

Set
\[
\widehat{\mathrm{pr}}''^{\mathrm{K}}_{p'',x''}
:=
\left\{
\widehat{\mathrm{pr}}''^{\mathrm{K}}_{p'',x'',m}
\right\}_{m\geq1}.
\]

For a homogeneous element
$a''\in\bigwedge\nolimits^{-i}\Gamma(E''_{p''}|_{W''_{x''}})$,
we define
\begin{equation}\label{xxbswb}
\widehat{\mathrm{pr}}''^{\mathrm{K}}_{p'',x'',m}(a'')
:=
\begin{cases}
\left(
\pi'^*
\left(
\widehat{f}^{\mathrm{K}}_{p,x,1}
\circ
\underline{\widehat{\eta}}^{-1}(\widehat{g}_{p,x})_1
\circ
\left(
\overline{\mathrm{em}}^{-1}_p
\left(
a''|_{\operatorname{Im}(\overline{\mathrm{em}}_p)}
\right)
\right)
\right),
0
\right)
+
(0,\pi''^*[a''])
& \text{if }m=1,
\\
0
& \text{if }m\geq2.
\end{cases}
\end{equation}
Here we use the notation from \eqref{bhuaiuhnn}
\[
\widehat{g}_{p,x}
:=
\widehat{\eta}^{\vee}(\widehat{g}_{p,x})
\circ
\underline{\widehat{\eta}}^{-1}(\widehat{g}_{p,x})
\]
for the $L_{\infty}[1]$-morphisms associated to the fibrations in
(iii) of Definition \ref{adfaf} (a).

By the compatibility of the embedding $\overline{\mathrm{em}}_p$ with
the sections, the definition of the $L_{\infty}$-structure on the
fiber product, and the fact that
$\widehat{f}^{\mathrm{K}}_{p,x,1}$ and
$\underline{\widehat{\eta}}^{-1}(\widehat{g}_{p,x})_1$ are
$L_{\infty}[1]$-morphisms, it is straightforward to verify that
$\widehat{\mathrm{pr}}''^{\mathrm{K}}_{p'',x''}$ indeed defines an
$L_{\infty}[1]$-morphism. In particular, the following relation holds:
\[
\widehat{\mathrm{pr}}''^{\mathrm{K}}_{p'',x'',1}
\left(
\iota_{s''_{p''}}(a'')
\right)
=
\iota_{\pi'^*s'_{p'}\oplus\pi''^*[s''_{p''}]}
\left(
\widehat{\mathrm{pr}}''^{\mathrm{K}}_{p'',x'',1}(a'')
\right)
=
\iota_{s^{\mathrm{f}}_{p',p''}}
\left(
\widehat{\mathrm{pr}}''^{\mathrm{K}}_{p'',x'',1}(a'')
\right).
\]
We omit the verification.

\begin{lem}
The tuples
\[
\begin{split}
\overline{\mathrm{Pr}}'
&:=
\Bigl(
\widehat{\mathcal{U}}^{\mathrm{f}}_{p',p''},
\widehat{\mathcal{U}}'_{p'},
\mathrm{pr}',
\{\,\mathrm{pr}'_{(p',p'')}\,\},
\widehat{\mathrm{pr}}'_{(p',p''),(x',x'')}
\Bigr),
\\
\overline{\mathrm{Pr}}''
&:=
\Bigl(
\widehat{\mathcal{U}}^{\mathrm{f}}_{p',p''},
\widehat{\mathcal{U}}''_{p''},
\mathrm{pr}'',
\{\,\mathrm{pr}''_{(p',p'')}\,\},
\widehat{\mathrm{pr}}''_{(p',p''),(x',x'')}
\Bigr)
\end{split}
\]
are pre-morphisms.
\end{lem}

\begin{proof}
We first consider the case of $\overline{\mathrm{Pr}}'$. For
$p',q'\in X'$ and $p'',q''\in X''$ such that
\[
\operatorname{Im}\psi'_{p'}\cap\operatorname{Im}\psi'_{q'}
\neq\emptyset,
\qquad
\operatorname{Im}\psi''_{p''}\cap\operatorname{Im}\psi''_{q''}
\neq\emptyset,
\]
we verify the required conditions.

\begin{enumerate}[label=(\roman*)]
\item
We have
\[
\phi'_{p'q'}\circ\mathrm{pr}'_{(p',p'')}
=
\mathrm{pr}'_{(q',q'')}
\circ
\phi^{\mathrm{f}}_{p'q',p''q''}
\]
on
\[
(s^{\mathrm{f}}_{(p',p'')})^{-1}(0)
\cap
U^{\mathrm{f}}_{(p',p'')(q',q'')}.
\]
Indeed,
\[
\begin{split}
\phi'_{p'q'}\circ\mathrm{pr}'_{(p',p'')}(x',x'')
&=
\phi'_{p'q'}(x')
\\
&=
\mathrm{pr}'_{(q',q'')}
\bigl(
\phi'_{p'q'}(x'),
\phi''_{p''q''}(x'')
\bigr)
\\
&=
\mathrm{pr}'_{(q',q'')}
\circ
\phi^{\mathrm{f}}_{p'q',p''q''}(x',x'')
\end{split}
\]
for every
\[
(x',x'')
\in
(s^{\mathrm{f}}_{(p',p'')})^{-1}(0)
\cap
U^{\mathrm{f}}_{(p',p'')(q',q'')}.
\]

\item
We have
\[
\psi'_{p'}\circ\mathrm{pr}'_{(p',p'')}
=
\mathrm{pr}\circ\psi^{\mathrm{f}}_{(p',p'')}
\]
on
\[
(s^{\mathrm{f}}_{(p',p'')})^{-1}(0)
\cap
U^{\mathrm{f}}_{(p',p'')(q',q'')}.
\]
Indeed,
\[
\begin{split}
\psi'_{p'}\circ\mathrm{pr}'_{(p',p'')}(x',x'')
&=
\psi'_{p'}(x')
\\
&=
\mathrm{pr}
\bigl(
\psi'_{p'}(x'),
\psi''_{p''}(x'')
\bigr)
\\
&=
\mathrm{pr}\circ
\psi^{\mathrm{f}}_{(p',p'')}(x',x'')
\end{split}
\]
for every
\[
(x',x'')
\in
(s^{\mathrm{f}}_{(p',p'')})^{-1}(0)
\cap
U^{\mathrm{f}}_{(p',p'')(q',q'')}.
\]

\item
For each
\[
(x',x'')
\in
(s^{\mathrm{f}}_{(p',p'')})^{-1}(0)
\cap
U^{\mathrm{f}}_{(p',p'')(q',q'')},
\]
we have
\[
\widehat{\phi}^{\mathrm{f}}_{p'q',p''q'',(x',x'')}
\circ
\widehat{\mathrm{pr}}'_{(q',q''),(x',x'')}
\sim
\widehat{\mathrm{pr}}'_{(p',p''),(x',x'')}
\circ
\widehat{\phi}'_{p'q',x'}.
\]
Indeed,
\[
\begin{split}
&
\widehat{\phi}^{\mathrm{f}}_{p'q',p''q'',(x',x'')}
\circ
\widehat{\mathrm{pr}}'_{(q',q''),(x',x'')}(a,\xi)
\\
&\quad=
\widehat{\phi}^{\mathrm{f}}_{p'q',p''q'',(x',x'')}
\left(
\left(\pi^{\prime*}a,0\right),
\left(\pi^{\prime*}\xi,0\right)
\right)
\\
&\quad=
\left(
\widehat{\phi}^{\mathrm{f},\mathrm{K}}_{p'q',p''q'',(x',x'')}
(\pi^{\prime*}a,0),
\widehat{\phi}^{\mathrm{f},\mathrm{dR}}_{p'q',p''q'',(x',x'')}
(\pi^{\prime*}\xi,0)
\right)
\\
&\quad=
\left(
\left(
\pi^{\prime*}
\left(
\widehat{\phi}^{\prime\mathrm{K}}_{p'q',x'}(a)
\right),
0
\right),
\left(
\pi^{\prime*}
\left(
\widehat{\phi}^{\prime\mathrm{dR}}_{p'q',x'}(\xi)
\right),
0
\right)
\right)
\\
&\quad=
\widehat{\mathrm{pr}}'_{(p',p''),(x',x'')}
\left(
\widehat{\phi}^{\prime\mathrm{K}}_{p'q',x'}(a),
\widehat{\phi}^{\prime\mathrm{dR}}_{p'q',x'}(\xi)
\right)
\\
&\quad=
\widehat{\mathrm{pr}}'_{(p',p''),(x',x'')}
\circ
\widehat{\phi}'_{p'q',x'}(a,\xi).
\end{split}
\]
Thus, the required compatibility follows.
\end{enumerate}

The case of $\overline{\mathrm{Pr}}''$ follows in the same way, except
for the $L_{\infty}$-component
$\widehat{\mathrm{pr}}^{\prime\prime\mathrm{K}}_{p'',x''}$.
In this case, we need to verify
\[
\widehat{\phi}^{\mathrm{f}}_{p'q',p''q'',(x',x'')}
\circ
\widehat{\mathrm{pr}}''_{(q',q''),(x',x'')}
\sim
\widehat{\mathrm{pr}}''_{(p',p''),(x',x'')}
\circ
\widehat{\phi}''_{p''q'',x''}.
\]
This follows from axiom~(iv) of Definition~\ref{adfaf}, which gives the
compatibility between the $L_{\infty}$-component of a fibration and the
coordinate changes, together with the definition in~\eqref{xxbswb}.
\end{proof}

\begin{defn}\label{dfnvfb}
We call the morphisms
\[
\begin{cases}
\mathrm{Pr}' :
\mathfrak{X}' {}_{F}\!\times_{G} \mathfrak{X}''
\rightarrow
\mathfrak{X}',
\\
\mathrm{Pr}'' :
\mathfrak{X}' {}_{F}\!\times_{G} \mathfrak{X}''
\rightarrow
\mathfrak{X}'',
\end{cases}
\]
determined by the above pre-morphisms
$\overline{\mathrm{Pr}}'$ and $\overline{\mathrm{Pr}}''$, the
\emph{canonical projections}.
\end{defn}

Observe that we obtain the commutative diagram
\[
\begin{tikzcd}
\mathfrak{X}' \,{}_{F}\!\times_{G}\, \mathfrak{X}''
\arrow[r,"\mathrm{Pr}'"]
\arrow[d,"\mathrm{Pr}''"']
&
\mathfrak{X}'
\arrow[d,"F"]
\\
\mathfrak{X}''
\arrow[r,"G"']
&
\mathfrak{X}.
\end{tikzcd}
\]

\begin{lem}
We have
\[
F\circ\mathrm{Pr}'
=
G\circ\mathrm{Pr}''.
\]
\end{lem}

\begin{proof}
The morphism $\mathrm{Pr}'$ is represented by the pre-morphism
\[
\overline{\mathrm{Pr}}'
=
\left(
\widehat{\mathcal{U}}^\mathrm{f},
\widehat{\mathcal{U}}',
\mathrm{pr}',
\left\{\mathrm{pr}'_{(p',p'')}\right\},
\left\{
\widehat{\mathrm{pr}}'_{(p',p''),(x',x'')}
\right\}
\right),
\]
while $F$ is represented by
\[
\overline{F}
=
\left(
\underline{\widehat{\mathcal{U}}}',
\widehat{\mathcal{U}},
f,
\left\{f_{p'}\right\},
\left\{\widehat{f}_{p',x'}\right\}
\right).
\]
Thus,
\[
F\circ\mathrm{Pr}'
:=
\left[
\left(
\widehat{\mathcal{U}}^\mathrm{f},
\widehat{\mathcal{U}},
f\circ\mathrm{pr}',
\left\{
f_{p'}\circ\mathrm{pr}'_{(p',p'')}
\right\},
\left\{
\widehat{\mathrm{pr}}'_{(p',p''),(x',x'')}
\circ
\widehat{f}_{p',x'}
\right\}
\right)
\right].
\]

Similarly,
\[
G\circ\mathrm{Pr}''
:=
\left[
\left(
\widehat{\mathcal{U}}^\mathrm{f},
\widehat{\mathcal{U}},
g\circ\mathrm{pr}'',
\left\{
g_{p''}\circ\mathrm{pr}''_{(p',p'')}
\right\},
\left\{
\widehat{\mathrm{pr}}''_{(p',p''),(x',x'')}
\circ
\widehat{g}_{p'',x''}
\right\}
\right)
\right].
\]

By the definition of the fiber product bases, we have
\[
f\circ\mathrm{pr}'
=
g\circ\mathrm{pr}''
\]
on
\[
X' {}_{f}\!\times_{g} X'',
\]
and
\[
f_{p'}\circ\mathrm{pr}'_{(p',p'')}
=
g_{p''}\circ\mathrm{pr}''_{(p',p'')}
\]
on
\[
U'_{p'} {}_{f_{p'}}\!\times_{g_{p''}} U''_{p''}.
\]
These equalities follow from the usual fiber products in the
topological and smooth categories, respectively.

For the Koszul $L_{\infty}$-components, we have
\[
\widehat{\mathrm{pr}}^{\prime\mathrm{K}}_{(p',p''),(x',x'')}
\circ
\widehat{f}^{\mathrm{K}}_{p',x'}
\sim
\widehat{\mathrm{pr}}^{\prime\prime\mathrm{K}}_{(p',p''),(x',x'')}
\circ
\widehat{g}^{\mathrm{K}}_{p'',x''},
\]
which is an immediate consequence of the definition of
$\widehat{\mathrm{pr}}^{\prime\prime\mathrm{K}}_{(p',p''),(x',x'')}$ in  (\ref{xxbswb}).
For the de Rham components, the corresponding compatibility follows
from \cite[Theorem A.11]{Kim1} and \cite[Corollary 4.6]{Kim2}, which
state that any two quasi-isomorphisms are $L_{\infty}[1]$-homotopic.
Therefore,
\[
F\circ\mathrm{Pr}'
=
G\circ\mathrm{Pr}''.
\]
\end{proof}

\begin{prop}\label{fbprokur}
In the above setting, let
$G : \mathfrak{X}'' \rightarrow \mathfrak{X}$
be a fibration. Then
\[
\mathfrak{X}' {}_F \times_G \mathfrak{X}''
:=
\left(
X'_f \times_g X'',
\left[
\left(
\left\{
\mathcal{U}^\mathrm{f}_{(p',p'')}
\right\},
\left\{
\widetilde{\Phi}^\mathrm{f}_{(p'q',p''q'')}
\right\}
\right)
\right]
\right)
\]
is the fiber product of $\mathfrak{X}'$ and $\mathfrak{X}''$ over
$\mathfrak{X}$ in $\mathbf{Kur}$.
\end{prop}

\begin{proof}
Let
\[
\mathfrak{Y}
=
\left(
Y,
\left[\widehat{\mathcal{U}}^Y\right]
\right)
\]
be a Kuranishi space. Suppose that we are given a commutative diagram
\[
\begin{tikzcd}
\mathfrak{Y} \arrow[r, "J"] \arrow[d, "I"'] &
\mathfrak{X}'' \arrow[d, "G"]
\\
\mathfrak{X}' \arrow[r, "F"'] &
\mathfrak{X},
\end{tikzcd}
\]
where $I : \mathfrak{Y} \rightarrow \mathfrak{X}'$ and
$J : \mathfrak{Y} \rightarrow \mathfrak{X}''$ are morphisms of
Kuranishi spaces represented by the pre-morphisms
\[
\begin{cases}
\overline{I}
:=
\left(
\widehat{\mathcal{U}}^Y,
\widehat{\mathcal{U}}',
i,
\left\{i_r\right\},
\left\{\widehat{i}_{r,y}\right\}
\right),
\\
\overline{J}
:=
\left(
\widehat{\mathcal{U}}^Y,
\widehat{\mathcal{U}}'',
j,
\left\{j_r\right\},
\left\{\widehat{j}_{r,y}\right\}
\right),
\end{cases}
\]
respectively.

We define a morphism
\[
\mathfrak{Y}
\xrightarrow{K}
\mathfrak{X}' {}_F \times_G \mathfrak{X}''
\]
by considering the pre-morphism
\[
\overline{K}
=
\left(
\widehat{\mathcal{U}}^Y,\;
\widehat{\mathcal{U}''}_{G} \times_F \widehat{\mathcal{U}}',\;
k,\;
\left\{k_r\right\},\;
\left\{\widehat{k}_{r,y}\right\}
\right).
\]
Thus, the underlying maps are
\begin{enumerate}
\item[--]
\[
k :
Y \rightarrow X' {}_f \times_g X'',
\qquad
r \mapsto \bigl(i(r),j(r)\bigr),
\]

\item[--]
\[
k_r :
U^Y_r
\rightarrow
U'_{i(r)} {}_{f_{i(r)}}\!\times_{g_{j(r)}} U''_{j(r)},
\qquad
y \mapsto \bigl(i_r(y),j_r(y)\bigr),
\]

\item[--]
\[
\widehat{k}_{r,y} :
\mathcal{C}^{\mathrm{f}}_{k(r),k_r(y)}
\rightarrow
\mathcal{C}^Y_{r,y},
\qquad
\widehat{k}_{r,y}
=
\left\{
\widehat{k}_{r,y,m}
\right\}_{m\geq1}.
\]
\end{enumerate}

For a homogeneous element
\[
a
=
\pi'^*a'
\otimes
\pi^{\prime\prime*}a''
\]
of
\[
\bigwedge\nolimits^{-i}
\Gamma
\left(
\pi^{\prime*}E_{p'}^{\prime*}
\big|_{W^{\mathrm{f}}_{(i_r(y),j_r(y))}}
\right)
\otimes
\bigwedge\nolimits^{-j}
\Gamma
\left(
\pi^{\prime\prime*}E_{p''}^{\prime\prime*}
\big|_{W^{\mathrm{f}}_{(i_r(y),j_r(y))}}
\right),
\]
we define the Koszul component
\[
\widehat{k}^{\mathrm{K}}_{r,y}
=
\left\{
\widehat{k}^{\mathrm{K}}_{r,y,m}
\right\}_{m\geq1}
\]
by
\begin{equation}\label{khtob1}
\widehat{k}^{\mathrm{K}}_{r,y,m}(a)
:=
\begin{cases}
\widehat{i}^{\mathrm{K}}_{i(r),i_r(y),m}(a')
& \text{if }j=0,\ 1\leq i,
\\[4pt]
\widehat{j}^{\mathrm{K}}_{j(r),j_r(y),m}(a'')
& \text{if }i=0,\ 1\leq j,
\\[4pt]
\widehat{i}^{\mathrm{K}}_{i(r),i_r(y),m}
\bigl(\mathrm{p}'_0(a')\bigr)
+
\widehat{j}^{\mathrm{K}}_{j(r),j_r(y),m}
\bigl(\mathrm{p}''_0(a'')\bigr)
& \text{if }i=j=0,
\\[4pt]
0
& \text{otherwise.}
\end{cases}
\end{equation}
Here, the projection maps $\mathrm{p}'_0$ and $\mathrm{p}''_0$ are
defined as in Lemma~\ref{wnqb}(ii).

\begin{lem}
The collection
\[
\widehat{k}^{\mathrm{K}}_{r,y}
=
\left\{
\widehat{k}^{\mathrm{K}}_{r,y,m}
\right\}_{m\geq1}
\]
is a well-defined $L_{\infty}[1]$-morphism.
\end{lem}

\begin{proof}
This follows from the fact that
$\widehat{i}^{\mathrm{K}}_{i(r),i_r(y)}$ and
$\widehat{j}^{\mathrm{K}}_{j(r),j_r(y)}$ are
$L_{\infty}[1]$-morphisms.
\end{proof}

For the de Rham component
$\widehat{k}^{\mathrm{dR}}_{r,y}$, recall that the discussion following
\eqref{ahbbbsxxxx} gives
\[
dp'(\ker\alpha^{\mathrm{f}})
\subset
\ker\alpha',
\qquad
dp''(\ker\alpha^{\mathrm{f}})
\subset
\ker\alpha''.
\]
Thus, the maps
$\widehat{\mathrm{pr}}'^{\mathrm{dR}}$ and
$\widehat{\mathrm{pr}}''^{\mathrm{dR}}$ are defined by pullback.
We then set
\[
\widehat{k}^{\mathrm{dR}}_{r,y}
:=
\left(
\widehat{i}^{\mathrm{dR}}_{i(r),i_r(y)}
\oplus
\widehat{j}^{\mathrm{dR}}_{j(r),j_r(y)}
\right)
\circ
\left(
\left(
\widehat{\mathrm{pr}}'^{\mathrm{dR}}_{k(r),k_r(y)}
\right)^{-1}
\oplus
\left(
\widehat{\mathrm{pr}}''^{\mathrm{dR}}_{k(r),k_r(y)}
\right)^{-1}
\right)^{-1},
\]
where $(\cdots)^{-1}$ denotes a homotopy inverse of each
quasi-isomorphism.

\begin{lem}\label{kfprm}
The tuple $\overline{K}$ defines a pre-morphism.
\end{lem}

\begin{proof}
Consider two points $r,z\in Y$ such that
\[
\operatorname{Im}\psi^Y_r
\cap
\operatorname{Im}\psi^Y_z
\neq
\emptyset.
\]
We verify the required conditions.

\begin{enumerate}[label=(\roman*)]
\item
We have
\[
\phi^{\mathrm{f}}_{k(r)k(z)}
\circ
k_r
=
k_z
\circ
\phi^Y_{rz}
\]
on
\[
(s^Y_r)^{-1}(0)\cap U^Y_{rz}.
\]
Indeed,
\[
\begin{split}
\phi^{\mathrm{f}}_{k(r)k(z)}
\circ k_r(y)
&=
\phi^{\mathrm{f}}_{k(r)k(z)}
\bigl(i_r(y),j_r(y)\bigr)
\\
&=
\left(
\phi_{i(r)i(z)}\circ i_r(y),
\phi_{j(r)j(z)}\circ j_r(y)
\right)
\\
&=
\left(
i_z\circ\phi^Y_{rz}(y),
j_z\circ\phi^Y_{rz}(y)
\right)
\\
&=
k_z\circ\phi^Y_{rz}(y)
\end{split}
\]
for every
\[
y\in(s^Y_r)^{-1}(0)\cap U^Y_{rz}.
\]

\item
We have
\[
\psi^{\mathrm{f}}_{k(r)}
\circ
k_r
=
k\circ\psi^Y_r
\]
on
\[
(s^Y_r)^{-1}(0)\cap U^Y_{rz}.
\]
Indeed,
\[
\begin{split}
\psi^{\mathrm{f}}_{k(r)}
\circ k_r(y)
&=
\psi^{\mathrm{f}}_{k(r)}
\bigl(i_r(y),j_r(y)\bigr)
\\
&=
\left(
\psi'_{i(r)}\circ i_r(y),
\psi''_{j(r)}\circ j_r(y)
\right)
\\
&=
\left(
i\circ\psi^Y_r(y),
j\circ\psi^Y_r(y)
\right)
\\
&=
k\circ\psi^Y_r(y)
\end{split}
\]
for every
\[
y\in(s^Y_r)^{-1}(0)\cap U^Y_{rz}.
\]

\item
We have
\[
\widehat{\phi}^Y_{rz,y}
\circ
\widehat{k}_{z,\phi^Y_{rz}(y)}
\sim
\widehat{k}_{r,y}
\circ
\widehat{\phi}^{\mathrm{f}}_{k(r)k(z),k_r(y)}
\]
for each
\[
y\in(s^Y_r)^{-1}(0)\cap U^Y_{rz}.
\]

Consider first a homogeneous element
\[
\begin{split}
a
&=
\pi^{\prime*}a'
\otimes
\pi^{\prime\prime*}a''
\\
&\in
\bigwedge\nolimits^{-i}
\Gamma
\left(
\pi^{\prime*}E_{p'}^{\prime*}
\big|_{W^{\mathrm{f}}_{(i_r(y),j_r(y))}}
\right)
\otimes
\bigwedge\nolimits^{-j}
\Gamma
\left(
\pi^{\prime\prime*}E_{p''}^{\prime\prime*}
\big|_{W^{\mathrm{f}}_{(i_r(y),j_r(y))}}
\right).
\end{split}
\]
When $j=0$ and $i\geq1$, equation~\eqref{khtob1} gives
\[
\begin{split}
\widehat{\phi}^Y_{rz,y}
\circ
\widehat{k}_{z,\phi^Y_{rz}(y)}(a)
&=
\widehat{\phi}^Y_{rz,y}
\circ
\widehat{i}_{z,\phi^Y_{rz}(y)}(a')
\\
&\sim
\widehat{i}_{r,i_r(y)}
\circ
\widehat{\phi}_{i(r)i(z),i_r(y)}(a')
\\
&=
\widehat{k}_{r,y}
\circ
\widehat{\phi}^{\mathrm{f}}_{k(r)k(z),k_r(y)}(a).
\end{split}
\]
The case $i=0$ and $j\geq1$, involving the other
$L_{\infty}[1]$-morphism
$\widehat{j}_{j(r),j_r(y)}$, follows in the same way. The case $i=0$and $j = 0$ follows from the linearity of the maps.

The homotopy commutativity for the de Rham component follows from
Theorem~\ref{anhp}, applied to the relevant quasi-isomorphisms between
the acyclic $L_{\infty}[1]$-algebras.
\end{enumerate}
\end{proof}

We remark that the relations
\[
I=\mathrm{Pr}'\circ K,
\qquad
J=\mathrm{Pr}''\circ K
\]
hold. The underlying topological and base components are immediate,
while the $L_{\infty}$-components follow directly from the definition
in~\eqref{khtob1}.

The universal properties for the zero-locus map $k$ and the base map
$k_r$ for each $r\in Y$ are immediate, since the corresponding fiber
products in the topological and smooth categories are standard.

It remains to verify the universal property for the
$L_{\infty}$-component. Let
\[
\underline{\widehat{k}}_{r,y} :
\mathcal{C}^\mathrm{f}_{k(r),k_r(y)}
\longrightarrow
\mathcal{C}^Y_{r,y}
\]
be an $L_{\infty}[1]$-morphism satisfying
\[
\begin{aligned}
\underline{\widehat{k}}_{r,y}
\circ
\widehat{\mathrm{pr}}'_{k(r),k_r(y)}
&=
\widehat{i}_{i(r),i_r(y)},
\\
\underline{\widehat{k}}_{r,y}
\circ
\widehat{\mathrm{pr}}''_{k(r),k_r(y)}
&=
\widehat{j}_{j(r),j_r(y)}.
\end{aligned}
\]

Then we have:
\begin{lem}
The $L_\infty[1]$-morphisms $\widehat{k}_{r,y}$ and $\widehat{\underline{k}}_{r,y}$ are $L_\infty[1]$-homotopic.
\end{lem}
\begin{proof}
It suffices to construct an $L_\infty[1]$-morphism
\[
\widehat{H} \colon \mathcal{C}^\mathrm{f}_{k(r),k_r(y)} \longrightarrow \mathcal{C}^Y_{r,y} \otimes \Omega^*(\Delta^1)
\]
satisfying
\[
\mathrm{Eval}_0 \circ \widehat{H} = \widehat{k}_{r,y}
\qquad\text{and}\qquad
\mathrm{Eval}_1 \circ \widehat{H} = \widehat{\underline{k}}_{r,y},
\]
where the codomain of $\widehat{H}$ is the model of $\mathcal{C}^Y_{r,y} \times \Delta^1$ equipped with the $L_{\infty}[1]$-algebra structure as in Example~\ref{hhex}. We only need to verify this for the Koszul part.

To construct the desired Koszul part of the $L_\infty[1]$-homotopy, we need to construct a family of $\mathbb{R}$-linear maps $\widehat{H}^{\mathrm{K}} = \left\{\widehat{H}^{\mathrm{K}}_m\right\}_{m \ge 1}$ that satisfies \eqref{ombb}.

For each homogeneous element
\[
\begin{split}
a_i &= a_i' \otimes a_i'' \\
&\in \bigwedge\nolimits^{-i}\Gamma\left(\pi^{\prime*} E'_{p'}\big|_{W'_{\phi'_{p'q'}(x')} \times W''_{\phi''_{p''q''}(x'')}}\right)_{\phi'_{p'q'} \times \phi''_{p''q''}} \\
&\qquad\qquad\qquad\qquad \otimes \bigwedge\nolimits^{-j}\Gamma\left(\pi^{\prime\prime*}\left[E''_{p''}\right]\big|_{W'_{\phi'_{p'q'}(x')} \times W''_{\phi''_{p''q''}(x'')}}\right)_{\phi'_{p'q'} \times \phi''_{p''q''}} \\
&= \bigwedge\nolimits^{-\bullet}\Gamma\left(\pi^{\prime*}E'_{q'} \oplus \pi^{\prime\prime*}\left[E''_{q''}\right]\big|_{W'_{\phi'_{p'q'}(x')} \times W''_{\phi''_{p''q''}(x'')}}\right)_{\phi'_{p'q'} \times \phi''_{p''q''}},
\end{split}
\]
we claim that we can define
\begin{equation}\label{bqqqqq}
\widehat{H}^\mathrm{K}_m(a_1,\ldots,a_m) :=
\begin{cases}
\widehat{k}^\mathrm{K}_{r,y,m}(a_1,\ldots,a_m) & \text{if either $i = 0$ or $j = 0$}, \\[6pt]
\widehat{\underline{k}}^\mathrm{K}_{r,y,m}(a_1,\ldots,a_m)\, t + c_m(a_1,\ldots,a_m)\, dt & \text{otherwise,}
\end{cases}
\end{equation}
where $c_m$ is an $\mathbb{R}$-linear map of degree $-1$ given by
\[
c_m \colon \left(\bigwedge\nolimits^{-\bullet} \Gamma\left(E^\mathrm{f}_{(p', p'')}\big|_{W^\mathrm{f}_{(x', x'')}}\right)\right)^{\otimes m} \longrightarrow \bigwedge\nolimits^{-\bullet} \Gamma\left(E^Y_{r}\big|_{W^Y_{y}}\right).
\]

It then remains to verify the $L_\infty$-relation
\begin{equation}\label{bjsps}
    l_1\bigl(\widehat{H}^\mathrm{K}_m(a_1,\ldots,a_m)\bigr)
    = \sum_i \pm \widehat{H}^\mathrm{K}_m\bigl(l_1(a_i),\ldots,\widehat{a_i},\ldots,a_m\bigr)
\end{equation}
for each $m \geq 1$.

For the first case in \eqref{bqqqqq}, where either $i = 0$ or $j = 0$, relation \eqref{bjsps} holds automatically because $\widehat{k}^\mathrm{K}_{r,y}$ is an $L_\infty[1]$-morphism. Thus, it suffices to prove \eqref{bjsps} for the second case in \eqref{bqqqqq}; that is, for each homogeneous element $a_i = a_i' \otimes a_i''$ with $|a_i'| \cdot |a_i''| \neq 0$, we need to verify the equality between
\begin{equation}\label{ombb}
\begin{split}
    l_1\bigl(\widehat{H}^{\mathrm{K}}_m(a_1,\ldots,a_m)\bigr)
    &= l_1\Bigl( t \cdot \underline{\widehat{k}}^\mathrm{K}_{r,y,m}(a_1,\ldots,a_m) + c_m(a_1,\ldots,a_m)\, dt \Bigr) \\
    &= t \cdot l_1\bigl(\underline{\widehat{k}}^\mathrm{K}_{r,y,m}(a_1,\ldots,a_m)\bigr) + \underline{\widehat{k}}^\mathrm{K}_{r,y,m}(a_1,\ldots,a_m)\, dt \\
    &\quad + l_1\bigl(c_m(a_1,\ldots,a_m)\bigr)\, dt
\end{split}
\end{equation}
and
\[
\begin{split}
    (\widehat{H}^{\mathrm{K}}_m \circ l_1)(a_1,\ldots,a_m)
    &= \sum_i \pm \widehat{H}^{\mathrm{K}}_m\bigl(l_1(a_i),\, a_1,\ldots,\widehat{a_i},\ldots,a_m\bigr) \\
    &= \sum_i \pm \Bigl[
        t \cdot \underline{\widehat{k}}^\mathrm{K}_{r,y,m}\bigl(l_1(a_i),a_1,\ldots,\widehat{a_i},\ldots,a_m\bigr) \\
        &\qquad\qquad + c_m\bigl(l_1(a_i),a_1,\ldots,\widehat{a_i},\ldots,a_m\bigr)\, dt
    \Bigr].
\end{split}
\]
Comparing the coefficients of $t$ and $dt$ separately, we obtain
\begin{equation}\label{bqpdmbm}
    l_1\bigl(\underline{\widehat{k}}^\mathrm{K}_{r,y,m}(a_1,\ldots,a_m)\bigr)
    = \sum_i \pm \underline{\widehat{k}}^\mathrm{K}_{r,y,m}\bigl(l_1(a_i),a_1,\ldots,\widehat{a_i},\ldots,a_m\bigr)
\end{equation}
and
\begin{equation}\label{bkklss}
    \underline{\widehat{k}}^\mathrm{K}_{r,y,m}(a_1,\ldots,a_m) + l_1\bigl({c}_m(a_1,\ldots,a_m)\bigr)
    = \sum_i \pm {c}_m\bigl(l_1(a_i),a_1,\ldots,\widehat{a_i},\ldots,a_m\bigr).
\end{equation}
Notice that \eqref{bqpdmbm} is precisely the $L_{\infty}$-relation for $\underline{\widehat{k}}^\mathrm{K}_{r,y}$, which is already valid. Thus, it remains to construct the maps $c_m$ satisfying \eqref{bkklss}. For this, by taking expansions of the atlas for $\mathfrak{Y}$ if necessary, we may assume without loss of generality that $\underline{k}_p$ is surjective for each $p \in Y$, so that $\widehat{\phi}_{pq,x} = \widehat{\phi}_{pq,x}^{\mathrm{c}}$. Since $\overline{\underline{K}}$ is a pre-morphism, we have
\begin{equation}\label{bajh}
    \widehat{\kappa} \circ \widehat{\underline{k}}^\mathrm{K}_p \circ \widehat{\phi}^{\mathrm{f},\mathrm{K}}_{k(p)k(q)}
    \sim
    \widehat{\kappa} \circ \widehat{\phi}^\mathrm{K}_{pq,x} \circ \widehat{\underline{k}}^\mathrm{K}_{k(q)}
    \sim
    \widehat{\underline{k}}^\mathrm{K}_{k(q)},
\end{equation}
where $\widehat{\kappa}$ is a homotopy inverse of $\widehat{\phi}^\mathrm{K}_{pq}$ under the surjectivity assumption on $\underline{k}_p$.

Relation \eqref{bajh} then guarantees the existence of an $\mathbb{R}$-linear map of degree $-1$,
\[
c_m \colon \left(\bigwedge\nolimits^{-\bullet} \Gamma\left(E^\mathrm{f}_{(p', p'')}\big|_{W^\mathrm{f}_{(x', x'')}}\right)\right)^{\otimes m} \longrightarrow \bigwedge\nolimits^{-\bullet} \Gamma\left(E^Y_{r}\big|_{W^Y_{y}}\right)
\]
for each $m \geq 1$. In particular, for $a_i = a_i' \otimes a_i''$ with $|a_i'| \cdot |a_i''| \neq 0$, relation \eqref{bkklss} holds because the leftmost term $\widehat{\kappa} \circ \widehat{\underline{k}}^\mathrm{K}_p \circ \widehat{\phi}^{\mathrm{f},\mathrm{K}}_{k(p)k(q)}$ in \eqref{bajh} vanishes on such elements by the definition of the coordinate change $\widehat{\phi}^{\mathrm{f},\mathrm{K}}_{k(p)k(q)}$.
\end{proof}

\end{proof}

The following propositions play a key role in the construction of the
category of fibrant objects in Section 6.

\begin{prop}\label{hnuun}
\begin{enumerate}[label=(\roman*)]
    \item Fibrations are preserved under pullbacks.
    \item Trivial fibrations are preserved under pullbacks.
\end{enumerate}
\end{prop}

\begin{proof}
(i) Let
\[
G : \mathfrak{X}'' \rightarrow \mathfrak{X}
\]
be a fibration represented by
\[
G =
\left[
\left(
\widehat{\mathcal{U}}'',
\widehat{\mathcal{U}},
g,
\left\{g_{p''}\right\},
\left\{\widehat{g}_{p'',x''}\right\}
\right)
\right]
\]
with the commutative diagram
\[
\begin{tikzcd}
    \mathfrak{X}' {}_F \times_G \mathfrak{X}''
    \arrow[r, "\mathrm{Pr}''"]
    \arrow[d, "\mathrm{Pr}'"']
    &
    \mathfrak{X}''
    \arrow[d, "G"]
    \\
    \mathfrak{X}'
    \arrow[r, "F"']
    &
    \mathfrak{X}.
\end{tikzcd}
\]

We show that the canonical projection
\[
\mathrm{Pr}' :
\mathfrak{X}' {}_F \times_G \mathfrak{X}''
\rightarrow
\mathfrak{X}'
\]
represented by
\[
\mathrm{Pr}'
=
\Bigl[
\left(
\widehat{\mathcal{U}}^\mathrm{f},
\widehat{\mathcal{U}}',
\mathrm{pr}',
\left\{\mathrm{pr}'_{(p',p'')}\right\},
\left\{
\widehat{\mathrm{pr}}'_{(p',p''),(x',x'')}
\right\}
\right)
\Bigr]
\]
is a fibration.

\begin{itemize}
    \item[--]
    The map
    \[
    \mathrm{pr}'_{(p',p'')} :
    U'_{p'} {}_{f_{p'}}\!\times_{g_{p''}} U''_{p''}
    \rightarrow
    U'_{p'}
    \]
    is surjective, since $g_{p''}$ is a surjective submersion.

    \item[--]
    There exists an embedding
    \[
    \mathrm{em}'_{(p',p'')} :
    U'_{p'}
    \hookrightarrow
    U'_{p'} {}_{f_{p'}}\!\times_{g_{p''}} U''_{p''},
    \qquad
    y' \mapsto
    \bigl(y',\epsilon_{p',p''}(y')\bigr),
    \]
    for some embedding
    \[
    \epsilon_{p',p''} :
    U'_{p'} \hookrightarrow U''_{p''}
    \]
    satisfying
    \[
    f_{p'}(z')
    =
    g_{p''}\bigl(\epsilon_{p',p''}(z')\bigr).
    \]
    The existence of such an embedding follows from the surjectivity
    of $g_{p''}$.

    \item[--]
    The map
    \[
    \overline{\mathrm{em}}'_{(p',p'')} :
    E'_{p'}
    \rightarrow
    {\pi'}^*E'_{p'}
    \oplus
    {\pi''}^*\left[E''_{p''}\right]
    \]
    is a bundle embedding over
    $\mathrm{em}'_{(p',p'')}$, given by
    \[
    \begin{split}
    \overline{\mathrm{em}}'_{(p',p'')} :
    E'_{p'}
    &\rightarrow
    {\pi'}^*E'_{p'}
    \oplus
    {\pi''}^*\left[E''_{p''}\right],
    \\
    (y',e')
    &\mapsto
    \left(
    \bigl(y',\epsilon_{p',p''}(y')\bigr),
    \left(
    \pi^{\prime *}e',
    \pi^{\prime\prime *}
    \left[
    s''_{p''}\bigl(\epsilon_{p',p''}(y')\bigr)
    \right]
    \right)
    \right).
    \end{split}
    \]
    We then have the commutative diagram
    \begin{equation}\label{ahhheeeeefq}
    \begin{tikzcd}[column sep=large]
        E'_{p'}
        \arrow[r, hook, "{\overline{\mathrm{em}}'_{(p',p'')}}"]
        &
        {\pi'}^*E'_{p'}
        \oplus
        {\pi''}^*\bigl[E''_{p''}\bigr]
        \\
        \arrow[u, swap, "s'_{p'}"']
        U'_{p'}
        \arrow[r, hook, "{\mathrm{em}'_{(p',p'')}}"']
        &
        U'_{p'} {}_{f_{p'}}\!\times_{g_{p''}} U''_{p''}.
        \arrow[u, swap,
        "{{\pi'}^*s'_{p'}
        \oplus
        {\pi''}^*\left[s''_{p''}\right]}"]
    \end{tikzcd}
    \end{equation}
    and the compatibility
    \[
    \overline{\mathrm{em}}'_{(p',p'')}
    \circ
    s'_{p'}
    =
    \left(
    {\pi'}^*s'_{p'}
    \oplus
    {\pi''}^*\left[s''_{p''}\right]
    \right)
    \circ
    \mathrm{em}'_{(p',p'')}
    \]
    follows immediately.

    The map
    $\widehat{\mathrm{pr}}^{\prime\mathrm{K}}_{(p',p''),(x',x'')}$
    is defined by \eqref{bzqqqqgbbbq}, using the above bundle embedding
    and Lemma~\ref{pokhgjb}.

    \item[--]
    The definition of the $L_{\infty}$-coordinate change
    $\widehat{\phi}^{\mathrm{f}}_{pq,x}$ of the fiber product given by
    \eqref{tvev} and \eqref{bbexx} ensures the existence of the two
    $L_{\infty}[1]$-morphisms in \eqref{hqrwbb}.

    \item[--]
    We have
    \[
    \left[\pi''^*s''_{p''}\right]^{-1}(0)
    \cap
    \left(\mathrm{pr}'_{(p',p'')}\right)^{-1}
    \left({s'_{p'}}^{-1}(0)\right)
    =
    \left(
    \pi'^*s'_{p'}
    \oplus
    \pi''^*s''_{p''}
    \right)^{-1}(0),
    \]
    as follows directly from the definition of the fiber product
    section.
\end{itemize}

(ii) Suppose, in addition, that
\[
G =
\left[
\left(
\widehat{\mathcal{U}}'',
\widehat{\mathcal{U}},
g,
\left\{\widehat{g}_{p''}\right\},
\left\{\widehat{g}_{p'',x''}\right\}
\right)
\right]
\]
is a weak equivalence. Thus, $g$ is a homeomorphism and
$\widehat{g}_{p'',x''}$ is a quasi-isomorphism.

Observe that
\[
X' {}_f\!\times_g X'' \simeq X'
\]
and that the projection
\[
\mathrm{pr}' :
X' {}_f\!\times_g X''
\rightarrow
X',
\qquad
(p',p'')\mapsto p',
\]
is a homeomorphism. Indeed, since $g$ is a homeomorphism, for each
$p'\in X'$ there is a unique $p''\in X''$ satisfying
\[
f(p')=g(p'').
\]

It remains to show that
$\widehat{\mathrm{pr}}'_{(p',p''),(x',x'')}$
is a quasi-isomorphism for each
\[
x'\in (s'_{p'})^{-1}(0),
\qquad
x''\in
\left[\pi''^*s''_{p''}\right]^{-1}(0).
\]

Since
\[
\widehat{g}^{\vee\mathrm{K}}_{p'',x'',1} :
\bigwedge\nolimits^{-\bullet}
\Gamma\left(
E_{g(p'')}^*
\big|_{W_{g_{p''}(x'')}}
\right)
\rightarrow
\bigwedge\nolimits^{-\bullet}
\Gamma\left(
E_{p''}^{\prime\prime *}
\big|_{W''_{x''}}
\right)
\]
is the injective quasi-isomorphic chain map induced by
Lemma~\ref{pokhgjb} and Corollary~\ref{ygsdnn}, the quotient complex
\[
\frac{
\bigwedge\nolimits^{-\bullet}
\Gamma\left(
E_{p''}^{\prime\prime *}
\big|_{W''_{x''}}
\right)
}{
\widehat{g}^{\vee\mathrm{K}}_{p'',x'',1}
\left(
\bigwedge\nolimits^{-\bullet}
\Gamma\left(
E_{g(p'')}^*
\big|_{W_{g_{p''}(x'')}}
\right)
\right)
}
\]
is acyclic.

This induces the following acyclic chain complex:
\begin{equation}\label{wes2}
\cdots
\rightarrow
\bigwedge\nolimits^{i+1}
\Gamma\left(
\left[E''_{p''}\right]^*
\big|_{W''_{x''}}
\right)
\xrightarrow{\iota_{s''_{p''}}^{(i+1)}}
\bigwedge\nolimits^i
\Gamma\left(
\left[E''_{p''}\right]^*
\big|_{W''_{x''}}
\right)
\xrightarrow{\iota_{s''_{p''}}^{(i)}}
\bigwedge\nolimits^{i-1}
\Gamma\left(
\left[E''_{p''}\right]^*
\big|_{W''_{x''}}
\right)
\rightarrow
\cdots.
\end{equation}

In particular, in degree $0$, we obtain the surjectivity of the linear
map
\begin{equation}\label{gecwp}
\Gamma\left(
\left[E''_{p''}\right]^*
\big|_{W''_{x''}}
\right)
\xrightarrow{\iota_{s''_{p''}}}
\frac{
\Gamma\left(
\bigwedge^0
\left(
E_{p''}^{\prime\prime *}
\big|_{W''_{x''}}
\right)
\right)
}{
\widehat{g}^{\vee\mathrm{K}}_{p'',x'',1}
\Gamma\left(
\bigwedge^0
E_{g(p'')}^*
\big|_{W_{g_{p''}(x'')}}
\right)
}.
\end{equation}

Moreover,
\begin{equation}\label{cse}
\begin{split}
&
\frac{
\Gamma\left(
\bigwedge^0
\left(
E_{p''}^{\prime\prime *}
\big|_{W''_{x''}}
\right)
\right)
}{
\widehat{g}^{\vee\mathrm{K}}_{p'',x'',1}
\Gamma\left(
\bigwedge^0
E_{g(p'')}^*
\big|_{W_{g_{p''}(x'')}}
\right)
}
\\
&\qquad =
\frac{
C^\infty(W''_{x''})
}{
\widehat{g}^{\vee\mathrm{K}}_{p'',x'',1}
\left(
C^\infty(W_{g_{p''}(x'')})
\right)
}
\\
&\qquad =
\frac{
C^\infty(W''_{x''})
}{
\pi^*
\circ
\left(\mathrm{em}_{p''}^{-1}\right)^*
\left(
C^\infty(W_{g_{p''}(x'')})
\right)
}
\overset{*}{\simeq}
I_{\mathrm{em}_{p''}},
\end{split}
\end{equation}
where
\[
I_{\mathrm{em}_{p''}}
:=
\left\{
h\in C^\infty(W''_{x''})
\mid
h|_{\operatorname{Im}\mathrm{em}_{p''}}
\equiv 0
\right\}.
\]

The isomorphism $*$ follows from the following lemma.

\begin{lem}\cite[Lemma 2.33 (i)]{Kim1}
There exists an $\mathbb{R}$-isomorphism
\begin{equation}\nonumber
\frac{
C^\infty(W''_{x''})
}{
\pi^*
\circ
\left(\mathrm{em}_{p''}^{-1}\right)^*
\left(
C^\infty(W_{g_{p''}(x'')})
\right)
}
\xlongrightarrow{\simeq}
\left\{
h\in C^\infty(W''_{x''})
\mid
h|_{\operatorname{Im}\mathrm{em}_{p''}}
\equiv 0
\right\}.
\end{equation}
\end{lem}

\begin{proof}
Let
\[
\pi :
W''_{x''}
\twoheadrightarrow
\mathrm{em}_{p''}
\left(
W_{g_{p''}(x'')}
\right)
\]
be a choice of retraction between the contractible spaces. Thus,
\[
\pi|_{\mathrm{em}_{p''}(W_{g_{p''}(x'')})}
=
\mathrm{id}_{\mathrm{em}_{p''}(W_{g_{p''}(x'')})}.
\]
Consider the map
\begin{equation}\nonumber
\begin{split}
\kappa :
C^\infty(W''_{x''})
&\longrightarrow
\left\{
h\in C^\infty(W''_{x''})
\mid
h|_{\operatorname{Im}\mathrm{em}_{p''}}
\equiv 0
\right\},
\\
a'
&\longmapsto
a'
-
\pi^*
\left(
a'|_{\operatorname{Im}\mathrm{em}_{p''}}
\right).
\end{split}
\end{equation}
This map is well-defined, since
\[
\left.
\left(
a'
-
\pi^*
\left(
a'|_{\operatorname{Im}\mathrm{em}_{p''}}
\right)
\right)
\right|_{\operatorname{Im}\mathrm{em}_{p''}}
=
a'|_{\operatorname{Im}\mathrm{em}_{p''}}
-
a'|_{\operatorname{Im}\mathrm{em}_{p''}}
=
0.
\]

The map $\kappa$ is obviously surjective, since
\[
\kappa|_{I_{\mathrm{em}_{p''}}}
=
\mathrm{id}.
\]
Its kernel consists of all elements of the form $\pi^*(b)$ for some
\[
b\in
C^\infty
\left(
\mathrm{em}_{p''}(W_{g_{p''}(x'')})
\right).
\]
Since $\mathrm{em}_{p''}$ is an embedding, every such $b$ can be written
as
\[
b=
\left(\mathrm{em}_{p''}^{-1}\right)^*(b')
\]
for some
\[
b'\in
C^\infty(W_{g_{p''}(x'')}).
\]
Hence
\[
\ker(\kappa)
=
\pi^*
\circ
\left(\mathrm{em}_{p''}^{-1}\right)^*
\left(
C^\infty(W_{g_{p''}(x'')})
\right).
\]
The result follows from the first isomorphism theorem.
\end{proof}

The infinitesimal consequence of the surjectivity of the map
\eqref{gecwp} is
\begin{equation}\label{bjsjhhh}
\left(
E''_{p''}\big|_{\operatorname{Im}(\mathrm{em}_{p''})}
\right)_x
\xrightarrow{
(d[s''_{p''}]_x)^{-1},\,\simeq
}
\left(
N_{W''_{x''}}\operatorname{Im}(\mathrm{em}_{p''})
\right)_x.
\end{equation}
In other words, the tangent bundle condition is satisfied at each
zero point
\[
y\in\operatorname{Im}(\mathrm{em}_{p''}).
\]

\begin{lem}\label{agohaoigigi}
In the above situation, $\widehat{\operatorname{pr}}_{(p',p''), (x',x'')}^{\prime \mathrm{K}}$ is a quasi-isomorphism.
\end{lem}

\begin{proof}
Recall that the $L_{\infty}$ Koszul morphism of a fibration is, by definition, induced by a pair of bundle embeddings 
\[
(\overline{\mathrm{em}}_p, \mathrm{em}_p) \qquad\text{and}\qquad (\overline{\mathrm{em}}'_p, \mathrm{em}'_p),
\]
fitting into the commutative diagram
\[
\begin{tikzcd}
E'_{f(p)} \arrow{d} &  \underline{E}^{\prime}_{f(p)} \arrow{d} \arrow[hook]{l}[swap]{\overline{\mathrm{em}}'_p} \arrow[hook]{r}{\overline{\mathrm{em}}_p}  & E_p  \arrow{d} \\
U'_{f(p)} & \underline{U}'_{f(p)} \arrow[hook]{l}[swap]{{\mathrm{em}}'_p} \arrow[hook]{r}{\mathrm{em}_p} & U_p
\end{tikzcd}
\]
for each $p$, and satisfying the compatibility conditions
\begin{equation}\label{gadgashhhh2}
\begin{split}
\left(s_p \circ \mathrm{em}_p\right)\big|_{\mathrm{Im}\left(\overline{\mathrm{em}}_p\right)} &= \overline{\mathrm{em}}_p \circ \underline{s}^{\prime}_{f(p)}, \\
s_p \circ \mathrm{em}'_p &= \overline{\mathrm{em}}'_p \circ s^{\prime \mathrm{c}}_{f(p)},
\end{split} 
\end{equation}
together with the tangent bundle condition for $\mathrm{em}'_p$.

By Lemma~\ref{pokhgjb}, we obtain the $L_{\infty}[1]$-morphisms
\[
\widehat{\eta}^{\vee}_{p,x} \colon \bigwedge\nolimits^{-\bullet}\Gamma\left(\underline{E}^{\prime *}_{f(p)}\big|_{\underline{W}'_{f(p)}}\right) \longrightarrow \bigwedge\nolimits^{-\bullet}\Gamma\left({E}^{*}_{p}\big|_{W_{p}}\right)
\]
and
\[
\underline{\widehat{\eta}}_{p,x}^{-1} \colon \bigwedge\nolimits^{-\bullet}\Gamma\left({E}^{\prime *}_{f(p)}\big|_{W'_{f(p)}}\right) \xrightarrow{\;\simeq\;} \bigwedge\nolimits^{-\bullet}\Gamma\left(\underline{E}^{\prime *}_{f(p)}\big|_{\underline{W}'_{f(p)}}\right).
\]
In particular, $\underline{\widehat{\eta}}_{p,x}^{-1}$ is a quasi-isomorphism by Lemma~\ref{ygsdnn}, as it is obtained from its homotopy inverse $\underline{\widehat{\eta}}_{p,x}$. 

We then define
\begin{equation}\label{bhuaiuhnn1}
\widehat{\eta}_{p,x} := \widehat{\eta}^{\vee}_{p,x} \circ \underline{\widehat{\eta}}_{p,x}.
\end{equation}
By hypothesis, $\widehat{\eta}_{p,x}$ is a weak equivalence, which implies that $\widehat{\eta}^{\vee}_{p,x}$ is also a weak equivalence.

\medskip
\noindent(1) We first prove the desired statement for the special case where $(\overline{\mathrm{em}}'_p, \mathrm{em}'_p)$ is given by
\[
\left( U'_{f(p)}, E'_{f(p)} \right) = \left(\underline{U}'_{f(p)}, \underline{E}^{\prime}_{f(p)}\right)
\qquad\text{and}\qquad
(\overline{\mathrm{em}}'_p, \mathrm{em}'_p) = (\mathrm{id}, \mathrm{id}),
\]
so that we have
\[
\widehat{\eta}_{p,x} = \widehat{\eta}^{\vee}_{p,x}.
\]
It is then straightforward to verify the tangent bundle condition for \eqref{ahhheeeeefq} as follows:
\[
\begin{split}
\frac{\left(\pi^{\prime*} E'_{p'} \oplus \bigl[ \pi^{\prime\prime*} E''_{p''} \bigr]\right)_{\mathrm{em}_{(p',p'')}(x')}}{\overline{\mathrm{em}}'_{(p',p''), x'}\left( (E'_{p'})_{x'} \right)} 
&\simeq \frac{(E'_{p'})_{x'} \oplus \bigl[ E''_{p''} \bigr]_{\epsilon_{p',p''}(x')}}{(E'_{p'})_{x'} \oplus \{0\}} \\
&\simeq \bigl[ E''_{p''} \bigr]_{\epsilon_{p',p''}(x')} \\
&\xrightarrow[\eqref{bjsjhhh}]{(d[s''_{p''}]_x)^{-1},\, \simeq} \left(N_{W_x}W''_{\phi(x)}\right)_{\epsilon_{p',p''}(y')} \\
&\simeq \frac{TU''_{p''}}{\mathrm{em}_{p''*}(TU_{g(p'')})} \\
&\simeq \frac{T_{(x',\epsilon_{p',p''}(x'))}\left(U'_{p'} {}_{f_{p'}}{\times}_{g_{p''}} U''_{p''}\right)}{\mathrm{em}_{(p',p'')*}\left(T_{x'}U'_{p'}\right)}
\end{split}
\]
for each $x' \in (s'_{p'})^{-1}(0)$.
\end{proof}

Moreover, Lemma~\ref{agohaoigigi} together with Proposition~\ref{afec} yields the desired quasi-isomorphism property of $\widehat{\mathrm{pr}}^{\prime \mathrm{K}}_{(p',p''),(x',x'')}$.

\ \

\noindent (2)
We next consider the general case, where we have a nontrivial embedding $(\overline{\mathrm{em}}'_p, \mathrm{em}'_p)$ and the $L_{\infty}[1]$-morphism $\underline{\widehat{\eta}}_{p,x}$.

For the base component, we have the following commutative diagram:
\begin{equation}\label{bsmcb}
\begin{tikzcd}[
    row sep=2.5em,
    column sep=5em,
    arrows={->}
]    
U'_{p'} \times_{\underline{U}_{g(p'')}} U''_{p''} \arrow[r, "\mathrm{r}"] \arrow[rr, bend left=20, "\mathrm{pr}''_{(p',p'')}"] \arrow[rdd, swap, "\underline{\mathrm{pr}}'_{(p',p'')}"] & U'_{p'} \times_{U_{g(p'')}} U''_{p''} \arrow[r, "\mathrm{pr}''_{(p',p'')}"] \arrow[dd, "\mathrm{pr}'_{(p',p'')}"] & U''_{p''} \arrow[d, "\exists \widetilde{g}_{p''}"] \arrow[dd, bend left=30, "g_{p''}"] \\
{} & {} & \underline{U}_{g(p'')} \arrow[u, bend left=20, "\mathrm{em}_{(p',p'')}"] \arrow[d, swap, "\mathrm{em}'_{(p',p'')}"] \\
{} & U'_{p'} \arrow[r, swap, "f_{p'}"] \arrow[ru, "\exists \widetilde{f}_{p'}"] & U_{g(p'')},
\end{tikzcd}
\end{equation}
where $\mathrm{r}$ is the retraction map
\[
U'_{p'} \times_{\underline{U}_{g(p'')}} U''_{p''} \longrightarrow U'_{p'} \times_{U_{g(p'')}} U''_{p''}
\]
between contractible spaces. This restricts to the map between local neighborhoods near the origin (or zero section):
\[
\mathrm{r} \colon W'_{x'} \times_{\underline{W}_x} W''_{x''} \twoheadrightarrow W'_{x'} \times_{{W}_x} W''_{x''}.
\]
Note that the existence of the lifts $\widetilde{f}_{p'}$ and $\widetilde{g}_{p''}$ is guaranteed by the contractibility of $U'_{p'}$, $U''_{p''}$, and $U_{g(p'')}$.

We remark that \eqref{bsmcb} consists of two fiber product diagrams:
\[
\begin{tikzcd}[
    row sep=1.5em,
    column sep=0.5em,
    arrows={->}
]    
 U'_{p'} \times_{U_{g(p'')}} U''_{p''} \arrow[r] \arrow[dd] & U''_{p''} \arrow[dd, bend left=20] \\
{} & {} \\
U'_{p'} \arrow[r] & U_{g(p'')},
\end{tikzcd}
\qquad\qquad 
\begin{tikzcd}[
    row sep=1.5em,
    column sep=1.5em,
    arrows={->}
]    
U'_{p'} \times_{\underline{U}_{g(p'')}} U''_{p''} \arrow[rr, bend left=20] \arrow[rdd] & {} & U''_{p''} \arrow[d] \\
{} & {} & \underline{U}_{g(p'')} \\
{} & U'_{p'} \arrow[ru] & {}
\end{tikzcd}
\]
which reduce to the fiber product diagrams for local neighborhoods
\[
\begin{tikzcd}[
    row sep=1.5em,
    column sep=0.5em,
    arrows={->}
]    
W'_{x'} \times_{W_x} W''_{x''} \arrow[r] \arrow[dd] & W''_{x''} \arrow[dd, bend left=20] \\
{} & {} \\
W'_{x'} \arrow[r] & W_x,
\end{tikzcd}
\qquad\qquad 
\begin{tikzcd}[
    row sep=1.5em,
    column sep=1.5em,
    arrows={->}
]    
W'_{x'} \times_{\underline{W}_x} W''_{x''} \arrow[rr, bend left=20] \arrow[rdd] & {} & W''_{x''} \arrow[d] \\
{} & {} & \underline{W}_{x} \\
{} & W'_{x'} \arrow[ru] & {}
\end{tikzcd}
\]
respectively.

\medskip

For the $L_{\infty}$-component, we have the commutative diagram:
\[
\begin{tikzcd}[
    row sep=2.5em,
    column sep=5em,
    arrows={->}
]    
    \underline{\mathcal{C}}^{\mathrm{f}}_{(p',p''), (x',x'')} & \mathcal{C}^{\mathrm{f}}_{(p',p''), (x',x'')} \arrow[l, "\widehat{\mathrm{r}}"'] & \mathcal{C}''_{p'', x''} \arrow[l, "\widehat{\mathrm{pr}}^{\prime\prime}_{(p',p''), (x',x'')}"'] \arrow[ll, swap, bend right=20, "\underline{\widehat{\mathrm{pr}}}^{\prime\prime}_{(p',p''), (x',x'')}"] \\
{} & {} & \underline{\mathcal{C}}_{p,x} \arrow[u, "{ \underline{\widehat{\eta}},\,\simeq }"'] \arrow[ld] \arrow[d, swap, "{ \widehat{\eta}^{\vee},\,\simeq }"'] \\
{} & \mathcal{C}'_{p', x'} \arrow[luu, swap, "\underline{\widehat{\mathrm{pr}}}^{\prime}_{(p',p''), (x',x'')}"'] \arrow[uu, swap, "\widehat{\mathrm{pr}}'_{(p',p''), (x',x'')}"] & \mathcal{C}_{p,x} \arrow[l, "\widehat{f}_{p',x'}"].
\end{tikzcd}
\]

From case (1), we already know that $\underline{\widehat{\mathrm{pr}}}^{\prime \mathrm{K}}_{(p',p''),(x',x'')}$ is a quasi-isomorphism.
On the other hand, $\widehat{\mathrm{r}}$ is defined as the chain map (and hence an $L_{\infty}[1]$-morphism) induced from the linear map of sections
\begin{equation}\label{qx}
\begin{split}
\Gamma\left(\pi^{\prime*} E^{\prime*}_{p'} \oplus \pi^{\prime\prime*} \bigl[ E^{\prime \prime*}_{p''} \bigr]\big|_{W'_{x'} \times_{{W}_x} W''_{x''}}\right) &\longrightarrow \Gamma\left(\pi^{\prime*} E^{\prime*}_{p'} \oplus \pi^{\prime\prime*} \bigl[ E^{\prime\prime*}_{p''} \bigr]\big|_{W'_{x'} \times_{\underline{W}_x} W''_{x''}}\right), \\
a &\longmapsto \mathrm{r}^*a.
\end{split}
\end{equation}
Here, the pullback by $\mathrm{r}$ is given by
\[
\mathrm{r}^*a(u) := a(\mathrm{r}(u))
\]
for $u \in W'_{x'} \times_{\underline{W}_x} W''_{x''}$ and $a \in \Gamma\left(\pi^{\prime*} E^{\prime*}_{p'} \oplus \pi^{\prime\prime*} \bigl[ E^{\prime \prime*}_{p''} \bigr]\big|_{W'_{x'} \times_{{W}_x} W''_{x''}}\right)$. Using \eqref{bhuaiuhnn1}, one can verify that \eqref{qx} satisfies the chain map condition.

By the definitions of $\widehat{\mathrm{pr}}^{\prime \mathrm{K}}_{(p',p''), (x',x'')}$ and $\underline{\widehat{\mathrm{pr}}}^{\prime \mathrm{K}}_{(p',p''), (x',x'')}$ together with the commutativity of the above diagram, we have
\[
\underline{\widehat{\mathrm{pr}}}^{\prime \mathrm{K}}_{(p',p''), (x',x'')} = \widehat{\mathrm{r}} \circ \widehat{\mathrm{pr}}^{\prime \mathrm{K}}_{(p',p''), (x',x'')}.
\]

Let $\sim$ denote the equivalence relation on linear maps defined by equality modulo composition with isomorphisms. Since $\underline{\widehat{\mathrm{pr}}}^{\prime \mathrm{K}}_{(p',p''), (x',x'')}$ is a quasi-isomorphism, we then have
\begin{equation}\label{mbjjjm}
\widehat{\mathrm{id}}_* \sim \widehat{\mathrm{r}}_* \circ \widehat{\mathrm{pr}}^{\prime \mathrm{K}}_{(p',p''), (x',x'')*}.
\end{equation}

Note that the map $\widehat{\mathrm{r}}$ admits a left inverse $\widehat{\mathrm{h}}$ induced by the restriction of sections, satisfying
\[
\widehat{\mathrm{h}} \circ \widehat{\mathrm{r}} = \widehat{\mathrm{id}},
\]
and hence
\begin{equation}\label{kkoobb}
\widehat{\mathrm{h}}_* \circ \widehat{\mathrm{r}}_* = \widehat{\mathrm{id}}_*.
\end{equation}

By composing \eqref{mbjjjm} with $\widehat{\mathrm{h}}_*$ and applying \eqref{kkoobb}, we obtain
\[
\widehat{\mathrm{h}}_* \sim \widehat{\mathrm{pr}}_{(p',p''), (x',x'')*}^{\prime \mathrm{K}},
\]
which, together with \eqref{kkoobb}, further implies
\begin{equation}\label{bahohdf}
\widehat{\mathrm{pr}}_{(p',p''), (x',x'')*}^{\prime \mathrm{K}} \circ \widehat{\mathrm{r}}_* \sim \widehat{\mathrm{id}}_*.
\end{equation}

From \eqref{mbjjjm} and \eqref{bahohdf}, we conclude that $\widehat{\mathrm{pr}}_{(p',p''), (x',x'')}^{\prime \mathrm{K}}$ is a quasi-isomorphism.
\end{proof}

Note that $\mathfrak{X}' {}_{F}\!\times_{G} \mathfrak{X}''$ does not
define a fiber product in $\mathbf{Kur}$ in the strict sense, since its
definition depends on the choices of representatives for $F$ and $G$.
In general, there is no canonical choice, and different choices need
not give rise to isomorphic Kuranishi spaces. Instead, we will show that
all such choices result in the same objects up to weak equivalence.

\begin{prop}\label{eobjkk}
The resulting Kuranishi spaces
$\mathfrak{X}' {}_{F}\!\times_{G} \mathfrak{X}''$
obtained from different choices of representatives $\overline{F}$ and
$\overline{G}$ are related by weak equivalences.
\end{prop}

\begin{proof}
Let
\begin{equation}\nonumber
\begin{cases}
\overline{G}_1 =
\left(
{\widehat{\mathcal{U}_1''}},
\widehat{\mathcal{U}_1},
g_1,
\left\{g_{1,p''}\right\},
\left\{\widehat{g}_{1,p'',x''}\right\}
\right),
\\
\overline{G}_2 =
\left(
{\widehat{\mathcal{U}_2''}},
\widehat{\mathcal{U}_2},
g_2,
\left\{g_{2,p''}\right\},
\left\{\widehat{g}_{2,p'',x''}\right\}
\right)
\end{cases}
\end{equation}
be pre-morphisms such that
\[
\overline{G}_1 \sim \overline{G}_2.
\]
In particular,
\[
{\widehat{\mathcal{U}_1''}}
\sim
{\widehat{\mathcal{U}_2''}}.
\]
We consider two cases, described below, to obtain a weak equivalence
\[
K :
\mathfrak{X}' {}_{F}\!\times_{[\overline{G}_1]} \mathfrak{X}''
\longrightarrow
\mathfrak{X}' {}_{F}\!\times_{[\overline{G}_2]} \mathfrak{X}''.
\]

\begin{enumerate}[label=(\alph*)]
\item
Suppose that
\[
\widehat{\mathcal{U}_1''}
\sim
\widehat{\mathcal{U}_2''}
\]
with an equality
\[
U''_{2,p''}
=
U''_{1,p''}\times\mathbb{R}^{n''}
\]
for some $n''\geq 2$, satisfying
\[
g_{p''}
\equiv
g^{n''}_{p''}\big|_{U''_{p''}},
\]
where
\[
g^{n''}_{p''} :
U''_{p''}\times\mathbb{R}^{n''}
\longrightarrow
U_p
\]
is the map given by
\[
g^{n''}_{p''}(y,r)
:=
g_{p''}(y).
\]

The map between the underlying spaces is the obvious homeomorphism
\[
k :
X' {}_f\!\times_g X''
\xlongrightarrow{\simeq}
X' {}_f\!\times_g X''.
\]

In this case, we consider the embedding
\[
\begin{split}
k_{(p',p'')} :
U'_{p'} {}_{f_{p'}}\!\times_{g_{p''}} U''_{p''}
&\longrightarrow
U'_{p'} {}_{f_{p'}}\!\times_{g^{n''}_{p''}}
\left(U''_{p''}\times\mathbb{R}^{n''}\right)
\\
&\simeq
\left(
U'_{p'} {}_{f_{p'}}\!\times_{g^{n''}_{p''}} U''_{p''}
\right)
\times\mathbb{R}^{n''},
\\
(y',y'')
&\longmapsto
\bigl(y',(y'',0)\bigr).
\end{split}
\]

Observe that the image of the bundle embedding
$\overline{\mathrm{em}}_{p''}$ for the fibration
$g^{n''}_{p''}$ in Definition~\ref{adfaf} (a) (iii) lies inside
$E''_{p''}\subset E''_{p''}\times\mathbb{R}^{n''}$. Hence, we can define
the Koszul $L_{\infty}[1]$-morphism
$\widehat{k}^{\mathrm{K}}_{(p',p''),(x',x'')}$ induced by the projection
\[
\begin{split}
&
\Gamma\Bigl(
\pi'^* E_{p'}'^*
\oplus
\pi''^*
\left[
E_{p''}''^*\oplus\mathbb{R}^{n''}
\right]
\Bigr)
\Big|_{
W'_{x'} {}_{f_{p'}}\times_{g_{p''}}
(W''_{x''}\times\mathbb{R}^{n''})
}
\\
&\hspace{2cm}
\twoheadrightarrow
\Gamma\Bigl(
\pi'^* E_{p'}'^*
\oplus
\pi''^*
\left[
E_{p''}''^*
\right]
\Bigr)
\Big|_{
W'_{x'} {}_{f_{p'}}\times_{g_{p''}} W''_{x''}
}.
\end{split}
\]

Together with the de Rham component, the resulting
$L_{\infty}[1]$-morphism
$\widehat{k}_{(p',p''),(x',x'')}$ is a quasi-isomorphism. In fact, the
data
\[
\overline{K}
:=
\left(
\widehat{\mathcal{U}''_1},
\widehat{\mathcal{U}_1''}\times\mathbb{R}^{n''},
k,
\left\{k_{(p',p'')}\right\},
\left\{
\widehat{k}_{(p',p''),(x',x'')}
\right\}
\right)
\]
determine a pre-morphism, which yields a weak equivalence. The
compatibility with coordinate changes is immediate from the
definitions.

\item
Suppose that
\[
\widehat{\mathcal{U}_1''}
\sim
\widehat{\mathcal{U}_2''}
\]
with $n_1=n_2=1$ in \eqref{urur}. Then, for each $p''$, there exists a
diffeomorphism
\begin{equation}\label{pigbbs}
U''_{1,p''}\cong U''_{2,p''}
\end{equation}
such that
\[
g_{1,p''}\big|_{\overset{\circ}{U}_{1,p''}}
\equiv
g_{2,p''}\big|_{\overset{\circ}{U}_{2,p''}}
\]
for some open subsets
\[
\overset{\circ}{U}_{i,p''}\subset U_{i,p''},
\qquad i=1,2,
\]
and the diffeomorphism restricts to a diffeomorphism
\[
\epsilon :
\overset{\circ}{U}_{1,p''}
\xlongrightarrow{\simeq}
\overset{\circ}{U}_{2,p''}
\]
as described in the following diagram:
\[
\begin{tikzcd}
    U''_{1,p''}\arrow[r, "\simeq"] & U''_{2,p''} \\
    \overset{\circ}{U''}_{1,p''}
    \arrow[r, "\epsilon", "\simeq"']
    \arrow[u, hook]
    &
    \overset{\circ}{U''}_{2,p''}
    \arrow[u, hook]
\end{tikzcd}
\]

We then define the base component map by
\[
\begin{split}
&
U'_{p'} {}_{f_{p'}}\!\times_{g_{1,p''}}
\overset{\circ}{U''}_{1,p''}
\xrightarrow{
\mathrm{id}_{U'_{p'}}\times\epsilon,\ \simeq
}
U'_{p'} {}_{f_{p'}}\!\times_{g_{2,p''}}
\overset{\circ}{U''}_{2,p''},
\\
&
(y',y''_1)
\longmapsto
\bigl(y',\epsilon(y''_1)\bigr),
\end{split}
\]
for $y'\in U'_{p'}$ and
$y''_1\in\overset{\circ}{U''}_{1,p''}$ satisfying
\[
f_{p'}(y')
=
g_{1,p''}(y''_1)
=
g_{2,p''}\bigl(\epsilon(y''_1)\bigr).
\]

The underlying topological map is given by the homeomorphism
\begin{equation}\label{esssv}
X' {}_f\!\times_g X''
\xlongrightarrow{\simeq}
X' {}_f\!\times_g X'',
\end{equation}
obtained from (i).

For the $L_{\infty}$-component, we use the fact that the isomorphism
\eqref{bnhh} in Definition~\ref{eqats}, together with axiom (vii),
restricts to an isomorphism between the local
$L_{\infty}[1]$-algebras on the corresponding open subdata.

It is immediate to verify that these data determine a pre-morphism by
the axioms in Definition~\ref{eqats}. Moreover, the resulting morphism
is a weak equivalence by \eqref{esssv} and condition (iv) in
Definition~\ref{eqats}.
\end{enumerate}

The same argument applies to different choices of representatives for
$F$, so we omit the details.

\end{proof}

By Lemma~\ref{genbvf}, with an appropriate choice of domain atlases,
fiber products of morphisms with manifold targets always exist.
However, this does not imply that we obtain a fiber product in the
strict sense, since its definition involves noncanonical choices, and
the resulting fiber product may depend on these choices. The preceding
proposition shows that it nevertheless defines an honest fiber product
in the homotopy category $\mathrm{Ho}(\mathbf{Kur})$. In other words,
it is a homotopy fiber product.

\begin{cor}
For any two morphisms with a common manifold target, a homotopy fiber
product can be defined in $\mathbf{Kur}$. In other words, it gives a
fiber product in the homotopy category $\mathrm{Ho}(\mathbf{Kur})$.
\end{cor}

\subsection{Examples of fiber products}

In this subsection, we present several examples of fiber products of Kuranishi spaces induced by fibrations. These include generalized intersections of smooth manifolds and boundary descriptions for moduli spaces of pseudoholomorphic discs. We also show that the Weinstein category can be formulated within the homotopy category $\operatorname{Ho}(\mathbf{Kur})$. Finally, we discuss product Kuranishi spaces, interpreting them as a special case of fiber products.

\begin{exam-defn}[Generalized intersections of smooth manifolds]
Let $f \colon M' \to M$ and $g \colon M'' \to M$ be smooth maps between manifolds. We define the fiber product $M' {}_f \!\times_g M''$ to be the object in $\mathbf{Kur}$ determined by the chart-level construction established above. Specifically, we set
\[
(M', f) \cap (M'', g) := M' \,{}_{f}\!\times_{g} M''.
\]
In particular, when $M'$ and $M''$ are embedded submanifolds of $M$ via embeddings $f$ and $g$, respectively, we call this fiber product the \textit{generalized intersection} of $M'$ and $M''$ in $M$.
\end{exam-defn}

\begin{exam}[The moduli space of pseudoholomorphic discs]
Consider the Kuranishi space structure on the moduli space $\mathcal{M}_k(L; \beta)$ in Example~\ref{moduli}, equipped with evaluation morphisms $\mathrm{ev}_j \colon \mathcal{M}_k(L; \beta) \to L$. We define the fiber product
\[
\mathcal{M}_{k_1}(L; \beta_1) {}_{\mathrm{ev}_{j_1}} \!\times_{\mathrm{ev}_{j_2}} \mathcal{M}_{k_2}(L; \beta_2)
\]
to be the object in $\mathbf{Kur}$ determined by the aforementioned construction. This description is expected to be useful in the study of pseudoholomorphic discs, where the boundary strata are canonically identified with fiber products of moduli spaces with fewer marked boundary points.
\end{exam}

\begin{defn}[Symplectic manifolds and Lagrangian Kuranishi spaces]
Symplectic manifolds equipped with the smooth atlases defined in Example~\ref{man} are naturally regarded as objects in $\mathbf{Kur}$. They satisfy the following conditions:
\begin{itemize}
\item[--] $\alpha_p$ is nondegenerate for each $p \in M$;
\item[--] $\alpha_p = \phi_{pq}^* \alpha_q$ on $s_p^{-1}(0)$ for each pair $p, q \in M$.
\end{itemize}
In this case, the foliation de Rham component of the local $L_\infty[1]$-algebra vanishes except in degree zero (smooth functions).

Let $M$ be a symplectic manifold regarded as an object in $\mathbf{Kur}$. We say that another object $\mathfrak{L}$, equipped with a choice of atlas
\[
\widehat{\mathcal{U}}_{\mathfrak{L}} = \left(\left\{\widehat{\mathcal{U}}_{\mathfrak{L},p}\right\}_{p \in L}, \left\{\Phi_{pq}\right\}_{p,q \in L}\right)
\]
for the underlying topological space $L$, is a \textit{Lagrangian Kuranishi space} in $M$ if there exists an embedding $I \colon \mathfrak{L} \to M$ represented by
\[
\left(\widehat{\mathcal{U}}_{\mathfrak{L}}, \widehat{\mathcal{U}}^{\mathrm{man}}_{M}, i, \{i_p\}, \{\widehat{i}_{p, x}\}\right)
\]
such that $i_p^*(\alpha_{i(p)}) = 0$ for every $p \in L$.
\end{defn}

\begin{exam-defn}[Weinstein category]\label{bjsjbhb}
Symplectic manifolds and their Lagrangian submanifolds serve as standard examples of symplectic Kuranishi spaces and their Lagrangians, respectively. Let
\[
\begin{split}
\operatorname{pr}_1 &\colon M \times M' \to M, \\
\operatorname{pr}_2 &\colon M \times M' \to M'
\end{split}
\]
denote the canonical projections.

\textbf{The category $\mathbf{Wein}$:}
We define the Weinstein category $\mathbf{Wein}$ as consisting of:
\[
\begin{split}
\operatorname{Ob}(\mathbf{Wein}) &:= \left\{ \text{symplectic manifolds in } \operatorname{Ho}(\mathbf{Kur}) \right\}, \\
\operatorname{Mor}_{\mathbf{Wein}}(M, M') &:= \left\{ \mathfrak{L} \overset{I}{\hookrightarrow} M^- \times M' \mathrel{\Big|} (\mathfrak{L}, I) \text{ is a Lagrangian Kuranishi space} \right\}.
\end{split}
\]
Composition is given by the fiber product of Kuranishi spaces:
\[
\begin{split}
\circ \colon \operatorname{Mor}_{\mathbf{Wein}}(M', M'') \times \operatorname{Mor}_{\mathbf{Wein}}(M, M') &\longrightarrow \operatorname{Mor}_{\mathbf{Wein}}(M, M''), \\
\left((\mathfrak{L}', I'), (\mathfrak{L}, I)\right) &\longmapsto \mathfrak{L} \,{}_{\operatorname{pr}_2 \circ I}\!\times_{\operatorname{pr}_1' \circ I'} \mathfrak{L}',
\end{split}
\]
induced by the morphisms
\[
\begin{split}
\operatorname{pr}_2 \circ I &\colon \mathfrak{L} \longrightarrow M', \\
\operatorname{pr}_1' \circ I' &\colon \mathfrak{L}' \longrightarrow M',
\end{split}
\]
with identities provided by the diagonal Lagrangians $\Delta_M \overset{I_{\Delta_M}}{\hookrightarrow} M^- \times M$.
\end{exam-defn}

\begin{prop}
$\mathbf{Wein}$ is a well-defined category.
\end{prop}

\begin{proof}
The associativity of composition follows from the associativity of fiber products up to canonical isomorphism in $\operatorname{Ho}(\mathbf{Kur})$. It therefore suffices to verify that the pair $(\Delta_M, I_{\Delta_M})$ acts as the identity morphism.

Observe that the composition
\[
\Delta_M \overset{I_{\Delta_M}}{\hookrightarrow} M \times M \xrightarrow{\operatorname{pr}_1} M
\]
is a diffeomorphism between smooth manifolds, and hence a trivial fibration. By Proposition~\ref{hnuun}~(ii), the fiber product
\[
\mathfrak{L} \,{}_{\operatorname{pr}_2 \circ I}\!\times_{\operatorname{pr}_1 \circ I_{\Delta_M}} \Delta_M
\]
is weakly equivalent to $\mathfrak{L}$. In the homotopy category $\operatorname{Ho}(\mathbf{Kur})$, this weak equivalence yields a canonical isomorphism. Furthermore, the universal property and commutativity of the fiber product diagram imply that the projections satisfy
\[
\operatorname{pr}_1 \circ \left(I \times_{I_{\Delta_M}} \mathrm{id}\right) = \operatorname{pr}_1 \circ I,
\qquad
\operatorname{pr}_2 \circ \left(I \times_{I_{\Delta_M}} \mathrm{id}\right) = \operatorname{pr}_2 \circ I,
\]
which proves that $(\Delta_M, I_{\Delta_M})$ acts as the identity, as claimed.
\end{proof}

\begin{exam}[Product Kuranishi spaces]
As a special case of the fiber product, we introduce the product of two Kuranishi spaces, which corresponds to the case where both spaces are mapped to the terminal object, namely a point equipped with the trivial chart. Note that in this case, the axioms for a fibration hold vacuously. Hence, unlike the general fiber product, no additional conditions such as the fibration property are required. Since most of the structures are essentially identical to those in the fiber product construction, we omit redundant details.

Let
\[
\mathfrak{X}' = \left(X',\,\left[\widehat{\mathcal{U}}'\right]\right)
\qquad\text{and}\qquad
\mathfrak{X}'' = \left(X'',\,\left[\widehat{\mathcal{U}}''\right]\right)
\]
be Kuranishi spaces.

Denote by $*$ the trivial Kuranishi space whose underlying space is a single point, equipped with the zero-rank vector bundle and trivial chart data. By Lemma~\ref{finob}, $*$ is the terminal object in $\mathbf{Kur}$. Moreover, the canonical morphisms
\[
F_0 \colon \mathfrak{X}' \longrightarrow *,
\qquad
G_0 \colon \mathfrak{X}'' \longrightarrow *
\]
are fibrations, as the axioms in Definition~\ref{adfaf}~(a) hold trivially.

We then define the \textit{product Kuranishi space} $\mathfrak{X}' \times \mathfrak{X}''$ as the fiber product
\[
\mathfrak{X}' \times \mathfrak{X}'' := \mathfrak{X}' {}_{F_0}\times_{G_0} \mathfrak{X}''.
\]
For later use in Sections~5 and~6, and to fix notation, we record the charts and coordinate changes explicitly.

The product atlas on $\mathfrak{X}' \times \mathfrak{X}''$ is given by
\[
\widehat{\mathcal{U}}' \times \widehat{\mathcal{U}}''
:=
\left(
\left\{\,\widehat{\mathcal{U}}^\times_{(p',p'')}\,\right\}_{(p',p'') \in X' \times X''},
\;
\left\{\,
\Phi^\times_{p'q',\,p''q''}
\,\right\}_{p',q' \in X',\; p'',q'' \in X''}
\right),
\]
where to each pair $(p',p'') \in X' \times X''$, we assign the product chart
\[
\mathcal{U}^\times_{(p',p'')}
:= \left( U^\times_{(p',p'')},\, E^\times_{(p',p'')},\, s^\times_{(p',p'')},\, \psi^\times_{(p',p'')} \right)
\]
consisting of:
\begin{itemize}
    \item[--] $U^\times_{(p',p'')} := U'_{p'} \times U''_{p''}$, which is a smooth manifold, equipped with the closed two-form $\pi^{\prime*}\alpha'_{W'_{x'}} - \pi^{\prime\prime*}\alpha''_{W''_{x''}}$;
    \item[--] $E^\times_{(p',p'')} := \pi^{\prime*} E'_{p'} \oplus \pi^{\prime\prime*} E''_{p''}$;
    \item[--] $s^\times_{(p',p'')} := \pi^{\prime*} s'_{p'} \oplus \pi^{\prime\prime*} s''_{p''} \colon U_{(p',p'')}^\times \longrightarrow E_{(p',p'')}^\times$, given explicitly by
    \[
    (y', y'') \longmapsto \left(\pi^{\prime*} s'_{p'}(y'),\, \pi^{\prime\prime*} s''_{p''}(y'')\right);
    \]
    \item[--] $\psi^\times_{(p',p'')}$, defined by the composition
    \begin{equation}\nonumber
    \begin{split}
    \psi^\times_{(p',p'')} \colon \left(s^\times_{(p',p'')}\right)^{-1}(0) &= \left(\pi^{\prime*} s'_{p'} \oplus \pi^{\prime\prime*} s''_{p''}\right)^{-1}(0) \\
    &\simeq (s'_{p'})^{-1}(0) \times (s''_{p''})^{-1}(0) \\
    &\longrightarrow X' \times X'', \\
    (y',y'') &\longmapsto \left(\psi'_{p'}(y'),\, \psi''_{p''}(y'')\right).
    \end{split}    
    \end{equation}
\end{itemize}

The local presymplectic neighborhood for each $(x', x'') \in \left(s^\times_{(p',p'')}\right)^{-1}(0)$ is given by an open subset
\[
W^\times_{(x',x'')} \subset W'_{x'} \times W''_{x''}.
\]
Notice that $\pi^{\prime*}\alpha'_{W'_{x'}} - \pi^{\prime\prime*}\alpha''_{W''_{x''}}$ has constant rank, and thus defines a presymplectic form.

The local $L_{\infty}[1]$-algebra for each $(x', x'') \in \left(s^\times_{(p',p'')}\right)^{-1}(0)$ is defined by
\begin{equation}\nonumber
\begin{split}
\mathcal{C}^\times_{(p',p''),(x',x'')}
&:=
\bigwedge\nolimits^{-\bullet} \Gamma
\left( \left(\pi^{\prime*}E_{p'}^{\prime*} \oplus \pi^{\prime\prime*}E_{p''}^{\prime\prime*}\right) \Big|_{W^\times_{(x',x'')}} \right)
\;\oplus\;
\Omega^{\bullet+1}_{\mathrm{aug}}\left(\mathcal{F}^\times_{(x',x'')}\right) \\
&= \left(\bigoplus_{i \geq 0}\bigwedge\nolimits^{-\bullet +i}\Gamma\left(\pi^{\prime*} (E'_{p'})^* \Big|_{W^\times_{(x',x'')}}\right) \otimes \bigwedge\nolimits^{-i}\Gamma\left(\pi^{\prime\prime*} (E''_{p''})^* \Big|_{W^\times_{(x',x'')}}\right)\right) \\
&\quad\quad\quad\quad\quad\quad\quad\quad\quad\quad\quad\quad\quad\quad\quad\quad\quad\quad\quad\quad\quad\quad \oplus \Omega^{\bullet+1}_{\mathrm{aug}}\left(\mathcal{F}^\times_{(x',x'')}\right),
\end{split}
\end{equation}
where $\Omega^{\bullet+1}_{\mathrm{aug}}\left(\mathcal{F}^\times_{(x',x'')}\right)$ denotes the foliation de Rham complex with respect to the foliation $T\mathcal{F}^\times_{(x',x'')} \subset TW^\times_{(x',x'')}$ given by the kernel of the closed two-form $\left(\pi^{\prime*}\alpha'_{W'_{x'}} - \pi^{\prime\prime*}\alpha''_{W''_{x''}}\right)\big|_{W^\times_{(x',x'')}}$.

Its $L_{\infty}$-operations $l^\times_{(p',p''),(x',x'')} := \left\{l^\times_{(p',p''),(x',x''),k}\right\}_{k \geq 1}$,
\[
l^\times_{(p',p''),(x',x''),k} \colon
\left(\mathcal{C}^\times_{(p',p''),(x',x'')}\right)^{\otimes k}
\longrightarrow
\mathcal{C}^\times_{(p',p''),(x',x'')},
\]
decompose into Koszul and de Rham parts:
\[
l^\times_{(p',p''),(x',x''),k} = l^{\times, \mathrm{K}}_{(p',p''),(x',x''),k} \oplus l^{\times, \mathrm{dR}}_{(p',p''),(x',x''),k}.
\]

The Koszul part $l^{\times, \mathrm{K}}_{(p',p''),(x',x''),m}$ is defined for each homogeneous element
\[
a = \pi^{\prime*}a' \otimes \pi^{\prime\prime*}a'' \in \bigwedge\nolimits^{-i}\Gamma\left(\pi^{\prime*} (E'_{p'})^* \Big|_{W^\times_{(x',x'')}}\right) \otimes \bigwedge\nolimits^{-j}\Gamma\left(\pi^{\prime\prime*} (E''_{p''})^* \Big|_{W^\times_{(x',x'')}}\right)
\]
by
\[
l^{\times, \mathrm{K}}_{(p', p''), (x', x''), m}(a) := 
\begin{cases} 
\pi^{\prime*} \left(l^{\prime, \mathrm{K}}_{p',x', m}(a')\right) & \text{if } j=0 \text{ and } 1 \leq i, \\
\pi^{\prime\prime*} \left(l^{\prime\prime, \mathrm{K}}_{p'', x'', m}(a'')\right) & \text{if } i=0 \text{ and } 1 \leq j, \\
0 & \text{otherwise}.
\end{cases}
\]

The de Rham part $l^{\times, \mathrm{dR}}_{(p',p''),(x',x'')}$ is given by the $L_\infty[1]$-algebra structure on $\Omega^{\bullet+1}_{\mathrm{aug}}\left(\mathcal{F}^\times_{(x',x'')}\right)$ associated with the closed two-form $\left(\pi^{\prime*}\alpha'_{W'_{x'}} - \pi^{\prime\prime*}\alpha''_{W''_{x''}}\right)\big|_{W^\times_{(x',x'')}}$ (cf.~Proposition~\ref{augomega}).

We now define coordinate changes. For $p',q' \in X'$, $p'',q'' \in X''$, and $(x', x'') \in \left(s^\times_{(p', p'')}\right)^{-1}(0) \cap \left(s^\times_{(q', q'')}\right)^{-1}(0)$, we consider the tuple
\[
\Phi^\times_{p'q',\,p''q''}
:=
\left(
U^\times_{p'q',\,p''q''},
\;
\phi^\times_{p'q',\,p''q''},
\;
\left\{\widehat{\phi}^\times_{p'q',\,p''q'',(x',x'')}\right\}
\right),
\]
consisting of:
\begin{itemize}
\item[--] $U^\times_{p'q',\,p''q''} := U'_{p'q'} \times U''_{p''q''}$;
\item[--] the smooth map
\[
\phi^\times_{p'q',\,p''q''} \colon U^\times_{p'q',\,p''q''} \longrightarrow U'_{q'} \times U''_{q''}
\]
defined by
\[
\phi^\times_{p'q',\,p''q''}(y',y'')
:=
\left(
\phi'_{p'q'}(y'),
\;
\phi''_{p''q''}(y'')
\right),
\]
where $\phi'_{p'q'} \colon U'_{p'q'} \to U'_{q'}$ and $\phi''_{p''q''} \colon U''_{p''q''} \to U''_{q''}$ are the coordinate changes of $\mathfrak{X}'$ and $\mathfrak{X}''$, respectively;
\item[--] the $L_\infty[1]$-morphism
\[
\widehat{\phi}^\times_{p'q',\,p''q'',(x',x'')}
\colon
\mathcal{C}^\times_{(q',q''),\phi^\times_{p'q',\,p''q''}(x',x'')}
\longrightarrow
\mathcal{C}^\times_{(p',p''),(x',x'')}
\]
defined analogously to the general fiber product case.
\end{itemize}

The canonical projections
\[
\begin{split}
\operatorname{Pr}' &\colon \mathfrak{X}' \times \mathfrak{X}'' \longrightarrow \mathfrak{X}', \\
\operatorname{Pr}'' &\colon \mathfrak{X}' \times \mathfrak{X}'' \longrightarrow \mathfrak{X}''
\end{split}
\]
are defined in the same manner. Specifically, they are represented by the pre-morphisms
\[
\begin{split}
\overline{\operatorname{Pr}}' &= \left(\widehat{\mathcal{U}}' \times \widehat{\mathcal{U}}'',\, \widehat{\mathcal{U}}',\, \operatorname{pr}',\, \left\{\operatorname{pr}'_{(p',p'')}\right\},\, \left\{ \widehat{\operatorname{pr}}'_{(p',p''),(x',x'')} \right\}\right), \\
\overline{\operatorname{Pr}}'' &= \left(\widehat{\mathcal{U}}' \times \widehat{\mathcal{U}}'',\, \widehat{\mathcal{U}}'',\, \operatorname{pr}'',\, \left\{\operatorname{pr}''_{(p',p'')}\right\},\, \left\{\widehat{\operatorname{pr}}''_{(p',p''),(x',x'')}\right\}\right).
\end{split}
\]

In particular, the Koszul component of the second projection morphism,
\[
\widehat{\operatorname{pr}}^{\prime\prime \mathrm{K}}_{(p',p''),(x',x'')} \colon \bigwedge\nolimits^{-\bullet}\Gamma\left((E''_{p''})^* \Big|_{W''_{x''}}\right) \longrightarrow \bigwedge\nolimits^{-\bullet}\Gamma\left((E^\times_{(p',p'')})^* \Big|_{W^\times_{(x',x'')}}\right),
\]
is given on homogeneous elements $a''$ by
\begin{equation}\label{cqqqqq}
\widehat{\operatorname{pr}}^{\prime\prime \mathrm{K}}_{(p',p''),(x',x''),m}(a'') := 
\begin{cases}
\left(0,\, \pi^{\prime\prime*}a''\right) & \text{if } m = 1, \\
0 & \text{if } m \geq 2.
\end{cases}
\end{equation}
This agrees with the general fiber product formula for $\widehat{\operatorname{pr}}''_{(p',p''),(x',x'')}$ in \eqref{xxbswb} when $\mathfrak{X}$ is taken to be the terminal object $*$ with its trivial Kuranishi structure.

We omit the verification that $\overline{\operatorname{Pr}}'$ and $\overline{\operatorname{Pr}}''$ define valid pre-morphisms, as the proof is identical to that for the fiber product.

In complete analogy with the fiber product construction, we obtain the following universal property:

\begin{lem}\label{dudhbbjbg}
The product Kuranishi space
\[
\mathfrak{X}' \times \mathfrak{X}''
:=
\left(
X' \times X'',\;
\left[\widehat{\mathcal{U}}' \times \widehat{\mathcal{U}}''\right]
\right)
\]
is the categorical product in $\mathbf{Kur}$; that is, for any Kuranishi space $\mathfrak{Y}$ and morphisms $H_1 \colon \mathfrak{Y} \to \mathfrak{X}'$, $H_2 \colon \mathfrak{Y} \to \mathfrak{X}''$, there exists a unique morphism $K \colon \mathfrak{Y} \to \mathfrak{X}' \times \mathfrak{X}''$ making the following diagram commute:
\[
\begin{tikzcd}
\mathfrak{Y} \arrow[rd, "{}^{\exists!} K"] \arrow[bend right=30, swap, "H_1"]{rdd} \arrow[bend left=20, "H_2"]{rrd} & & \\
& \mathfrak{X}' \times \mathfrak{X}''
\arrow[r, "\operatorname{Pr}''"]
\arrow[d, "\operatorname{Pr}'"']
& \mathfrak{X}'' \\
& \mathfrak{X}'. &
\end{tikzcd}
\]
\end{lem}
\end{exam}

\begin{rem}
We will apply the product Kuranishi space construction in Subsection~5.3 to establish the factorization property for categories of fibrant objects.
\end{rem}

\section{A model of path objects and the factorization axiom}

One of the key axioms for a category of fibrant objects is the factorization of any diagonal morphism into a weak equivalence followed by a fibration. In this section, we establish this axiom for the category $\mathbf{Kur}$ by constructing an explicit path object for any given Kuranishi space and showing that the diagonal morphism naturally factors through it.

\subsection{The path object}

In this subsection, for a given object of $\mathbf{Kur}$, we construct its path object.

On a smooth manifold $U$, we fix a Riemannian metric $g$ and its associated Levi-Civita connection. We define the manifold of short geodesics $\mathrm{P}_gU$, which is characterized by the following lemma in \cite{BLX}.

\begin{lem}\cite[Proposition 2.1]{BLX}\label{beaxx}
Let $I$ be an open interval containing $[0,1]$.
There exists a smooth manifold $\mathrm{P}_gU$ parameterizing geodesic paths
\[
\gamma \colon I \longrightarrow U
\]
such that the following diagram commutes:
\begin{equation}\label{pgutu}
\begin{tikzcd}
{} & \mathrm{P}_gU \arrow{ld}[swap]{\gamma(0)\times\gamma'(0)} \arrow{rd}{\gamma(0)\times\gamma(1)} & {}\\
TU & {} & U\times U \\
{} & U  \arrow{uu}{\mathrm{const.}} \arrow{lu}{0} \arrow{ru}[swap]{\Delta} & {}
\end{tikzcd}
\end{equation}
where the two upper diagonal maps 
\[
\gamma(0)\times\gamma'(0)
\quad\text{and}\quad
\gamma(0)\times\gamma(1)
\]
are open embeddings.
\end{lem}

\begin{proof}
We follow the argument in \cite[Proposition~2.1]{BLX}. Consider an open neighborhood $V$ of the zero section in $TU$ on which the exponential map is well-defined. We may further assume that the map $(t,v) \mapsto \exp(tv)$ is defined for all $t \in I$ and $v \in V$. Thus, $V$ parameterizes a smooth family of geodesics on $I$:
\[
    V \times I \longrightarrow U, \qquad
    (x,v,t) \longmapsto \gamma_{x,v}(t) := \exp_x(tv).
\]
Evaluating the differentials of the smooth maps
\[
    V \longrightarrow TU, \qquad
    (x,v) \longmapsto \left(\gamma_{x,v}(0),\, \gamma'_{x,v}(0)\right),
\]
and
\[
    V \longrightarrow U \times U, \qquad
    (x,v) \longmapsto \left(\gamma_{x,v}(0),\, \gamma_{x,v}(1)\right),
\]
along the zero section, we obtain isomorphisms on tangent spaces. Hence, by the inverse function theorem, there exists a smaller open neighborhood $U_0 \subset V$ of the zero section such that both induced maps $U_0 \to TU$ and $U_0 \to U \times U$ are diffeomorphisms onto open submanifolds. Setting $\mathrm{P}_g U := U_0$ yields the desired smooth manifold.
\end{proof}

\begin{notation}
We denote by 
\[
\gamma_{x,y} \in \mathrm{P}_gU
\]
the unique short geodesic satisfying $\gamma_{x,y}(0) = x$ and $\gamma_{x,y}(1) = y$.
\end{notation}

By means of the upper-left diagonal open embedding in \eqref{pgutu}, $\mathrm{P}_gU$ is identified up to diffeomorphism with an open neighborhood of the zero section of the tangent bundle $TU$. For each $t \in I$, we define the evaluation map
\[
\mathrm{ev}_t \colon \mathrm{P}_gU \longrightarrow U, \qquad  
\gamma \longmapsto \gamma(t).
\]
Given a vector bundle $F \to U$, we denote the corresponding pullback bundle by $\mathrm{ev}_t^*F \to \mathrm{P}_gU$.

\begin{defn}[Path Kuranishi space]\label{mawehe}
Let 
\[
\mathfrak{X} = \left(X,\, \left[\left\{{\mathcal{U}}_p\right\}, \left\{\Phi_{pq}\right\}\right]\right)
\]
be a Kuranishi space. We define its \textit{path space} $\mathrm{P}\mathfrak{X}$ by the collection of charts
\[
\mathrm{P}{\mathcal{U}}_p := \left(\mathrm{P}_gU_p \times W,\, {\mathrm{ev}_1}^*E_p \oplus {\mathrm{ev}_0}^*TU_p \oplus E_{p,W},\, 
s^\mathrm{P}_p,\, \psi^\mathrm{P}_p \right)
\]
specified by the following data:
\begin{itemize}
\item[--]
$W$ is a real vector space of dimension $\dim W = \operatorname{rk} E_{p}$, and $E_{p,W} \to \mathrm{P}_gU_p \times W$ is the trivial vector bundle with fiber $W$, so that $\operatorname{rk} E_{p,W} = \operatorname{rk} E_{p}$.

\item[--]
The base manifold $\mathrm{P}_gU_p \times W$ is equipped with the closed two-form $\alpha^\mathrm{P} \in \Omega^2(\mathrm{P}_gU_p \times W)$ defined by
\[
\alpha^\mathrm{P}_{(\gamma,w)} := \left(\operatorname{pr}_1^* \mathrm{ev}_0^* \alpha_p\right)_{(\gamma,w)}
\]
for each $(\gamma, w) \in \mathrm{P}_gU_p \times W$, where $\operatorname{pr}_1 \colon \mathrm{P}_gU_p \times W \to \mathrm{P}_gU_p$ is the projection.

\item[--] The section $s^\mathrm{P}_p \in \Gamma\left({\mathrm{ev}_1}^*E_p \oplus {\mathrm{ev}_0}^*TU_p \oplus E_{p,W}\right)$ is defined by
\[
\begin{split}
s^\mathrm{P}_p(\gamma, w) &:= \left(w + (\mathrm{ev}_0^* s_p)_\gamma + (\mathrm{ev}_1^* s_p)_\gamma\,, \sigma_p(\gamma),\, 2w 
 + (\mathrm{ev}_0^* s_p)_\gamma + (\mathrm{ev}_1^* s_p)_\gamma \right) \\
&= \left(w + s_{p}(\gamma(0)) + s_{p}(\gamma(1)),\, \dot{\gamma}(0),\, 2w + s_{p}(\gamma(0)) + s_{p}(\gamma(1))\right)
\end{split}
\]
for each $(\gamma, w) \in \mathrm{P}_gU_p \times W$, where $\sigma_p \in \Gamma({\mathrm{ev}}_0^*TU_p)$ is the canonical section given by $\sigma_{p}(\gamma) := \dot\gamma(0)$.

\item[--] 
For the zero locus, the condition $\dot\gamma(0) = 0$ implies that $\gamma$ is a constant path, $\gamma(t) \equiv x \in U_p$. Vanishing of the remaining components then yields $w = 0$ and $s_p(x) = 0$. Consequently,
\begin{equation}\label{zze}
\left(s^\mathrm{P}_p\right)^{-1}(0) = s_p^{-1}(0) \times \{0\} \;\simeq\; s_p^{-1}(0),
\end{equation}
and the underlying topological space of $\mathrm{P}\mathfrak{X}$ canonically identifies with $X$. We therefore set
\[
\psi^\mathrm{P}_p := \psi_p.
\]

\item[--] For each $x \in s_p^{-1}(0)$, we take the presymplectic neighborhood to be a contractible open subset
\begin{equation}\label{onoaong}
W_x^\mathrm{P} \subset \mathrm{ev}_0^{-1}(W_x) \times W
\end{equation}
equipped with the presymplectic form $\alpha_{W_x^\mathrm{P}}^\mathrm{P} := \left(\operatorname{pr}_1^*\mathrm{ev}_0^* \alpha_{W_x}\right)\big|_{W_x^\mathrm{P}}$ and the characteristic foliation
\[
T\mathcal{F}^\mathrm{P}_x := \ker \alpha_{W_x^\mathrm{P}}^\mathrm{P} = \left(\operatorname{pr}_1^*\mathrm{ev}_0^* T\mathcal{F}_x\right)\!\Big|_{W_x^\mathrm{P}} \oplus TW.
\]

\item[--] The local $L_\infty[1]$-algebra is given by
\[
\mathcal{C}^\mathrm{P}_x
:=
\bigwedge\nolimits^{-\bullet} \Gamma\left(
\left({\mathrm{ev}_1}^*E_p \oplus {\mathrm{ev}_0}^*TU_p \oplus E_{p,W}\right)^*\Bigg|_{W^\mathrm{P}_x}
\right)
\;\oplus\;
\Omega_{\mathrm{aug}}^{\bullet+1}\left(\mathcal{F}^\mathrm{P}_x\right).
\]

\item[--] The $L_{\infty}[1]$-algebra structure on $\mathcal{C}^\mathrm{P}_x$ is given by operations $l^\mathrm{P}_{x,k} \colon \left(\mathcal{C}^\mathrm{P}_x\right)^{\otimes k} \to \mathcal{C}^\mathrm{P}_x$ ($k \geq 1$) decomposing as
\[
l^\mathrm{P}_{x,k} := l^\mathrm{P,K}_{x,k} \oplus l^\mathrm{P,dR}_{x,k},
\]
where the Koszul part is
\begin{equation}\nonumber
l^\mathrm{P,K}_{x,k} := \begin{cases}
\iota_{s^\mathrm{P}_p} &\text{if } k = 1, \\
0 &\text{if } k \geq 2,
\end{cases}
\end{equation}
and $l^\mathrm{P,dR}_{x,k}$ denotes the $L_{\infty}[1]$-algebra structure on $\Omega^{\bullet+1}_{\mathrm{aug}}\left(\mathcal{F}^\mathrm{P}_x\right)$ given in Proposition~\ref{augomega}.
\end{itemize}
\end{defn}

For $p,q \in X$, we define the coordinate changes
\begin{equation}\label{bacccvbz}
\mathrm{P}\Phi_{pq}
:=
\left(
\mathrm{P}_gU_{pq} \times W_p,\;
\mathrm{P}\phi_{pq},\;
\left\{\widehat{\mathrm{P}\phi}_{pq,x}\right\}
\right)
\end{equation}
from $\mathrm{P}\mathcal{U}_p$ to $\mathrm{P}\mathcal{U}_q$ as follows:
\begin{itemize}
\item[--]
The open subset $\mathrm{P}_gU_{pq}$ is defined by
\[
\mathrm{P}_gU_{pq} := \left(\mathrm{P}_gU_p\right)\!\big|_{U_{pq}},
\]
where $\mathrm{P}_gU_p$ is viewed as an open neighborhood of the zero section of $TU_p$ via Lemma~\ref{beaxx}.

\item[--]
The smooth map on the base
\begin{equation}\label{efsvscs}
\mathrm{P}\phi_{pq} \colon \mathrm{P}_gU_{pq} \times W_p \longrightarrow \mathrm{P}_gU_q \times W_q
\end{equation}
is given by
\[
\mathrm{P}\phi_{pq}\left(\gamma, w \right) := \left(\phi_{pq}(\gamma)_g , \iota_{pq}(w)\right),
\]
where $\phi_{pq}(\gamma)_g$ denotes the short geodesic in $U_q$ from $\phi_{pq}\left(\gamma(0)\right)$ to $\phi_{pq}\left(\gamma(1)\right).$ Since $\phi_{pq}$ is an embedding, we have $\dim U_p \leq \dim U_q.$ Since we assume the virtual dimensions coincide $\dim U_p - \mathrm{rk}E_p = \dim U_q - \mathrm{rk}E_q$ for each coorinate change pair, so that we have $\mathrm{rk}E_p \leq \mathrm{rk}E_q.$ Then we denote a choice of embedding of vector spaces $\iota_{pq} : W_p \rightarrow W_q.$

\item[--] The compatibility conditions for the base coordinate changes, namely, (i) through (iv) in Definition~\ref{kstr}, hold for the following reasons: (i) is clear. The remaining three conditions are evaluated at the zero locus $(\gamma_{x,x}, 0)$ for each $x \in s_{p}^{-1}(0)$ and $p$, where the compatibilities reduce to those concerning $\phi_{pq}$ and $\psi_p$, which are already known to hold.

\item[--]
The $L_\infty[1]$-morphism
\[
\widehat{\mathrm{P}\phi}_{pq,x}^{\mathrm{c}} \colon
\mathcal{C}^\mathrm{P}_{q,\,\mathrm{P}\phi_{pq}(x),\,\mathrm{P}\phi_{pq}}
\longrightarrow
\mathcal{C}^\mathrm{P}_{p,\,x}
\]
is defined componentwise by $\widehat{\mathrm{P}\phi}^{\mathrm{c}}_{pq,x,k} := \widehat{\mathrm{P}\phi}^{\mathrm{c, K}}_{pq,x,k} \oplus \widehat{\mathrm{P}\phi}^{\mathrm{c, dR}}_{pq,x,k}$.
\end{itemize}

We write $\widehat{\mathrm{P}\phi}^{\mathrm{c}}_{pq,x}$ in the composition of three $L_{\infty}$-morphisms as follows:
\begin{equation}\label{bwccc}
\widehat{\mathrm{P}\phi}^{\mathrm{c}}_{pq,x} \colon \mathcal{C}^\mathrm{P}_{q,\,\mathrm{P}\phi_{pq}(x),\,\mathrm{P}\phi_{pq}} \xrightarrow{(3)} \overset{\circ}{\mathcal{C}}{}^\mathrm{P}_{q,\,\phi_{pq}(x),\,\phi_{pq}} \xrightarrow{(2)} \overset{\circ}{\mathcal{C}}{}^\mathrm{P}_{p,x} \xrightarrow{(1)} \mathcal{C}^\mathrm{P}_{p,x},
\end{equation}
where we denote
\[
\mathcal{C}^\mathrm{P}_{p,x}
:=
\bigwedge\nolimits^{-\bullet} \Gamma\left(
\left({\mathrm{ev}_1}^*E_p \oplus {\mathrm{ev}_0}^*T^*U_p \oplus E_{p,W}\right)^*\Bigg|_{W^\mathrm{P}_x}
\right)
\;\oplus\;
\Omega_{\mathrm{aug}}^{\bullet+1}\left(\mathcal{F}^\mathrm{P}_x\right),
\]
\[
\overset{\circ}{\mathcal{C}}{}^\mathrm{P}_{p,x}
:=
\bigwedge\nolimits^{-\bullet} \Gamma\left({\mathrm{ev}_1}^*E^*_p \Big|_{W_x}
\right)
\;\oplus\;
\Omega_{\mathrm{aug}}^{\bullet+1}\left(\overset{\circ}{\mathcal{F}}{}^\mathrm{P}_x\right),
\]
\[
\mathcal{C}^\mathrm{P}_{q,\,\mathrm{P}\phi_{pq}(x),\,\mathrm{P}\phi_{pq}} :=\left(\bigwedge\nolimits^{-\bullet} \Gamma\left({\mathrm{ev}_1}^*E^*_q \oplus \mathrm{ev}_0T^*U_q \oplus E^*_{q,W} \Big|_{W_{\phi_{pq}(x)}^{\mathrm{P}}}
\right)
\;\oplus\;
\Omega_{\mathrm{aug}}^{\bullet+1}\left({\mathcal{F}}{}^\mathrm{P}_{\phi_{pq}(x)}\right)\right)_{\mathrm{P}\phi_{pq}},
\]
and
\[
\overset{\circ}{\mathcal{C}}{}^\mathrm{P}_{q,\,(\phi_{pq}(x),0),\,\phi_{pq}} := \left(\bigwedge\nolimits^{-\bullet} \Gamma\left({\mathrm{ev}_1}^*E^*_q \Big|_{W'_{\phi_{pq}(x)}}
\right)
\;\oplus\;
\Omega_{\mathrm{aug}}^{\bullet+1}\left(\overset{\circ}{\mathcal{F}}{}^\mathrm{P}_{\phi_{pq}(x)}\right)\right)_{\phi_{pq}}.
\]

In what follows, we verify that each of those maps is a quasi-isomorphism.

\noindent (1) : $\overset{\circ}{\mathcal{C}}{}^\mathrm{P}_{p,x} \xrightarrow{(1)} \mathcal{C}^\mathrm{P}_{p,x}$

\noindent(\textit{de Rham part})

We first consider the base embedding
\[
\begin{split}
{W}_x &\xhookrightarrow{\quad} W^\mathrm{P}_x \\
y &\longmapsto (\gamma_{y,y}, 0),
\end{split}
\]
where we take
\[
W^\mathrm{P}_x \simeq {W}_x \times \left(W_x \times W_p\right).
\]

This base embedding respect the presymplectic forms $\alpha_p$ and $\alpha^\mathrm{P} := \operatorname{pr}_1^* \mathrm{ev}_0^* \alpha_p,$ hence induces an injective map between the foliation tangent bundle
\[
T\overset{\circ}{\mathcal{F}^\mathrm{P}}_{p, x} \xhookrightarrow{\quad} T\mathcal{F}^\mathrm{P}_{p, x}.
\]
This further indcues an $L_{\infty}[1]$-morphism between the foliation de Rham complex
\[
\Omega^{\bullet+1}\left(\mathcal{F}^\mathrm{P}_{p, x}\right) \longrightarrow \Omega^{\bullet+1}\left(\overset{\circ}{\mathcal{F}^\mathrm{P}}_{p, x}\right),
\]
which yields the de Rham $L_{\infty}$-coordinate change of (1), that is, the map between the augmentation by \cite[Lemma B.20]{Kim1}
\[
\Omega^{\bullet+1}_{\text{aug}}\left(\mathcal{F}^\mathrm{P}_{p, x}\right) \longrightarrow \Omega^{\bullet+1}_{\text{aug}}\left(\overset{\circ}{\mathcal{F}^\mathrm{P}}_{p, x}\right),
\]
which is necessarily a quasi-isomorphisms between acyclic $L_{\infty}[1]$-algebras.

\noindent (\textit{Koszul part})

The Koszul part quasi-isomorphism is obtained from the following bundle embedding

\begin{equation}\label{qvsvs}
\begin{tikzcd}[column sep=large, row sep=large]
E_p \arrow[r, hook, "\overline{\delta}_{1,p}"]
&
\mathrm{ev}_1^*E_p \oplus \mathrm{ev}_0^*TU_p \oplus E_{p,W}
\\
U_p \arrow[u, "s_p"] \arrow[r, hook, "\delta_{1,p}"] 
&
\mathrm{P}_gU_p \times W_p \arrow[u, swap, "s_p^{\mathrm{P}}"] 
\end{tikzcd}
\end{equation}
defined on each fiber over $y \in U_p$ by
\begin{equation}\label{nannnn}
\overline{\delta}_{1,p}(y, v) := \left(2v,\, 0,\, 2v\right) \in (E_p)_y \oplus T_y U_p \oplus W_p.
\end{equation}
Along the image $\delta_{1,p}(y) = (\gamma_{y,y}, 0)$, we have $\dot{\gamma}_{y,y}(0) = 0$ and $\mathrm{ev}_0(\gamma_{y,y}) = \mathrm{ev}_1(\gamma_{y,y}) = y$. Substituting $w = 0$ into the definition of $s_p^\mathrm{P}$ yields
\[
\begin{split}
\left(s_p^{\mathrm{P}} \circ \delta_{1,p}\right)(y)
&= s_p^{\mathrm{P}}(\gamma_{y,y}, 0)\\
&= \left(2s_p(y),\, \dot{\gamma}_{y,y}(0),\, 2s_p(y) \right)\\
&= \left(2s_p(y),\, 0,\, 2s_p(y)\right)\\
& = \left(\overline{\delta}_{1,p} \circ s_p \right)(y).
\end{split}
\]
We will revisit this bundle map in the proof of the factorization property of the diagonal morphism in Theorem \ref{facpro}.

\begin{lem}
There exists a bundle isomorphism:
\[
\begin{tikzcd}[column sep=large, row sep=large]
F \oplus F \arrow[r, "{\overline{\delta}, \simeq}"]
&
\Delta \oplus F^-
\\
U \arrow[u, "s_1 \oplus  s_2"] \arrow[r, "\mathrm{id}_U"] 
&
U, \arrow[u, swap, "{(s_1, s_1) \oplus (s_1 - s_2)}"]
\end{tikzcd}
\]
compatible with the section on both sides.
\end{lem}

\begin{proof}
We can take the bundle map to be
\[
\begin{split}
\overline{\delta} \colon F \oplus F &\xlongrightarrow{\simeq} \Delta \oplus F^- ,\\
(x, (a, b)) & \longmapsto (x, (a, a), (a-b)).
\end{split}
\]
It is immediate to show that this is an embedding. Moreover, it is compatible with the sections
\[
\left(x, \left(s_1(x), s_2(x)\right)\right) \xmapsto{\; \overline{\delta} \;} \left(x, \left(s_1(x),s_1(x)\right), \left(s_1(x) -s_2(x)\right)\right).
\]
\end{proof}

Now composing the above two bundle maps to obtain:
\[
\begin{tikzcd}[column sep=large, row sep=large]
E_p \arrow[r, hook, "\overline{\delta}_{1,p}"]
&
\mathrm{ev}_1^*E_p \oplus \mathrm{ev}_0^*TU_p \oplus E_{p,W} \arrow[r, "{\overline{\delta}, \simeq}"]
&
\Delta \oplus E^- \oplus \mathrm{ev}_0^*TU_p
\\
U_p \arrow[u, "s_p"] \arrow[r, hook, "\delta_{1,p}"] 
&
\mathrm{P}_gU_p \times W_p \arrow[u, swap, "s_p^{\mathrm{P}}"] \arrow[r, "\mathrm{id}"] 
&
\mathrm{P}_gU_p \times W_p, \arrow[u, swap, "s_p^{\Delta}"] 
\end{tikzcd}
\]
where $s_p^{\Delta}$ is the section given by
\[
\begin{split}
s_p^{\Delta} \colon \mathrm{P}_gU_p \times W_p &\longrightarrow \Delta \oplus E^- \oplus \mathrm{ev}_0^*TU_p,\\
(\gamma, w) &\longmapsto \left(\left(w + s_p(\gamma(0)) + s_p(\gamma(1)), w + s_p(\gamma(0)) + s_p(\gamma(1))\right), w, \dot{\gamma}(0)\right)
\end{split}
\]
obtained from the previous Lemma. We also have the compatibility
\begin{equation}\label{ogkksjsg}
\overline{\delta} \circ \overline{\delta}_{1,p} \circ s_p = s_p^{\Delta} \circ \delta_{1,p}.
\end{equation}

We now verify the tangent bundle condition along the zero locus $x \in s_p^{-1}(0)$. As we remarked earlier, the bundle embedding $(\delta_{1,p}, \overline{\delta}_{1,p})$ will be discussed later in the proof of Theorem~\ref{facpro} in Subsection 5.2. There, we verify the tangent bundle condition for the embedding \eqref{qvsvs}. The case of the isomorphism $(\mathrm{id}, \overline{\delta})$ with section compatibility is trivial, so we omit its proof. It then follows immediately that their composition satisfies the tangent bundle condition.

Therefore, by Theorem~\ref{afec} (i.e., \cite[Theorem~2.36]{Kim1}), this bundle embedding induces a quasi-isomorphism (or its homotopy inverse) on the Koszul complexes:
\[
\bigwedge\nolimits^{-\bullet} \Gamma\left(
\left(\mathrm{ev}_1^*E_p^* \oplus \mathrm{ev}_0^*T^*U_p \oplus E_{p,W}^*\right)\!\Big|_{W^\mathrm{P}_x}
\right)
\xrightarrow{\;\simeq\;}
\bigwedge\nolimits^{-\bullet} \Gamma\left(E_p^*\Big|_{W_x}\right).
\]
Composing this with the $L_{\infty}[1]$-isomorphism arising from the bundle isomorphism $(\mathrm{id}, \overline{\delta})$, we obtain 
\[
\widehat{\delta}^{\mathrm{K}}_{p,x,1} \colon
\bigwedge\nolimits^{-\bullet} \Gamma\left(
\left(\Delta^* \oplus E^{-*} \oplus \mathrm{ev}_0^*T^*U_p\right)\!\Big|_{W^\mathrm{P}_x}
\right)
\xrightarrow{\;\simeq\;}
\bigwedge\nolimits^{-\bullet} \Gamma\left(E_p^*\Big|_{W_x}\right).
\]
Setting the higher operations to be trivial ($\widehat{\delta}^{\mathrm{K}}_{p,x,k} = 0$ for $k \geq 2$), we define the Koszul $L_{\infty}$-coordinate change $\widehat{\delta}^{\mathrm{K}}_{p,x}$.

\noindent (2) : $\overset{\circ}{\mathcal{C}}{}^\mathrm{P}_{q,\,\phi_{pq}(x),\,\phi_{pq}} \xrightarrow{(2)} \overset{\circ}{\mathcal{C}}{}^\mathrm{P}_{p,x}$

Consider the $L_{\infty}$-coordinate
\[
\widehat{\phi}_{pq,x} \colon \overset{\circ}{\mathcal{C}}{}^\mathrm{P}_{q,\,\phi_{pq}(x),\,\phi_{pq}} \xrightarrow{(2)} \overset{\circ}{\mathcal{C}}{}^\mathrm{P}_{p,x},
\]
for each $p,q,$ with $x,$ and have it as (2) in \eqref{bwccc}. We already know it is a quasi-isomorphism.

\subsection{Construction of the quasi-isomorphism (3)} In this subsection, we continue to construct the $L_{\infty}$-coordinate change of path Kuranishi spaces with component (3).

\ \

\noindent (3) : $\mathcal{C}^\mathrm{P}_{q,\,\mathrm{P}\phi_{pq}(x),\,\mathrm{P}\phi_{pq}} \xrightarrow{(3)} \overset{\circ}{\mathcal{C}}{}^\mathrm{P}_{q,\,\phi_{pq}(x),\,\phi_{pq}} $

We first construct a map in the opposite direction
\[
\overset{\circ}{\mathcal{C}}{}^\mathrm{P}_{q,\,\phi_{pq}(x),\,\phi_{pq}}
\longrightarrow
\mathcal{C}^\mathrm{P}_{q,\,\mathrm{P}\phi_{pq}(x),\,\mathrm{P}\phi_{pq}},   
\]
show its quasi-isomorphism property, and consider its homotopy inverse to obtain (3) in \eqref{bwccc}.

\ \

\noindent(\textit{de Rham part})

For the base embedding
\[
\begin{split}
{W}'_{\phi_{pq}(x)} &\xhookrightarrow{\quad} W^\mathrm{P}_{\phi_{pq}(x)}, \\
y &\longmapsto (\gamma_{y,y}, 0),
\end{split}
\]
with a choice
\begin{equation}\label{wghssgwb}
W^\mathrm{P}_{\phi_{pq}(x)} \simeq {W}'_{\phi_{pq}(x)} \times \left({W}'_{\phi_{pq}(x)} \times W_q\right),
\end{equation}
we have the compatibility with the presymplectic forms $\alpha_p$ and $\alpha^\mathrm{P} := \operatorname{pr}_1^* \mathrm{ev}_0^* \alpha_q.$ Hence we obtain an injective map between the foliation tangent bundle
\[
T\overset{\circ}{\mathcal{F}^{\mathrm{P}}}_{q, \phi_{pq}(x)} \longrightarrow T\mathcal{F}^\mathrm{P}_{q, \phi_{pq}(x)},
\]
which further induces an $L_{\infty}[1]$-morphism
\[
\Omega^{\bullet+1}\left(\mathcal{F}^{\mathrm{P}}_{q, \phi_{pq}(x)}\right) \longrightarrow \Omega^{\bullet+1}\left(\overset{\circ}{\mathcal{F}^\mathrm{P}}_{q, \phi_{pq}(x)}\right).
\]
Then it yields the map between the augmentation by \cite[Lemma B.20]{Kim1}
\[
\Omega^{\bullet+1}_{\text{aug}}\left(\mathcal{F}^{\mathrm{P}}_{q, \phi_{pq}(x)}\right) \longrightarrow \Omega^{\bullet+1}_{\text{aug}}\left(\overset{\circ}{\mathcal{F}^\mathrm{P}}_{q, \phi_{pq}(x)}\right),
\]

Recall that there are canonical maps (cf. (\ref{varep}) \& Lemma~\ref{vecpf})
\[
\Omega^{\bullet+1}\left(\overset{\circ}{\mathcal{F}^{\mathrm{P}}}_{q, \phi_{pq}(x)}\right) \longrightarrow \Omega^{\bullet+1}\left(\overset{\circ}{\mathcal{F}^{\mathrm{P}}}_{q, \phi_{pq}(x)}\right)_{\phi_{pq}},
\]
\[
\Omega^{\bullet+1}\left(\mathcal{F}^\mathrm{P}_{q, \phi_{pq}(x)}\right) \longrightarrow \Omega^{\bullet+1}\left(\mathcal{F}^\mathrm{P}_{q, \phi_{pq}(x)}\right)_{\mathrm{P}\phi_{pq}}
\]
with their augmentations:
\begin{equation}\label{basshshh}
\Omega^{\bullet+1}_{\text{aug}}\left(\overset{\circ}{\mathcal{F}^{\mathrm{P}}}_{q, \phi_{pq}(x)}\right) \xlongrightarrow{\simeq} \Omega^{\bullet+1}_{\text{aug}}\left(\overset{\circ}{\mathcal{F}^{\mathrm{P}}}_{q, \phi_{pq}(x)}\right)_{\phi_{pq}},
\end{equation}
\[
\Omega^{\bullet+1}_{\text{aug}}\left(\mathcal{F}^\mathrm{P}_{q, \phi_{pq}(x)}\right) \xlongrightarrow{\simeq} \Omega^{\bullet+1}_{\text{aug}}\left(\mathcal{F}^\mathrm{P}_{q, \phi_{pq}(x)}\right)_{\mathrm{P}\phi_{pq}},
\]
which are quasi-isomorphisms between acyclic $L_{\infty}[1]$-algebras. We remark that in (\ref{basshshh}) the completion $(\cdots)_{\phi_{pq}}$ along the image of $\phi_{pq}$ that lies in the first component of the decomposition (\ref{wghssgwb}).

Noting the $L_\infty[1]$-morphism $(\ref{basshshh})$ is a quasi-isomorphism, we consider its homotopy inverse and the compositions
\[
\begin{split}
\Omega^{\bullet+1}_{\text{aug}}\left(\overset{\circ}{\mathcal{F}^{\mathrm{P}}}_{q, \phi_{pq}(x)}\right)_{\phi_{pq}} &\xlongrightarrow{\simeq} \Omega^{\bullet+1}_{\text{aug}}\left(\overset{\circ}{\mathcal{F}^{\mathrm{P}}}_{q, \phi_{pq}(x)}\right)\\ &\xlongrightarrow{\simeq} \Omega^{\bullet+1}_{\text{aug}}\left(\mathcal{F}^\mathrm{P}_{q, \phi_{pq}(x)}\right) \xlongrightarrow{\simeq} \Omega^{\bullet+1}_{\text{aug}}\left(\mathcal{F}^\mathrm{P}_{q, \phi_{pq}(x)}\right)_{\mathrm{P}\phi_{pq}},
\end{split}
\]
which we define to be the $L_\infty$-de Rham coordinate change. The quasi-isomorphism property of the de Rham part follows from the fact that it is an $L_{\infty}[1]$-morphism between acyclic $L_{\infty}[1]$-algebras.

\noindent(\textit{Koszul part})

We first construct a quasi-isomorphism
\[
\bigwedge\nolimits^{-\bullet}\Gamma\left(E^*_q\big|_{W_x}\right)_{\phi_{pq}} \longrightarrow \bigwedge\nolimits^{-\bullet}\Gamma\left(\mathrm{ev}_1^*E^*_q \oplus \mathrm{ev}_0^*T^*U_q \oplus E^*_{q,W}\right)_{\mathrm{P}\phi_{pq}},
\]
and take its homotopy inverse to define the Koszul $L_{\infty}$-coordinate change.

Recall that we can choose $W^{\mathrm{P}}_{\phi_{pq}(x)}$ in \eqref{onoaong} so that it can be decomposed as
\begin{equation}\label{oagjg}
W^{\mathrm{P}}_{\phi_{pq}(x)} \simeq W'_{\phi_{pq}(x)} \times W^{\prime \perp}_{\phi_{pq}(x)}.
\end{equation}
We can guarantee its existence by choosing $\iota_{pq}$ properly.
We then obtain an expression for $C^{\infty}\left(W^\mathrm{P}_{\phi_{pq}(x)}\right)$:
\begin{equation}\label{qqoottt}
C^{\infty}\left(W^\mathrm{P}_{\phi_{pq}(x)}\right) \simeq C^{\infty}\left(W'_{\phi_{pq}(x)}\right) \widehat{\otimes}_{\mathbb{R}} C^{\infty}\left(W^{\prime \perp}_{\phi_{pq}(x)}\right)
\end{equation}
as the completed tensor product between the smooth function spaces on the two components, with respect to the $C^{\infty}$-topology on $C^{\infty}\left(W^\mathrm{P}_{\phi_{pq}(x)}\right)$.

From the definition of the base coordinate change
\[
\mathrm{P}\phi_{pq}\big|_{W^\mathrm{P}_x \times W_p} \colon W^\mathrm{P}_x \times W_p \longrightarrow W^\mathrm{P}_{\phi_{pq}(x)} \times W_q,
\]
we also have the corresponding decomposition from \eqref{efsvscs} and \eqref{oagjg},
\begin{equation}\label{bbsgjgggff}
\mathrm{P}\phi_{pq} = \phi_{pq} \times \phi^\perp_{pq},
\end{equation}
into the starting and endpoint components.

Recall the completion
\[
C^{\infty}_{\mathrm{P}\phi_{pq}}\left(W^\mathrm{P}_{\phi_{pq}(x)}\right) := \lim\limits_{\longleftarrow} \frac{C^{\infty}\left(W^\mathrm{P}_{\phi_{pq}(x)}\right)}{\left(I_{\mathrm{P}\phi_{pq}}\right)^n},
\]
where $I_{\mathrm{P}\phi_{pq}}$ is the ideal in $C^{\infty}\left(W^\mathrm{P}_{\phi_{pq}(x)}\right)$ given by
\[
I_{\mathrm{P}\phi_{pq}} := \left\{ h \in C^{\infty}\left(W^\mathrm{P}_{\phi_{pq}(x)}\right) \Big| \ h\big|_{\mathrm{Im}(\mathrm{P}\phi_{pq})} = 0 \right\}.
\]
Note that under the identification \eqref{qqoottt}, we have
\[
I_{\mathrm{P}\phi_{pq}} \simeq I_{\phi_{pq}} \widehat{\otimes}_{\mathbb{R}} C^{\infty}\left(W^{\prime\perp}_{\phi_{pq}(x)}\right) + C^{\infty}\left(W'_{\phi_{pq}(x)}\right) \widehat{\otimes}_{\mathbb{R}} I^{\perp}_{\phi_{pq}},
\]
where $I_{\phi_{pq}}$ and $I^{\perp}_{\phi_{pq}}$ denote the ideals defined similarly to $I_{\mathrm{P}\phi_{pq}}$ with respect to the decomposition \eqref{bbsgjgggff}.

The following lemma provides a way to identify this space and justifies (1):
\begin{lem}
We have
\[
C^{\infty}_{\mathrm{P}\phi_{pq}}\left(W^\mathrm{P}_{\phi_{pq}(x)}\right) \simeq C^{\infty}_{\phi_{pq}}\left(W'_{\phi_{pq}(x)}\right) \widehat{\otimes}_{\mathbb{R}} C^{\infty}_{\phi^{\perp}_{pq}}\left(W^{\prime \perp}_{\phi_{pq}(x)}\right).
\]
\end{lem}

\begin{proof}[Proof-sketch]
It is straightforward to show that
\[
\begin{split}
&\frac{C^{\infty}\left(W'_{\phi_{pq}(x)}\right)}{\left(I_{\phi_{pq}}\right)^i} \;\widehat{\otimes}_{\mathbb{R}}\; \frac{C^{\infty}\left(W^{\prime \perp}_{\phi_{pq}(x)}\right)}{\left(I^{\perp}_{\phi_{pq}}\right)^j} \\
&\qquad\qquad \simeq \frac{C^{\infty}\left(W'_{\phi_{pq}(x)}\right) \widehat{\otimes}_{\mathbb{R}} C^{\infty}\left(W^{\prime \perp}_{\phi_{pq}(x)}\right)}{\left(I_{\phi_{pq}}\right)^i \widehat{\otimes}_{\mathbb{R}} C^{\infty}\left(W^{\prime\perp}_{\phi_{pq}(x)}\right) + C^{\infty}\left(W'_{\phi_{pq}(x)}\right) \widehat{\otimes}_{\mathbb{R}} \left(I^{\perp}_{\phi_{pq}}\right)^j}
\end{split}
\]
and
\[
\begin{split}
&\frac{C^{\infty}\left(W'_{\phi_{pq}(x)}\right) \widehat{\otimes}_{\mathbb{R}} C^{\infty}\left(W^{\prime \perp}_{\phi_{pq}(x)}\right)}{\left( I_{\mathrm{P}\phi_{pq}} \right)^k} \\
&\qquad\qquad \simeq \frac{C^{\infty}\left(W'_{\phi_{pq}(x)}\right) \widehat{\otimes}_{\mathbb{R}} C^{\infty}\left(W^{\prime \perp}_{\phi_{pq}(x)}\right)}{\left(I_{\phi_{pq}} \widehat{\otimes}_{\mathbb{R}} C^{\infty}\left(W^{\prime\perp}_{\phi_{pq}(x)}\right) + C^{\infty}\left(W'_{\phi_{pq}(x)}\right) \widehat{\otimes}_{\mathbb{R}} I^{\perp}_{\phi_{pq}}\right)^k}.
\end{split}
\]

From the inclusions
\[
\begin{split}
&\left(I_{\phi_{pq}} \widehat{\otimes}_{\mathbb{R}} C^{\infty}\left(W^{\prime\perp}_{\phi_{pq}(x)}\right) + C^{\infty}\left(W'_{\phi_{pq}(x)}\right) \widehat{\otimes}_{\mathbb{R}} I^{\perp}_{\phi_{pq}}\right)^{i+j} \\
&\qquad\qquad \subset \left(I_{\phi_{pq}}\right)^i \widehat{\otimes}_{\mathbb{R}} C^{\infty}\left(W^{\prime\perp}_{\phi_{pq}(x)}\right) + C^{\infty}\left(W'_{\phi_{pq}(x)}\right) \widehat{\otimes}_{\mathbb{R}} \left(I^{\perp}_{\phi_{pq}}\right)^j \\
&\qquad\qquad \subset \left(I_{\phi_{pq}} \widehat{\otimes}_{\mathbb{R}} C^{\infty}\left(W^{\prime\perp}_{\phi_{pq}(x)}\right) + C^{\infty}\left(W'_{\phi_{pq}(x)}\right) \widehat{\otimes}_{\mathbb{R}} I^{\perp}_{\phi_{pq}}\right)^{\min\{i,j\}},
\end{split}
\]
we see that the two inverse systems are cofinal. Thus, we have
\[
\begin{split}
C^{\infty}_{\mathrm{P}\phi_{pq}}\left(W^\mathrm{P}_{\phi_{pq}(x)}\right) &:= \lim\limits_{\underset{k}{\longleftarrow}} \frac{C^{\infty}\left(W^\mathrm{P}_{\phi_{pq}(x)}\right)}{\left(I_{\mathrm{P}\phi_{pq}}\right)^k} \\
&= \lim\limits_{\underset{k}{\longleftarrow}} \frac{C^{\infty}\left(W'_{\phi_{pq}(x)}\right) \widehat{\otimes}_{\mathbb{R}} C^{\infty}\left(W^{\prime \perp}_{\phi_{pq}(x)}\right)}{\left(I_{\phi_{pq}} \widehat{\otimes}_{\mathbb{R}} C^{\infty}\left(W^{\prime\perp}_{\phi_{pq}(x)}\right) + C^{\infty}\left(W'_{\phi_{pq}(x)}\right) \widehat{\otimes}_{\mathbb{R}} I^{\perp}_{\phi_{pq}}\right)^k} \\
&= \lim\limits_{\underset{i,j}{\longleftarrow}} \frac{C^{\infty}\left(W'_{\phi_{pq}(x)}\right) \widehat{\otimes}_{\mathbb{R}} C^{\infty}\left(W^{\prime \perp}_{\phi_{pq}(x)}\right)}{\left(I_{\phi_{pq}}\right)^i \widehat{\otimes}_{\mathbb{R}} C^{\infty}\left(W^{\prime\perp}_{\phi_{pq}(x)}\right) + C^{\infty}\left(W'_{\phi_{pq}(x)}\right) \widehat{\otimes}_{\mathbb{R}} \left(I^{\perp}_{\phi_{pq}}\right)^j} \\
&= \lim\limits_{\underset{i,j}{\longleftarrow}} \left(\frac{C^{\infty}\left(W'_{\phi_{pq}(x)}\right)}{\left(I_{\phi_{pq}}\right)^i} \;\widehat{\otimes}_{\mathbb{R}}\; \frac{C^{\infty}\left(W^{\prime \perp}_{\phi_{pq}(x)}\right)}{\left(I^{\perp}_{\phi_{pq}}\right)^j}\right) \\
&= C^{\infty}_{\phi_{pq}}\left(W'_{\phi_{pq}(x)}\right) \widehat{\otimes}_{\mathbb{R}} C^{\infty}_{\phi^{\perp}_{pq}}\left(W^{\prime \perp}_{\phi_{pq}(x)}\right).
\end{split}
\]
\end{proof}We then define an $\mathbb{R}$-linear map
\[
\begin{split}
\widehat{\eta}^{\mathrm{K}}_{x,1} \colon \bigwedge\nolimits^{-\bullet} \Gamma\left(E_q^*\big|_{W'_{\phi_{pq}(x)}}\right)_{\phi_{pq}} &= \bigwedge\nolimits^{-\bullet} \Gamma\left(E_q^*\big|_{W'_{\phi_{pq}(x)}}\right) \otimes_{C^\infty(W'_{\phi_{pq}(x)})} C^\infty_{\phi_{pq}}\left(W'_{\phi_{pq}(x)}\right) \\
&\longrightarrow \bigwedge\nolimits^{-\bullet} \Gamma\left(\left(\Delta^*_q \oplus E_q^{-*} \oplus \operatorname{ev}_0^* T^* U_q\right)\!\Big|_{W^\mathrm{P}_{\phi_{pq}(x)}}\right) \\
&\qquad\qquad\qquad\qquad\qquad \otimes_{C^\infty(W^\mathrm{P}_{\phi_{pq}(x)})} C_{\mathrm{P}\phi_{pq}}^\infty\left(W^\mathrm{P}_{\phi_{pq}(x)}\right)
\end{split}
\]
by
\[
a \otimes h \longmapsto \left(\overline{\delta} \circ \overline{\delta}_1 \circ \pi^* \circ \left(\delta_1^{-1}\right)^*(a)\right) \otimes \left(\pi^* \circ \left(\delta_1^{-1}\right)^*(h)\right)
\]
for each
\[
a \in \bigwedge\nolimits^{-\bullet} \Gamma\left(E_q^*\big|_{W'_{\phi_{pq}(x)}}\right) \qquad\text{and}\qquad h \in C^\infty_{\phi_{pq}}\left(W'_{\phi_{pq}(x)}\right),
\]
where
\[
\begin{split}
\pi \colon W^\mathrm{P}_{\phi_{pq}(x)} &\twoheadrightarrow \delta_{1}\left(W'_{\phi_{pq}(x)}\right), \\
(y, y', w) &\longmapsto (\gamma_{y,y'}, 0)
\end{split}
\]
is a choice of retraction map.

We can verify that $\widehat{\eta}^{\mathrm{K}}_{x,1}$ is a chain map using the section compatibility \eqref{ogkksjsg} and the definition of pullback dual sections:
\[
\begin{split}
\overline{\delta} \circ \overline{\delta}_1 \circ \pi^* \circ \left(\delta_1^{-1}\right)^*\!\left(\iota_{s'_q}(a)\right)
&= \overline{\delta} \circ \overline{\delta}_1 \circ \iota_{\pi^* \left(\delta_1^{-1}\right)^*(s'_q)} \left(\pi^* \circ \left(\delta_1^{-1}\right)^*(a)\right) \\
&= \iota_{s^{\Delta}_q} \left(\overline{\delta} \circ \overline{\delta}_1 \circ \pi^* \circ \left(\delta_1^{-1}\right)^*(a)\right).
\end{split}
\]
Moreover, the embedding property of $\overline{\delta} \circ \overline{\delta}_1$ together with the definition of pullbacks immediately implies that $\widehat{\eta}^{\mathrm{K}}_{x,1}$ is injective.

We want $\widehat{\eta}^{\mathrm{K}}_{x,1}$ to be a quasi-isomorphism. To show this, it suffices to prove the acyclicity of the cokernel
\[
\operatorname{coker}\left(\widehat{\eta}^{\mathrm{K}}_{x,1}\right) = \frac{\bigwedge\nolimits^{-\bullet} \Gamma\left(\left(\Delta^*_q \oplus E_q^{-*} \oplus \operatorname{ev}_0^* T^* U_q\right)\!\Big|_{W^\mathrm{P}_{\phi_{pq}(x)}}\right)_{\mathrm{P}\phi_{pq}}}{\widehat{\eta}^{\mathrm{K}}_{x,1}\left(\bigwedge\nolimits^{-\bullet} \Gamma\left(E_q^*\big|_{W'_{\phi_{pq}(x)}}\right)_{\phi_{pq}}\right)}.
\]

For this purpose, we need some preparation. Recall first that, by the definition of the pullback of dual sections, the $C^{\infty}(W^\mathrm{P}_{\phi_{pq}(x)})$-module action on $\widehat{\eta}^{\mathrm{K}}_{x,1}\left(\bigwedge\nolimits^{i}\Gamma\left(E_q^{\prime*}\big|_{W'_{\phi_{pq}(x)}}\right)_{\phi_{pq}}\right)$ by $f \in C^{\infty}(W^\mathrm{P}_{\phi_{pq}(x)})$ is given by
\[
\begin{split}
&\left(\overline{\delta} \circ \overline{\delta}_1 \circ \pi^* \circ \left(\delta_1^{-1}\right)^*(a)\right) \otimes \left(\pi^* \circ \left(\delta_1^{-1}\right)^*(h)\right) \\
&\quad \longmapsto \left(\overline{\delta} \circ \overline{\delta}_1 \circ \pi^* \circ \left(\delta_1^{-1}\right)^*\left(f\big|_{\mathrm{Im}\left(\delta_{1}\big|_{W'_{\phi_{pq}(x)}}\right)} \cdot a\right)\right) \otimes \left(\pi^* \circ \left(\delta_1^{-1}\right)^*\left(f\big|_{\mathrm{Im}\left(\delta_{1}\big|_{W'_{\phi_{pq}(x)}}\right)} \cdot h\right)\right),
\end{split}
\]
while the action on $\bigwedge\nolimits^{i} \Gamma\left(\Delta_q^* \big|_{W^\mathrm{P}_{\phi_{pq}(x)}}\right)_{\mathrm{P}\phi_{pq}}$ is given by the standard module multiplication. It is straightforward to see that these two module structures are not compatible. Therefore, the inclusion 
\begin{equation}\label{b3wgw}
\widehat{\eta}^{\mathrm{K}}_{x,1}\left(\bigwedge\nolimits^{i}\Gamma\left(E_q^{\prime*}\big|_{W'_{\phi_{pq}(x)}}\right)_{\phi_{pq}}\right) \xhookrightarrow{\quad} \bigwedge\nolimits^{i} \Gamma\left(\Delta_q^* \big|_{W^\mathrm{P}_{\phi_{pq}(x)}}\right)_{\mathrm{P}\phi_{pq}}
\end{equation}
can be understood only $\mathbb{R}$-linearly.

Observe that we have the following decomposition
\[
\begin{split}
\left(\Delta_q \oplus E_q^{-} \oplus \mathrm{ev}_0^*TU_q\right)\!\Big|_{W^\mathrm{P}_{\phi_{pq}(x)}}
&\simeq \Delta_q \big|_{W^\mathrm{P}_{\phi_{pq}(x)}} \oplus \left(E_q^{-} \oplus \mathrm{ev}_0^*TU_q\right)\!\Big|_{W^\mathrm{P}_{\phi_{pq}(x)}} \\
&\simeq \overline{\delta} \circ \overline{\delta}_1 \circ \pi^* \left(\left(\delta_1^{-1}\right)^*\left(E_q\big|_{W'_{\phi_{pq}(x)}}\right)\right) \oplus E^\mathrm{c},
\end{split}
\]
where we use the identification
\[
\Delta_q \big|_{W^\mathrm{P}_{\phi_{pq}(x)}} \simeq \overline{\delta} \circ \overline{\delta}_1 \circ \pi^* \left(\left(\delta_1^{-1}\right)^*\left(E_q\big|_{W'_{\phi_{pq}(x)}}\right)\right),
\]
and define
\[
E^\mathrm{c} := \left(E_q^{-} \oplus \mathrm{ev}_0^*TU_q\right)\!\Big|_{W^\mathrm{P}_{\phi_{pq}(x)}}.
\]

Similarly, for the dual bundle, we have
\[
\begin{split}
\left(\Delta_q^* \oplus E_q^{-*} \oplus \mathrm{ev}_0^*T^*U_q\right)\!\Big|_{W^\mathrm{P}_{\phi_{pq}(x)}}
&\simeq \Delta_q^* \big|_{W^\mathrm{P}_{\phi_{pq}(x)}} \oplus \left(E_q^{-*} \oplus \mathrm{ev}_0^*T^*U_q\right)\!\Big|_{W^\mathrm{P}_{\phi_{pq}(x)}} \\
&\simeq \overline{\delta} \circ \overline{\delta}_1 \circ \pi^* \left(\left(\delta_1^{-1}\right)^*\left(E^*_q\big|_{W'_{\phi_{pq}(x)}}\right)\right) \oplus E^{\mathrm{c}}.
\end{split}
\]
By abuse of notation, we write $E^\mathrm{c}$ for both cases. The section 
\[
s^{\Delta}_q \in \Gamma\left(\Delta_q \Big|_{W^\mathrm{P}_{\phi_{pq}(x)}} \oplus \left(E_q^{-} \oplus \operatorname{ev}_0^* T U_q\right)\!\Big|_{W^\mathrm{P}_{\phi_{pq}(x)}}\right)
\]
decomposes accordingly as
\[
s^{\Delta}_q =: s^{\Delta}_{\delta} \oplus s^{\Delta}_{\mathrm{c}}.
\]

Note that there exists a double complex structure
\[
\begin{split}
\bigwedge\nolimits^{-\bullet} &\Gamma\left(\left(\Delta^*_q \oplus E_q^{-*} \oplus \operatorname{ev}_0^* T^* U_q\right)\!\Big|_{W^\mathrm{P}_{\phi_{pq}(x)}}\right)_{\mathrm{P}\phi_{pq}} \\
&= \bigoplus_{i,j} \bigwedge\nolimits^{i,j} \Gamma\left(\left(\Delta_q^* \oplus E_q^{-*} \oplus \operatorname{ev}_0^* T^* U_q\right)\!\Big|_{W^\mathrm{P}_{\phi_{pq}(x)}}\right)_{\mathrm{P}\phi_{pq}},
\end{split}
\]
where each $(i,j)$-component is given by
\[
\begin{split}
&\bigwedge\nolimits^{i,j} \Gamma\left(\left(\Delta_q^* \oplus E_q^{-*} \oplus \operatorname{ev}_0^* T^* U_q\right)\!\Big|_{W^\mathrm{P}_{\phi_{pq}(x)}}\right)_{\mathrm{P}\phi_{pq}} \\
&\qquad := \bigwedge\nolimits^{i} \Gamma\left(\Delta_q^* \Big|_{W^\mathrm{P}_{\phi_{pq}(x)}}\right)_{\mathrm{P}\phi_{pq}} \otimes \bigwedge\nolimits^{j} \Gamma\left(\left(E_q^{-*} \oplus \operatorname{ev}_0^* T^* U_q\right)\!\Big|_{W^\mathrm{P}_{\phi_{pq}(x)}}\right)_{\mathrm{P}\phi_{pq}},
\end{split}
\]
equipped with the differential
\[
d = \iota_{s^{\Delta}_q} = \iota_{s^{\Delta}_{\delta}} + (-1)^i \iota_{s^{\Delta}_{\mathrm{c}}}
\]
on the $(i,j)$-component.

With this differential, $\operatorname{coker} \left(\widehat{\eta}^{\mathrm{K}}_{x,1}\right)$ also inherits a double complex structure:
\[
\begin{split}
&\frac{\bigwedge\nolimits^{-\bullet} \Gamma\left(\left(\Delta_q^* \oplus E_q^{-*} \oplus \operatorname{ev}_0^* T^* U_q\right)\!\Big|_{W^\mathrm{P}_{\phi_{pq}(x)}}\right)_{\mathrm{P}\phi_{pq}}}{\widehat{\eta}^{\mathrm{K}}_{x,1}\left(\bigwedge\nolimits^{-\bullet}\Gamma\left(E_q^*\Big|_{W'_{\phi_{pq}(x)}}\right)_{\phi_{pq}}\right)} \\ 
&\quad = \left( \bigoplus_{i \ge 0,\, j > 0} \bigwedge\nolimits^{i,j} \Gamma\left(\left(\Delta_q^* \oplus E_q^{-*} \oplus \operatorname{ev}_0^* T^* U_q\right)\!\Big|_{W^\mathrm{P}_{\phi_{pq}(x)}}\right)_{\mathrm{P}\phi_{pq}} \right) \\
&\qquad\qquad \oplus \bigoplus_{i \ge 0} \frac{\bigwedge\nolimits^{i,0} \Gamma\left(\left(\Delta_q^* \oplus E_q^{-*} \oplus \operatorname{ev}_0^* T^* U_q\right)\!\Big|_{W^\mathrm{P}_{\phi_{pq}(x)}}\right)_{\mathrm{P}\phi_{pq}}}{\widehat{\eta}^{\mathrm{K}}_{x,1}\left(\bigwedge\nolimits^{i}\Gamma\left(E_q^{\prime*}\Big|_{W'_{\phi_{pq}(x)}}\right)_{\phi_{pq}}\right)} \\
&\quad = \left( \bigoplus_{i \ge 0,\, j > 0} \bigwedge\nolimits^{i} \Gamma\left(\Delta_q^* \Big|_{W^\mathrm{P}_{\phi_{pq}(x)}}\right)_{\mathrm{P}\phi_{pq}} \otimes \bigwedge\nolimits^{j} \Gamma\left(\left(E_q^{-*} \oplus \operatorname{ev}_0^* T^* U_q\right)\!\Big|_{W^\mathrm{P}_{\phi_{pq}(x)}}\right)_{\mathrm{P}\phi_{pq}} \right) \\
&\qquad\qquad \oplus \bigoplus_{i \ge 0} \frac{\bigwedge\nolimits^{i} \Gamma\left(\Delta_q^* \Big|_{W^\mathrm{P}_{\phi_{pq}(x)}}\right)_{\mathrm{P}\phi_{pq}}}{\widehat{\eta}^{\mathrm{K}}_{x,1}\left(\bigwedge\nolimits^{i}\Gamma\left(E_q^{\prime*}\Big|_{W'_{\phi_{pq}(x)}}\right)_{\phi_{pq}}\right)}.
\end{split}
\]

By a standard argument in homological algebra, the acyclicity of this bounded double complex follows from that of its row and column complexes. Thus, to establish the acyclicity of the quotient complex, it suffices to show the acyclicity of each column complex:
\begin{equation}\label{diccpx}
\begin{split}
0 \longrightarrow \bigwedge\nolimits^{i,j} &\Gamma\left(\left(\Delta_q^* \oplus E_q^{-*} \oplus \operatorname{ev}_0^* T^* U_q\right)\!\Big|_{W^\mathrm{P}_{\phi_{pq}(x)}}\right)_{\mathrm{P}\phi_{pq}} \\
&\xrightarrow{\ \iota_{s^{\Delta}_\mathrm{c}}\ } \bigwedge\nolimits^{i,j-1} \Gamma\left(\left(\Delta_q^* \oplus E_q^{-*} \oplus \operatorname{ev}_0^* T^* U_q\right)\!\Big|_{W^\mathrm{P}_{\phi_{pq}(x)}}\right)_{\mathrm{P}\phi_{pq}} \dots \\
&\dots \xrightarrow{\ \iota_{s^{\Delta}_\mathrm{c}}\ } \frac{\bigwedge\nolimits^{i,0} \Gamma\left(\left(\Delta_q^* \oplus E_q^{-*} \oplus \operatorname{ev}_0^* T^* U_q\right)\!\Big|_{W^\mathrm{P}_{\phi_{pq}(x)}}\right)_{\mathrm{P}\phi_{pq}}}{\widehat{\eta}^{\mathrm{K}}_{x,1}\left(\bigwedge\nolimits^{i}\Gamma\left(E_q^*\Big|_{W'_{\phi_{pq}(x)}}\right)_{\phi_{pq}}\right)} \longrightarrow 0
\end{split}
\end{equation}
for each $i$. Note that the final differential $\iota_{s^{\Delta}_\mathrm{c}}$, by abuse of notation, denotes the composition of the interior product with $s^{\Delta}_\mathrm{c}$ and the projection onto the quotient module modulo the $\mathbb{R}$-subspace $\widehat{\eta}^{\mathrm{K}}_{x,1}\left(\bigwedge\nolimits^{i}\Gamma\left(E_q^*\Big|_{W'_{\phi_{pq}(x)}}\right)_{\phi_{pq}}\right)$.

Then, as an $\mathbb{R}$-module, the quotient term
\[
\bigoplus_{i \ge 0} \frac{\bigwedge\nolimits^{i} \Gamma\left(\Delta_q^* \Big|_{W^\mathrm{P}_{\phi_{pq}(x)}}\right)_{\mathrm{P}\phi_{pq}}}{\widehat{\eta}^{\mathrm{K}}_{x,1}\left(\bigwedge\nolimits^{i}\Gamma\left(E_q^{\prime*}\Big|_{W'_{\phi_{pq}(x)}}\right)_{\phi_{pq}}\right)}
\]
is computed as follows:
\[
\begin{split}
&\frac{\bigwedge\nolimits^{i,0} \Gamma\left(\left(\Delta_q^* \oplus E_q^{-*} \oplus \operatorname{ev}_0^* T^* U_q\right)\!\Big|_{W^\mathrm{P}_{\phi_{pq}(x)}}\right) \otimes_{C^\infty(W^\mathrm{P}_{\phi_{pq}(x)})} C_{\mathrm{P}\phi_{pq}}^\infty\left(W^\mathrm{P}_{\phi_{pq}(x)}\right)}{\widehat{\eta}^{\mathrm{K}}_{x,1}\left(\bigwedge\nolimits^{i} \Gamma\left(E_q^*\Big|_{W'_{\phi_{pq}(x)}}\right) \otimes_{C^\infty(W'_{\phi_{pq}(x)})} C^\infty_{\phi_{pq}}\left(W'_{\phi_{pq}(x)}\right)\right)} \\
&\quad \simeq \frac{\bigwedge\nolimits^{i} \Gamma\left(\Delta_q^* \Big|_{W^\mathrm{P}_{\phi_{pq}(x)}}\right) \otimes_{C^\infty(W^\mathrm{P}_{\phi_{pq}(x)})} C_{\mathrm{P}\phi_{pq}}^\infty\left(W^\mathrm{P}_{\phi_{pq}(x)}\right)}{\left(\overline{\delta} \circ \overline{\delta}_1 \circ \pi^* \circ \left(\delta_1^{-1}\right)^*\left(\bigwedge\nolimits^{i} \Gamma\left(E_q^*\Big|_{W'_{\phi_{pq}(x)}}\right)\right)\right) \otimes_{C^\infty(W'_{\phi_{pq}(x)})} \left(C^\infty_{\phi_{pq}}\left(W'_{\phi_{pq}(x)}\right) \widehat{\otimes}_{\mathbb{R}} \{1\}\right)} \\
&\quad \simeq \frac{\bigwedge\nolimits^{i} \Gamma\left(\overline{\delta} \circ \overline{\delta}_1 \circ \pi^* \circ \left(\delta_1^{-1}\right)^*\left(E_q^*\Big|_{W^\mathrm{P}_{\phi_{pq}(x)}}\right)\right) \otimes_{C^\infty(W^\mathrm{P}_{\phi_{pq}(x)})} \left(C^\infty_{\phi_{pq}}\left(W'_{\phi_{pq}(x)}\right) \widehat{\otimes}_{\mathbb{R}} C^{\infty}_{\phi_{pq}^{\perp}}\left(W_{\phi_{pq}(x)}^{\perp}\right)\right)}{\left(\overline{\delta} \circ \overline{\delta}_1 \circ \pi^* \circ \left(\delta_1^{-1}\right)^*\left(\bigwedge\nolimits^{i} \Gamma\left(E_q^*\Big|_{W'_{\phi_{pq}(x)}}\right)\right)\right) \otimes_{C^\infty(W'_{\phi_{pq}(x)})} \left(C^\infty_{\phi_{pq}}\left(W'_{\phi_{pq}(x)}\right) \widehat{\otimes}_{\mathbb{R}} \{1\}\right)} \\
&\quad \simeq \frac{\left(\bigwedge\nolimits^{i} \Gamma\left(E_q^* \Big|_{W'_{\phi_{pq}(x)}}\right) \otimes_{C^\infty(W'_{\phi_{pq}(x)})} C^{\infty}\left(W^{\mathrm{P}}_{\phi_{pq}(x)}\right)\right) \otimes_{C^\infty(W^\mathrm{P}_{\phi_{pq}(x)})} \left(C^\infty_{\phi_{pq}}\left(W'_{\phi_{pq}(x)}\right) \widehat{\otimes}_{\mathbb{R}} C^{\infty}_{\phi_{pq}^{\perp}}\left(W_{\phi_{pq}(x)}^{\perp}\right)\right)}{\left(\bigwedge\nolimits^{i} \Gamma\left(E_q^*\Big|_{W'_{\phi_{pq}(x)}}\right)\right) \otimes_{C^\infty(W'_{\phi_{pq}(x)})} \left(C^\infty_{\phi_{pq}}\left(W'_{\phi_{pq}(x)}\right) \widehat{\otimes}_{\mathbb{R}} \{1\}\right)} \\
&\quad \simeq \frac{\bigwedge\nolimits^{i} \Gamma\left(E_q^* \Big|_{W'_{\phi_{pq}(x)}}\right) \otimes_{C^\infty(W'_{\phi_{pq}(x)})} \left(C^\infty_{\phi_{pq}}\left(W'_{\phi_{pq}(x)}\right) \widehat{\otimes}_{\mathbb{R}} C^{\infty}_{\phi_{pq}^{\perp}}\left(W_{\phi_{pq}(x)}^{\perp}\right)\right)}{\left(\bigwedge\nolimits^{i} \Gamma\left(E_q^*\Big|_{W'_{\phi_{pq}(x)}}\right)\right) \otimes_{C^\infty(W'_{\phi_{pq}(x)})} \left(C^\infty_{\phi_{pq}}\left(W'_{\phi_{pq}(x)}\right) \widehat{\otimes}_{\mathbb{R}} \{1\}\right)} \\
&\quad \simeq \bigwedge\nolimits^{i} \Gamma\left(E_q^*\Big|_{W'_{\phi_{pq}(x)}}\right) \otimes_{C^\infty(W'_{\phi_{pq}(x)})} \frac{C^\infty_{\phi_{pq}}\left(W'_{\phi_{pq}(x)}\right) \widehat{\otimes}_{\mathbb{R}} C^{\infty}_{\phi_{pq}^{\perp}}\left(W_{\phi_{pq}(x)}^{\perp}\right)}{C^\infty_{\phi_{pq}}\left(W'_{\phi_{pq}(x)}\right) \widehat{\otimes}_{\mathbb{R}} \{1\}}.
\end{split}
\]

We remark that the $C^\infty\left(W'_{\phi_{pq}(x)}\right)$-module
\[
\frac{C^\infty_{\phi_{pq}}\left(W'_{\phi_{pq}(x)}\right) \widehat{\otimes}_{\mathbb{R}} C^{\infty}_{\phi_{pq}^{\perp}}\left(W_{\phi_{pq}(x)}^{\perp}\right)}{C^\infty_{\phi_{pq}}\left(W'_{\phi_{pq}(x)}\right) \widehat{\otimes}_{\mathbb{R}} \{1\}}
\]
consists of the elements in
\[
C^\infty_{\phi_{pq}}\left(W'_{\phi_{pq}(x)}\right) \widehat{\otimes}_{\mathbb{R}} C^{\infty}_{\phi_{pq}^{\perp}}\left(W_{\phi_{pq}(x)}^{\perp}\right) \simeq C_{\mathrm{P}\phi_{pq}}^{\infty}\left(W'_{\phi_{pq}(x)} \times W^{\perp}_{\phi_{pq}(x)}\right)
\]
that vanish on
\[
W'_{\phi_{pq}(x)} \times \{0\} \subset W'_{\phi_{pq}(x)} \times W^{\perp}_{\phi_{pq}(x)}.
\]

To prove the quasi-isomorphism property, recall that the bundle embeddings
\[
\begin{tikzcd}[column sep=large, row sep=large]
E_q \arrow[r, hook, "\overline{\delta}_{1,q}"] \arrow[d, swap, "s_q"]
& \operatorname{ev}_1^* E_q \oplus \operatorname{ev}_0^* T U_q \oplus E_{q,W} \arrow[r, "\overline{\delta}", "\simeq"' swap] \arrow[d, swap, "s_q^{\mathrm{P}}"]
& \Delta_q \oplus E_q^- \oplus \operatorname{ev}_0^* T U_q \arrow[d, swap, "s_q^{\Delta}"] \\
U_q \arrow[r, hook, "\delta_{1,q}"]
& \mathrm{P}_{g'} U_q \times W_q \arrow[r, "\mathrm{id}"]
& \mathrm{P}_{g'} U_q \times W_q
\end{tikzcd}
\]
satisfy the tangent bundle condition as in Case (1). Consequently, the decomposition
\[
s^{\Delta}_{q} = s^{\Delta}_{\delta} \oplus s^{\Delta}_{\mathrm{c}} \colon \mathrm{P}_{g'} U_q \times W_q \longrightarrow \Delta_q \oplus E_q^- \oplus \operatorname{ev}_0^* T U_q,
\]
where
\[
s^{\Delta}_{\mathrm{c}} := \sigma_q \oplus \mathrm{id} \colon \mathrm{P}_{g'} U_q \times W_q \longrightarrow E_{\mathrm{c}} := \operatorname{ev}_0^* T U_q \oplus E_q^-,
\]
has full rank at each zero point $\left(\gamma_{y,y}, 0\right) \in \left(s_q^{\mathrm{P}}\right)^{-1}(0)$. In other words, the differential map at $\left(\gamma_{y,y}, 0\right)$,
\begin{equation}\label{dgxdd}
\left(d s^{\Delta}_{\mathrm{c}}\right)_{\left(\gamma_{y,y}, 0\right)} = d_y \sigma_q \oplus \mathrm{id}_0 \colon T_y \mathrm{P}_g U_q \oplus T_0 W_q \longrightarrow T_{\sigma_q(y)} \operatorname{ev}_0^* T U_q \oplus T_{0} E_{q,W},
\end{equation}
has full rank.

To this situation, we apply the following lemma:

\begin{lem}\label{bkkjjhhh}
Let $F \to V$ be a vector bundle and $c \in \Gamma(F)$ a smooth section. For each zero point $y$ of $c$, there exists a neighborhood $W \subset V$ of $y$ such that the following holds: Let $\bigwedge\nolimits^{-\bullet}\Gamma(F^*|_W)$ be the Koszul complex with differential $\iota_c|_W$. If $c$ is transversal to the zero section at $y$, then the cohomology $H^{i}\left(\bigwedge\nolimits^{-\bullet}\Gamma(F^*|_W)\right)$ of the Koszul complex vanishes at every negative degree $i < 0$ on $W$.
\end{lem}

\begin{proof}
Express $c$ in an orthonormal frame $\{e_1, \dots, e_r\}$ for the section space $\Gamma(F)$ as
\begin{equation}\label{ssce}
c = \sum_{j=1}^r c_j e_j, \qquad c_j \in C^{\infty}(W).
\end{equation}

We first show that the tuple $(c_1, \dots, c_r)$ from \eqref{ssce} forms a \textit{regular sequence} in $C^\infty(W)$; that is, for each $1 \le i \le r$, $c_i$ is not a zero-divisor in the quotient ring
\[
\frac{C^{\infty}(W)}{\langle c_1, \dots, c_{i-1} \rangle}.
\]
This property is independent of the choice of frame $\{e_j\}$ and local components $\{c_j\}$.

Suppose, to the contrary, that $c_i$ is a zero-divisor in $C^{\infty}(W)/\langle c_1, \dots, c_{i-1}\rangle$. Then there exists $b_i \in C^{\infty}(W)$ such that $b_i \notin \langle c_1, \dots, c_{i-1} \rangle$ and $b_i c_i \in \langle c_1, \dots, c_{i-1} \rangle$, which allows us to write
\begin{equation}\label{zdld}
b_i c_i = \sum_{j=1}^{i-1} b_j c_j
\end{equation}
for some $b_j \in C^{\infty}(W)$, $j = 1, \dots, i-1$.

Since $c$ is transversal to the zero section at $y$, the matrix $\left( \frac{\partial c_k}{\partial x_l} (y) \right)_{k,l}$ has full rank $r$. By shrinking $W$ if necessary, we can choose a coordinate system around $y$ that extends $\{c_1, \dots, c_r\}$. Under these coordinates, \eqref{zdld} implies that
\[
b_i \Big|_{\{c_1 = \dots = c_{i-1} = 0\}} = 0.
\]
By Taylor's theorem with remainder on a sufficiently small convex neighborhood $W$, it follows that
\[
b_i \in \langle c_1, \dots, c_{i-1} \rangle,
\]
which contradicts the choice of $b_i$. Since $i$ was arbitrary, $(c_1, \dots, c_r)$ is a regular sequence.

Now consider the chain complex
\[
\mathcal{D}_0 \colon 0 \longrightarrow \overbrace{\bigwedge\nolimits^{r}\Gamma(F^*) \xrightarrow{\ \iota_{c}\ } \dots \xrightarrow{\ \iota_{c}\ } \Gamma(F^*)}^{\deg < 0} \xrightarrow{\ \iota_{c}\ } \overbrace{C^{\infty}(W)}^{\deg = 0} \longrightarrow 0.
\]

For any degree $i < 0$, consider the sequence
\begin{equation}\label{wes}
\dots \longrightarrow \bigwedge\nolimits^{i+1}\Gamma(F^*) \xrightarrow{\ \iota_{c}^{(i+1)}\ } \bigwedge\nolimits^{i}\Gamma(F^*) \xrightarrow{\ \iota_{c}^{(i)}\ } \bigwedge\nolimits^{i-1}\Gamma(F^*) \longrightarrow \dots,
\end{equation}
where the exactness $\ker \iota_c^{(i)} = \operatorname{Im} \iota_c^{(i+1)}$ follows from standard homological algebra under the condition that $(c_1, \dots, c_r)$ is a regular sequence (see, e.g., Chapter 17 of \cite{Eisenbud}). Since exactness holds over $C^{\infty}(W)$, the cohomology vanishes as an $\mathbb{R}$-vector space as well.
\end{proof}

Denoting $E^\mathrm{c} := \left(E_q^{-} \oplus \operatorname{ev}_0^* T U_q\right)\!\Big|_{W^\mathrm{P}_{\phi_{pq}(x)}}$, we now establish the acyclicity of $(\mathcal{D}_0)_{\phi}$:
\[
\begin{split}
(\mathcal{D}_0)_{\phi} \colon 0 \longrightarrow &\overbrace{\left(\bigwedge\nolimits^{r}\Gamma\left(E^{\mathrm{c}}\right)\right)_{\mathrm{P}\phi_{pq}} \xrightarrow{\ \iota_{s^{\Delta}_{\mathrm{c}}}\ } \dots \xrightarrow{\ \iota_{s^{\Delta}_{\mathrm{c}}}\ } \Gamma\left(E^\mathrm{c}\right)_{\mathrm{P}\phi_{pq}}}^{\deg < 0} \\
&\quad \xrightarrow{\ \iota_{s^{\Delta}_{\mathrm{c}}}\ } \overbrace{\frac{\bigwedge\nolimits^{0,0} \Gamma\left(\left(\Delta_q^* \oplus E_q^{-*} \oplus \operatorname{ev}_0^* T^* U_q\right)\!\Big|_{W^\mathrm{P}_{\phi_{pq}(x)}}\right)_{\mathrm{P}\phi_{pq}}}{\widehat{\eta}^{\mathrm{K}}_{x,1}\left(\bigwedge\nolimits^{0}\Gamma\left(E_q^*\Big|_{W'_{\phi_{pq}(x)}}\right)_{\phi_{pq}}\right)}}^{\deg = 0} \longrightarrow 0.
\end{split}
\]
The proof for $(\mathcal{D}_{i>0})_{\phi}$ is virtually identical and is therefore omitted.

For the degree $< 0$ case, we first consider the sequence
\begin{equation}\label{wes}
\dots \longrightarrow \bigwedge\nolimits^{j+1}\Gamma\left(E^\mathrm{c}\Big|_{W^\mathrm{P}_{\phi_{pq}(x)}}\right) \xrightarrow{\ \iota_{s^{\Delta}_{\mathrm{c}}}^{(j+1)}\ } \bigwedge\nolimits^{j}\Gamma\left(E^\mathrm{c}\Big|_{W^\mathrm{P}_{\phi_{pq}(x)}}\right) \xrightarrow{\ \iota_{s^{\Delta}_{\mathrm{c}}}^{(j)}\ } \bigwedge\nolimits^{j-1}\Gamma\left(E^\mathrm{c}\Big|_{W^\mathrm{P}_{\phi_{pq}(x)}}\right) \longrightarrow \dots,
\end{equation}
where $\ker \iota_{s^{\Delta}_{\mathrm{c}}}^{(j)} = \operatorname{Im} \iota_{s^{\Delta}_{\mathrm{c}}}^{(j+1)}$ follows directly from Lemma \ref{bkkjjhhh}.

Note that $s^{\Delta}_{\mathrm{c}}$, as a section of the complement bundle, satisfies $(s^{\Delta}_{\mathrm{c}})^{k} \in I_\phi \setminus I_\phi^2$ for each component $k$ (when expanded in an orthonormal frame of $\Gamma(E^\mathrm{c})$ as $s^{\Delta}_{\mathrm{c}} = \sum_k (s^{\Delta}_{\mathrm{c}})^k e_k$) by the tangent bundle condition and section compatibility. By \cite[Lemma 2.9]{Kim1}, the exactness shown in \eqref{wes} carries over to the completed complex:
\begin{equation}\label{wes2}
\dots \longrightarrow \left(\bigwedge\nolimits^{j+1}\Gamma\left(E^\mathrm{c}\Big|_{W^\mathrm{P}_{\phi_{pq}(x)}}\right)\right)_{\mathrm{P}\phi_{pq}} \xrightarrow{\ \iota_{s^{\Delta}_{\mathrm{c}}}^{(j+1)}\ } \left(\bigwedge\nolimits^{j}\Gamma\left(E^\mathrm{c}\Big|_{W^\mathrm{P}_{\phi_{pq}(x)}}\right)\right)_{\mathrm{P}\phi_{pq}} \xrightarrow{\ \iota_{s^{\Delta}_{\mathrm{c}}}^{(j)}\ } \left(\bigwedge\nolimits^{j-1}\Gamma\left(E^\mathrm{c}\Big|_{W^\mathrm{P}_{\phi_{pq}(x)}}\right)\right)_{\mathrm{P}\phi_{pq}} \longrightarrow \dots.
\end{equation}

For the degree $= 0$ case, we observe that
\begin{equation}\label{cse}
\begin{split}
&\frac{\bigwedge\nolimits^{0,0} \Gamma\left(\left(\Delta_q^* \oplus E_q^{-*} \oplus \operatorname{ev}_0^* T^* U_q\right)\!\Big|_{W^\mathrm{P}_{\phi_{pq}(x)}}\right)_{\mathrm{P}\phi_{pq}}}{\widehat{\eta}^{\mathrm{K}}_{x,1}\left(\bigwedge\nolimits^{0}\Gamma\left(E_q^*\Big|_{W'_{\phi_{pq}(x)}}\right)_{\phi_{pq}}\right)} \\
&\quad = \frac{C_{\mathrm{P}\phi_{pq}}^{\infty}\left(W^{\mathrm{P}}_{\phi_{pq}(x)}\right)}{\widehat{\eta}^{\mathrm{K}}_{x,1}\left(C^{\infty}_{\phi_{pq}}\left(W'_{\phi_{pq}(x)}\right)\right)} \\
&\quad = \frac{C_{\mathrm{P}\phi_{pq}}^{\infty}\left(W^{\mathrm{P}}_{\phi_{pq}(x)}\right)}{\pi^* \circ \left(\delta_1^{-1}\right)^*\left(C^{\infty}_{\phi_{pq}}\left(W'_{\phi_{pq}(x)}\right)\right)} \\
&\quad \simeq \frac{C^{\infty}_{\phi_{pq}}\left(W'_{\phi_{pq}(x)}\right) \widehat{\otimes}_{\mathbb{R}} C_{\phi_{pq}^{\perp}}^{\infty}\left(W_{\phi_{pq}(x)}^{\perp}\right)}{C^{\infty}_{\phi_{pq}}\left(W'_{\phi_{pq}(x)}\right) \widehat{\otimes}_{\mathbb{R}} \{1\}} \\
&\quad \simeq \left\{ h \in C_{\mathrm{P}\phi_{pq}}^{\infty}\left(W^{\mathrm{P}}_{\phi_{pq}(x)}\right) \; \middle| \; h \Big|_{\operatorname{Im}\left(\delta_{1,q}\big|_{W'_{\phi_{pq}(x)}}\right)} = 0 \right\},
\end{split}
\end{equation}
where the embedding $\delta_{1,q}\big|_{W'_{\phi_{pq}(x)}}$ is given by
\[
\delta_{1,q}\big|_{W'_{\phi_{pq}(x)}} \colon W'_{\phi_{pq}(x)} \times \{0\} \xhookrightarrow{\ \delta_{1,q}\ } W^{\mathrm{P}}_{\phi_{pq}(x)}.
\]

It remains to prove the surjectivity of the map
\begin{equation}\label{gecwp}
\begin{split}
\Gamma\left(E^\mathrm{c}\Big|_{W^{\mathrm{P}}_{\phi_{pq}(x)}}\right) \otimes_{C^{\infty}\left(W^{\mathrm{P}}_{\phi_{pq}(x)}\right)} C_{\mathrm{P}\phi_{pq}}^{\infty}\left(W^{\mathrm{P}}_{\phi_{pq}(x)}\right) \xrightarrow{\ \iota_{s^{\Delta}_{\mathrm{c}}}\ } \left\{ h \in C_{\mathrm{P}\phi_{pq}}^{\infty}\left(W^{\mathrm{P}}_{\phi_{pq}(x)}\right) \; \middle| \; h \Big|_{\operatorname{Im}\left(\delta_{1,q}\big|_{W'_{\phi_{pq}(x)}}\right)} = 0 \right\}
\end{split}
\end{equation}
under the identification \eqref{cse}, defined by
\[
a \otimes h \longmapsto h \cdot \iota_{s^{\Delta}_{\mathrm{c}}}(a)
\]
for $h \in C_{\mathrm{P}\phi_{pq}}^{\infty}\left(W^{\mathrm{P}}_{\phi_{pq}(x)}\right)$ and $a \in \Gamma\left(E^\mathrm{c}\Big|_{W^{\mathrm{P}}_{\phi_{pq}(x)}}\right)$.
Indeed, $h \cdot \iota_{s^{\Delta}_{\mathrm{c}}}(a)$ vanishes on $\operatorname{Im}\left(\delta_{1,q}\big|_{W'_{\phi_{pq}(x)}}\right) \simeq W'_{\phi_{pq}(x)} \times \{0\}$, since
\[
\iota_{s^{\Delta}_{\mathrm{c}}}(a)\Big|_{W'_{\phi_{pq}(x)} \times \{0\}} = a\left(s^{\Delta}_{\mathrm{c}}\right)\Big|_{W'_{\phi_{pq}(x)} \times \{0\}} = a\Big|_{W'_{\phi_{pq}(x)} \times \{0\}}\left(s^{\Delta}_{\mathrm{c}}\Big|_{W'_{\phi_{pq}(x)} \times \{0\}}\right) = 0
\]
by section compatibility \eqref{ogkksjsg} and the definition of $s^{\Delta}_{\mathrm{c}}$. Hence, \eqref{gecwp} is well-defined.

Let $m$ be the dimension of $W'_{\phi_{pq}(x)}$ and denote $r := m + \operatorname{rk} E'_q = m + \dim W_q$. At each point $(y_1, \dots, y_m) \in \operatorname{Im}\left(\delta_{1,p}\big|_{W'_{\phi_{pq}(x)}}\right) \simeq W'_{\phi_{pq}(x)} \times \{0\}$, we consider the restricted section
\[
\underline{s^{\Delta}_{\mathrm{c}}} := s^{\Delta}_{\mathrm{c}}\Big|_{\pi^{-1}(y_1, \dots, y_m)},
\]
defined on fiber neighborhoods as
\[
\underline{s^{\Delta}_{\mathrm{c}}} \colon W^{\mathrm{P}}_{\phi_{pq}(x)} \supset \pi^{-1}(y_1, \dots, y_m) \longrightarrow \mathbb{R}^r, \qquad (w_1, \dots, w_r) \longmapsto \left(\underline{s^{\Delta}_{\mathrm{c}}}_1(\vec{y},\vec{w}), \dots, \underline{s^{\Delta}_{\mathrm{c}}}_r(\vec{y},\vec{w})\right),
\]
where $\pi \colon W^{\mathrm{P}}_{\phi_{pq}(x)} \twoheadrightarrow \operatorname{Im}\left(\delta_{1,q}\big|_{W'_{\phi_{pq}(x)}}\right)$ is a chosen retraction between contractible spaces. Here, we write $(w_1, \dots, w_r)$ for $(y_1, \dots, y_m; w_1, \dots, w_r)$.

By the tangent bundle condition for the section $s^{\Delta}_{\mathrm{c}}$ at its zero points, the inverse function theorem implies the existence of a local inverse
\begin{equation}\label{spmo}
\underline{s^{\Delta}_{\mathrm{c}}}^{-1} \colon \left(\underline{s^{\Delta}_{\mathrm{c}}}_1, \dots, \underline{s^{\Delta}_{\mathrm{c}}}_r\right) \longmapsto (w_1, \dots, w_r)
\end{equation}
defined on a smaller open neighborhood $\overset{\circ}{W}'_{\phi_{pq}(x)} \subset W'_{\phi_{pq}(x)}$. Importantly, by virtue of Example \ref{dcotopwx}, we may assume without loss of generality that $\overset{\circ}{W}'_{\phi_{pq}(x)}$ coincides with $W'_{\phi_{pq}(x)}$.

Note that any element
\[
h \in \left\{ h \in C_{\mathrm{P}\phi_{pq}}^{\infty}\left(W^{\mathrm{P}}_{\phi_{pq}(x)}\right) \; \middle| \; h \Big|_{\operatorname{Im}\left(\delta_{1,q}\big|_{W'_{\phi_{pq}(x)}}\right)} = 0 \right\}
\]
can be expressed as
\begin{equation}\label{hhy}
h = h\left[\left[ y_1, \dots, y_m; w_1, \dots, w_r \right]\right]
\end{equation}
in the base coordinates $(y_1, \dots, y_m)$ of $W'_{\phi_{pq}(x)} \times \{0\}$ and normal coordinates $(w_1, \dots, w_r)$ in $\{0\} \times W^{\perp}_{\phi_{pq}(x)}$.

Using the local inverse \eqref{spmo}, we substitute $w_i = \underline{s^{\Delta}_{\mathrm{c}}}^{-1}_i\left(\underline{s^{\Delta}_{\mathrm{c}}}_1\left(\vec{y},\vec{w}\right), \dots, \underline{s^{\Delta}_{\mathrm{c}}}_r\left(\vec{y},\vec{w}\right)\right)$ into \eqref{hhy} to obtain
\begin{equation}\label{heee}
\begin{split}
h &= h\left[\left[ y_1, \dots, y_m; \underline{s^{\Delta}_{\mathrm{c}}}_1\left(\vec{y},\vec{w}\right), \dots, \underline{s^{\Delta}_{\mathrm{c}}}_r \left(\vec{y},\vec{w}\right) \right]\right] \\
&\overset{(1)}{=} h(y_1, \dots, y_m, \vec{0}) + \sum_{|\vec{\alpha}| \ge 1} \frac{1}{\vec{\alpha}!} \frac{\partial^{|\vec{\alpha}|} h(\vec{y}, \vec{\underline{s^{\Delta}_{\mathrm{c}}}})}{\partial \underline{s^{\Delta}_{\mathrm{c}}}^{\vec{\alpha}}}\Bigg|_{\vec{\underline{s^{\Delta}_{\mathrm{c}}}} = 0} \cdot \underline{s^{\Delta}_{\mathrm{c}}}^{\vec{\alpha}} \\
&\overset{(2)}{=} \sum_{|\vec{\alpha}| \ge 1} \frac{1}{\vec{\alpha}!} \frac{\partial^{|\vec{\alpha}|} h(\vec{y}, \vec{\underline{s^{\Delta}_{\mathrm{c}}}})}{\partial \underline{s^{\Delta}_{\mathrm{c}}}^{\vec{\alpha}}}\Bigg|_{\vec{\underline{s^{\Delta}_{\mathrm{c}}}} = 0} \cdot \underline{s^{\Delta}_{\mathrm{c}}}^{\vec{\alpha}},
\end{split}
\end{equation}
where equality $(1)$ follows from the Taylor expansion with respect to the section coordinates, and $(2)$ follows from the vanishing condition $h \big|_{\operatorname{Im}\delta_{1,q}} = 0$. Here, we use the multi-index notation $\underline{s^{\Delta}_{\mathrm{c}}}^{\vec{\alpha}} := \underline{s^{\Delta}_{\mathrm{c}}}_1^{\alpha_1} \cdots \underline{s^{\Delta}_{\mathrm{c}}}_r^{\alpha_r}$ for $\vec{\alpha} = (\alpha_1, \dots, \alpha_r)$.

Consequently, $h \in \operatorname{Im} \iota_{\underline{s^{\Delta}_{\mathrm{c}}}}$, showing that $\iota_{\underline{s^{\Delta}_{\mathrm{c}}}}$ is surjective. Indeed, the final line of \eqref{heee} displays $h$ as a linear combination of $\big\{\underline{s^{\Delta}_{\mathrm{c}}}^{\vec{\alpha}}\big\}$ over $C^{\infty}_{\mathrm{P}\phi_{pq}}\left(W^{\mathrm{P}}_{\phi_{pq}(x)}\right)$, since
\[
\frac{\partial^{|\vec{\alpha}|} h(\vec{y}, \vec{\underline{s^{\Delta}_{\mathrm{c}}}})}{\partial \underline{s^{\Delta}_{\mathrm{c}}}^{\vec{\alpha}}}\Bigg|_{\vec{\underline{s^{\Delta}_{\mathrm{c}}}} = 0} \in C^{\infty}_{\mathrm{P}\phi_{pq}}\left(W^{\mathrm{P}}_{\phi_{pq}(x)}\right)
\]
for each multi-index $\vec{\alpha}$.

Recalling the decomposition \eqref{ssce}, we choose an element of $\Gamma\left(E^\mathrm{c}\big|_{W'_{\phi_{pq}(x)}}\right)$ expressed in the dual frame $\{{e'_1}^*, \dots, {e'_r}^*\}$. For a non-zero multi-index $\vec{\alpha} = (\alpha_1, \dots, \alpha_r)$, we define
\[
m(\vec{\alpha}) := \min \{ i \mid \alpha_i \neq 0, \; 1 \le i \le r \}
\]
and
\[
\vec{\alpha}' := \vec{\alpha} - (0, \dots, \overbrace{1}^{m(\vec{\alpha})}, \dots, 0).
\]

We then construct a preimage section
\[
\tau(h) := \sum_{|\vec{\alpha}| \ge 1} \frac{1}{\vec{\alpha}!} \left( \frac{\partial^{|\vec{\alpha}|} h}{\partial \underline{s^{\Delta}_{\mathrm{c}}}^{\vec{\alpha}}}\Bigg|_{\vec{\underline{s^{\Delta}_{\mathrm{c}}}} = 0} \cdot \underline{s^{\Delta}_{\mathrm{c}}}^{\vec{\alpha}'} \right) {e'_{m(\vec{\alpha})}}^*,
\]
which explicitly satisfies $h = \iota_{\underline{s^{\Delta}_{\mathrm{c}}}}(\tau(h))$. This confirms that the contraction map $\iota_{\underline{s^{\Delta}_{\mathrm{c}}}}$ is surjective.

We then define
\[
\widehat{\eta}^{\mathrm{K}}_{x,m} := \begin{cases} \widehat{\eta}^{\mathrm{K}}_{x,1} & \text{if } m = 1, \\ 0 & \text{if } m \ge 2. \end{cases}
\]

Taking into account the bundle isomorphism
\[
\operatorname{ev}_1^* E'_q \oplus \operatorname{ev}_0^* T U_q \oplus E_{q,W} \simeq \Delta_q \oplus E_q^- \oplus \operatorname{ev}_0^* T U_q,
\]
we define the quasi-isomorphism
\[
\mathcal{C}^\mathrm{P}_{q,\,\mathrm{P}\phi_{pq}(x),\,\mathrm{P}\phi_{pq}} \xrightarrow{\ (3)\ } \overset{\circ}{\mathcal{C}}{}^\mathrm{P}_{q,\,\phi_{pq}(x),\,\phi_{pq}}
\]
as a homotopy inverse of $\widehat{\eta}^{\mathrm{K}}_{x} := \left\{\widehat{\eta}^{\mathrm{K}}_{x,m}\right\}$.

This completes the construction of the quasi-isomorphism (3).

\subsection{Factorization property}

In this subsection, we show that the path space $\mathrm{P}\mathfrak{X}$ associated with a Kuranishi space $\mathfrak{X}$ serves as an explicit path object for the diagonal morphism $\Delta \colon \mathfrak{X} \to \mathfrak{X} \times \mathfrak{X}$, thereby factoring $\Delta$ into a weak equivalence followed by a fibration.

Let $\mathfrak{X} = \left(X,\, \left[\widehat{\mathcal{U}}\right]\right)$ be a Kuranishi space. We define the \textit{diagonal morphism}
\[
\Delta \colon \mathfrak{X} \longrightarrow \mathfrak{X} \times \mathfrak{X}
\]
as the unique morphism induced by the universal property of the product Kuranishi space $\mathfrak{X} \times \mathfrak{X}$:
\[
\begin{tikzcd}
\mathfrak{X} \arrow[rd, "{}^{\exists!} \Delta"] \arrow[bend right=30, swap, "\mathrm{id}_{\mathfrak{X}}"]{rdd} \arrow[bend left=20, "\mathrm{id}_{\mathfrak{X}}"]{rrd} & & \\
& \mathfrak{X} \times \mathfrak{X}
\arrow[r, "\operatorname{Pr}''"]
\arrow[d, "\operatorname{Pr}'"']
& \mathfrak{X} \\
& \mathfrak{X}, &
\end{tikzcd}
\]
which is represented by the pre-morphism
\[
\overline{\Delta} := \left(\widehat{\mathcal{U}},\, \left(\widehat{\mathcal{U}} \times \widehat{\mathcal{U}}\right) \times \mathbb{R},\, \delta,\, \{\delta_p\},\, \left\{\widehat{\delta}_{p,x}\right\}\right).
\]
Here, the underlying continuous map $\delta \colon X \to X \times X$ and the base smooth embeddings $\delta_p \colon U_p \to U_p \times U_p$ are given by the standard diagonal maps
\[
\delta(x) := (x, x)
\qquad\text{and}\qquad
\delta_p(y) := (y, y).
\]

For the expanded atlas $(\widehat{\mathcal{U}} \times \widehat{\mathcal{U}}) \times \mathbb{R}$, the local $L_{\infty}[1]$-algebra is given by
\[
\mathcal{C}^{\times}_{(p,p),(x,x,0)}
:=
\bigwedge\nolimits^{-\bullet} \Gamma\left(
\left(\pi^{\prime*}E_p^* \oplus \pi^{\prime\prime*}E_{p}^* \oplus E_{\mathbb{R}}\right)\!\Big|_{W^{\times}_{(x,x,0)}}
\right)
\;\oplus\;
\Omega_{\mathrm{aug}}^{\bullet+1}\left(\mathcal{F}^\times_{(x,x,0)}\right).
\]
The $L_\infty[1]$-morphism
\[
\widehat{\delta}_{p,x} \colon
\mathcal{C}^{\times}_{(p,p),(x,x,0)}
\longrightarrow
\mathcal{C}_{p,x}
\]
is determined uniquely by the universal property of Lemma~\ref{dudhbbjbg}.

\begin{lem}
$\overline{\Delta}$ defines a valid pre-morphism.
\end{lem}

\begin{proof}
This is a special case of Lemma~\ref{kfprm}.
\end{proof}

\begin{thm}[Factorization property]\label{facpro}
The diagonal morphism $\Delta$ factors through the path object $\mathrm{P}\mathfrak{X}$. Specifically, there exist morphisms
\[
\mathfrak{X}
\xrightarrow{\ \Delta_1\ }
\mathrm{P}\mathfrak{X}
\xrightarrow{\ \Delta_2\ }
\mathfrak{X} \times \mathfrak{X}
\]
such that
\[
\Delta_2 \circ \Delta_1 = \Delta.
\]
Moreover, $\Delta_1$ is a weak equivalence and $\Delta_2$ is a fibration.
\end{thm}

\begin{proof}
\noindent(\textit{$\Delta_1$ is a weak equivalence.})
Let $\mathfrak{X} = \left(X,\, [\widehat{\mathcal{U}}]\right)$ be a Kuranishi space. We define the morphism
\[
\Delta_1 \colon \mathfrak{X} \longrightarrow \mathrm{P}\mathfrak{X}
\]
by the equivalence class of the pre-morphism
\[
\overline{\Delta}_1 := \left(\widehat{\mathcal{U}},\, \mathrm{P}\widehat{\mathcal{U}},\, \delta_1,\, \left\{\delta_{1,p}\right\},\, \left\{\widehat{\delta}_{1,p,x}\right\}\right).
\]
Here, each path chart $\mathrm{P}{\mathcal{U}}_p$ is given by
\[
\mathrm{P}{\mathcal{U}}_p := \left(\mathrm{P}_gU_p \times W,\, \mathrm{ev}_1^*E_p \oplus \mathrm{ev}_0^*TU_p \oplus E_{p,W},\, s^{\mathrm{P}}_p,\, \psi_p^\mathrm{P}\right)
\]
as in Definition~\ref{mawehe}, with coordinate changes $\{\mathrm{P}\Phi_{pq}\}$ given by \eqref{bacccvbz}.

The continuous map on the underlying topological spaces is the identity
\begin{equation}\label{pwghvbiib}
\delta_1 \colon X \longrightarrow X, \qquad x \longmapsto x,
\end{equation}
and the smooth embedding on each base is given by
\[
\delta_{1,p} \colon U_p \hookrightarrow \mathrm{P}_gU_p \times W, \qquad y \mapsto (\gamma_{y,y},\, 0),
\]
where $\gamma_{y,y} \in \mathrm{P}_gU_p$ denotes the constant geodesic at $y$.

For each $x \in s_p^{-1}(0)$, the local $L_\infty[1]$-algebra of $\mathrm{P}\mathfrak{X}$ is
\[
\mathcal{C}^\mathrm{P}_{(x, 0)}
:=
\bigwedge\nolimits^{-\bullet} \Gamma\left(
\left(\mathrm{ev}_1^*E_p^* \oplus \mathrm{ev}_0^*T^*U_p \oplus E_{p,W}^*\right)\!\Big|_{W^\mathrm{P}_x}
\right)
\;\oplus\;
\Omega_{\mathrm{aug}}^{\bullet+1}\left(\mathcal{F}^\mathrm{P}_x\right).
\]
The $L_{\infty}[1]$-morphism
\[
\widehat{\delta}_{1,p,x} \colon
\mathcal{C}^{\mathrm{P}}_{\delta_1(p),\,\delta_{1,p}(x)}
\longrightarrow
\mathcal{C}_{p,x}
\]
decomposes into Koszul and de Rham parts, defined as follows.

For the Koszul part, we consider the vector bundle embedding
\[
\begin{tikzcd}[column sep=large, row sep=large]
E_p \arrow[r, hook, "\overline{\delta}_{1,p}"]
&
\mathrm{ev}_1^*E_p \oplus \mathrm{ev}_0^*TU_p \oplus E_{p,W}
\\
U_p \arrow[u, "s_p"] \arrow[r, hook, "\delta_{1,p}"] 
&
\mathrm{P}_gU_p \times W \arrow[u, swap, "s_p^{\mathrm{P}}"] 
\end{tikzcd}
\]
defined on each fiber over $y \in U_p$ by
\begin{equation}\label{nannnn}
\overline{\delta}_{1,p}(y, v) := \left(2v,\, 0,\, 2v\right) \in (E_p)_y \oplus T_y U_p \oplus W.
\end{equation}
Along the image $\delta_{1,p}(y) = (\gamma_{y,y}, 0)$, we have $\dot{\gamma}_{y,y}(0) = 0$ and $\mathrm{ev}_0(\gamma_{y,y}) = \mathrm{ev}_1(\gamma_{y,y}) = y$. Substituting $w = 0$ into the definition of $s_p^\mathrm{P}$ yields
\[
\begin{split}
\left(s_p^{\mathrm{P}} \circ \delta_{1,p}\right)(y)
&= s_p^{\mathrm{P}}(\gamma_{y,y}, 0)\\
&= \left(2s_p(y),\, \dot{\gamma}_{y,y}(0),\, 2s_p(y) \right)\\
&= \left(2s_p(y),\, 0,\, 2s_p(y)\right)\\
& = \left(\overline{\delta}_{1,p} \circ s_p \right)(y).
\end{split}
\]

We now verify the tangent bundle condition along the zero locus $x \in s_p^{-1}(0)$. The differential $d_{(x,0)} s_p^{\mathrm{P}}$ induces a linear map on normal spaces:
\begin{equation}\label{beacccc}
\left[d_{(x,0)} s_p^{\mathrm{P}}\right] \colon \frac{T_{(x,0)}\left(\mathrm{P}_gU_p \times W\right)}{\delta_{1,p*}\left(T_xU_p\right)} \xrightarrow{\ \simeq\ } \frac{\left(\mathrm{ev}_1^*E_p \oplus \mathrm{ev}_0^*TU_p \oplus E_{p,W}\right)_{(\gamma_{x,x},0)}}{\operatorname{Im}\left(\overline{\delta}_{1,p}\right)_{(x,0)}}.
\end{equation}
Via the canonical identification $T_{\gamma_{x,x}}(\mathrm{P}_gU_p) \simeq T_xU_p \oplus T_xU_p$, the domain normal space splits as
\[
\frac{T_{(x,0)}\left(\mathrm{P}_gU_p \times T_0W\right)}{\delta_{1,p*}\left(T_xU_p\right)}
\simeq
\frac{\left(T_xU_p \oplus T_xU_p\right) \oplus T_0W}{T_xU_p \oplus \{0\} \oplus \{0\}}
\simeq
T_xU_p \oplus W.
\]
On the target normal space, since $s_p(x) = 0$, the fiber of $\operatorname{Im}(\overline{\delta}_{1,p})$ at $(x,0)$ is $(E_p)_x \oplus \{0\} \oplus \{0\}$, yielding
\[
\frac{\left(\mathrm{ev}_1^*E_p \oplus \mathrm{ev}_0^*TU_p \oplus E_{p,W}\right)_{(\gamma_{x,x},0)}}{\operatorname{Im}\left(\overline{\delta}_{1,p}\right)_{(x,0)}}
\simeq
\frac{(E_p)_x \oplus T_xU_p \oplus W}{(E_p)_x \oplus \{0\} \oplus \{0\}}
\simeq
T_xU_p \oplus W.
\]
Under these identifications, the normal differential $\left[d_{(x,0)} s_p^{\mathrm{P}}\right]$ acts by
\[
({v},\, w) \longmapsto \left({v},\, 2w\right) \in T_xU_p \oplus W,
\]
which is an isomorphism because ${v} \mapsto {v}$ is the identity on $T_xU_p$ and $ w \mapsto 2w$ is an isomorphism on $W$. Therefore, by Proposition~\ref{afec}, this bundle embedding induces a quasi-isomorphism on the Koszul complexes:
\[
\widehat{\delta}^{\mathrm{K}}_{1,p,x} \colon
\bigwedge\nolimits^{-\bullet} \Gamma\left(
\left(\mathrm{ev}_1^*E_p^* \oplus \mathrm{ev}_0^*T^*U_p \oplus E_{p,W}^*\right)\!\Big|_{W^\mathrm{P}_x}
\right)
\xrightarrow{\ \simeq\ }
\bigwedge\nolimits^{-\bullet} \Gamma\left(E_p^* \Big|_{W_x}\right).
\]

For the de Rham part, the differential $d\delta_{1,p}$ satisfies
\[
\alpha^\mathrm{P}\left(d\delta_{1,p}(X)\right) = (\operatorname{pr}_1^* \mathrm{ev}_0^* \alpha_p)(X, 0) = \alpha_p\left(\mathrm{ev}_{0*} X\right) = \alpha_p(X)
\]
for every $X \in \Gamma(TU_p)$, which implies that
\[
d\delta_{1,p}\left(\ker \alpha_p\right) \subset \ker \alpha^\mathrm{P}.
\]
Consequently, pullback along $\delta_{1,p}$ induces a map between the foliation de Rham complexes:
\[
\widehat{\delta}^{\mathrm{dR}}_{1,p,x} \colon \Omega_{\mathrm{aug}}^{\bullet+1}\left(\mathcal{F}^\mathrm{P}_x\right) \rightarrow \Omega_{\mathrm{aug}}^{\bullet+1}\left(\mathcal{F}_x\right).
\]
Combining the Koszul and de Rham parts, $\widehat{\delta}_{1,p,x}$ is a quasi-isomorphism for every $x$, which proves that $\Delta_1$ is a weak equivalence.

\begin{lem}
$\overline{\Delta}_1$ defines a valid pre-morphism.
\end{lem}

\begin{proof}
Let $p, q \in X$ satisfy $\operatorname{Im}\psi_p \cap \operatorname{Im}\psi_q \neq \emptyset$. We verify the three defining conditions for a pre-morphism:

\begin{enumerate}[label = (\roman*)]
\item ($\mathrm{P}\phi_{\delta_1(p)\delta_1(q)} \circ \delta_{1,p} = \delta_{1,q} \circ \phi_{pq}$ \textit{on} $s_{p}^{-1}(0) \cap U_{pq}$)

Recall that $\delta_1 = \mathrm{id}_X$, so $\delta_1(p) = p$ and $\delta_1(q) = q$. For each $x \in s_{p}^{-1}(0) \cap U_{pq}$, the base embedding evaluates to $\delta_{1,p}(x) = (\gamma_{x,x}, 0)$, where $\gamma_{x,x}$ denotes the constant geodesic at $x$. Applying the path coordinate change $\mathrm{P}\phi_{pq}(\gamma, w) = (\phi_{pq}(\gamma)_g, w)$, we obtain
\[
\begin{split}
\mathrm{P}\phi_{pq} \circ \delta_{1,p}(x) &= \mathrm{P}\phi_{pq}(\gamma_{x,x}, 0) = \left(\phi_{pq}(\gamma_{x,x})_g,\, 0\right) = \left(\gamma_{\phi_{pq}(x), \phi_{pq}(x)},\, 0\right) \\
&= \delta_{1,q}\left(\phi_{pq}(x)\right) = \left(\delta_{1,q} \circ \phi_{pq}\right)(x).
\end{split}
\]

\item ($\psi^\mathrm{P}_{\delta_1(p)} \circ \delta_{1,p} = \delta_1 \circ \psi_p$ \textit{on} $s_p^{-1}(0)$)

For each $x \in s_p^{-1}(0)$, we have $\delta_{1,p}(x) = (\gamma_{x,x}, 0)$. By Definition~\ref{mawehe}, the footprint identification satisfies $\psi_p^\mathrm{P}(\gamma_{x,x}, 0) = \psi_p(x)$. Since $\delta_1 = \mathrm{id}_X$, we immediately have
\[
\psi^\mathrm{P}_p\left(\delta_{1,p}(x)\right)
= \psi^\mathrm{P}_p(\gamma_{x,x}, 0)
= \psi_p(x)
= \delta_1\left(\psi_p(x)\right)
= \left(\delta_1 \circ \psi_p\right)(x).
\]

\item ($\widehat{\phi}_{pq,x} \circ \widehat{\delta}_{1, q, \phi_{pq}(x)} \sim \widehat{\delta}_{1,p,x} \circ \widehat{\mathrm{P}\phi}_{pq,\delta_{1,p}(x)}$ \textit{for each} $x \in s_p^{-1}(0) \cap U_{pq}$)

Since the $L_\infty[1]$-morphisms $\widehat{\delta}_{1,p,x}$ and $\widehat{\delta}_{1,q,\phi_{pq}(x)}$ are homotopy equivalences (by Proposition~\ref{afec}), they admit homotopy inverses. Thus, relation~(iii) is equivalent up to $L_\infty[1]$-homotopy to showing that
\begin{equation}\label{bnmzb}
\widehat{\delta}_{1,p,x}^{-1} \circ \widehat{\phi}_{pq,x} \;\sim\; \widehat{\mathrm{P}\phi}_{pq,\delta_{1,p}(x)} \circ \widehat{\delta}_{1, q, \phi_{pq}(x)}^{-1}.
\end{equation}
The de Rham components commute strictly as pullback maps of differential forms under smooth maps. For the Koszul components, the homotopy inverses $\widehat{\delta}_{1,p,x}^{-1}$ are explicitly induced by the bundle embeddings $\overline{\delta}_{1,p}$ of \eqref{nannnn} via the retraction provided by Proposition~\ref{afec}. Then, relation~\eqref{bnmzb} follows directly from the definition of the path coordinate change $\widehat{\mathrm{P}\phi}_{pq,\delta_{1,p}(x)}$ in \eqref{bwccc}, which was defined precisely to make this diagram commute up to homotopy.
\end{enumerate}
\end{proof}

\begin{lem}\label{dlwq}
$\Delta_1 := [\overline{\Delta}_1]$ is a weak equivalence.
\end{lem}

\begin{proof}
From the observation on the zero locus in \eqref{zze} and the definition of the underlying topological map \eqref{pwghvbiib}, the continuous map $\delta_1 = \mathrm{id}_X$ is trivially a homeomorphism.
As established above, each Koszul component $\widehat{\delta}^{\mathrm{K}}_{1,p,x}$ is a quasi-isomorphism by Proposition~\ref{afec}. The de Rham component $\widehat{\delta}^{\mathrm{dR}}_{1,p,x}$ is likewise a quasi-isomorphism, as it is a chain map between acyclic foliation de Rham complexes. Consequently, each $\widehat{\delta}_{1,p,x}$ is a quasi-isomorphism, which proves that $\Delta_1$ is a weak equivalence.
\end{proof}

\noindent(\textit{$\Delta_2$ is a fibration.})

Let $\mathfrak{X} = \left(X,\, [\widehat{\mathcal{U}}]\right)$ be a Kuranishi space. We define the morphism
\[
\Delta_2 \colon \mathrm{P}\mathfrak{X} \longrightarrow \mathfrak{X} \times \mathfrak{X}
\]
by the equivalence class of the pre-morphism
\[
\overline{\Delta}_2 := \left(\mathrm{P}\widehat{\mathcal{U}},\, (\widehat{\mathcal{U}} \times \widehat{\mathcal{U}}) \times \mathbb{R},\, \delta_2,\, \{\delta_{2,p}\},\, \left\{\widehat{\delta}_{2,p,x}\right\}\right),
\]
where the continuous map on the underlying topological spaces is
\[
\delta_2 \colon X \longrightarrow X \times X, \qquad p \longmapsto (p, p),
\]
and the smooth base maps are defined by
\[
\begin{split}
\delta_{2,p} \colon \mathrm{P}_gU_p \times W &\longrightarrow U_p \times U_p \times \mathbb{R}, \\
(\gamma, w) &\longmapsto \left(\gamma(0),\, \gamma(1),\, |w|^2\right).
\end{split}
\]
By choosing each chart domain $U_p$ to be sufficiently small (geodesically convex), the restricted map $\delta_{2,p}|_{\mathrm{P}_gU_p \times \{0\}} \colon \gamma \mapsto (\gamma(0), \gamma(1))$ is a diffeomorphism onto an open neighborhood of the diagonal in $U_p \times U_p$. This property is preserved under shrinkings, and therefore holds for representing atlases in $\mathcal{A}_0(\mathfrak{X})$ (cf.~\eqref{cevb}).

Recall that the local $L_{\infty}[1]$-algebras are given by
\[
\mathcal{C}^{\times}_{(p,p),(x,x,0)}
:=
\bigwedge\nolimits^{-\bullet} \Gamma\left(
\left(\pi^{\prime*}E_p^* \oplus \pi^{\prime\prime*}E_p^* \oplus E_{\mathbb{R}}^*\right)\!\Big|_{W^{\times}_{(x,x,0)}}
\right)
\;\oplus\;
\Omega_{\mathrm{aug}}^{\bullet+1}\left(\mathcal{F}^\times_{(x,x,0)}\right)
\]
and
\[
\mathcal{C}^\mathrm{P}_{p,x}
:=
\bigwedge\nolimits^{-\bullet} \Gamma\left(
\left(\mathrm{ev}_1^*E_p^* \oplus \mathrm{ev}_0^*T^*U_p \oplus E_{p,W}^*\right)\!\Big|_{W^\mathrm{P}_x}
\right)
\;\oplus\;
\Omega_{\mathrm{aug}}^{\bullet+1}\left(\mathcal{F}^\mathrm{P}_x\right).
\]
For each $p \in X$ and $x \in s_p^{-1}(0)$, the $L_\infty[1]$-morphism
\[
\widehat{\delta}_{2,p,x} \colon
\mathcal{C}^{\times}_{(p,p),(x,x,0)}
\longrightarrow
\mathcal{C}^{\mathrm{P}}_{p,x}
\]
is constructed as follows.

Consider the canonical vector bundle embedding
\[
\pi^{\prime*}E_p \oplus \pi^{\prime\prime*}E_p \hookrightarrow \pi^{\prime*}E_p \oplus \pi^{\prime\prime*}E_p \oplus E_{\mathbb{R}}
\]
over the inclusion $U_p \times U_p \hookrightarrow U_p \times U_p \times \mathbb{R}$ as the zero slice. Since the tangent bundle condition is satisfied, Proposition~\ref{afec} yields a quasi-isomorphism
\[
\left(\widehat{\eta}^{\vee}\right)^{-1} \colon \bigwedge\nolimits^{-\bullet}\Gamma\left(
\left(\pi^{\prime*}E^*_p \oplus \pi^{\prime\prime*}E^*_p \oplus E^*_{\mathbb{R}}\right)\!\Big|_{W^\times_{(x,x,0)}}
\right) \xrightarrow{\ \simeq\ } \bigwedge\nolimits^{-\bullet}\Gamma\left(
\left(\pi^{\prime*}E^*_p \oplus \pi^{\prime\prime*}E^*_p\right)\!\Big|_{W^\times_{x,x}}
\right).
\]

Next, consider the vector bundle embedding
\[
\begin{tikzcd}[column sep=large, row sep=large]
\pi^{\prime*}E_p \oplus \pi^{\prime\prime*}E_p \arrow[r, hook, "\overline{\mathrm{em}}_{p,p}"] & \mathrm{ev}_1^*E_p \oplus \mathrm{ev}_0^*TU_p \oplus E_{p,W} \\
U_p \times U_p \arrow[r, hook, "\mathrm{em}_{p,p}"] \arrow[u, "\pi^{\prime*}s_p \times \pi^{\prime\prime*}s_p"] & \mathrm{P}_gU_p \times W \arrow[u, swap, "s^{\mathrm{P}}_p"]
\end{tikzcd}
\]
defined on fibers by
\begin{equation}\label{bwzzzzbwq}
\overline{\mathrm{em}}_{p,p}\left((x,y),\, (e', e'')\right) := \left(\left(\gamma_{x,y}, 0\right),\, \left(e' + e'',\, 0,\, e' + e''\right)\right)
\end{equation}
covering the base embedding
\[
\mathrm{em}_{p,p} \colon (x, y) \longmapsto \left(\gamma_{x,y},\, 0\right),
\]
where $\gamma_{x,y} \in \mathrm{P}_gU_p$ is the unique short geodesic from $x$ to $y$.

Evaluating the section $s_p^\mathrm{P}$ along the image of $\mathrm{em}_{p,p}$ gives
\begin{equation}\label{gadgashhhh2}
\begin{split}
\left(s_p^{\mathrm{P}} \circ \mathrm{em}_{p,p}\right)|_{\mathrm{Im}\overline{\mathrm{em}}_{p,p}}(x, y)
&= s_p^{\mathrm{P}}\left(\gamma_{x,y},\, 0\right)|_{\mathrm{Im}\overline{\mathrm{em}}_{p,p}} \\
&= \left(s_p(x) + s_p(y),\, \dot{\gamma}_{x,y}(0),\, s_p(x) + s_p(y)\right)|_{\mathrm{Im}\overline{\mathrm{em}}_{p,p}} \\
&= \left(s_p(x) + s_p(y),\, 0,\, s_p(x) + s_p(y)\right)\\
&= \overline{\mathrm{em}}_{p,p}\left(s_p(x),\, s_p(y)\right) \\
&= \left(\overline{\mathrm{em}}_{p,p} \circ \left(\pi^{\prime*}s_p \times \pi^{\prime\prime*}s_p\right)\right)(x, y),
\end{split}
\end{equation}
which establishes the required section compatibility. Furthermore, on the base manifolds we have
\[
\left(\delta_{2,p} \circ \mathrm{em}_{p,p}\right)(x, y) = \delta_{2,p}\left(\gamma_{x,y},\, 0\right) = \left(\gamma_{x,y}(0),\, \gamma_{x,y}(1),\, 0\right) = (x, y, 0).
\]

Let $\widehat{\eta}(\overline{\mathrm{em}}_{p,p}|_{W_{x,x}^{\times}})$ denote the $L_{\infty}[1]$-morphism induced by the bundle embedding $\overline{\mathrm{em}}_{p,p}|_{W_{x,x}^{\times}}$ according to \eqref{haahb} and Lemma~\ref{pokhgjb}:
\[
\begin{split}
\widehat{\eta}\left(\overline{\mathrm{em}}_{p,p}\Big|_{W_{x,x}^{\times}}\right) \colon
\bigwedge\nolimits^{-\bullet}&\Gamma\left(\left(\pi^{\prime*}E^*_p \oplus \pi^{\prime\prime*}E^*_p\right)\!\Big|_{W^{\times}_{x,x}}\right)\\ 
&\quad\quad\quad\quad\quad\quad\rightarrow \bigwedge\nolimits^{-\bullet}\Gamma\left(\left(\mathrm{ev}_1^*E^*_p \oplus \mathrm{ev}_0^*T^*U_p \oplus E^*_{p,W}\right)\!\Big|_{W^\mathrm{P}_x}\right).
\end{split}
\]
We define the Koszul component of $\widehat{\delta}_{2,p,x}$ as the composite
\[
\widehat{\delta}_{2,p,x}^\mathrm{K} := \widehat{\eta}\left(\overline{\mathrm{em}}_{p,p}\Big|_{W_{x,x}^{\times}}\right) \circ \left(\widehat{\eta}^{\vee}\right)^{-1}.
\]
Together with the pullback on the foliation de Rham complexes 
\[
\widehat{\delta}_{2,p,x}^\mathrm{dR} \colon \Omega_{\mathrm{aug}}^{\bullet+1}(\mathcal{F}^\times_{(x,x,0)}) \to \Omega_{\mathrm{aug}}^{\bullet+1}(\mathcal{F}^\mathrm{P}_x),
\]
this specifies $\widehat{\delta}_{2,p,x} := \widehat{\delta}_{2,p,x}^\mathrm{K} \oplus \widehat{\delta}_{2,p,x}^\mathrm{dR}$.

Finally, we verify that
\begin{equation}\label{eq:delta2_zero_locus}
\left(s^\mathrm{P}_p\right)^{-1}(0) \;\cap\; \delta_{2,p}^{-1}\left(\left(s_p \times s_p \times \mathrm{id}_{\mathbb{R}}\right)^{-1}(0)\right) = \left(s^\mathrm{P}_p\right)^{-1}(0).
\end{equation}
The inclusion $\subset$ is immediate; thus, it suffices to establish the reverse inclusion. Recall that $(s^\mathrm{P}_p)^{-1}(0)$ consists precisely of constant geodesics $\gamma_{x,x}$ at points $x \in s_p^{-1}(0)$ with $w = 0$. Since
\[
\delta_{2,p}^{-1}\left(\left(s_p \times s_p \times \mathrm{id}_{\mathbb{R}}\right)^{-1}(0)\right) = \left\{ (\gamma, w) \in \mathrm{P}_gU_p \times W \;\middle|\; s_p(\gamma(0)) = 0,\; s_p(\gamma(1)) = 0,\; w = 0 \right\},
\]
every element of $(s^\mathrm{P}_p)^{-1}(0)$ automatically satisfies these conditions, proving equality in \eqref{eq:delta2_zero_locus}.

\begin{lem}
$\overline{\Delta}_2$ defines a valid pre-morphism.
\end{lem}

\begin{proof}
Let $p, q \in X$ satisfy $\operatorname{Im}\psi_p \cap \operatorname{Im}\psi_q \neq \emptyset$. We verify the three defining conditions for a pre-morphism:

\begin{enumerate}[label = (\roman*)]
\item ($\left(\phi^{\times}_{pq} \times \mathrm{id}_{\mathbb{R}}\right) \circ \delta_{2,p} = \delta_{2,q} \circ \mathrm{P}\phi_{pq}$ \textit{on} $(s_p^\mathrm{P})^{-1}(0) \cap \mathrm{P}_gU_{pq}$)

Here, $\phi^\times_{pq} := \phi_{pq} \times \phi_{pq}$. For each $(\gamma_{x,x}, 0) \in (s_p^\mathrm{P})^{-1}(0) \cap \mathrm{P}_gU_{pq}$, we have
\[
\begin{split}
\left(\left(\phi_{pq} \times \phi_{pq} \times \mathrm{id}_{\mathbb{R}}\right) \circ \delta_{2,p}\right)(\gamma_{x,x}, 0)
&= \left(\phi_{pq} \times \phi_{pq} \times \mathrm{id}_{\mathbb{R}}\right)(x, x, 0) \\
&= \left(\phi_{pq}(x),\, \phi_{pq}(x),\, 0\right) \\
&= \delta_{2,q}\left(\gamma_{\phi_{pq}(x), \phi_{pq}(x)},\, 0\right) \\
&= \left(\delta_{2,q} \circ \mathrm{P}\phi_{pq}\right)(\gamma_{x,x}, 0).
\end{split}
\]

\item ($\psi^{\times}_{p,p} \circ \delta_{2,p} = \delta_2 \circ \psi_p^\mathrm{P}$ \textit{on} $(s_p^\mathrm{P})^{-1}(0)$)

For each $(\gamma_{x,x}, 0) \in (s_p^\mathrm{P})^{-1}(0)$, we compute
\[
\begin{split}
\left(\psi^\times_{p,p} \circ \delta_{2,p}\right)(\gamma_{x,x}, 0)
&= \psi^\times_{p,p}(x, x, 0) \\
&= \left(\psi_p(x),\, \psi_p(x)\right) \\
&= \delta_2\left(\psi_p(x)\right) \\
&= \left(\delta_2 \circ \psi_p^\mathrm{P}\right)(\gamma_{x,x}, 0).
\end{split}
\]

\item ($\widehat{\mathrm{P}\phi}_{pq,x} \circ \widehat{\delta}_{2,q, \phi_{pq}(x)} \sim \widehat{\delta}_{2,p,x} \circ \widehat{\phi}^{\times}_{(p,p)(q,q),\,\delta_{2,p}(x)}$ \textit{for each} $x \in s_p^{-1}(0) \cap U_{pq}$)

Recall from the factorization $\Delta = \Delta_2 \circ \Delta_1$ that $\widehat{\delta}_{2,p,x} \sim \widehat{\delta}_{1,p,x}^{-1} \circ \widehat{\delta}_{p,x}$ up to $L_\infty[1]$-homotopy. Using the commutativity relation established for $\overline{\Delta}_1$ in Lemma~\ref{bnmzb} and the pre-morphism property of $\overline{\Delta}$, we compute up to $L_\infty[1]$-homotopy:
\[
\begin{split}
\widehat{\mathrm{P}\phi}_{pq,x} \circ \widehat{\delta}_{2,q, \phi_{pq}(x)}
&\sim \widehat{\mathrm{P}\phi}_{pq,x} \circ \left(\widehat{\delta}_{1, q, \phi_{pq}(x)}^{-1} \circ \widehat{\delta}_{q, \phi_{pq}(x)}\right) \\
&\sim \left(\widehat{\mathrm{P}\phi}_{pq,x} \circ \widehat{\delta}_{1, q, \phi_{pq}(x)}^{-1}\right) \circ \widehat{\delta}_{q, \phi_{pq}(x)} \\
&\sim \left(\widehat{\delta}_{1, p, x}^{-1} \circ \widehat{\phi}_{pq,x}\right) \circ \widehat{\delta}_{q, \phi_{pq}(x)} \\
&= \widehat{\delta}_{1, p, x}^{-1} \circ \left(\widehat{\phi}_{pq,x} \circ \widehat{\delta}_{q, \phi_{pq}(x)}\right) \\
&\sim \widehat{\delta}_{1, p, x}^{-1} \circ \left(\widehat{\delta}_{p,x} \circ \widehat{\phi}^{\times}_{(p,p)(q,q),\,\delta_p(x)}\right) \\
&= \left(\widehat{\delta}_{1, p, x}^{-1} \circ \widehat{\delta}_{p,x}\right) \circ \widehat{\phi}^{\times}_{(p,p)(q,q),\,\delta_p(x)} \\
&\sim \widehat{\delta}_{2,p,x} \circ \widehat{\phi}^{\times}_{(p,p)(q,q),\,\delta_{2,p}(x)}.
\end{split}
\]
This completes the verification that $\overline{\Delta}_2$ is a pre-morphism.
\end{enumerate}
\end{proof}

\noindent(\textit{$\Delta = \Delta_2 \circ \Delta_1$})

By the definition of the composition of morphisms, the composite $\Delta_2 \circ \Delta_1$ is represented by the equivalence class of the pre-morphism
\begin{equation}\label{morcom1}
\overline{\Delta}_2 \circ \overline{\Delta}_1 := \left(\widehat{\mathcal{U}},\, (\widehat{\mathcal{U}} \times \widehat{\mathcal{U}}) \times \mathbb{R},\, \delta_2 \circ \delta_1,\, \left\{\delta_{2,\delta_1(p)} \circ \delta_{1,p}\right\},\, \left\{\widehat{\delta}_{1,p,x} \circ \widehat{\delta}_{2,\delta_1(p),\delta_{1,p}(x)}\right\}\right).
\end{equation}
By the universal property of the product Kuranishi space $\mathfrak{X} \times \mathfrak{X}$, it suffices to show that
\[
\operatorname{Pr}' \circ (\Delta_2 \circ \Delta_1) = \mathrm{id}_{\mathfrak{X}}
\qquad\text{and}\qquad
\operatorname{Pr}'' \circ (\Delta_2 \circ \Delta_1) = \mathrm{id}_{\mathfrak{X}},
\]
which uniquely characterizes the diagonal morphism $\Delta$.

First, on the underlying topological spaces, since $\delta_1 = \mathrm{id}_X$ and $\delta_2(p) = (p, p)$, we obtain
\[
\begin{split}
\left(\operatorname{pr}' \circ (\delta_2 \circ \delta_1)\right)(p) &= \operatorname{pr}'(p, p) = p = \mathrm{id}_X(p), \\
\left(\operatorname{pr}'' \circ (\delta_2 \circ \delta_1)\right)(p) &= \operatorname{pr}''(p, p) = p = \mathrm{id}_X(p).
\end{split}
\]

Next, on the base manifolds, for each chart $U_p$ with $\delta_1(p) = p$, the embedding $\delta_{1,p}$ sends each $y \in U_p$ to $(\gamma_{y,y}, 0) \in \mathrm{P}_gU_p \times W$, where $\gamma_{y,y}$ is the constant geodesic at $y$. Since $\delta_{2,p}(\gamma_{y,y}, 0) = (y, y, 0) \in U_p \times U_p \times \mathbb{R}$, composing with the canonical coordinate projections yields
\[
\begin{split}
\left(\operatorname{pr}'_{(p,p)} \circ (\delta_{2,p} \circ \delta_{1,p})\right)(y)
&= \operatorname{pr}'_{(p,p)}\left(\delta_{2,p}(\gamma_{y,y}, 0)\right)
= \operatorname{pr}'_{(p,p)}(y, y, 0)
= y = \mathrm{id}_{U_p}(y), \\
\left(\operatorname{pr}''_{(p,p)} \circ (\delta_{2,p} \circ \delta_{1,p})\right)(y)
&= \operatorname{pr}''_{(p,p)}\left(\delta_{2,p}(\gamma_{y,y}, 0)\right)
= \operatorname{pr}''_{(p,p)}(y, y, 0)
= y = \mathrm{id}_{U_p}(y).
\end{split}
\]

Finally, the compatibility of the $L_\infty[1]$-components follows once we show that the composites
\[
\widehat{\delta}_{1,p,x} \circ \widehat{\delta}_{2,p,(\gamma_{x,x},0)} \circ \widehat{\operatorname{pr}}'_{(p,p),(x,x,0)}
\qquad\text{and}\qquad
\widehat{\delta}_{1,p,x} \circ \widehat{\delta}_{2,p,(\gamma_{x,x},0)} \circ \widehat{\operatorname{pr}}''_{(p,p),(x,x,0)}
\]
are $L_\infty[1]$-homotopic to the identity morphism $\mathrm{id}_{\mathcal{C}_{p,x}}$ on each local complex. By construction, we recall that $\widehat{\delta}_{2,p,(\gamma_{x,x},0)} \sim \widehat{\delta}_{1,p,x}^{-1} \circ \widehat{\delta}_{p,x}$ up to homotopy. Substituting this relation into the first composite gives
\[
\widehat{\delta}_{1,p,x} \circ \left(\widehat{\delta}_{1,p,x}^{-1} \circ \widehat{\delta}_{p,x}\right) \circ \widehat{\operatorname{pr}}'_{(p,p),(x,x,0)}
\;\sim\;
\widehat{\delta}_{p,x} \circ \widehat{\operatorname{pr}}'_{(p,p),(x,x,0)}
\;\sim\;
\mathrm{id}_{\mathcal{C}_{p,x}},
\]
and similarly for the second projection. Consequently, we conclude that $\Delta_2 \circ \Delta_1 = \Delta$ as morphisms of Kuranishi spaces.

Recall that $\widehat{\delta}_{1,p,x}$ is a quasi-isomorphism by Lemma \ref{dlwq}. Since the de Rham morphism is always a quasi-isomorphism, it suffices to prove the following lemma:

\begin{lem}
$\widehat{\delta}^{\mathrm{K}}_{2,\delta_1(p),\delta_{1,p}(x)} \circ \widehat{\mathrm{pr}}^{\prime\mathrm{K}}_{(p',p''),(x',x'')}$ and $\widehat{\delta}^{\mathrm{K}}_{2,\delta_1(p),\delta_{1,p}(x)} \circ \widehat{\mathrm{pr}}^{\prime\prime\mathrm{K}}_{(p',p''),(x',x'')}$ are quasi-isomorphisms.
\end{lem}

\begin{proof}
We first show that $\widehat{\delta}^{\mathrm{K}}_{2,\delta_1(p),\delta_{1,p}(x)} \circ \widehat{\mathrm{pr}}^{\prime\mathrm{K}}_{(p',p''),(x',x'')}$ is a quasi-isomorphism. Considering the canonical projection
\[
\widehat{\mathrm{pr}}^{\prime\mathrm{K}}_{(p',p''),(x',x''),1} \colon a' \longmapsto \left(\pi^{\prime*}a', 0\right),
\]
it suffices to verify that the $L_{\infty}[1]$-morphism induced from the bundle embedding \eqref{bwzzzzbwq},
\[
\overline{\mathrm{em}}_{p,p} \colon \bigl((y,z), e', e''\bigr) \longmapsto \bigl((\gamma_{y,z},0), (e' + e'', 0, e' + e'')\bigr),
\]
and
\[
\begin{split}
\mathrm{em}_{p,p} \colon W_x \times W_x &\longrightarrow W_x^{\mathrm{P}} \times W, \\
(y, z) &\longmapsto (\gamma_{y, z},0),
\end{split}
\]
as depicted in the diagram
\[
\begin{tikzcd}[column sep=large, row sep=large]
  \pi^{\prime*}E_p \oplus \pi^{\prime\prime*}E_p \arrow[r, hook, "\overline{\mathrm{em}}_{p,p}"] \arrow[d, swap, "\pi^{\prime*}s_p \oplus \pi^{\prime\prime*}s_p"] & \mathrm{ev}_1^*E_p \oplus \mathrm{ev}_0^*TU_p \oplus E_W \arrow[d, swap, "s^{\mathrm{P}}_p"] \\
  U_p \times U_p \arrow[r, hook, "\mathrm{em}_{p,p}"] & \mathrm{P}U_p \times W
\end{tikzcd}
\]
is a quasi-isomorphism. 

It is straightforward to verify that the induced chain map (and hence the induced $L_{\infty}[1]$-morphism) is homotopic, as an $L_{\infty}[1]$-morphism, to the one induced by the composite bundle embedding
\[
\overline{\mathrm{em}}_{p,p} \circ \overline{\iota}'_x \quad \text{over} \quad \mathrm{em}_{p,p} \circ \iota'_x,
\]
as shown in the following commutative diagram:
\[
\begin{tikzcd}[column sep=large, row sep=large]
E_p\big|_{W_x} \arrow[r, hook, "\overline{\iota}'_x"] \arrow[d, swap, "s_p|_{W_x}"] & \pi^{\prime*}\left(E_p\big|_{W_x}\right) \oplus \pi^{\prime\prime*}\left(E_p\big|_{W_x}\right) \arrow[r, hook, "\overline{\mathrm{em}}_{p,p}"] \arrow[d, swap, "\pi^{\prime*}s_p \oplus \pi^{\prime\prime*}s_p|_{W_x \times W_x}"] & \mathrm{ev}_1^*\left(E_p\big|_{W_x}\right) \oplus \mathrm{ev}_0^*TW_x \oplus E_W \arrow[d, swap, "s^{\mathrm{P}}_p"] \\
W_x \arrow[r, hook, "\iota'_x"] & W_x \times W_x \arrow[r, hook, "\mathrm{em}_{p,p}"] & W^\mathrm{P}_x \times W
\end{tikzcd}
\]
where each map is given explicitly by
\[
\begin{split}
\overline{\iota}'_x \colon E_p\big|_{W_x} &\longrightarrow \left.\left(\pi^{\prime*}E_p \oplus \pi^{\prime\prime*}E_p\right)\right|_{W_x \times W_x}, \\
(y,e') &\longmapsto \bigl((y,y), 2e', 0\bigr),
\end{split}
\]
and
\[
\begin{split}
\iota'_x \colon W_x &\longrightarrow W_x \times W_x, \\
y &\longmapsto (y,y).
\end{split}
\]
Here, we choose the retraction map
\[
\begin{split}
\pi'_x \colon W^{\mathrm{P}}_x &\twoheadrightarrow \operatorname{Im}(\iota'_x), \\
(y,z) &\longmapsto (y,y).
\end{split}
\]

Observe that
\[
(\overline{\mathrm{em}}_{p,p} \circ \overline{\iota}'_x)(y,e') \colon (y,e') \longmapsto \bigl((y,y), 2e', 0\bigr) \longmapsto \bigl((\gamma_{y,y},0), (2e', 0, 2e')\bigr), 
\]
and
\[
(\mathrm{em}_{p,p} \circ \iota'_x)(y) \colon y \longmapsto (y,y) \longmapsto (\gamma_{y,y},0).
\]

Recall that $s^\mathrm{P}_p(\gamma, w)$ is defined by
\[
\begin{split}
s^\mathrm{P}_p(\gamma, w) &:= \left(w + (\mathrm{ev}_0^* s_p)_\gamma + (\mathrm{ev}_1^* s_p)_\gamma, \, (\sigma_p)_{\gamma}, \, 2w + (\mathrm{ev}_0^* s_p)_\gamma + (\mathrm{ev}_1^* s_p)_\gamma\right) \\ 
&= \left(w + s_{p}(\gamma(0)) + s_{p}(\gamma(1)), \, \dot{\gamma}(0), \, 2w + s_{p}(\gamma(0)) + s_{p}(\gamma(1))\right),
\end{split}
\]
where we identify the fibers of $\mathrm{ev}_0^*E_p$, $\mathrm{ev}_1^*E_p$, and $E_W$. We have the compatibility conditions with the sections:
\[
\begin{split}
\left(s^{\mathrm{P}}_p \circ \mathrm{em}_{p,p} \circ \iota'_x\right)(y) &= s^{\mathrm{P}}_p(\gamma_{y,y}, 0) \\ 
&= (2s_{p}(y), 0, 2s_{p}(y)) \\
&= \overline{\mathrm{em}}_{p,p}(2s_p(y), 0) \\
&= \left(\overline{\mathrm{em}}_{p,p} \circ \overline{\iota}'_x\right) \circ s_p(y).
\end{split}
\]

The differential $d s^{\mathrm{P}}_p$ at a zero point $(\gamma_{y,y}, 0)$ for $y$ satisfying $s_p(y) = 0$ is given by
\[
\begin{split}
d s^\mathrm{P}_{p,(\gamma_{y,y}, 0)} \colon T_{\gamma_{y,y}}\mathrm{P}W_x \oplus T_0 W &\longrightarrow (\mathrm{ev}_1^* E_p)_{\gamma_{y,y}} \oplus (\mathrm{ev}_0^* T W_x)_{\gamma_{y,y}} \oplus E_{W,0}, \\
(\xi, v) &\longmapsto \bigl(v + d s_p(\xi(0)) + d s_p(\xi(1)), \, \xi(0), \, 2v + d s_p(\xi(0)) + d s_p(\xi(1))\bigr).
\end{split}
\]

It suffices to show that it satisfies the tangent bundle condition, namely that the normal differential
\[
\left[d s_{p,(\gamma_{y,y}, 0)}^\mathrm{P}\right] \colon \frac{T_{\gamma_{y, y}}\mathrm{P}W_x \oplus T_0 W}{d\left(\mathrm{em}_{p,p} \circ \iota'_x\right)(T_y W_x)} \xrightarrow{\ \simeq\ } \frac{\left(\mathrm{ev}_1^* E_p \oplus \mathrm{ev}_0^* T W_x \oplus E_W\right)_{\mathrm{em}_{p,p} \circ \iota'_x(y)}}{\overline{\mathrm{em}}_{p,p} \circ \overline{\iota}'_x \left((E_p)_y\right)}
\]
is an isomorphism. Note that
\[
\frac{T_{\gamma_{y, y}}\mathrm{P}W_x \oplus T_0 W}{d\left(\mathrm{em}_{p,p} \circ \iota'_x\right)(T_y W_x)} \simeq \frac{T_y W_x \oplus T_y W_x \oplus T_0 W}{T_y W_x \oplus \{0\} \oplus \{0\}} \simeq T_y W_x \oplus W,
\]
and
\begin{equation}\label{jiiwdg}
\begin{split}
\frac{\left(\mathrm{ev}_1^* E_p \oplus \mathrm{ev}_0^* T W_x \oplus E_W\right)_{\mathrm{em}_{p,p} \circ \iota'_x(y)}}{\overline{\mathrm{em}}_{p,p} \circ \overline{\iota}'_x \left((E_p)_y\right)} &\simeq \frac{(\mathrm{ev}_1^* E_p)_{\gamma_{y,y}} \oplus (\mathrm{ev}_0^* T W_x)_{\gamma_{y,y}} \oplus E_{W,0}}{(\mathrm{ev}_1^* E_p)_{\gamma_{y,y}} \oplus \{0\} \oplus \{0\}} \\
&\simeq T_y W_x \oplus W,
\end{split}
\end{equation}
where we identify the diagonal $\Delta_{(\mathrm{ev}_1^* E_p)_{\gamma_{y,y}}}$ with $(\mathrm{ev}_1^* E_p)_{\gamma_{y,y}}$ in the denominator of \eqref{jiiwdg}.

Under these identifications, the normal differential $\left[d s_{p,(\gamma_{y,y}, 0)}^\mathrm{P}\right]$ acts as
\[
\begin{split}
\left[d s_{p,(\gamma_{y,y}, 0)}^\mathrm{P}\right] \colon T_y W_x \oplus T_0 W &\longrightarrow T_y W_x \oplus E_{W,0}, \\
(\zeta, v) &\longmapsto (\zeta, 2v),
\end{split}
\]
which is clearly an isomorphism. Therefore, by Proposition~\ref{afec}, this bundle embedding induces a quasi-isomorphism on the Koszul complexes:
\[
\begin{split}
\widehat{\delta}^{\mathrm{K}}_{2,\delta_1(p),\delta_{1,p}(x)} &\circ \widehat{\mathrm{pr}}^{\prime\mathrm{K}}_{(p',p''),(x',x'')} \colon \\
&\bigwedge\nolimits^{-\bullet} \Gamma\left(E_p^* \Big|_{W_x}\right) \xrightarrow{\ \simeq\ } \bigwedge\nolimits^{-\bullet} \Gamma\left(\left(\mathrm{ev}_1^* E_p^* \oplus \mathrm{ev}_0^* T^* U_p \oplus E_{p,W}^*\right)\!\Big|_{W^\mathrm{P}_x \times W}\right).
\end{split}
\]

Similarly, we can verify that $\widehat{\delta}^{\mathrm{K}}_{2,\delta_1(p),\delta_{1,p}(x)} \circ \widehat{\mathrm{pr}}^{\prime\prime\mathrm{K}}_{(p',p''),(x',x'')}$ is a quasi-isomorphism.

We consider the canonical projection
\[
\widehat{\mathrm{pr}}^{\prime\prime\mathrm{K}}_{(p',p''),(x',x''),1} \colon a'' \longmapsto \left(0, \pi^{\prime\prime*} a''\right).
\]

Again, it is straightforward to see that the induced chain map (and hence the induced $L_{\infty}[1]$-morphism) is the one induced by the composite bundle embedding
\[
\overline{\mathrm{em}}_{p,p} \circ \overline{\iota}''_x \quad \text{over} \quad \mathrm{em}_{p,p} \circ \iota''_x,
\]
as shown in the following commutative diagram:
\[
\begin{tikzcd}[column sep=large, row sep=large]
E_p\big|_{W_x} \arrow[r, hook, "\overline{\iota}''_x"] \arrow[d, swap, "s_p|_{W_x}"] & \pi^{\prime*}\left(E_p\big|_{W_x}\right) \oplus \pi^{\prime\prime*}\left(E_p\big|_{W_x}\right) \arrow[r, hook, "\overline{\mathrm{em}}_{p,p}"] \arrow[d, swap, "\pi^{\prime*}s_p \oplus \pi^{\prime\prime*}s_p|_{W_x \times W_x}"] & \mathrm{ev}_1^*\left(E_p\big|_{W_x}\right) \oplus \mathrm{ev}_0^* T W_x \oplus E_W \arrow[d, swap, "s^{\mathrm{P}}_p"] \\
W_x \arrow[r, hook, "\iota''_x"] & W_x \times W_x \arrow[r, hook, "\mathrm{em}_{p,p}"] & W^\mathrm{P}_x \times W
\end{tikzcd}
\]
where each map is given explicitly by
\[
\begin{split}
\overline{\iota}''_x \colon E_p\big|_{W_x} &\longrightarrow \left.\left(\pi^{\prime*} E_p \oplus \pi^{\prime\prime*} E_p\right)\right|_{W_x \times W_x}, \\
(y, e'') &\longmapsto \bigl((y,y), 0, 2e''\bigr),
\end{split}
\]
and
\[
\begin{split}
\iota''_x \colon W_x &\longrightarrow W_x \times W_x, \\
y &\longmapsto (y,y).
\end{split}
\]
Here, we choose the retraction map
\[
\begin{split}
\pi''_x \colon W^{\mathrm{P}}_x &\twoheadrightarrow \operatorname{Im}(\iota''_x), \\
(y,z) &\longmapsto (y,y).
\end{split}
\]
\end{proof}

In conclusion, we have shown that $\Delta = \Delta_2 \circ \Delta_1$.
\end{proof}

\section{$\mathbf{Kur}$ as a category of fibrant objects}

This subsection concludes the paper by proposing a structure of a category of fibrant objects on $\mathbf{Kur}$, which provides a way to compute the homotopy category $\mathrm{Ho}(\mathbf{Kur})$.

\subsection{Category of fibrant objects}

To construct the homotopy category, one usually considers a model category structure in the sense of Quillen \cite{Quillen}. However, we do not need the full machinery of model categories here. Instead, \cite{Brown} shows that the existence of path objects and pullbacks is sufficient for computing the homotopy category $\mathrm{Ho}(\mathbf{Kur})$.

We recall the definition introduced in \cite{Brown}.

\begin{defn}[Category of fibrant objects]\label{njjs}
Let ${C}$ be a category and let $\mathrm{Fib}$ and $\mathrm{W}$ be two classes of morphisms of ${C}$. We say that the triple $({C}, \mathrm{Fib}, \mathrm{W})$ is a \textit{category of fibrant objects} if the following axioms are satisfied:
\begin{enumerate}[label = (\roman*)]
\item ${C}$ has a final object, and the morphisms to the final object are fibrations.
\item All isomorphisms are fibrations.
\item Weak equivalences satisfy the two-out-of-three property.
\item The pullback of each fibration under an arbitrary morphism exists and is again a fibration.
\item The pullback of a trivial fibration is a trivial fibration.
\item Every diagonal morphism factors as a weak equivalence followed by a fibration.
\end{enumerate}
\noindent Here, a \textit{trivial fibration} means a morphism that is both a fibration and a weak equivalence.
\end{defn}

\begin{rem}
\cite{BLX} shows that their notion of derived differentiable manifolds has the structure of a category of fibrant objects.
\end{rem}

\begin{thm}\label{bownbbfab}
$(\mathbf{Kur}, \mathrm{Fib}, \mathrm{W})$ is a category of fibrant objects.
\end{thm}

\begin{proof}
\begin{enumerate}[label = (\roman*)]
\item (\textit{$\mathbf{Kur}$ has a final object, and the morphisms to the final object are fibrations.}) We identify a final object as follows. Define the Kuranishi chart at the point $*$ by
\[
\mathcal{U}_* :=
\left( *, *, \mathrm{id}_*, \{0\}, \mathrm{id}_* \right)
\]
with the trivial $L_{\infty}[1]$-algebra
\[
\mathcal{C}_* = 0.
\]

Let
\[
* := \left( \{*\},\; \left[\widehat{\mathcal U}_*\right] \right).
\]

\begin{lem}\label{finob}
$*$ is the final object in the category $\mathbf{Kur}$.
\end{lem}

\begin{proof}
Let
\[
\mathfrak{X} = \left(X, \left[\widehat{\mathcal U}\right]\right)
\]
be a Kuranishi space.

A morphism
\[
F : \mathfrak{X} \rightarrow *
\]
is given by
\[
\overline{F} := \left( \widehat{\mathcal{U}}_*, \widehat{\mathcal{U}}, f, \left\{f_p\right\},\, \left\{\widehat f_{p,*}\right\} \right),
\]
where
\[
f: X \rightarrow *,
\]
\[
f_p : U_p \rightarrow \{*\}; \  p \in X,
\]
and
\[
\widehat f_{p,*} :
\mathcal{C}_* = \{0\}
\rightarrow
\mathcal{C}_{p,x}; \ p \in X,\; x \in s_p^{-1}(0)
\]
are trivial maps in each corresponding sense. The conditions for $\overline{F}$ to be a pre-morphism are satisfied for trivial reasons. In other words, we obtain a morphism
\[
F := \left[\overline{F}\right]
\]
from $\mathfrak{X}$ to $*$.
\end{proof}

Also, $\overline{F}$ satisfies the axioms for a fibration for trivial reasons, so we omit the details.

\begin{enumerate}[label = (\roman*), start = 2]
\item (\textit{Every isomorphism is a fibration.}) See Lemma \ref{afdhjdasf}.

\item (\textit{Weak equivalences satisfy the two-out-of-three property.})

Let $\mathfrak{X} = \left(X, \left[\widehat{\mathcal{U}}\right]\right)$, $\mathfrak{X}' = \left(X', \left[\widehat{\mathcal{U}'}\right]\right)$, and $\mathfrak{X}'' = \left(X'', \left[\widehat{\mathcal{U}''}\right]\right)$ be Kuranishi spaces. Let
\[
F : \mathfrak{X} \rightarrow \mathfrak{X}', \qquad
G : \mathfrak{X}' \rightarrow \mathfrak{X}'',
\qquad
H : \mathfrak{X} \rightarrow \mathfrak{X}''
\]
be morphisms between them satisfying $H = G \circ F$.

Suppose that $F$, $G$, and $H$ are represented by the pre-morphisms
\begin{equation}\nonumber
\begin{cases}
\overline{F} = \left(\widehat{\mathcal{U}}, \widehat{\mathcal{U}'}, f, \left\{f_{p}\right\}, \left\{\widehat{f}_{p,x}\right\}\right),\\
\overline{G} = \left(\underline{\widehat{\mathcal{U}'}}, \widehat{\mathcal{U}''}, g, \left\{g_{f(p)}\right\}, \left\{\widehat{g}_{f(p),y}\right\}\right),\\
\overline{H} = \left(\underline{\widehat{\mathcal{U}}}, \underline{\widehat{\mathcal{U}''}}, h, \left\{h_{p}\right\}, \left\{\widehat{h}_{p,x}\right\}\right),
\end{cases}
\end{equation}
respectively, with
\[
\left[\widehat{\mathcal{U}'}\right]
=
\left[\underline{\widehat{\mathcal{U}'}}\right]
\qquad\text{and}\qquad
\left[\widehat{\mathcal{U}}\right]
=
\left[\underline{\widehat{\mathcal{U}}}\right].
\]
Without loss of generality, we may assume that
\[
\widehat{\mathcal{U}'} = \underline{\widehat{\mathcal{U}'}}
\qquad\text{and}\qquad
\widehat{\mathcal{U}} = \underline{\widehat{\mathcal{U}}},
\]
so that
\begin{equation}\label{hgf23h}
h = g \circ f,
\end{equation}
and
\begin{equation}\label{hgf23q}
\widehat{h}_{p,x}
=
{\widehat{f}}_{p,x}
\circ
{\widehat{g}}_{f(p),f_p(x)}
\end{equation}
up to $L_{\infty}[1]$-homotopy for each $p \in X$ and $x \in s_p^{-1}(0)$.

The two-out-of-three property for homeomorphisms and quasi-isomorphisms follows immediately from \eqref{hgf23h} and \eqref{hgf23q}, respectively. In particular, for the $L_{\infty}$-component, we use the fact that any $L_{\infty}[1]$-morphism homotopic to a quasi-isomorphism is itself a quasi-isomorphism, which follows directly from the definition of $L_{\infty}[1]$-homotopy.

\item (\textit{The pullback of each fibration under an arbitrary morphism exists and is again a fibration.})
See Proposition \ref{hnuun} (i).

\[
\begin{tikzcd}
\mathfrak{Y} \arrow[rd, "{}^{\exists!} \widetilde{H}"] \arrow[bend right=30,swap, "H'"]{rdd} \arrow[bend left=20, "H''"]{rrd} & {} & {}\\
{} & \mathfrak{X}' \,{}_{F}\!\times_{G}\, \mathfrak{X}''
\arrow[r,"\mathrm{Pr}''"]
\arrow[d,"\mathrm{Pr}'"']
& \mathfrak{X}''
\arrow[d,"G"]
\\
{} & \mathfrak{X}'
\arrow[r,"F"']
& \mathfrak{X}.
\end{tikzcd}
\]

\item (\textit{The pullback of a trivial fibration is a trivial fibration.})
See Proposition \ref{hnuun} (ii).

\item (\textit{Every diagonal morphism factors as a weak equivalence followed by a fibration.})
See Proposition \ref{facpro}.
\end{enumerate}
\end{enumerate}
\end{proof}

\begin{defn}[Homotopy category of Kuranishi spaces]
We define the \textit{homotopy category of Kuranishi spaces} $\mathrm{Ho}(\mathbf{Kur})$ to be the resulting category obtained from the construction in Subsection \ref{pppo} and Definition \ref{dthcaa}.
\end{defn}

\subsection{Homotopy category of a category of fibrant objects}\label{pppo}

In this subsection, we briefly revisit the construction of the homotopy category of a category of fibrant objects, following \cite{Brown} for the reader's convenience.

Let $(C, \mathrm{Fib}, \mathrm{W})$ be a category of fibrant objects. For an object $Y$, the diagonal map
\[
\Delta : Y \rightarrow Y \times Y
\]
can be factored as a weak equivalence followed by a fibration:
\[
Y \xrightarrow{\Delta_1} \mathrm{P}Y \xrightarrow{\Delta_2} Y \times Y,
\]
where $\mathrm{P}Y$ is the path object for $Y$. Two morphisms $f, g: X \rightarrow Y$ are said to be \textit{right homotopic}, denoted by $f \sim_R g$, if there exists a morphism
\[
H: X \rightarrow \mathrm{P}Y
\]
called a homotopy, making the following diagram commutes:
\[
\begin{tikzcd} 
{} & X \arrow[d,"H"] \arrow[dl,"f"'] \arrow[dr,"g"] &{}\\ 
Y & \mathrm{P}Y \arrow[l,"\mathrm{ev}_0"'] \arrow[r,"\mathrm{ev}_1"] &Y,
\end{tikzcd}
\]
where we denote
\[
\mathrm{ev}_0
:=
\operatorname{pr}_1\circ\Delta_2, \ 
\mathrm{ev}_1
:=
\operatorname{pr}_2\circ\Delta_2
:
\mathrm{P}Y\rightarrow Y.
\]

By \cite{Brown}, the right homotopy relation $\sim_R$ is an equivalence relation on the morphism set $\mathrm{Hom}_{{C}}(X, Y)$. We then denote by $C/\sim$ the category consisting of the following data:
\begin{enumerate}
\item[--] $\mathrm{Ob}(C/\sim) := \mathrm{Ob}(C),$
\item[--] $\mathrm{Hom}_{{C}/\sim}(X, Y) := \mathrm{Hom}_{{C}}(X, Y) / \sim_R.$
\end{enumerate}

\begin{lem}{\cite[Brown's factorization lemma]{Brown}}
Let $({C}, \mathrm{Fib}, \mathrm{W})$ be a category of fibrant objects. If $f: X \rightarrow Y$ is any morphism in ${C}$, then $f$ can be factored as
\[
X \xrightarrow{i} Z \xrightarrow{p} Y,
\]
where $i: X \rightarrow Z$ is a weak equivalence and $p: Z \rightarrow Y$ is a fibration. Furthermore, there exists a trivial fibration $q: Z \rightarrow X$ such that
\[
q \circ i = \mathrm{id}_X.
\]
\end{lem}

To obtain the homotopy category $\mathrm{Ho}({C}) = {C}[\mathrm{W}^{-1}]$, we need to formally invert the weak equivalences via the \textit{calculus of fractions}. For an arbitrary category, this may require handling infinite and uncomputable zig-zags of morphisms. However, in a category of fibrant objects, Brown's Factorization Lemma ensures that the category ${C}/\sim$ satisfies a right calculus of fractions with respect to trivial fibrations. Consequently, a morphism from $X$ to $Y$ in $\mathrm{Ho}({C})$ can be explicitly represented by a \textit{roof}, i.e., a diagram of the form
\[
\begin{tikzcd}
& Z \arrow[dl, "~"', "w"'] \arrow[dr, "f"] & \\
X & & Y.
\end{tikzcd}
\]
Here $w: Z \xrightarrow{\sim} X$ is a weak equivalence, and $f: Z \rightarrow Y$ is any morphism in ${C}$. Denote by $\mathrm{Roof}(X, Y)$ the collection of all roofs between $X$ and $Y$.

Two roofs $(w: Z \rightarrow X, \, f: Z \rightarrow Y)$ and $(w': Z' \rightarrow X, \, f': Z' \rightarrow Y)$ are defined to be equivalent if they can be dominated by a common roof in the homotopy category ${C}/\sim$, that is, if there exists a third roof $(w'': Z'' \xrightarrow{\sim} X, \, f'': Z'' \rightarrow Y)$ and right homotopy classes of maps $h: Z'' \rightarrow Z$ and $h': Z'' \rightarrow Z'$ making the following diagram commute in ${C}/\sim$:

\[
\begin{tikzcd}[row sep=large, column sep=large]
  {} & {} & Z'' \arrow[ld, "h"'] \arrow[rd, "h'"] & {} & \\
    {} & Z \arrow[rd, "f"] \arrow[ld, "w"'] & {} & Z' \arrow[ld, "f'"'] \arrow[rd, "w'"] & {} \\
    X & {} & Y & {} & X,
\end{tikzcd}
\]
where we require
\[
w \circ h \sim_R w' \circ h'
\qquad\text{and}\qquad
f \circ h \sim_R f' \circ h'.
\]
Because ${C}$ satisfies a right calculus of fractions, this dominance condition is in fact an equivalence relation.

\begin{defn}\cite[The homotopy category $\mathrm{Ho}(C)$]{Brown}\label{dthcaa}
We denote by $\mathrm{Ho}(C)$ the \textit{homotopy category} of the category of fibrant objects $C$, which consists of the following data:
\begin{enumerate}
\item[--] $\mathrm{Ob}(\mathrm{Ho}(C)) := \mathrm{Ob}(C),$
\item[--] $\mathrm{Hom}_{\mathrm{Ho}({C})}(X, Y) := \mathrm{Roof}(X, Y) / \sim.$
\end{enumerate}
\end{defn}

Using Theorem \ref{bownbbfab}, we apply this construction to $(\mathbf{Kur}, \mathrm{Fib}, \mathrm{W})$ to obtain the homotopy category $\operatorname{Ho}(\mathbf{Kur}).$

\begin{prop}
There exists a canonical embedding
\[
\mathbf{Man} \hookrightarrow \mathrm{Ho}\left(\mathbf{Kur}\right).
\]
\end{prop}

\begin{proof}
We can compose the functor
\[
\mathbf{Man} \hookrightarrow \mathbf{Kur}
\]
in Theorem \ref{popl} with the localization functor
\[
\mathbf{Kur} \rightarrow \mathrm{Ho}\left(\mathbf{Kur}\right).
\]
We claim that this composition is an embedding. Regarding manifolds as objects in $\mathrm{Ho}\left(\mathbf{Kur}\right),$ observe the following. When restricted to manifolds, a weak equivalence consists of a homeomorphism between the underlying topological spaces together with an isomorphism induced by the pullback on the function spaces of local neighborhoods. Hence, a weak equivalence between manifolds is precisely a diffeomorphism.
\end{proof}

\appendix

\section{$L_{\infty}[1]$-algebras and their homotopy theory}

In this section, we briefly introduce $L_{\infty}[1]$-algebras and their homotopy theories, following \cite{Kim2}. We first recall the notion of \textit{graded symmetric algebra} $SC$ of a vector space $C$ over a field $\mathbf{k},$
$$
S C := TC / \langle v \otimes v' - (-1)^{|v|\cdot|v'|} v' \otimes v \rangle,
$$
with its degree $k$ component $S^k C := \{v \in S C \mid v \text{ is homogeneous of degree } k\}.$
We have a decomposition
$$
S C = \bigoplus\limits_{k=0}^{\infty} S^k C
$$
with the induced product $\odot$ on each component.
We denote by $\text{Sh}(i, k-i)$ the set of $(i, k-i)$-unshuffles, and the sign $\sgn(\tau)$ for $\tau \in \text{Sh}(i, k-i)$ is defined for homogeneous elements $a_1, \cdots, a_k \in C$, we write
$$
a_{\tau(1)} \odot \cdots \odot a_{\tau(k)} = \sgn(\tau) a_1 \odot \cdots \odot a_k.
$$

\begin{defn}
An \textit{$L_{\infty}[1]$-algebra} is a pair $\left(C, \{l_k\}\right)$ consisting of a vector space $C$ and a family of degree 1 linear maps
$$
l_k : S^k C \rightarrow C, \ k \geq 0,
$$
satisfying the relations
\begin{equation}\label{quadrel}
\sum\limits_{i = 0}^{k} \sum\limits_{\tau \in {\text{Sh}}(i, k-i)} {\sgn(\tau)} l_{k-i+1}\left(l_i(a_{\tau(1)}, \cdots, a_{\tau(i)}),a_{\tau(i+1)}, \cdots, a_{\tau(k)}\right) = 0.
\end{equation}
\end{defn}

\begin{defn}
Let $(C,\{l_k\})$ and $(C', \{l'_k\})$ be two $L_{\infty}[1]$-algebras. An $L_{\infty}[1]$\textit{-algebra morphism}, or simply $L_{\infty}[1]$\textit{-morphism}
\begin{equation}\label{lrel}
f : C \rightarrow C'
\end{equation}
is a family of degree 0 linear maps
$$
f_k : S^kC \rightarrow C', \ k \geq 0,
$$
satisfying the relations
\begin{equation}\label{frel}
\begin{split}
\sum\limits_{i = 0}^{k}& \sum\limits_{\tau \in {\text{Sh}}(i, k-i)} {\sgn(\tau)} f_{k-i+1}\left(l_i(a_{\tau(1)}, \cdots, a_{\tau(i)}),a_{\tau(i+1)}, \cdots, a_{\tau(k)}\right)\\
&= \sum\limits_{\substack{t, j_1, \cdots, j_t \geq 1,\\ j_1 + \cdots + j_t = k}} \sum\limits_{\tau \in S_k}  \frac{\sgn(\tau)}{t! j_1! \cdots j_t!} \  l'_{t}\bigl(f_{j_1}(a_{\tau(1)}, \cdots, a_{\tau(j_1)}), \cdots, \\
& \quad \quad \quad \quad \quad \quad \quad \quad \quad \quad \quad \quad f_{j_t}(a_{\tau(k -(j_1 + \cdots + j_{t-1}))}, \cdots, a_{\tau(k)})\bigr).
\end{split}
\end{equation}
Here, $S_k$ denotes the symmetric group of permutations of $k$ elements. 
\end{defn}

\begin{defn}
For two $L_{\infty}[1]$-morphisms
$$
f : C \rightarrow C', \ g: C' \rightarrow C'',
$$
we define their \textit{composition}
$$
g \circ f : C \rightarrow C''
$$
by a family of linear maps of degree 0 for $k \geq 0$
\begin{equation}\nonumber
\begin{split}
(g \circ f)_k := &\sum\limits_{i = 0}^{k} \sum\limits_{\tau \in S_k}  \frac{\sgn(\tau)}{t! j_1! \cdots j_t!} \ g_{t}\bigl(f_{j_1}(a_{\tau(1)}, \cdots, a_{\tau(j_1)}), \cdots,\\
&\quad \quad\quad \quad\quad \quad\quad \quad f_{j_t}(a_{\tau(k -(j_1 + \cdots + j_{t-1}))}, \cdots, a_{\tau(k)})\bigr).
\end{split}
\end{equation}
It is straightforward to verify that $\{(g \circ f)_k\}_{k \geq 0}$ satisfies the relation (\ref{frel}).
\end{defn}

\begin{defn}
We say an $L_{\infty}[1]$-algebra $\{l_k\}_{k \geq 0}$ is \textit{strict} if $l_0 = 0.$ Otherwise, we say it is \textit{curved.} We similarly define \textit{strict/curved} $L_{\infty}[1]$-morphisms.
\end{defn}

In the strict case, the relations (\ref{lrel}) and (\ref{frel}) coincide with the differential and the chain map relations, respectively. That is, they satisfy
\[
l_1 \left( l_1(a) \right) = 0; \ l'_1 \left(f_1 (a) \right) = f_1 \left( l_1(a) \right).
\]
\begin{defn}
We say that a strict $L_{\infty}[1]$-algebra $(C,\{l_k\})$ is \textit{acyclic} if its cohomology for each degree vanishes, that is, if 
\[
H^*(C) = \frac{\ker{l_1}}{\mathrm{Im}l_1} = 0.
\]
We say that a strict $L_{\infty}[1]$-morphism $\{f_k\}_{k \geq 1}$ between strict $L_{\infty}[1]$-algebras is a \textit{quasi-isomorphism} if $f_1$ is a quasi-isomorphic chain map.
\end{defn}

From now on and elsewhere in this paper, we always assume the strictness of $L_{\infty}[1]$-algebras and their morphisms.

To define homotopies between two $L_{\infty}[1]$-morphisms, we need to introduce the following notion:
\begin{defn}[Models of $\Delta^1 \times C$]\label{mdl1}
Let $C$ be an $L_{\infty}[1]$-algebra. We say an $L_{\infty}[1]$-algebra $\mathfrak{C}$ is a \textit{model} of $\Delta^1 \times C$ if there exist $L_{\infty}[1]$-morphisms
\[
\Eval_j : \mathfrak{C} \rightarrow C, \ \ j = 0,1\\
\]
and
\[
\Incl : {C} \rightarrow \mathfrak{C},
\]
with the following properties:
\begin{enumerate}[label = (\roman*)]
\item $\left(\Eval_j\right)_{k \geq 2} \equiv 0, \ j =0,1, \ \Incl_{k \geq 2} \equiv 0.$
\item $\Eval_j, \ j =0,1$ and  $\Incl$ are quasi-isomorphisms.
\item $(\Eval_j)_1 \circ \Incl = \mathrm{id}_C.$
\item $(\Eval_0)_1 \oplus (\Eval_{1})_1 : \mathfrak{C} \rightarrow {C} \oplus C$ is surjective.
\end{enumerate}
\end{defn}

Using the notion of models, we can define homotopies between $L_{\infty}[1]$-morphisms:

\begin{defn}[Homotopy]\label{defn:homotopy}
We say that two $L_{\infty}[1]$-morphisms $f_0, f_1 : (C, \{l_k\}) \rightarrow (C', \{l'_k\})$ are \textit{homotopic} if there exist a \textit{model of} $\Delta^1 \times C'$ denoted by $\mathfrak{C}'$ and an $L_{\infty}[1]$-morphism $h: C \rightarrow \mathfrak{C}'$ such that we have $f_j = \Eval_j \circ h, \ j = 0,1.$
\end{defn}

\begin{exam}\label{hhex}
Let $\Delta^1$ be the unit interval and $\Omega^*(\Delta^1)$ its de Rham complex over a field. For an $L_{\infty}[1]$-algebra $C$, consider the tensor product
\[
\Omega^*(\Delta^1) \otimes C,
\]
which inherits an $L_{\infty}[1]$-algebra structure with operations
\[
l_k \colon (\Omega^*(\Delta^1) \otimes C)^{\otimes k} \longrightarrow \Omega^*(\Delta^1) \otimes C, \quad k \geq 1,
\]
given by
\[
l_k(\alpha_1 \otimes x_1, \dots, \alpha_k \otimes x_k) :=
\begin{cases}
d\alpha_1 \otimes x_1 + (-1)^{|\alpha_1|} \alpha_1 \otimes l_1(x_1) & \text{if } k = 1, \\[1ex]
(-1)^{|\vec{\alpha}|} \alpha_1 \wedge \dots \wedge \alpha_k \otimes l_k(x_1, \dots, x_k) & \text{if } k \geq 2,
\end{cases}
\]
for $\alpha_i \in \Omega^*(\Delta^1)$, $x_i \in C$ ($i = 1, \dots, k$), and $k \geq 1$. Here, the sign factor is defined by
\[
|\vec{\alpha}| := \sum_{i=1}^{k-1} |x_i| \left( \sum_{j=i+1}^k |\alpha_j| \right) + \sum_{i=1}^k |\alpha_i|.
\]

The $L_{\infty}[1]$-morphisms $\mathrm{Eval}_i = \{(\mathrm{Eval}_i)_k\}_{k \geq 1}$ (for $i = 0, 1$) and $\mathrm{Incl} = \{\mathrm{Incl}_k\}_{k \geq 1}$ are given by
\[
(\mathrm{Eval}_i)_k \colon (\Omega^*(\Delta^1) \otimes C)^{\otimes k} \longrightarrow C,
\]
\[
(\mathrm{Eval}_i)_k :=
\begin{cases}
\left.\mathrm{id}\right|_{t=i} \otimes \mathrm{id}_C & \text{if } k = 1, \\[1ex]
0 & \text{if } k \geq 2,
\end{cases}
\]
and
\[
\begin{split}
\mathrm{Incl}_k \colon C^{\otimes k} &\longrightarrow \Omega^*(\Delta^1) \otimes C, \\
\mathrm{Incl}_k &:=
\begin{cases}
1 \otimes \mathrm{id}_C & \text{if } k = 1, \\[1ex]
0 & \text{if } k \geq 2,
\end{cases}
\end{split}
\]
respectively. It follows immediately that this data satisfies the axioms in Definition~\ref{mdl1}. In particular, $\mathrm{Eval}_i$ ($i=0,1$) and $\mathrm{Incl}$ are $L_{\infty}[1]$-algebra morphisms. 

Moreover, they are quasi-isomorphisms, as sketched below. Since $\mathrm{Incl}_1$ is injective, it being a quasi-isomorphism is equivalent to the acyclicity of the quotient complex
\[
\frac{\Omega^*(\Delta^1) \otimes C}{\mathrm{Incl}_1(C)} \simeq \frac{\Omega^*(\Delta^1) \otimes C}{\{1\} \otimes C} \simeq \frac{\Omega^*(\Delta^1)}{\mathbb{R}\cdot 1} \otimes C,
\]
which follows from the acyclicity of $\Omega^*(\Delta^1)/(\mathbb{R}\cdot 1)$ together with the Künneth formula. Finally, axiom (iii) of Definition~\ref{mdl1} implies that $\mathrm{Eval}_i$ is also a quasi-isomorphism for each $i$.
\end{exam}

\begin{lem}\cite[Lemma 3.3]{Kim2}
Homotopies define an equivalence relation.
\end{lem}

We introduce  homotopy equivalences for $L_\infty[1]$-algebras.

\begin{defn}
An $L_{\infty}[1]$-morphism $f : C \rightarrow C'$ is a \textit{homotopy equivalence} if there exists another $L_{\infty}[1]$-morphism $g : C' \rightarrow C$ such that $g \circ f$ and  $f \circ g$ are homotopic to $\mathrm{id}_C$ and $\mathrm{id}_{C'},$ respectively.
\end{defn}

\begin{prop}\cite[Proposition 3.7]{Kim2}\label{prop:heer}
Homotopy equivalences define an equivalence relation.
\end{prop}

We show that a key theorem on quasi-isomorphisms, which shows the usefulness of our definition.
\begin{thm}\label{anhp}
Any two quasi-isomorphic $L_{\infty}[1]$-morphisms $f_0, f_1 : C \rightarrow C'$ are homotopic.
\end{thm}

\begin{proof}
This theorem is a special case of \cite[Corollary 4.6]{Kim2} for two morphisms.
\end{proof}

We close this section by stating an $L_{\infty}[1]$-algebra version of the Whitehead theorem over a field.

\begin{thm}[Whitehead theorem]\cite[Theorem 3.13]{Kim2}\label{wht}
Over a field and for strict $L_{\infty}[1]$-algebras, a quasi-isomorphism is a homotopy equivalence.
\end{thm}

\section{Presymplectic Neighborhoods and Local $L_{\infty}[1]$-Algebras}

In this section, we provide the details on the structures that we put on a Kuranishi chart. For more detail, we refer the reader to \cite{Kim1}.

Let a Kuranishi chart be given with the base $U.$ We equip $U$ with a closed two-form $\beta \in \Omega^2(U).$ 

We require that $U$ has a decomposition 
\begin{equation}\label{aabbbb}
U = \bigcup\limits_i \mathcal{S}_i,
\end{equation}
into (possibly nonconnected) submanifolds,
\[
\mathcal{S}_i := \{x \in U \mid \text{rk} (\ker\beta_x) =i \}, \ 0 \leq i \leq \dim U
\]
together with their tubular neighborhoods:
\[
\begin{cases}
\iota_{i} : N_i \rightarrow  U, \text{ an open neighborhood of each }  \mathcal{S}_{i} \text{ in }U,\\
\pi_{i} : N_i \rightarrow \mathcal{S}_{i}, \text{ the associated projection}.
\end{cases}
\]

\begin{rem}
\cite[Corollary 6.2]{KO} proves that there exists a residual subset of such closed two-forms. Furthermore, the stratification structure turns out to be Whitney (cf. \cite[Definition 6.5 \& Theorem 6.6]{KO}), which leads to the construction of \textit{stratified $L_{\infty}$-spaces.}
\end{rem}

\subsubsection*{The presymplectic open neighborhood $W_x$} We associate an open contractible submanifold $W_x$ to each zero point $x \in s^{-1}(0).$ 

For each zero point $x \in s^{-1}(0) \cap \mathcal{S}_i$ for some $i$, we choose $\overset{\circ}{W}_x \subset \mathcal{S}_i,$ an open ball containing $x$ in $\mathcal{S}_i.$ For the inclusion ${\iota_x} : \overset{\circ}{W}_x {\hookrightarrow} U,$ we denote $\beta|_{\overset{\circ}{W}_x} := \iota_x^*\beta.$ We have $d\left(\beta|_{\overset{\circ}{W}_x}\right) = d\iota_x^*\beta = \iota_x^* d\beta = 0$ and that $\beta|_{\overset{\circ}{W}_x}$ is of constant rank by construction. In other words, $\left(\overset{\circ}{W_x}, \beta|_{\overset{\circ}{W_x}}\right)$ is a presymplectic manifold.

Then we obtain another presymplectic manifold
\begin{equation}\label{adslfjkdh}
W_x = (W_x, \beta_{W_x}) := 
\left(\pi_i^{-1}(\overset{\circ}{W_x}), \pi^*_i(\beta|_{\overset{\circ}{W_x}})\right)
\end{equation}
and call it a \textit{local presymplectic neighborhood of} $x \in s^{-1}(0).$ 

We write 
\[
T\mathcal{F}_x := \ker \beta_{W_x}
\]
for the regular foliation (i.e., each leaf having the same dimension) determined by the kernel of $\beta_{W_x}.$ In \cite[Lemma B.2]{Kim1}, we equip $\Omega^{\bullet +1}(\mathcal{F}_x)$ with an $L_{\infty}[1]$-algebra structure $\{l^{\mathcal{F}}_k\}.$ Furthermore, we have:

\subsubsection*{The $L_{\infty}[1]$-algebra $\mathcal{C}_x$}
At each zero point $x \in s^{-1}(0),$ we associate a \textit{local $L_{\infty}[1]$-algebra,}
\begin{equation}\nonumber 
\mathcal{C}_x := \overbrace{\bigwedge\nolimits^{-\bullet}\Gamma(E^*|_{W_x})}^{\text{Koszul}} \oplus \overbrace{\Omega^{\bullet + 1}_{\mathrm{aug}}(\mathcal{F}_x)}^{\text{de Rham}},
\end{equation}
which consists of the two parts: Koszul and de Rham.

The Koszul part, $\bigwedge\nolimits^{-\bullet}\Gamma(E^*|_{W_x})$ is the Koszul complex,
\[
0 \rightarrow \overbrace{\bigwedge\nolimits^{r}\Gamma(E^*|_{W_x})}^{\text{deg} = -r} \xrightarrow{\iota_{s|_{W_x}}} \cdots \xrightarrow{\iota_{s|_{W_x}}}  \overbrace{\Gamma(E^*|_{W_x})}^{\text{deg} = -1} \xrightarrow{\iota_{s|_{W_x}}}  \overbrace{C^{\infty}(W_x)}^{\text{deg} = 0} \rightarrow 0
\]
with the differential $l_1^{\mathrm{K}} := \iota_{s|_{W_x}},$ given by:
\[
l_1^{\mathrm{K}} : a_1 \wedge \cdots \wedge a_m \mapsto \sum\limits_{i=1}^{m} (-1)^{i+1}a_i(s|_{W_x}) \cdot a_1 \wedge \cdots \wedge \widehat{a_i} \wedge \cdots \wedge a_m,
\]
with all higher $l^{\mathrm{K}}_{k \geq 2}$ being set to zero.

The de Rham part, $\Omega^{\bullet + 1}_{\mathrm{aug}}(\mathcal{F}_x)$ is the augmented foliation de Rham complex degree shifted by 1 equipped with the $L_{\infty}[1]$-algebra structure $\{l^{\mathrm{dR}}_k\}_{k \geq 1}$ obtained from Proposition \ref{augomega}. 

\begin{prop}\cite[Corollary 2.10, Proposition B.16]{Kim1}\label{augomega}
In the above situation, there exists an $L_{\infty}[1]$-algebra structure on the chain complex $\Omega^{\bullet +1}_{\mathrm{aug}}(\mathcal{F}_x)$ that extends $\{l^{\mathcal{F}}_k\}$ on $\Omega^{\bullet +1}(\mathcal{F}_x).$ Moreover, this $L_{\infty}[1]$-algebra has trivial cohomology.
\end{prop}

The $L_{\infty}[1]$-structure on $\mathcal{C}_x$ is then given by
\[
l_k : \mathcal{C}_x^{\otimes k} \rightarrow \mathcal{C}_x; \ l_k := l^{\mathrm{K}}_k \oplus l^{\mathrm{dR}}_k,
\]
where the direct sum notation indicates that the operations on the two components are defined separately. It is immediate that the family $\{l_k\}_{k \geq 1}$ satisfies the $L_{\infty}[1]$-relation.

\begin{lem}\cite[Lemma 2.3]{Kim1}\label{lemdw}
For different choices of $\overset{\circ}{W}_x$, we obtain isomorphic de Rham $L_\infty[1]$-algebras.  
\end{lem}

For the definition of chart morphisms we need the notion of completed algebras:
\begin{defn}[Completed algebras]\label{ladef}
Given a smooth map $\phi: V \rightarrow U,$ We define the \textit{completion of} $\mathcal{C}_{x}$ by
\begin{equation}\nonumber 
\mathcal{C}_{x, \phi} := \left(\bigwedge\nolimits^{-\bullet}\Gamma(E^*|_{W_x})\right)_{\phi} \oplus \Omega^{\bullet + 1}_{\mathrm{aug}, \phi}(\mathcal{F}_x),
\end{equation}
where the Koszul part
\begin{equation}\label{kdzpt}
\left(\bigwedge\nolimits^{-\bullet}\Gamma(E^*|_{W_x})\right)_{\phi} :=  C^{\infty}_{\phi}(W_x) \otimes_{C^{\infty}(W_x)} \bigwedge\nolimits^{-\bullet}\Gamma(E^*|_{W_x}),
\end{equation}
where we consider the inverse limit
\begin{equation}\label{ivlmt}
C_{\phi}^{\infty}(W) := \lim_{\longleftarrow} C^{\infty}(W_x)/ I_{\phi}^{j} \cdot C^{\infty}(W_x),
\end{equation}
where we denote $I_{\phi} := \left\{ f \in C^{\infty}(W_x) \mid f|_{\text{Im}\phi} \equiv 0\right\}.$

Its $L_{\infty}[1]$-structure
\[
l^{\mathrm{K}, \phi}_k: {\left(\bigwedge\nolimits^{i}\Gamma(E^*|_{W_x})\right)_{\phi}}^{\otimes k} \rightarrow \left(\bigwedge\nolimits^{i-1}\Gamma(E^*|_{W_x})\right)_{\phi},\\
\]
for each $1 \leq i \leq r$ is defined as follows: For each $j\geq 1,$ $h \in C^{\infty}(W_x)^{(j)},$ and $a \in \bigwedge\nolimits^{-\bullet}\Gamma(E^*|_{W_x}),$ we set
\[
(j) : l^{\mathrm{K}, \phi}_1(h \otimes a) := [1]_{j-2} \otimes \iota_{s|_{W_x}}(\widetilde{h}a),
\]
where $\widetilde{h}$ is a choice of representative in $C^{\infty}(W_x)$ such that $h = \widetilde{h} + I_{\phi}^j,$ and we set $[1]_{j-2} := 0$ for $j \leq 2$ by definition. All higher $l_{\phi, k \geq 2}^{\mathrm{K}}$'s are set to zero, so the $L_{\infty}$-relation of $\{l^{\mathrm{K}}_{\phi, k}\}_{k \geq 1}$ holds trivially.

The de Rham part $\Omega^{\bullet + 1}_{\mathrm{aug}, \phi}(\mathcal{F}_x)$ is the completed foliation de Rham complex \textit{with augmentation,} given by
\[
\Omega^{\bullet + 1}_{\mathrm{aug}, \phi}(\mathcal{F}_x) := \overbrace{\Omega^{\bullet + 1}(\mathcal{F}_x)_{\phi}}^{\deg \geq -1} \oplus \overbrace{\left(C^{\infty}(W_x)_{\mathcal{F}_x}\right)_{\phi}}^{\deg = -2},
\]
where we denote
\[
\begin{cases}
\Omega^{\bullet + 1}(\mathcal{F}_x)_{\phi} := C_{\phi}^{\infty}(W_x) \otimes_{C^{\infty}(W_x)} \Omega^{\bullet + 1}(\mathcal{F}_x),\\
\big(C^{\infty}(W_x)_{\mathcal{F}_x}\big)_{\phi} : = \ker \big(l_1^{\mathrm{dR}} : \Omega^{-1}(\mathcal{F}_x)[1]_{\phi} \rightarrow \Omega^{0}(\mathcal{F}_x)[1]_{\phi}\big).
\end{cases}
\]
The de Rham part $L_{\infty}[1]$-structure $l_{\phi, k}^{\mathrm{dR}}$ is obtained by applying the following Proposition \ref{augomega}. The formula can be given similarly to $l^{\mathrm{K}, \phi}_{\phi, k}$ in the setting of $V$-algebras in Appendix B of \cite{Kim1}. We omit its precise formulation here. For more details, we refer the reader to \cite[Definition B.19]{Kim1}.

In conclusion, $\mathcal{C}_{x,\phi}$ with $\left\{l_{\phi, k} := l^{\mathrm{K}}_{\phi, k} \oplus l^{\mathrm{dR}}_{\phi, k}\right\}$ is an $L_{\infty}[1]$-algebra with the $L_{\infty}[1]$-relation, which can be verified in a straightforward manner.
\end{defn}

Given a local algebra $\mathcal{C}_x,$ there exists a natural map to its completion:
For each $k \geq 1$, We define
\begin{equation}\label{varep}
\widehat{\varepsilon}_{\phi(x), \phi,k} : \mathcal{C}_x^{\otimes k} \rightarrow \mathcal{C}_{x, \phi}
\end{equation}
by
\begin{equation}\nonumber
\widehat{\varepsilon}_{\phi(x), \phi, k}\big((a_1, \xi_1), \cdots, (a_k,\xi_k)\big) := \begin{cases}
1 \otimes (a_1, \xi_1) = (1 \otimes a_1, 1 \otimes \xi_1) &\text{ if } k = 1,\\
0 &\text{ if } k \geq 2,
\end{cases}
\end{equation}
and consider the family $\widehat{\varepsilon}_{\phi(x), \phi} := \left\{\widehat{\varepsilon}_{\phi(x), \phi,k}\right\}_{k \geq 1}.$
\begin{lem}\cite[Lemma 2.8]{Kim1}\label{vecpf}
$\widehat{\varepsilon}_{\phi(x), \phi}$ is an $L_{\infty}[1]$-morphism.
\end{lem}

In the context of Definition \ref{ladef} regarding local $L_{\infty}[1]$-algebras, several special cases are worth discussing. First, we have a simple type of completions for surjective $\phi.$
\begin{lem}\cite[Lemma 2.12]{Kim1}\label{surjcp}
If $\phi$ is surjective, then we have an isomorphism
\[
\mathcal{C}_{\phi(x),\phi} \simeq \mathcal{C}_{\phi(x)}.
\]
\end{lem}

Second, we consider the completion for open subcharts: Let  $o : U \hookrightarrow U'.$ be an open inclusion and $\mathcal{U} := \mathcal{U}'|_U$ the open subchart on $U.$

\begin{lem}\cite[Lemma 2.13]{Kim1}\label{itavsteai}
In the above situation, there exists an $L_\infty[1]$-quasi-isomorphism:
\[
\widehat{o}_x : \mathcal{C}'_{o(x),o} \simeq \mathcal{C}_x.
\]
\end{lem}

\section{Equivalence relations}

In this section, we review the equivalence relations that are needed in the definition of Kuranishi atlases and spaces. More details are provided in \cite{Kim1}.

We define expansion of a chart as follows:
\begin{defn}[Expansion of a chart]\label{expvs}
Let $\mathcal{U} = (U, E, s, \Gamma, \psi)$ be a Kuranishi chart on $X$ as in Definition \ref{kurdef}, and let $V$ be a finite-dimensional vector space. From $\mathcal{U}$ we construct another chart on $X$, called the \textit{expansion of $\mathcal{U}$ by $V$},
\[
\mathcal{U} \times V := (U \times V, E \times V, s \times \mathrm{id}_V, \Gamma, \psi),
\]
consisting of the following data:
\begin{itemize}
\item[--] $U \times V$, equipped with the closed two-form $\pi^* \beta$, where $\pi: U \times V \to U$ denotes projection onto the $U$-component.
\item[--] The vector bundle $E \times V \to U \times V$, obtained in the natural way from $E \to U.$
\item[--] The section $s \times \text{id}_V: U \times V \longrightarrow E \times V$, $(y,v) \mapsto \big(s(y), v\big).$
\item[--] The group $\Gamma$, acting only on the $U$-component of $U \times V.$
\item[--] The homeomorphism $\psi : {(s \times \text{id}_V)^{-1}(0) }/{\Gamma}\simeq {s^{-1}(0)}/{\Gamma} \xhookrightarrow{\psi} X$, coinciding with the $\psi$ of $\mathcal{U}.$
\item[--] The open neighborhood $W_{(x,0)} := W_x \times V$ of the zero point $(x,0).$
\item[--] The local $L_\infty[1]$-algebra at $(x,0) \in (s \times \mathrm{id}_V)^{-1}(0)$,
\[
\mathcal{C}^V_{(x,0)} := \bigwedge\nolimits^{-\bullet} \Gamma\left(\left(\pi^* E \oplus V\right)^*\big|_{W_{(x,0)}}\right) \oplus \Gamma_{\mathrm{aug}}\left(\bigwedge\nolimits^{\bullet+1}\left(\pi^*T\mathcal{F} \oplus V\right)^*\big|_{W_{(x,0)}}\right).
\]
\item[--] The $L_{\infty}[1]$-operations $\left\{l_{(x,0),k}^V: {\mathcal{C}^V}_{(x,0)}^{\otimes k} \rightarrow \mathcal{C}_{(x,0)}^V\right\}_{k \geq 1}$ on $\mathcal{C}^V_{(x,0)}$,
given by $l_{(x,0), k}^{V, \mathrm{K}} \oplus l_{(x,0), k}^{V, \mathrm{dR}},$ with each component given by
\[
\begin{cases}
&\begin{split}
l_{(x,0), k}^{V, \mathrm{K}}&\big((a_1, w_1^*) , \cdots , (a_k, w_k^*)\big)\\
&\quad \quad := \begin{cases}
\iota_{s \times \mathrm{id}_V}(a_1,w^*_1) = \left\{\iota_s(a_1)(x), \iota_{v}w_1^*(x,v)\right\}_{(x,v) \in W_x \times V} &\text{ if } k =1,\\
0 &\text{ if } k \geq 2,
\end{cases}
\end{split}\\
&l^{V,\mathrm{dR}}_{(x,0),k} \big((\xi_1, \tau_1), \cdots , (\xi_k, \tau_k)\big) := \left(l^{\mathrm{dR}}_{x,k} \left(\xi_1, \cdots,\xi_k \right), l^V_{k}(\tau_1, \cdots, \tau_k)\right).
\end{cases}
\]
Here, we define
\[
 l^V_{k}(\tau_1, \cdots, \tau_k) := \begin{cases}
\sum_i\tau_1(v_i) & \text{with respect to a fixed basis } \{v_i\} \text{ of } V, \text{ if } k = 1,\\
0  & \text{ if } k \geq 2,
\end{cases}
\]
for $a_i \in \Gamma\left(\pi^*E^*\big|_{W_{(x,v)}}\right), \ \xi_i \in \pi^*\left(\Omega^{\bullet +1}\left(\mathcal{F}\big|_{W_{(x,v)}}\right)\right)$, and $w^*_i, \tau_i \in \Gamma\left(V^*|_{W_{(x,v)}}\right)$. It follows immediately that the family $\{l_{(x,0),k}^V\}$ forms an $L_\infty[1]$-algebra. Moreover, we have
\[
H^*\left(\Gamma\big(\bigwedge\nolimits^{\bullet+1}(\pi^*T\mathcal{F} \oplus V)^*\big|_{W_{(x,0)}}\big), l^{V, \mathrm{dR}}_{(x,0),1}\right) \simeq H^*\big(\Omega^{\bullet + 1}(\mathcal{F}_x), l^{\mathrm{dR}}_{x,1}\big),
\]
which can be verified as follows: since $l^{\mathrm{dR}}_{x,k}$ and $l^V_k$ are defined independently, one may apply the K\"{u}nneth theorem together with the fact that the cohomology computed by $l_1$ on an affine space is trivial. We then add the augmentation to the de Rham part, and equip it with an $L_{\infty}[1]$-structure using Proposition \ref{augomega}.

\end{itemize}
\end{defn}

\begin{defn}[Expanded atlases]\label{expatla}

Given a Kuranishi atlas $\widehat{\mathcal{U}}$ and a nonnegative number $m,$ we define the \textit{expanded atlas} of $\widehat{\mathcal{U}}$ by
\[
\widehat{\mathcal{U}} \times \mathbb{R}^{{m}}
:= 
\Big(
\left\{\mathcal{U}_p \times \mathbb{R}^{{m}} \right\}_p,\,
\left\{ \left(U_{pq} \times \mathbb{R}^{{m}}, \phi_{pq}^{\mathbb{R}^{{m}}}, \left\{\phi_{pq,x}^{\mathbb{R}^{{m}}} \right\}\right)\right\}_{p,q}
\Big),
\]
where each component is given by:
\begin{itemize}
    \item[--] $\mathcal{U}_p \times \mathbb{R}^{{m}}$ is the expanded chart for each $p \in X$.

    \item[--] $U_{pq} \times \mathbb{R}^{{m}}$ is an open subset of $U_p \times \mathbb{R}^{{m}}$.
    
    \item[--] $\phi_{pq}^{\mathbb{R}^{{m}}} : U_{pq} \times \mathbb{R}^{{m}} \to U_q \times \mathbb{R}^{{m}}$ is the base coordinate change given by
    \[
    \phi_{pq}^{\mathbb{R}^{{m}}} := \phi_{pq} \times \mathrm{id}_{\mathbb{R}^{{m}}}.
    \]
   \item[--] $\widehat{\phi}_{pq,x}^{\mathbb{R}^{m}} : \mathcal{C}_{q,\,(\phi_{pq}(x),0)}^{\mathbb{R}^{m}}
\to \mathcal{C}_{p,(x,0)}$, for each $x \in s_{p}^{-1}(0)$, is the $L_{\infty}[1]$-coordinate change given by the composition
\[
\mathcal{C}^{\mathbb{R}^{m}}_{q,\,(\phi_{pq}(x),0)}
\xrightarrow{\widehat{\pi}_{q, (\phi_{pq}(x),0)}^{-1}, \simeq}
\mathcal{C}_{q,\,\phi_{pq}(x)}
\xrightarrow{\widehat{\phi}_{pq,x}, \simeq}
\mathcal{C}_{p,x} \xrightarrow{\widehat{\pi}_{p, (x,0)}, \simeq} \mathcal{C}^{\mathbb{R}^{m}}_{p, (x,0)},
\]
Here, the $L_{\infty}[1]$-quasi-isomorphisms $(1)$ and $(2)$ are defined as in Example 2.23 and Lemma 2.24 of \cite{Kim1}.
\end{itemize}
\end{defn}

Any chart naturally restricts on an open subset of the base under the condition that the group action is closed:

\begin{defn}[Open subchart]
Let $\mathcal{U}= (U, E, s, \Gamma, \psi)$ be a Kuranishi chart of $X$ and $U_0 \subset U$ an open subset with $\Gamma \cdot U_0 \subset U_0.$ Then the restricted tuple 
\[
\mathcal{U}|_{U_0}= (U_0, E|_{U_{0}}, s|_{U_{0}}, \Gamma, \psi |_{(U_0 \cap s^{-1}(0))/\Gamma})
\]
canonically determines a Kuranishi chart called the \textit{open subchart} of $\mathcal{U}$ on $U_0.$
\end{defn}

We can consider completions for open subcharts. Let $o : U \hookrightarrow U'$ be an open inclusion, and let $\mathcal{U} := \mathcal{U}'|_U$ denote the open subchart on $U.$

\begin{lem}[Completion at an open embedding with open subchart data]\label{itavsteai}\cite[Lemma 2.13]{Kim1}
In the situation above, there exists an $L_\infty[1]$-quasi-isomorphism
\[
\widehat{o}_x : \mathcal{C}'_{o(x),o} \simeq \mathcal{C}_x.
\]
\end{lem}

\begin{exam}[Different choices for the open neighborhood $W_x$]\label{dcotopwx}
Let $\mathcal{U}$ and $\mathcal{U}'$ be Kuranishi charts that are identical \textit{except} for the choice of open neighborhoods $W_x \underset{\text{open}} {\subset} W'_x$ at a zero point $x \in s^{-1}(0)$. Then Kuranishi charts are independent of such choices, in the following sense: there exists a morphism of charts $\Phi: \mathcal{U} \rightarrow \mathcal{U}'$ given by
\begin{itemize}
    \item[--] $\phi: U \rightarrow U$, the identity map;
    \item[--] $\widehat{\phi}^{\mathrm{c}}_x: \mathcal{C}'_{\phi(x),\phi} \xrightarrow{\simeq} \mathcal{C}_x$, the quasi-isomorphism from Lemma \ref{itavsteai}.
\end{itemize}
\end{exam}

\begin{notation}
Let $\left(X, \widehat{\mathcal{U}}\right)$ be a Kuranishi atlas. We write 
\[
\left(X, \widehat{\mathcal{U}}^0\right) < \left(X, \widehat{\mathcal{U}}\right), \text{ or simply } \widehat{\mathcal{U}}^0 <\widehat{\mathcal{U}}, 
\]
for its open subatlas $\left(X, \widehat{\mathcal{U}}^0\right)$ that consist of the open subcharts and the coordinate changes induced to them.
\end{notation}
With this notation, we define an equivalence relation between the atlases.

\begin{defn}[Equivalence of atlases]\label{eqats}
Let $\left(X, \widehat{\mathcal{U}}_1\right)$ and $\left(X, \widehat{\mathcal{U}}_2\right)$ be Kuranishi atlases. We say that they are \textit{equivalent} and write
\[
\left(X, \widehat{\mathcal{U}}_1\right) \sim \left(X, \widehat{\mathcal{U}}_2\right), \text{ or simply } \widehat{\mathcal{U}}_1 \sim \widehat{\mathcal{U}}_2
\]
if
\begin{equation}\label{urur}
\widehat{\mathcal{U}}^0_1 \times \mathbb{R}^{n_1} = \widehat{\mathcal{U}}^0_2 \times \mathbb{R}^{n_2}
\end{equation}
by which we mean that the following conditions hold:

\begin{enumerate}[label = (\roman*)]
    \item There exists a commutative diagram as follows
    \[
    \begin{tikzcd}[row sep=large, column sep=large]
    E^0_{1,p}|_{U_{1,p}^0 \times \mathbb{R}^{n_1}} \arrow[r, "\simeq"] 
    & E^0_{2,p}|_{U_{2,p}^0 \times \mathbb{R}^{n_2}}\\
    U_{1,p}^0 \times \mathbb{R}^{n_1} \arrow[r, "\simeq"] \arrow[u, "s_{1,p}^0 \times \mathrm{id}_{\mathbb{R}^{n_1}}"'] 
    & U_{2,p}^0 \times \mathbb{R}^{n_2} \arrow[u, "s_{2,p}^0 \times \mathrm{id}_{\mathbb{R}^{n_2}}"]\\
   (s_{1,p}^0 \times \mathrm{id}_{\mathbb{R}^{n_1}})^{-1}(0) \arrow{r}{(1), \simeq} \arrow[u, hook] 
& (s_{2,p}^0 \times \mathrm{id}_{\mathbb{R}^{n_2}})^{-1}(0). \arrow[u, hook]
    \end{tikzcd}
    \]

    \item There exists a group isomorphism $\Gamma^0_{1,p} \simeq \Gamma^0_{2,p}$.
    
    \item There exists a commutative diagram as follows
  \[
    \begin{tikzcd}[row sep=large, column sep=large]
     \frac{(s_{1,p}^0)^{-1}(0)}{\Gamma_{1,p}}
    \arrow[r, "\simeq"] \arrow[d, swap, "\simeq"] 
    & \frac{(s_{1,p}^0 \times \mathrm{id}_{\mathbb{R}^{n_1}})^{-1}(0)}{\Gamma_{1,p}} \arrow[r, hook, "\psi_{1,p}^0"] & X  \\
   \frac{(s_{2,p}^0)^{-1}(0)}{\Gamma_{2,p}} \arrow[r, "\simeq"] 
    &  \frac{(s_{2,p}^0 \times \mathrm{id}_{\mathbb{R}^{n_2}})^{-1}(0)}{\Gamma_{2,p}}. \arrow[ru, hook, "\psi_{2,p}^0"] &{}
    \end{tikzcd}
    \]
    \item For each pair of the zero points $x_1 \overset{(1), \simeq}{\leftrightarrow} x_2$, there exists an \textit{isomorphism}
    \begin{equation}\label{bnhh}
    \mathcal{C}_{{1,p}, (x_1, 0)}^{0,\mathbb{R}^{n_1}} \xrightarrow{\simeq} \mathcal{C}_{{2,p}, (x_2, 0)}^{0,\mathbb{R}^{n_2}}.
    \end{equation}
    
    \item There exists a commutative diagram as follows
    \[
    \begin{tikzcd}[row sep=large, column sep=large]
   U_{1,p}^0 \times \mathbb{R}^{n_1} \arrow[r, "\simeq"] 
    & U_{2,p}^0 \times \mathbb{R}^{n_2}\\
    U_{1,pq}^0 \times \mathbb{R}^{n_1} 
    \arrow{r}{(2), \simeq} \arrow[u, hook] 
    & U_{2,pq}^0 \times \mathbb{R}^{n_2}. \arrow[u, hook]
    \end{tikzcd}
    \]
    
    \item We have 
\[
\begin{split}
\phi_{1,pq}^{0, \mathbb{R}^{n_1}} = \phi_{2,pq}^{0, \mathbb{R}^{n_2}}
\end{split}
\] 
modulo the diffeomorphism $(2)$ in (v).

\item We have 
\[
\begin{split}
\widehat{\phi}_{1,pq,x}^{0, \mathbb{R}^{n_1}} = \widehat{\phi}_{2,pq,x}^{0, \mathbb{R}^{n_2}}
\end{split}
\] 
for each $x \in s_{p}^{-1}(0)$ modulo the isomorphism $(\ref{bnhh})$ in (iv).

\end{enumerate}

\end{defn}

We list some of the properties of the above-mentioned equivalence by the following lemma whose proof can be found in \cite[Lemma 3.8]{Kim1}.
\begin{lem}\cite[Lemma 3.8]{Kim1}
We have:
\begin{enumerate}[label=(\roman*)]
\item $\widehat{\mathcal{U}}^0 \sim \widehat{\mathcal{U}}$ for an open subatlas $\widehat{\mathcal{U}}^0 < \widehat{\mathcal{U}}.$
\item $\widehat{\mathcal{U}} \sim \widehat{\mathcal{U}} \times V$ for a finite dimensional vector space $V.$
\item $\sim$ is an equivalence relation.
\item $(\widehat{\mathcal{U}} \times \mathbb{R}^{n}) \times \mathbb{R}^{n'} = \widehat{\mathcal{U}} \times \mathbb{R}^{n + n'}$ for all $n, n' \geq 0.$
\end{enumerate}
\end{lem}

We now consider a pair of pre-morphisms from $\mathfrak{X} = \left(X, \left[\widehat{\mathcal{U}}\right]\right)$ to $\mathfrak{X}' = \left(X', \left[\widehat{\mathcal{U}}'\right]\right)$
\begin{equation}\label{f1f2}
\begin{cases}
\overline{F}_1 = \left(\widehat{\mathcal{U}}_1, \widehat{\mathcal{U}'}_1, f_1, \left\{f_{1,p}\right\}, \left\{\widehat{f}_{1,p,x}\right\}\right), \\
\overline{F}_2 = \left(\widehat{\mathcal{U}}_2, \widehat{\mathcal{U}'}_2, f_2, \left\{f_{2,p}\right\}, \left\{\widehat{f}_{2,p,x}\right\}\right)
\end{cases}
\end{equation}
with the properties:
\[
\left[\widehat{\mathcal{U}}_1\right] = \left[\widehat{\mathcal{U}}_2\right] = \left[\widehat{\mathcal{U}}\right]; \ \ \left[\widehat{\mathcal{U}}'_1\right] = \left[\widehat{\mathcal{U}}'_2\right] = \left[\widehat{\mathcal{U}}'\right].
\]

Note that $\overline{F}_1$ and $\overline{F}_2$ can be extended to
\begin{equation}\label{extfi}
\begin{cases}
\overline{F}^{n_1,n'_1}_1 = \left(\widehat{\mathcal{U}}^0_1 \times \mathbb{R}^{n_1}, \widehat{\mathcal{U}'}_1 \times \mathbb{R}^{n'_1}, f_1, \left\{\widetilde{f}_{1,p}\right\}, \left\{\widetilde{\widehat{f}}_{1,p,(x,0)}\right\}\right) \\
\overline{F}^{n_2,n'_2}_2 = \left(\widehat{\mathcal{U}}^0_2 \times \mathbb{R}^{n_2}, \widehat{\mathcal{U}'}_2 \times \mathbb{R}^{n'_2}, f_2, \left\{\widetilde{f}_{2,p}\right\}, \left\{\widetilde{\widehat{f}}_{2,p,(x,0)}\right\}\right),
\end{cases}
\end{equation}
with following properties:
\begin{enumerate}[label=(\roman*)]
\item $\widehat{\mathcal{U}}^0_1 \times \mathbb{R}^{n_1} = \widehat{\mathcal{U}}^0_2 \times \mathbb{R}^{n_2},$
\item $n_i \geq n'_i, \ i= 1,2,$
\item $\widetilde{f}_{i,p} : U^{0}_{i,p} \times \mathbb{R}^{n_i} \rightarrow  U^{'0}_{i,p} \times \mathbb{R}^{n'_i}$ is a \textit{surjective} map that extends $\widetilde{f}_{i,p},$ that is, $\widetilde{f}_{i,p}|_{U_1^0 \times \{0\}} \equiv {f}_{i,p}$ (cf. the assumption $\max\limits_{p \in X}U_p < \infty$ in Definition \ref{kstr}).  In particular, we have $\widetilde{f}_{i,p}|_{s^{-1}_p(0) \times \{0\}} \equiv {f}_{i,p}|_{s_p^{-1}(0)}.$  
\end{enumerate}
We remark that having $\widetilde{f}_{i,p}$ of condition (ii) for each $i$ is always possible by the contractibility of the base $U_i^0.$ The $L_{\infty}[1]$-morphisms $\left\{\widetilde{\widehat{f}}_{i,p,x}\right\}, \ i =1,2$ in (\ref{f1f2}) are given by the following compositions:

We define the $L_\infty$-component of extended morphism by
\begin{equation}
\begin{split}
    \widetilde{\widehat{f}}&_{i, p, x} :=
   \widehat{\pi}_{i,p,(x,0)} \circ \widehat{f}_{i,p,x} \circ \widehat{\pi}_{i, f(p), (f_{i,p}(x),0)}^{-1}\\
& :{} \mathcal{C}^{'\mathbb{R}^{n_i'}}_{f(p),(f_{i,p}(x),0)} \xrightarrow{\widehat{\pi}_{i, f(p), (f_{i,p}(x),0)}^{-1}} \mathcal{C}'_{f(p), f_{i,p}(x)} \xrightarrow{\widehat{f}_{i,p,x}} \mathcal{C}_{p, x} \xrightarrow{\widehat{\pi}_{i, p, (x,0)}} \mathcal{C}^{\mathbb{R}^{n_i}}_{p, (x,0)},
\end{split}
\end{equation}
where $\widehat{\pi}_{i, f(p), (f_{i,p}(x),0)}^{-1}$ and $\widehat{\pi}_{i, p, (x,0)}, \ i =1,2$ are the $L_{\infty}[1]$-quasi-isomorphisms defined in \cite[Example 2.23 \& Lemma 2.24]{Kim1} . 

\begin{defn}[Equivalence of pre-morphisms]\label{defmoreq}
Without loss of generality, one can assume that $n_1 \geq n_2$ in (\ref{urur}). We say that two pre-morphisms are \textit{equivalent} and write
\begin{equation}\label{morrel}
\overline{F}_1 \sim \overline{F}_2
\end{equation}
if there exist extensions $\overline{F}^{n_1,n'_1}_1$ and $\overline{F}^{n_2,n'_2}_2$ as in (\ref{extfi}) such that the following hold:
\begin{enumerate}[label = (\roman*)]
\item $f_1 = f_2,$
\item $\widetilde{f}_{1, p}|_{V_1}  =  \widetilde{f}_{2, p}|_{V_2}$ 
for some open neighborhoods $V_1 \simeq V_2$ of $(s_{1,p}^{0})^{-1}(0) \times \{0\}$ and $(s_{2,p}^{0})^{-1}(0) \times \{0\},$ respectively in ${{U}}^0_1 \times \mathbb{R}^{n_1}$ and ${{U}}^0_2 \times \mathbb{R}^{n_2}$ (precise meaning provided below),
\item $\widetilde{\widehat{f}}_{1,p,x} = \widetilde{\widehat{f}}_{2,p,x}$ up to $L_{\infty}[1]$-homotopy, that is, the following diagram commutes up to $L_{\infty}[1]$-homotopy:
\begin{equation}\label{octa}
\begin{tikzcd}
\mathcal{C}'_{f(p),\widetilde{f}_{1,p}(x)} \arrow{rr}{{\widehat{f}}_{1,p,x}}  & {} & \mathcal{C}_{p, x}  \arrow{d}{\widehat{\pi}_{1,p,(x,0)}}\\
 \mathcal{C}^{'\mathbb{R}^{n_1'}}_{f(p),(\widetilde{f}_{1,p}(x),0)} \cong \mathcal{C}^{'\mathbb{R}^{n_2'}}_{f(p),(f_{2,p}(x),0)} \arrow{u}{\widehat{\pi}^{-1}_{1,f(p),(\widetilde{f}_{1,p}(x),0)}} \arrow{d}[swap]{\widehat{\pi}^{-1}_{2,f(p),(\widetilde{f}_{2,p}(x),0)}} &{}& \mathcal{C}^{\mathbb{R}^{n_1}}_{p, (x,0)} \cong \mathcal{C}^{\mathbb{R}^{n_2}}_{p, (x,0)} \\
\mathcal{C}'_{f(p),\widetilde{f}_{2,p}(x)}  \arrow{rr}{{\widehat{f}}_{2,p,x}}  & {} & \mathcal{C}_{p, x} \arrow{u}[swap]{\widehat{\pi}_{2,p,(x,0)}}
\end{tikzcd}
\end{equation}
for each $x \in (s_{1,p}^{0})^{-1}(0) \times \{0\}.$ 

\end{enumerate}
Here ${(s_{i,p}^{0})^{-1}(0) \times \{0\}}$'s in the conditions (ii) and (iii) are to be understood as the same subset of both ${{U}}^0_1 \times \mathbb{R}^{n_1}$ and ${{U}}^0_2 \times \mathbb{R}^{n_2}$ modulo the identification from (\ref{urur}).
\end{defn}

\begin{lem}\cite[Lemma 3.13]{Kim1}
$\sim$ is an equivalence relation.
\end{lem}

\label{appen:forms}


\begin{thebibliography}{9}

\bibitem[1]{BLX}
Kai Behrend, Hsuan-Yi Liao, Ping Xu, \emph{Derived Differentiable Manifolds}, Preprint, arXiv:2006.01376, 2020.

\bibitem[2]{Brown}
Kenneth Brown, \emph{Abstract Homotopy Theory and Generalized Sheaf Cohomology}, Transactions of the American Mathematical Society, Vol. 186, 419-458, 1973.

\bibitem[3]{Eisenbud}
David Eisenbud, \emph{Commutative algebra with a view toward algebraic geometry,} Graduate Texts in Mathematics 150, Springer, 2004.

\bibitem[4]{FO}
Kenji Fukaya, Kaoru Ono, \emph{Arnold conjecture and Gromov-Witten invariants}, Topology, Volume 38, Issue 5, Pages 933-1048, 1999.

\bibitem[5]{FOOO1}
Kenji Fukaya, Yong-Geun Oh, Hiroshi Ohta, Kaoru Ono, \emph{Kuranishi structures and Virtual fundamental chain}, Springer Monographs in Mathematics, Springer, 2020.

\bibitem[6]{FOOO2}
Kenji Fukaya, Yong-Geun Oh, Hiroshi Ohta, Kaoru Ono, \emph{Lagrangian Intersection Floer Theory : Anomaly and Obstruction Part I, II}, 2009.

\bibitem[7]{Kim1}
Taesu Kim, \emph{Categorical structures of Kuranishi spaces with $L_{\infty}[1]$-algebras,} Preprint, arXiv:arXiv:2607.01371, 2026.

\bibitem[8]{Kim2}
Taesu Kim, \emph{Homotopy models for $L_{\infty}[1]$-algebras in higher degrees,} Preprint, arXiv:2606.28985, 2026.

\bibitem[9]{Kim3}
Taesu Kim, \emph{$L_\infty$-Kuranishi spaces and the moduli space of pseudoholomorphic disks,} Preprint, arXiv:2511.05206, 2025.

\bibitem[10]{Kim4}
Taesu Kim, \emph{Kuranishi chart categories and higher cocycle conditions,} Preprint, arXiv:2607.01732, 2026.

\bibitem[11]{Kim4}
Taesu Kim, \emph{A canonical approach to Morse theory,} in preparation.

\bibitem[12]{KO}
Taesu Kim, Yong-Geun Oh, \emph{Stratifications associated to generic closed two-forms and stratified $L_{\infty}$ spaces,} Preprint, arXiv:2602.24099, 2026.

\bibitem[13]{Quillen}
Daniel, Quillen, \emph{Rational homotopy theory,} Ann. of Math. (2), 90, pp. 205-295, 1969.

\end{thebibliography}
\end{document}